\documentclass[a4paper,12pt]{amsart}
\usepackage[utf8]{inputenc}

\usepackage{url}
\usepackage[margin=1in]{geometry}  
\usepackage{epsfig}
\usepackage{color}
\usepackage{graphicx}              
\usepackage{amsmath}
\usepackage{amscd}               
\usepackage{amsfonts}
\usepackage{amssymb}             
\usepackage{amsthm}                
\usepackage{epsfig}
\usepackage{mathrsfs}
\usepackage{hyperref}
\usepackage{enumerate}
\usepackage{xcolor}

\numberwithin{equation}{section}

\newtheorem{theorem}{Theorem}

\newtheorem*{theorem*}{Theorem}
\newtheorem{lemma}{Lemma}
\newtheorem*{lemma*}{Lemma}

\newtheorem*{proposition*}{Proposition}
\theoremstyle{definition}
\newtheorem{definition}{Definition}

\theoremstyle{remark}

\newcommand{\supp}{\operatorname{supp}}

\newcommand{\Stab}{\operatorname{Stab}}

\newcommand{\nonmax}{\operatorname{non-max}}

\newcommand{\sym}{\operatorname{Sym}}

\newcommand{\tr}{\operatorname{tr}}

\newcommand{\bQ}{\mathbb{Q}}
\newcommand{\bR}{\mathbb{R}}

\newcommand{\zed}{\mathbb{Z}}

\newcommand{\bF}{\mathbf{F}}

\newcommand{\one}{\mathbf{1}}

\newcommand{\GL}{\mathrm{GL}}

\newcommand{\sM}{{\mathscr{M}}}

\newcommand{\sO}{{\mathscr{O}}}

\newcommand{\sS}{{\mathscr{S}}}

\newcommand{\sX}{{\mathscr{X}}}

\allowdisplaybreaks[1]

\title{The local zeta function in enumerating quartic fields}
\author{Robert Hough}
\address[Robert Hough]{Department of Mathematics, Stony Brook University, 100 Nicolls Road, Stony Brook, NY 11794}
\email{robert.hough@stonybrook.edu}
\thanks{Robert Hough is supported by NSF Grants DMS-1712682, ``Probabilistic methods in discrete structures
and applications,'' and DMS-1802336, ``Analysis of Discrete Structures and Applications.''}

\begin{document}

\maketitle

\begin{abstract}
An exact formula is obtained for the Fourier transform of the local condition of maximality modulo primes $p>3$ in the prehomogeneous vector space $2 \otimes \sym^2(\zed_p^3)$ parametrizing quartic fields, thus solving the local `quartic case' in enumerating quartic fields.
\end{abstract}

\section{Introduction}

In \cite{TT13b} Taniguchi and Thorne estimated the number of cubic fields with discriminant of size at most $X$, obtaining an asymptotic with a secondary main term and the current best known power saving error term. The method used the theory of Shintani zeta functions \cite{S72} enumerating the class numbers of binary cubic forms, together with a sieve sifting forms which are non-maximal at finitely many primes.  An important ingredient in the sieve was an evaluation of the Fourier transform of the local indicator function of cubic rings over $\zed_p$ which are non-maximal \cite{TT13a}, which showed that this signal is strongly localized.  The purpose of this article is to obtain a corresponding exact formula for the Fourier transform of the indicator function of non-maximal quartic rings over $\zed_p$, thus solving the `local quartic case' in enumerating number fields by discriminant.  This formula may be used to obtain refined estimates for the equidistribution of the shapes of quartic fields paired with their cubic resolvent rings \cite{H17}.

One reason that a refined error term is known for the count of cubic fields is that cubic fields may be identified with maximal cubic rings via the ring of integers, and cubic rings are parametrized by the prehomogeneous vector space of binary cubic forms, which has a rich algebraic structure.  In the quartic case, there is a parallel structure in which pairs $(Q, R)$ where $Q$ is a quartic ring and $R$ is its cubic resolvent ring are parametrized by the prehomogeneous vector space of pairs of ternary quadratic forms \cite{B04b}. In \cite{SS74} Sato and Shintani developed the theory of zeta functions enumerating class numbers of integral orbits in general prehomogeneous  vector spaces.  In \cite{WY92} Wright and Yukie explain how zeta functions of prehomogeneous vector spaces may be used to enumerate etal\'{e} quadratic, cubic, quartic and quintic number fields, and in \cite{Y93} Yukie developed many of the analytic properties in the quartic case, identifying the order and location of the poles of the quartic zeta function.  In \cite{H17} the author introduced an analogue of the zeta function in the quartic case, twisted by a cusp form evaluated at the lattice shape of each ring, and proved that the resulting object is entire.  This avoided some technical difficulties in treating the quartic zeta function, where the residues are still unknown, and made it possible to obtain equidistribution results for the shape of quartic fields via the zeta function method.  The current work permits obtaining refined estimates in the cuspidal spectrum of the shape of quartic fields, since an exact formula for the coefficients may be used in the dual zeta function which appears there.  

In \cite{B04a}, \cite{B04b}, and \cite{B08} Bhargava obtained parameterizations of integral cubic, quartic and quintic rings, in particular giving local conditions modulo $p^2$ describing the maximality of such rings. Here the exact formula for the Fourier transform of the set of maximal quartic rings over $\zed_p$ is obtained by describing the relevant orbit structure of $\GL_3\times \GL_2(\zed/p^2\zed)$ acting on $2 \otimes \sym^2((\zed/p^2\zed)^3)$ and evaluating the orbital exponential sums. This method of evaluating the Fourier transform  has the advantage that it extends to treat other local conditions besides the condition of maximality.

\subsection{The space $2 \otimes \sym^2(R^3)$} Let $R$ be a commutative  ring with unit. A quartic ring over $R$ is a free rank 4 $R$-module with a ring structure.
Elements of the space of ternary quadratic forms over $R$, $\sym^2(R^3)$ are written
\begin{equation}
 A = \begin{pmatrix} a & \frac{b}{2} & \frac{c}{2}\\ \frac{b}{2} & d & \frac{e}{2}\\ \frac{c}{2} & \frac{e}{2} &f \end{pmatrix},
\end{equation}
corresponding to the form 
\begin{equation}
 (u,v,w) A \begin{pmatrix} u\\v\\w\end{pmatrix} = a u^2 + b uv + cuw + d v^2 + e vw + f w^2.
\end{equation}
Forms are typically indicated by showing only those elements on or above the diagonal.
There is a natural inner product on $\sym^2(R^3)$,
\begin{equation}
 [A,A'] = \tr A^t A' = aa' + \frac{bb'}{2} + \frac{cc'}{2} + dd' + \frac{ee'}{2} + ff'.
\end{equation}
An element $g \in \GL_3(R)$ acts on $\sym^2(R^3)$ by $g \cdot (A) = gAg^t$.  Note that
\begin{align}
 [g \cdot A, (g^{-1})^t \cdot A'] = \tr  g A^t g^t (g^{-1})^t A' g^{-1} = \tr g A^t A' g^{-1} = \tr A^t A' = [A,A'].
\end{align}
The space $V(R) = 2 \otimes \sym^2(R^3)$ consists of pairs of ternary quadratic forms $(A, B)$.  The group $G(R) = \GL_3(R) \times \GL_2(R)$ acts on $V(R)$ with the $\GL_3$ factor acting on each ternary quadratic form separately, and the $\GL_2(R)$ factor forming linear combinations of the two forms, that is $(g_3, g_2 = \begin{pmatrix} a&b\\c &d \end{pmatrix}) \in G(R)$ act by
\begin{equation}
 (g_3, g_2) \cdot (A, B) = (a g_3 A g_3 ^t + b g_3 B g_3^t, c g_3 A g_3^t + d g_3 B g_3^t).
\end{equation}
The inner product on $2 \otimes \sym^2(R^3)$ is
\begin{equation}
 [(A,B), (A',B')] = \tr \begin{pmatrix} [A, A']& [A,B']\\ [B,A']& [B,B']\end{pmatrix} = [A,A'] + [B, B'].
\end{equation}
Thus, with $(g^{-1})^t = ((g_3^{-1})^t, (g_2^{-1})^t)$,
\begin{equation}
 [g \cdot(A,B), (g^{-1})^t \cdot (A', B')] = [(A,B), (A', B')].
\end{equation}
With this action, $2 \otimes \sym^2(R^3)$ is a prehomogeneous vector space, see \cite{SS74}.

In \cite{B04b} the following theorem is proved parametrizing quartic rings.
\begin{theorem*}[\cite{B04b}, Theorem 1]
 There is a canonical bijection between the set of $\GL_3(\zed) \times \GL_2(\zed)$-orbits on the space $(2 \otimes \sym^2(\zed^3))^*$ of pairs of integral ternary quadratic forms and the set of isomorphism classes of pairs $(Q, R)$, where $Q$ is a quartic ring and $R$ is a cubic resolvent ring of $Q$. 
\end{theorem*}
\noindent
It is also verfied that the condition that a quartic ring is maximal may be checked locally by tensoring with $\zed_p$ for each prime $p$, and that the condition that a quartic ring over $\zed_p$ is maximal can be checked modulo $p^2$.  Call a pair of ternary quadratic forms over $\zed_p$ maximal if it is associated to quartic ring over $\zed_p$ which is maximal.

The basic object of study in this article are the \emph{orbital exponential sums}, for $x, \xi \in V(\zed/p^2\zed)$,\footnote{The notation $e(\cdot)$ indicates the additive character on $\bR$, $e(x) = e^{2\pi i x}$.  The notation, for positive integer $m$, $e_m(x) = e\left(\frac{x}{m}\right)$ is also used.}
\begin{equation}
 S_{p^2}(x, \xi) =\sum_{g \in G(\zed/p^2 \zed)} e\left(\frac{1}{p^2}[g \cdot x, \xi]\right);
\end{equation}
the Fourier transform of the indicator functions of maximal and non-maximal forms is a linear combination of such sums. The size and distribution of these sums for varying $\xi$ describe the distribution of the orbit containing the point $x$ within the space $V(\zed/p^2\zed)$. In \cite{TT16} the orbits of the space $V(\zed/p\zed)$ are enumerated and the orbital exponential sums, for $x, \xi \in V(\zed/p\zed)$,
\begin{equation}
 S_p(x,\xi) = \sum_{g \in G(\zed/p\zed)} e\left( \frac{1}{p}[g \cdot x, \xi]\right)
\end{equation}
are calculated. 
The orbits are listed in the table below. Throughout $\ell$ denotes a non-square modulo $p$ and
\begin{equation}
 s(a,b,c,d) = (p-1)^a p^b (p+1)^c (p^2+p+1)^{\frac{d}{2}}.
\end{equation}

\begin{equation}\label{orbit_table}
 \begin{array}{|l|l|l|l|}
  \hline
  \text{Orbit} & \text{Representative} & \text{Orbit size} & \text{Stabilizer size}\\ 
  \hline
  \sO_0 & (0,0) & 1 & s(5,4,2,2)\\
  \hline
  \sO_{D1^2} & (0, w^2) & s(1,0,1,2) & s(4,4,1,0)\\
  \hline
  \sO_{D11} & (0, vw) & s(1,1,2,2)/2 & 2s(4,3,0,0)\\
  \hline
  \sO_{D2} & (0, v^2- \ell w^2) & s(2,1,1,2)/2 & 2s(3,3,1,0)\\
  \hline
  \sO_{Dns} & (0, u^2 - vw) & s(2,2,1,2) & s(3,2,1,0)\\
  \hline
  \sO_{Cs} & (w^2, vw) & s(2,1,2,2) & s(3,3,0,0)\\
  \hline
  \sO_{Cns} & (vw, uw) & s(2,3,1,2) & s(3,1,1,0)\\
  \hline
  \sO_{B11} & (w^2, v^2) & s(2,2,2,2)/2 & 2 s(3,2,0,0)\\
  \hline
  \sO_{B2} & (vw, v^2 + \ell w^2) & s(3,2,1,2)/2 & 2 s(2,2,1,0)\\
  \hline
  \sO_{1^4} & (w^2, uw + v^2) & s(3,2,2,2) & s(2,2,0,0)\\
  \hline
  \sO_{1^31} & (vw, uw+v^2) & s(3,3,2,2) & s(2,1,0,0)\\
  \hline
  \sO_{1^21^2} & (w^2, uv) & s(2,4,2,2)/2 &2 s(3,0,0,0)\\
  \hline
  \sO_{2^2} & (w^2, u^2 - \ell v^2) & s(3,4,1,2)/2 & 2 s(2,0,1,0)\\
  \hline
  \sO_{1^211} & (v^2 -w^2, uw) & s(3,4,2,2)/2 & 2 s(2,0,0,0)\\
  \hline
  \sO_{1^2 2} &(v^2 - \ell w^2, uw) & s(3,4,2,2)/2 & 2s(2,0,0,0)\\
  \hline
  \sO_{1111} & (uw - vw, uv-vw) & s(4,4,2,2)/24 & 24 s(1,0,0,0)\\
  \hline
  \sO_{112} & (vw, u^2 -v^2 - \ell w^2) & s(4,4,2,2)/4 & 4 s(1,0,0,0)\\
  \hline
  \sO_{22} & (vw, u^2 - \ell v^2 - \ell w^2) &s(4,4,2,2)/8 & 8 s(1,0,0,0)\\
  \hline
  \sO_{13} & (uw - v^2, B_3) & s(4,4,2,2)/3 & 3s(1,0,0,0)\\
  \hline
  \sO_{4} & (uw-v^2, B_4) & s(4,4,2,2)/4 & 4s(1,0,0,0)\\
  \hline
 \end{array}
\end{equation}
The items $B_3$ and $B_4$ indicate $B_3 = uv + a_3 v^2 + b_3 vw + c_3 w^2$ and $B_4 = u^2 + a_4 uv + b_4 v^2 + c_4 vw + d_4 w^2$ where $X^3 + a_3 X^2 + b_3 X + c_3$ and $X^4 + a_4 X^3 + b_4 X^2 + c_4 X + d_4$ are irreducible over $\zed/p\zed$.

In this article the orbits modulo $p^2$ relevant to the condition of maximality are identified, along with the relevant exponential sums. An illustrative theorem from the calculation is as follows.
Let $\one_{\nonmax}$ be the indicator function of forms in $V(\zed/p^2\zed)$ which correspond to quartic rings non-maximal at $p$. 
\begin{theorem}\label{norm_theorem}
 For $p>3$ the Fourier transform \[\widehat{\one_{\nonmax}}(\xi) = \sum_{x \in V(\zed/p^2\zed)}\one_{\nonmax}(x) e_{p^2}([ x, \xi])\] is supported on the $\mod p$ orbits $\sO_0$, $\sO_{D1^2}$, $\sO_{D11}$ and $\sO_{D^2}$.  It satisfies
 \begin{align*}
  \left\|\widehat{\one_{\nonmax}}\right\|_1 &= 2p^{29}+2p^{28} + 4p^{27} - 8p^{26}-19p^{25} -2p^{24} + 20p^{23} + 24p^{22} -5p^{21} \\& \qquad-17p^{20} -5p^{19} + 3p^{18} + 2p^{17} -2p^{16}+p^{15} + p^{14},\\
   \left\|\widehat{\one_{\nonmax}}\right\|_2^2 &= p^{46} + 2p^{45} + 2p^{44} -3p^{43} -4p^{42}-p^{41}+3p^{40}+3p^{39}-p^{38}-p^{37},\\
   \left|\supp \widehat{\one_{\nonmax}}\right|& = 2p^{15} + p^{14} -2p^{13}-p^{12}+2p^{10}-p^8.
 \end{align*}

\end{theorem}
An exact formula for the Fourier transform is given following the calculation of the relevant orbital exponential sums in Section \ref{summation_section}.  For applications, an important feature of this calculation is that the Fourier transform has small support, and that the $L^1$ norm saves a power of $p^{1.5}$ compared to a naive application of Cauchy-Schwarz using the already strong bound for the support. In practice, the exact evaluation of the Fourier transform can quantify the intuition that the condition of maximality is captured significantly modulo $p$, with a small refinement modulo $p^2$.

One conceptual reason for the small size of the support is that if $\xi \not \equiv 0 \bmod p$, the orbital exponential sum $S_{p^2}(x, \xi)$ can be shown to vanish unless the $\mod p$ orbit of $x$ contains elements orthogonal to the $p$-adic tangent space of the orbit $\sO_\xi$ at $\xi \bmod p$.  A second key observation is that the $\mod p^2$ orbits above a $\mod p$ orbit $\sO_x$ are in bijection with the orbits of the action of the $\mod p$ stabilizer $G_x$ acting on a quotient by the tangent space.  These key observations are explained in Sections \ref{orb_exp_section} and \ref{quotient_action_section}, prior to the determination of the relevant orbits and evaluation of the corresponding exponential sums.

This paper contains a significant number of explicit matrix calculations, and, while self-contained, is intended to be read along with supporting Mathematica files currently available from the author's homepage \url{https://www.math.stonybrook.edu/~rdhough/}.  The corresponding Mathematica notebook \cite{W18} is indicated in \textbf{bold}. Small cases of the orbit decompositions were obtained in Magma \cite{BCP97}.

\section*{Background and notation}
Throughout, $p$ denotes a fixed odd prime, $p > 3$ and $\ell$ denotes a fixed quadratic non-residue $\mod p$. Given a commutative ring  $R$ with unit, $\sym^2(R^n)$ denotes $n$-ary quadratic forms with coefficients in $R$, identified with $n \times n$ symmetric matrices with $R$ coefficients.  In this work, $R$ will always be one of the rings $\zed, \zed_p, \zed/p\zed = \bF_p$ or $\zed/p^2\zed$. $M_n(R)$ denotes $n \times n$ matrices over $R$. $V(R)$ denotes the space $2 \otimes \sym^2(R^3)$ of pairs $(A, B)$ of ternary quadratic forms over $R$.  The algebraic group $G(R) = \GL_3(R)\times \GL_2(R)$ acts on $V(R)$ with $(g_3, g_2 = \begin{pmatrix} a&b\\c &d \end{pmatrix}) \in G(R)$ acting by
\begin{equation}
 (g_3, g_2) \cdot (A, B) = (a g_3 A g_3 ^t + b g_3 B g_3^t, c g_3 A g_3^t + d g_3 B g_3^t).
\end{equation}

The following lemma recalls the classical orbit description of $\GL_2(\bF_p)\times \GL_1(\bF_p)$ acting on $\sym^2(\bF_p^2)$, in which the $\GL_2$ factor acts by change of variable, and the $\GL_1$ factor acts by multiplication by a scalar.
\begin{lemma}\label{sym_2_action_lemma}
 The action of $\GL_2(\bF_p)\times \GL_1(\bF_p)$ on the space $\sym^2(\bF_p^2)$ of binary quadratic forms in variables $u, v$ has four orbits with representatives $0, u^2, uv,$ and $u^2 - \ell v^2$.
\end{lemma}
\begin{proof}
 If the form is non-zero and reducible, map one linear factor to $u$.  The second linear factor is either the same, which obtains $u^2$, or different, in which case it may be mapped to $v$.  If the form is non-zero and irreducible, then the coefficient on $u^2$ is non-zero, and thus, by completing the square, the $uv$ coefficient may be set to 0.  After scaling the variables and multiplying by a scalar, the form $u^2 - \ell v^2$ is obtained.
\end{proof}
A related smaller group action is also used.
\begin{lemma}\label{sym_2_action_smaller_lemma}
 When the subgroup $\begin{pmatrix} r &\\ a & s\end{pmatrix} \times \GL_1(\bF_p)$ of $\GL_2(\bF_p)\times \GL_1(\bF_p)$ acts  on the space $\sym^2(\bF_p^2)$ of binary quadratic forms in variables $u, v$ there six orbits with representatives $0, u^2, v^2, u(u+v), uv, u^2 - \ell v^2$.
\end{lemma}
\begin{proof}
 Given a non-zero quadratic form $x_1u^2 + x_2uv + x_3 v^2$, if the $x_1$ coordinate is non-zero then any equivalent form will have non-vanishing $u^2$ coordinate since $r \neq 0$ and $v$ may only be scaled by the action.  In this case, make a change of variables in $u$ to eliminate the $x_2$ coordinate.  After scaling and dilating $u$ and $v$, the form is equivalent to one of $u^2, u^2 - v^2$ or $u^2 - \ell v^2$, each of which is inequivalent since the first is a double line, the second is a quadratic reducible with distinct linear factors, and the third is irreducible.  Under a change of coordinates $u^2 - v^2$ is equivalent to $u(u+v)$.
 
 Those non-zero forms with $x_1 = 0$ are inequivalent to those above and contain $v$ as a factor.  After possibly making a change of variable in $u$ and after dilating and scaling the form is equivalent to either $uv$ or $v^2$, which are inequivalent.
\end{proof}

Given a group $G$ acting on a set $\sX$, $G_x = \Stab_G(x)$ denotes the stabilizer in $G$ of the point $x \in \sX$.  $\sO_x = \{g \cdot x: g \in G\}$ denotes the orbit containing $x$.  Given a pair of ternary quadratic forms $x$, $\sO_x \bmod p$ indicates the orbit of $x \in V(\bF_p)$ under the action of $G(\bF_p)$, while $\sO_x \bmod p^2$ indicates the orbit of $x$ in $V(\zed/p^2\zed)$ under the action of $G(\zed/p^2\zed)$.

Given a positive integer $m$, the notations $e(x) = e^{2\pi i x}$ and $e_m(x) = e^{\frac{2\pi i x}{m}}$ are used for complex exponentials.  The Legendre symbol is indicated $\left( \frac{\cdot}{p}\right)$, which satisfies \begin{equation}\sum_{a \bmod p} \left(\frac{a}{p}\right) = 0.\end{equation} The quadratic Gauss sum modulo $p$ is indicated
\begin{equation}
 \tau = \sum_{n \in \bF_p} e_p(n^2).
\end{equation}
This satisfies, for $a \not \equiv 0 \bmod p$, 
\begin{equation}
 \tau \left(\frac{a}{p}\right) = \sum_{n \in \bF_p} e_p(an^2).
\end{equation}
Also, $\tau^2 = \left(\frac{-1}{p}\right) p.$
\section{Conditions of maximality}
A quartic ring over $\zed$ is maximal if it is not a proper subring of another quartic ring.  Maximality is a condition which may be checked locally.  In \cite{B04b} it is shown that the condition of maximality may be checked modulo $p^2$, and congruences identifying maximal quartic rings are given. The conclusions of \cite{B04b}, Section 4 are summarized in the following lemma.
\begin{lemma}\label{maximal_ring_lemma}
 A quartic ring maximal at $p$ comes from one of the following orbits modulo $p$
\begin{equation}
 \sO_{1111}, \sO_{112}, \sO_{13}, \sO_{22}, \sO_4, \sO_{1^211}, \sO_{1^22}, \sO_{1^21^2}, \sO_{2^2}, \sO_{1^31}, \sO_{1^4}.
\end{equation}
  Of the orbits listed, all elements of the $\mod p$ orbits
\begin{equation}
 \sO_{1111}, \sO_{112}, \sO_{13}, \sO_{22}, \sO_4
\end{equation}
are maximal.  An element of $\sO_{1^211}, \sO_{1^22}, \sO_{1^21^2}, \sO_{1^31}, \sO_{1^4}$ is non-maximal if and only if it is $G(\zed/p^2\zed)$ equivalent to a form
\begin{equation}
 \begin{pmatrix} 0 \bmod p^2  & 0 \bmod p &0 \bmod p\\ & * &*\\&&* \end{pmatrix}\begin{pmatrix} 0 \bmod p &*&*\\ &*&*\\&&*\end{pmatrix}.
\end{equation}
A form in $\sO_{2^2}$ is non-maximal if and only if $(A, B)$ is equivalent to a form 
\begin{equation}
 \begin{pmatrix} 0 \bmod p^2  & 0 \bmod p^2 &0 \bmod p\\ & 0 \bmod p^2 & 0\bmod p\\&&* \end{pmatrix}\begin{pmatrix} * &*&*\\ &*&*\\&&*\end{pmatrix}.
\end{equation}
\end{lemma}
Note that the forms $A, B$ have been exchanged compared to \cite{B04b}.
\begin{proof} This is stated in geometric terms, rather than as orbits under the group action in \cite{B04b}.  The two notions coincide except in the case of a single point of intersection of multiplicity 4.
Note that the elements satisfying the geometric condition $T_p^{(1)}$ of \cite{B04b} correspond to the orbit $\sO_{1^4}$, and those satisfying the condition $T_p^{(2)}$ are the union of the two orbits $\sO_{B11} \sqcup \sO_{B2}$ and are non-maximal. Indeed, the representative $(w^2, uw + v^2)$ of the orbit $\sO_{1^4}$ has $w^2 = 0$, $uw + v^2 = 0$ intersecting at a single point $(1,0,0)$ of multiplicity 4, and has $uw + v^2$ irreducible, thus satisfying the geometric condition of $T_p^{(1)}$, while the representatives $(w^2, v^2)$ and $(vw, v^2 + \ell w^2)$ of $\sO_{B11}$ and $\sO_{B2}$ again have a single point $(1,0,0)$ of intersection of multiplicity 4, but are equivalent to forms with both forms reducible, meeting the geometric condition of $T_p^{(2)}$, see \cite{B04b} p. 1353. Since $|\sO_{1^4}| = (p-1)^3p^2(p+1)^2(p^2+p+1)$ and $|\sO_{B11}| +|\sO_{B2}| = (p-1)^2p^3(p+1)(p^2+p+1)$, these orbits make up the full count of elements in $T_p^{(1)}(1^4)$ and $T_p^{(2)}(1^4)$, see \cite{B04b}, Lemma 21.

\end{proof}

The density of maximal elements within the $\mod p$ orbits is also determined by \cite{B04b} and is summarized in the following lemma.
\begin{lemma}\label{mod_p_density_lemma}
 Those $\mod p$ orbits containing forms corresponding to maximal forms modulo $p^2$ have the following density of maximal forms,
  \begin{equation}\label{densities_table}
  \begin{array}{|l|l|}
   \hline
   \text{Orbit} & \text{Density} \\
   \hline
   \sO_{1111} & 1\\
   \hline
   \sO_{112} & 1\\
   \hline
   \sO_{13} & 1\\
   \hline
   \sO_{22} & 1\\
   \hline
   \sO_{4} & 1\\
   \hline
   \sO_{1^211} & \frac{p-1}{p}\\
   \hline
   \sO_{1^22} & \frac{p-1}{p}\\
   \hline 
   \sO_{1^21^2} & \left(\frac{p-1}{p}\right)^2\\
   \hline
   \sO_{2^2} & \frac{p^2-1}{p^2}\\
   \hline
   \sO_{1^31} & \frac{p-1}{p} \\
   \hline  
   \sO_{1^4} & \frac{p-1}{p} \\
   \hline
  \end{array}
 \end{equation}
\end{lemma}

\section{The orbital exponential sums}\label{orb_exp_section}
Let $\one_{\max}$ denote the indicator function on $V(\zed/p^2\zed)$ of those forms corresponding to rings maximal at $p$, $\one_{\nonmax} = 1 - \one_{\max}$.  For $\xi \in V(\zed/p^2\zed)$, the Fourier transforms are
\begin{equation}
 \widehat{\one_{\max}}(\xi) = \sum_{x \in V(\zed/p^2\zed)}\one_{\max}(x) e_{p^2}\left([x,\xi] \right),\qquad \widehat{\one_{\nonmax}} = p^{24}\delta_0 - \widehat{\one_{\max}}.
\end{equation}
Since the condition of maximality is invariant under the action of $G(\zed/p^2\zed)$, the Fourier transform is the linear combination of orbital exponential sums.
\begin{definition}
 Given $x, \xi \in V(\zed/p^2\zed)$, define their \emph{orbital exponential sum}
\begin{align}
 S_{p^2}(x, \xi) = \sum_{g \in G(\zed/p^2\zed)} e_{p^2} \left([g \cdot x, \xi] \right).
\end{align}
Define, also, the \emph{unweighted orbital exponential sum}
\begin{align}
 \Sigma_{p^2}(x, \xi) = \sum_{x' \in \sO_x \bmod p^2} e_{p^2}([x',\xi]) = \frac{S_{p^2}(x, \xi)}{|\Stab_{G(\zed/p^2\zed)}(x)|}.
\end{align}

\end{definition}

When $\xi = p \xi_1$ is an element of $pV(\zed)$, the exponential sum reduces to the $\mod p$ exponential sum
\begin{equation}
 S_{p^2}(x, \xi) = p^{13} S_p(x, \xi_1) = p^{13} \sum_{g \in G(\zed/p\zed)} e_p\left([g\cdot x, \xi_1]\right).
\end{equation}
These are obtained in \cite{TT16}.

As an intermediate step between the orbital exponential sums and the full Fourier transform of $\one_{\max}$ and $\one_{\nonmax}$, exponential sums taken over those maximal elements above a $\mod p$ orbit are calculated.
\begin{definition}
 Given $x \in V(\zed/p\zed)$ and $\xi \in V(\zed/p^2\zed)$, define the \emph{maximal exponential sum}
 \begin{equation}
  \sM(x, \xi) = \sum_{\substack{x' \in V(\zed/p^2\zed)\\ x' \in \sO_x \bmod p \\ \text{maximal}}} e_{p^2}([x',\xi]).
 \end{equation}

\end{definition}
Since the condition of maximality is $G(\zed/p^2\zed)$-invariant, the maximal exponential sums may be expressed as a sum over unweighted orbital exponential sums,
\begin{equation}
 \sM(x, \xi) = \sum_{\substack{\sO_{x'}\bmod p^2: \sO_{x'} \subset \sO_x \bmod p\\ \text{maximal}}} \Sigma_{p^2}(x', \xi).
\end{equation}

The following lemma reduces the number of orbits $\sO_x \bmod p$ which need to be considered in determining the Fourier transforms of $\one_{\max}$ and $\one_{\nonmax}$.
\begin{lemma}\label{full_mod_p_orbit_lemma}
 If $\xi \not \equiv 0 \mod p$, then for all $x$,
 \begin{equation}
  \sum_{\substack{x' \in V(\zed/p^2\zed)\\ x' \in \sO_x \bmod p}} e_{p^2}\left([x',\xi]\right) = 0,
 \end{equation}
and similarly if the roles of $x$ and $\xi$ are reversed. Thus
\begin{equation}
 \sM(x, \xi) = -\sum_{\substack{x' \in V(\zed/p^2\zed)\\ x' \in \sO_x \bmod p \\ \text{non-maximal}}} e_{p^2}([x',\xi]).
\end{equation}
In particular, if $\sO_x \bmod p$ contains only maximal, or only non-maximal elements, then for all $\xi \not \equiv 0 \bmod p$,
\begin{equation}
 \sM(x,\xi) = 0.
\end{equation}

\end{lemma}
\begin{proof}
 For the first claim, write
 \begin{align}
  \sum_{\substack{x' \in V(\zed/p^2\zed)\\ x' \in \sO_x \mod p}} e_{p^2}\left([x',\xi]\right) &= \frac{1}{p^{12}} \sum_{x'' \in V(\zed/p\zed)}\sum_{\substack{x' \in V(\zed/p^2\zed)\\ x' \in \sO_x \mod p}} e_{p^2}\left([x'+px'',\xi]\right) \\
  \notag &= \sum_{\substack{x' \in V(\zed/p^2\zed)\\ x' \in \sO_x \mod p}} e_{p^2}\left([x',\xi]\right)\frac{1}{p^{12}} \sum_{x'' \in V(\zed/p\zed)}e_p\left([x'', \xi]\right) = 0.
 \end{align}
The claim with the roles of $x$ and $\xi$ reversed follows by symmetry.  

The second statement is now obvious.

If $\sO_x \bmod p$ contains only non-maximal elements then the sum defining $\sM(x, \xi)$ is empty, hence 0, which combined with the second claim, proves the third claim.
\end{proof}

As a consequence of Lemmas \ref{mod_p_density_lemma} and \ref{full_mod_p_orbit_lemma}, when $\xi \not \equiv 0 \bmod p$,
\begin{equation}
 \widehat{\one_{\max}}(\xi) = \sum_x \sM(x, \xi)
\end{equation}
with $x$ ranging over representatives for  $\sO_{1^211}, \sO_{1^22}, \sO_{1^21^2}, \sO_{2^2}, \sO_{1^31}, \sO_{1^4} \bmod p$.

\subsection{Annihilator spaces}\label{annihilator_section} An important observation in passing from $\mod p$ orbits and orbital exponential sums to $\mod p^2$ orbits and orbital exponential sums is that the group action may be linearized $p$-adically.
Denote 
\begin{equation}M_3(\zed/p\zed) = \begin{pmatrix} \zed/p\zed & \zed/p\zed & \zed/p\zed\\ \zed/p\zed & \zed/p\zed & \zed/p\zed\\ \zed/p\zed & \zed/p\zed & \zed/p\zed\end{pmatrix}, \qquad M_2(\zed/p\zed) = \begin{pmatrix} \zed/p\zed & \zed/p\zed\\ \zed/p\zed & \zed/p\zed \end{pmatrix}.\end{equation}
\begin{definition} 
 Given a form $x \in V(\zed/p^2\zed)$, define the \emph{annihilator subspace}, or \emph{$p$-adic tangent space} $V_x < V(\zed/p\zed)$ associated to $x$ by
\begin{equation}
 \left(I + p M_3(\zed/p\zed), I + p M_2(\zed/p\zed) \right) \cdot x = x + p V_x.
\end{equation}
\end{definition}
This notion depends only on $x \bmod p$ and not on the residue modulo $p^2$.

The orbits $\sO_{1111}, \sO_{112}, \sO_{13}, \sO_{22}, \sO_4$, which have full dimension and contain only maximal elements, have annihilator subspaces equal to $V(\zed/p\zed)$.  This fact is not needed to calculate the Fourier transform and is not explicitly proved here.  However, a verification is provided in \textbf{annihilator\_spaces.nb}.  The annihilator subspaces of orbit representatives  of the remaining non-zero orbits are listed below, with computation postponed to Appendix \ref{annihilator_appendix}.

\begin{equation}
 \begin{array}{|l|l|l|}
  \hline
  \text{Orbit}  & \text{Representative} & \text{Annihilator subspace}\\
  \hline
  \sO_{D1^2} & (0, w^2) & \begin{pmatrix} 0&0&0\\ &0&0\\&&*\end{pmatrix}\begin{pmatrix} 0&0&*\\&0&*\\&&*\end{pmatrix}\\
  \hline
  \sO_{D11} & (0, vw) & \begin{pmatrix} 0&0&0\\ &0&*\\&&0\end{pmatrix}\begin{pmatrix} 0&*&*\\&*&*\\&&*\end{pmatrix}\\
  \hline
  \sO_{D2} & (0, v^2 - \ell w^2) & \begin{pmatrix} 0&0&0\\ &z&0\\&&- z\ell\end{pmatrix}\begin{pmatrix} 0&*&*\\&*&*\\&&*\end{pmatrix}\\
  \hline
  \sO_{Dns} & (0, u^2 - vw) & \begin{pmatrix} z&0&0\\ &0&-\frac{z}{2}\\&&0\end{pmatrix}\begin{pmatrix} *&*&*\\&*&*\\&&*\end{pmatrix}\\
  \hline
  \sO_{Cs} & (w^2, vw) & \begin{pmatrix} 0&0&z\\ &0&*\\&&*\end{pmatrix}\begin{pmatrix} 0&\frac{z}{2}&*\\&*&*\\&&*\end{pmatrix}\\
  \hline
  \sO_{Cns} & (vw, uw) & \begin{pmatrix} 0&\frac{z}{2}&*\\ &y&*\\&&*\end{pmatrix}\begin{pmatrix} z&\frac{y}{2}&*\\&0&*\\&&*\end{pmatrix}\\
  \hline
  \sO_{B11} & (w^2, v^2) & \begin{pmatrix} 0&*&0\\ &*&*\\&&*\end{pmatrix}\begin{pmatrix} 0&0&*\\&*&*\\&&*\end{pmatrix}\\
  \hline
  \sO_{B2} & (vw, v^2 + \ell w^2) & \begin{pmatrix} 0&\frac{z}{2}&\frac{y}{2}\\ &*&*\\&&*\end{pmatrix}\begin{pmatrix} 0&y& z\ell\\&*&*\\&&*\end{pmatrix}\\
  \hline
  \sO_{1^4} & (w^2, uw + v^2) & \begin{pmatrix} 0&0&y+ \frac{z}{2}\\ &z&*\\&&*\end{pmatrix}\begin{pmatrix} y&*&*\\&*&*\\&&*\end{pmatrix}\\
  \hline
  \sO_{1^31} & (vw, uw + v^2) & \begin{pmatrix} 0&\frac{z}{2}&*\\ &*&*\\&&*\end{pmatrix}\begin{pmatrix} z&*&*\\&*&*\\&&*\end{pmatrix}\\
  \hline
  \sO_{1^21^2} & (w^2, uv) & \begin{pmatrix} 0&*&*\\ &0&*\\&&*\end{pmatrix}\begin{pmatrix} *&*&*\\&*&*\\&&*\end{pmatrix}\\
  \hline
  \sO_{2^2} & (w^2, u^2-\ell v^2) & \begin{pmatrix} z&0&*\\ &- z\ell&*\\&&*\end{pmatrix}\begin{pmatrix} *&*&*\\&*&*\\&&*\end{pmatrix}\\
  \hline
  \sO_{1^211} & (v^2 - w^2, uw) & \begin{pmatrix} 0&*&*\\ &*&*\\&&*\end{pmatrix}\begin{pmatrix} *&*&*\\&*&*\\&&*\end{pmatrix}\\
  \hline
  \sO_{1^22} & (v^2 - \ell w^2, uw) & \begin{pmatrix} 0&*&*\\ &*&*\\&&*\end{pmatrix}\begin{pmatrix} *&*&*\\&*&*\\&&*\end{pmatrix}\\
  \hline
 \end{array}
\end{equation}

One important way in which the annihilator subspaces are used is to show that many of the orbital exponential sums vanish.  
\begin{lemma}\label{exponential_sum_lemma_1}
 For $x, \xi \in V(\zed/p^2\zed)$ the orbital exponential sum may be expressed
 \begin{equation}
  S_{p^2}(x, \xi)=\sum_{g \in G(\zed/p^2\zed)} e_{p^2} \left([g \cdot x, \xi]\right) \one(g \cdot x \in V_\xi^\perp \bmod p).
 \end{equation}

\end{lemma}

\begin{proof}
Since $ (I + p M_3(\zed/p\zed), I+ p M_2(\zed/p\zed))$ is a subgroup of $G(\zed/p^2\zed)$, the orbital exponential sums may be written
\begin{align}
 S_{p^2}(x, \xi)&= \sum_{g \in G(\zed/p^2\zed)} e_{p^2}\left([g\cdot x, \xi]\right)\\
 & \notag =\frac{1}{p^{13}} \sum_{g_1 \in M_3(\zed/p\zed) \times M_2 (\zed/p\zed)} \sum_{g \in G(\zed/p^2\zed)} e_{p^2}\left([(I+p g_1) g \cdot x, \xi] \right)\\
 &\notag =\frac{1}{p^{13}}\sum_{g_1 \in M_3(\zed/p\zed) \times M_2 (\zed/p\zed)} \sum_{g \in G(\zed/p^2\zed)} e_{p^2}\left([g \cdot x, (I + p g_1)^t \cdot \xi] \right)\\
 & \notag =\sum_{g \in G(\zed/p^2\zed)} e_{p^2} \left([g \cdot x, \xi]\right) \one(g \cdot x \in V_\xi^\perp \bmod p).
\end{align}
\end{proof}
By symmetry, the condition $g^t \cdot \xi \in V_x^\perp \bmod p$ may also be imposed.  In fact the conditions $g \cdot x \in V_\xi^\perp \bmod p$ and $g^t \cdot \xi \in V_x^\perp \bmod p$ are equivalent.
\begin{lemma}
 For any $x, \xi \in V(\zed/p^2\zed)$ and $g \in G(\zed/p^2\zed)$, $g \cdot x \not \in V_\xi^\perp$ if and only if $g^t \cdot \xi \in V_x^\perp$.
\end{lemma}
\begin{proof}
 One has
 \begin{align}
  g\cdot x \in V_\xi^\perp & \Leftrightarrow \forall g_1 \in M_3 \times M_2, \; [g\cdot x, (I + p g_1)^t \cdot \xi] \equiv [g \cdot x, \xi] \bmod p^2\\
  &\notag \Leftrightarrow \forall g_1 \in M_3 \times M_2, \; [(I + p g_1)g\cdot x,  \xi] \equiv [x, g^t \cdot \xi] \bmod p^2\\
  &\notag \Leftrightarrow \forall g_1 \in M_3 \times M_2, \;  [g(I + p g^{-1}g_1g)\cdot x,  \xi] \equiv [x, g^t \cdot \xi] \bmod p^2\\
  &\notag \Leftrightarrow \forall g_1 \in M_3 \times M_2, \;  [(I + p g^{-1}g_1g)\cdot x,  g^t\cdot \xi] \equiv [x, g^t \cdot \xi] \bmod p^2\\
  &\notag \Leftrightarrow g^t \cdot \xi \in V_x^\perp.
 \end{align}

\end{proof}
In addition to the notion of an annihilator subspace, a second useful notion in evaluating the exponential sums is that of an action set acting on a pair of forms.

\begin{definition}Given forms $x, \xi \in V(\zed/p\zed)$, define the \emph{action set} 
\begin{equation}
 G_{x, \xi}  = \{g \in G(\zed/p\zed):  g^t \cdot \xi \in V_x^\perp \bmod p\}.
\end{equation}
\end{definition}
Note that, by Lemma \ref{exponential_sum_lemma_1}, the condition $G_{x, \xi} \neq \emptyset$ is a necessary condition for $S_{p^2}(x,\xi) \neq 0$, a fact which restricts consideration to a small number of $\mod p$ orbit pairings.
\begin{lemma}\label{action_set_lemma}
 The following table lists all pairs $(\sO_x, \sO_\xi)$ of non-zero orbits modulo $p$ such that $\sO_x$ contains both maximal and non-maximal elements, and such that $x \in \sO_x$, $\xi \in \sO_\xi$ have $G_{x, \xi} \neq \emptyset$.
 \begin{equation}
  \begin{array}{|l|l|}
   \hline
   \sO_x & \sO_\xi\\
   \hline
   \sO_{1^4} & \sO_{D1^2}, \sO_{D11}, \sO_{1^4}\\
   \hline
   \sO_{1^31} & \sO_{D1^2}, \sO_{Cs}\\
   \hline
   \sO_{1^21^2} & \sO_{D1^2}, \sO_{D11}, \sO_{D2}\\
   \hline
   \sO_{2^2} & \sO_{D11}, \sO_{D2}\\
   \hline
   \sO_{1^211} & \sO_{D1^2}\\
   \hline
   \sO_{1^22} & \sO_{D1^2}\\
   \hline
  \end{array}
 \end{equation}

\end{lemma}

\begin{proof}
 The annihilator subspaces $V_{1^22}, V_{1^211}, V_{2^2}$ and $V_{1^21^2}$ associated to the standard representatives in the corresponding orbits annihilate the second quadratic form, so these orbits may be paired only with frequencies $\xi$ from orbits of type $D$. In the first quadratic form, each of these spaces annihilate all coefficients associated to the $w$ variable, and hence $G_{x, \xi} = \emptyset$ unless $\xi$ is equivalent to a $D$ orbit with a representative that depends only on $u$ and $v$.  This eliminates $\sO_{Dns}$. The spaces $V_{1^22}^\perp, V_{1^211}^\perp$ are one dimensional and contain only a double line, so $\sO_{1^22}$ and $\sO_{1^211}$ may be paired only with $\sO_{D1^2}$. The orbit $\sO_{D1^2}$ may not be paired with $\sO_{2^2}$ since a double line has coefficients on $u^2$, $v^2$ which are related by a square.  Since a quadratic form $\ell u^2 + a uv + v^2$ may be made reducible or irreducible by adjusting $a$, both $\sO_{D11}$ and $\sO_{D2}$ may be paired with $\sO_{2^2}$.  The pairs $(u^2,0)$, $(u^2-v^2,0)$ and $(u^2-\ell v^2,0)$ show that all $D$ orbits besides $\sO_{Dns}$ may be paired with $\sO_{1^21^2}$.
 
 When $x$ is the standard representative for $\sO_{1^31}$, $G_{x, \xi} \neq \emptyset$ implies that $\sO_\xi$  has a representative of the form $(c_1 u^2 + 2c_2 uv, -c_2 u^2)$, with $c_1, c_2$ scalars.  Using the $G$ action, one may have $c_1 \neq 0$ ($\sO_{D1^2}$) or $c_2 \neq 0$ ($\sO_{Cs}$) but not both, since if $c_2 \neq 0$ then a multiple of the second form may be added to the first to force $c_1 = 0$.  Thus these are all of the pairings with $G_{x, \xi} \neq \emptyset$.
 
When $x$ is the standard representative for $\sO_{1^4}$, $G_{x, \xi} \neq \emptyset$ implies that $\sO_{\xi}$ has a representative of the form 
 \begin{equation}
  (c_1 u^2 + c_2 uv + c_3(v^2 - 2uw), 2c_3 u^2)
 \end{equation}
 with $c_1, c_2, c_3$ scalars.
When $c_3 = 0$ one may obtain $\sO_{D1^2}$ and $\sO_{D11}$ by taking either $c_1$ or $c_2$ non-zero.  When $c_3 \neq 0$ one may impose $c_1 = 0$ by adding a multiple of the second form to the first, and $c_2 = 0$ by replacing $w$ with $w + a v$ for some scalar $a$.  The $c_3 \neq 0$ case obtains $\xi \in \sO_{1^4}$. 
\end{proof}
 
The action sets for each pair contained in Lemma \ref{action_set_lemma} are listed in the following table, with the proofs given in Appendix \ref{action_set_appendix}.
{\tiny
\begin{equation}\label{action_set_table}
 \begin{array}{|l|l|l|l|l|}
  \hline
  (\sO_x, \sO_\xi) &x_0 & \xi_0 & G_{x, \xi}^t & |G_{x, \xi}|\\
  \hline
 (\sO_{1^211}, \sO_{D1^2})& \begin{pmatrix} 0 & 0 &0\\ &1 &0\\ &&-1\end{pmatrix}\begin{pmatrix} 0 & 0 & \frac{1}{2}\\ &0 &0\\ &&0\end{pmatrix} & \begin{pmatrix} 1 & 0 &0\\ &0&0\\&&0\end{pmatrix}\begin{pmatrix} 0&0&0\\ &0&0\\&&0\end{pmatrix} & \begin{pmatrix} * & * & *\\ & * & *\\ & * & *\end{pmatrix}\begin{pmatrix} * & *\\ & *\end{pmatrix}& (p-1)^5p^4(p+1)\\
  \hline
 (\sO_{1^22}, \sO_{D1^2}) & \begin{pmatrix} 0 &0 &0\\ &1&0\\&&-\ell\end{pmatrix}\begin{pmatrix}0&0&\frac{1}{2}\\ &0&0\\&&0\end{pmatrix} &\begin{pmatrix} 1 &0 &0\\ &0&0\\&&0\end{pmatrix}\begin{pmatrix} 0&0&0\\&0&0\\&&0\end{pmatrix} & \begin{pmatrix} * & * & *\\ & * & *\\ & * & *\end{pmatrix}\begin{pmatrix} * & *\\ & *\end{pmatrix} & (p-1)^5p^4(p+1)\\
 \hline
 (\sO_{2^2}, \sO_{D11}) & \begin{pmatrix} 0 & 0 &0\\ & 0 & 0\\ & & 1\end{pmatrix}\begin{pmatrix} 1 & 0 &0\\ & - \ell & 0 \\ & & 0\end{pmatrix} & \begin{pmatrix}0 & 1 &0\\ & 0 & 0 \\ && 0\end{pmatrix}\begin{pmatrix} 0& 0 & 0\\ &0 &0\\ && 0\end{pmatrix} & \begin{array}{l} \begin{pmatrix} a & c\lambda \ell  & * \\ c & a\lambda  & * \\ & & *\end{pmatrix} \begin{pmatrix} * & *\\ & *\end{pmatrix}, \\ \lambda \in \bF_p^\times,\\ (a,c) \in \bF_p^2 \setminus \{(0,0)\}\end{array}& (p-1)^5p^3(p+1)\\
 \hline
 (\sO_{2^2}, \sO_{D2}) & \begin{pmatrix} 0 &0 &0\\ &0 &0\\ & &1\end{pmatrix}\begin{pmatrix} 1 & 0  & 0\\ &-\ell & 0\\ &&0\end{pmatrix}& \begin{array}{l}\begin{pmatrix} \ell & \beta &0\\ &1 &0\\ &&0\end{pmatrix} \begin{pmatrix} 0 &0 &0\\&0&0\\&&0\end{pmatrix}\\ \ell u^2 + 2\beta uv + v^2 \; \text{irred.}\end{array} & \begin{array}{l} \begin{pmatrix} a & b & * \\ c&d&*\\ &&*\end{pmatrix} \begin{pmatrix} * & * \\ & *\end{pmatrix},\\  \Bigl [  (\ell u^2 + 2\beta uv + v^2),\\\begin{pmatrix} a&c\\ b&d\end{pmatrix} \cdot(u^2 - \ell v^2)\Bigr ] = 0 \end{array}& (p-1)^4p^3(p+1)^2\\
 \hline
 (\sO_{1^21^2}, \sO_{D1^2}) & \begin{pmatrix} 0 & 0 &0\\&0&0\\&&1\end{pmatrix}\begin{pmatrix}0 & \frac{1}{2} &0\\&0&0\\&&0\end{pmatrix} & \begin{pmatrix} 1 &0 &0\\ &0&0\\&&0\end{pmatrix}\begin{pmatrix} 0&0&0\\&0&0\\&&0\end{pmatrix} & \begin{array}{l} \begin{pmatrix} * & * & *\\ & * & *\\ & * & *\end{pmatrix} \begin{pmatrix} * & *\\ & *\end{pmatrix}\\ \sqcup \begin{pmatrix} & * & *\\ * & * & *\\ & * & *\end{pmatrix} \begin{pmatrix} * & *\\ & *\end{pmatrix} \end{array} &2(p-1)^5p^4(p+1)\\
 \hline
 (\sO_{1^21^2}, \sO_{D11}) & \begin{pmatrix} 0 & 0 & 0\\ &0 &0\\ &&1\end{pmatrix} \begin{pmatrix} 0&\frac{1}{2} &0\\&0&0\\ &&0\end{pmatrix} & \begin{pmatrix} 1 &0&0\\ &-1 &0\\ &&0\end{pmatrix} \begin{pmatrix} 0 &0&0\\&0&0\\&&0\end{pmatrix} & \begin{array}{l}\begin{pmatrix} * & * & * \\ * & * & * \\ && * \end{pmatrix} \begin{pmatrix} * & * \\ & *\end{pmatrix}\\ u+v \mapsto \alpha u +\beta v\\ u-v \mapsto \lambda(\alpha u - \beta v)\end{array} & (p-1)^6p^3\\
 \hline
 (\sO_{1^21^2}, \sO_{D2}) & \begin{pmatrix} 0 &0&0\\ &0 &0\\ &&1\end{pmatrix} \begin{pmatrix} 0 & \frac{1}{2} &0\\ &0&0\\&&0\end{pmatrix} & \begin{pmatrix} -\ell &0&0\\&1 &0\\&&0\end{pmatrix}\begin{pmatrix} 0&0&0\\&0&0\\&&0\end{pmatrix} & \begin{array}{l}\begin{pmatrix} a & c\ell & *\\ c\lambda   & a\lambda & *\\ & & *\end{pmatrix} \begin{pmatrix} * & *\\ & *\end{pmatrix},\\ \lambda \in \bF_p^\times, \\ (a,c) \in \bF_p^2 \setminus \{(0,0)\}\end{array}& (p-1)^5p^3(p+1)\\
 \hline
 (\sO_{1^31}, \sO_{D1^2}) & \begin{pmatrix} 0 &0 &0\\ &0&\frac{1}{2}\\ &&0\end{pmatrix}\begin{pmatrix}0&0 & \frac{1}{2}\\ &1&0\\&&0\end{pmatrix} & \begin{pmatrix} 1 &0 &0\\ &0 &0\\ &&0\end{pmatrix}\begin{pmatrix} 0 & 0 &0\\ &0&0\\&&0\end{pmatrix} & \begin{pmatrix} * & * & *\\ & * & *\\ & * & *\end{pmatrix}\begin{pmatrix} * & * \\ & *\end{pmatrix}& (p-1)^5p^4(p+1)\\
 \hline
 (\sO_{1^31}, \sO_{Cs}) & \begin{pmatrix} 0 & 0 &0 \\ & 0 &\frac{1}{2} \\ &&0\end{pmatrix}\begin{pmatrix} 0 & 0 & \frac{1}{2}\\ & 1 &0\\ &&0\end{pmatrix} & \begin{pmatrix} 0 & -1 &0\\ &0&0\\&&0\end{pmatrix}\begin{pmatrix} 1& 0 &0\\ &0&0\\&&0\end{pmatrix} & \begin{pmatrix} a & * & *\\ & b & *\\ && *\end{pmatrix}\begin{pmatrix} c & *\\ & \frac{bc}{a}\end{pmatrix}&(p-1)^4p^4 \\
 \hline
 (\sO_{1^4}, \sO_{D1^2}) & \begin{pmatrix} 0 & 0 &0\\ &0 &0\\ &&1\end{pmatrix}\begin{pmatrix} 0 & 0 &\frac{1}{2}\\ &1&0\\&&0\end{pmatrix} & \begin{pmatrix} 1 &0 &0\\ &0&0\\&&0\end{pmatrix}\begin{pmatrix} 0&0&0\\&0&0\\&&0\end{pmatrix} & \begin{pmatrix} * & * & *\\ & * & * \\ & * & *\end{pmatrix}\begin{pmatrix} * & *\\ & *\end{pmatrix}& (p-1)^5p^4(p+1)\\
 \hline
 (\sO_{1^4}, \sO_{D11}) & \begin{pmatrix} 0 & 0 &0\\ & 0 &0\\ &&1\end{pmatrix}\begin{pmatrix} 0 & 0 & \frac{1}{2}\\ &1 &0\\ &&0\end{pmatrix} & \begin{pmatrix} 0 & 1 & 0\\ &0&0\\&&0\end{pmatrix}\begin{pmatrix} 0 & 0 &0\\ &0&0\\&&0\end{pmatrix} & \begin{array}{l} \begin{pmatrix} * & * & *\\ & * & *\\ &  & *\end{pmatrix}\begin{pmatrix} * & *\\ & * \end{pmatrix}\\\sqcup \begin{pmatrix} * & * & *\\ * & & *\\ && *\end{pmatrix}\begin{pmatrix} * & *\\ & *\end{pmatrix}\end{array}& 2(p-1)^5p^4\\
 \hline
 (\sO_{1^4}, \sO_{1^4}) & \begin{pmatrix} 0 & 0 & 0\\ &0&0\\&&1\end{pmatrix}\begin{pmatrix}0 & 0&\frac{1}{2}\\ &1&0\\&&0\end{pmatrix} & \begin{pmatrix} 0 & 0 & -1\\ &1&0\\&&0\end{pmatrix}\begin{pmatrix}2 &0&0\\&0&0\\&&0\end{pmatrix} &\begin{array}{l} \begin{pmatrix} a & * & *\\ & b & *\\ &&c\end{pmatrix} \begin{pmatrix} d & *\\ & e \end{pmatrix},\\ b^2=ac, b^2 d = a^2 e\end{array}& (p-1)^3p^4\\ \hline
 \end{array}
\end{equation}
}
The introduction of the action set permits a simplified representation of the orbital exponential sums.
\begin{lemma}\label{orb_exp_sum_lemma}
 Let $x = x_0 + px_1$, $\xi = \xi_0 + p\xi_1$ from the above table with $\|x_0\|_\infty, \|\xi_0\|_\infty \leq \frac{p}{2}$.  The orbital exponential sum $S_{p^2}(x,\xi)$ has evaluation $S_{p^2}(x, \xi) = p^{13}\sS(x,\xi)$ with
 \begin{equation}
  \sS(x, \xi)=\sum_{g \in G_{x, \xi}^t} e_p\left([x_1, g\cdot \xi_0] + [x_0, g\cdot \xi_1] \right).
 \end{equation}

\end{lemma}
\begin{proof}
 By examination of (\ref{action_set_table}), for $g \in G_{x, \xi}^t$, $[x_0, g \cdot \xi_0] = 0$. Hence
 \begin{align}
  S_{p^2}(x,\xi) &= \sum_{g \in G(\zed/p^2\zed), g\bmod p \in G_{x, \xi}^t} e_{p^2}\left(\left[x_0 + px_1 ,g \cdot (\xi_0 + p \xi_1) \right] \right)\\
  \notag &= \sum_{g \in G(\zed/p^2\zed), g\bmod p \in G_{x, \xi}^t} e_{p}\left(\left[x_0  ,g \cdot  \xi_1 \right]+ \left[ x_1, g \cdot \xi_0\right] \right)\\
  \notag &= p^{13} \sum_{g \in G_{x, \xi}^t \bmod p} e_{p}\left(\left[x_0  ,g \cdot  \xi_1 \right]+ \left[ x_1, g \cdot \xi_0\right] \right)\\
  \notag &= p^{13} \sS(x, \xi),
 \end{align}
 which proves the claim.
\end{proof}

Lemmas \ref{full_mod_p_orbit_lemma}, \ref{action_set_lemma} and \ref{orb_exp_sum_lemma} demonstrate that in order to obtain $\widehat{\one_{\max}}$ and $\widehat{\one_{\nonmax}}$ at $\xi \not \equiv 0 \bmod p$ it suffices to consider orbital exponential sums $S(x, \xi)$ in which $\sO_x \bmod p$ contains both maximal and non-maximal elements, and in which $\sO_\xi$ appears paired with $\sO_x$ in Lemma \ref{action_set_lemma}.  In particular, the maximal or non-maximal $\mod p^2$ orbits above $\sO_{1^4}, \sO_{1^31}, \sO_{1^21^2}, \sO_{2^2}, \sO_{1^211}$ and $\sO_{1^22} \bmod p$ need to be determined, and all $\mod p^2$ orbits need to be determined above $\sO_{D1^2}, \sO_{D11}, \sO_{D2}, \sO_{Cs}$ and $\sO_{1^4} \bmod p$.  This classification is performed in the next  Section.

\section{The quotient group action and $\mod p^2$ orbits}\label{quotient_action_section}
The following Section reduces the determination of the $\mod p^2$ orbits above a $\mod p$ orbit with representative $x$ to the determination of the orbits of a quotient space under the $\mod p$ stabilizer of $x$. Then the necessary $\mod p^2$ orbits are classified. 
The stabilizer groups of the following orbit representatives are needed below. 
\begin{equation}\label{stabilizer_table}
 \begin{array}{|l|l|l|l|}
  \hline
  \text{Orbit} & \text{Repr.} & \text{Stabilizer} & \text{Stab. size}\\
  \hline
  \sO_{D1^2} & (0, w^2) & \begin{array}{l} \begin{pmatrix} a &b &\\ c &d &\\ e&f &t\end{pmatrix}\begin{pmatrix} x &\\ y & \frac{1}{t^2} \end{pmatrix}, \\ x,t \in (\zed/p\zed)^\times, e,f,y \in \zed/p\zed, ad - bc \neq 0 \end{array} &  s(4,4,1,0) \\
  \hline
  \sO_{D11} & (0, vw) & \begin{array}{l} \begin{pmatrix} r &&\\ a & s &\\ b && t\end{pmatrix}\begin{pmatrix} x &\\ y & \frac{1}{st} \end{pmatrix} \sqcup \begin{pmatrix} r && \\ a && s\\ b &t &\end{pmatrix} \begin{pmatrix} x &\\y & \frac{1}{st} \end{pmatrix}, \\ x,r,s,t \in (\zed/p\zed)^\times, a,b,y \in \zed/p\zed \end{array} & 2s(4,3,0,0)\\
  \hline
  \sO_{D2} & (0, v^2 - \ell w^2) & \begin{array}{l} \begin{pmatrix} r && \\ a & c & e\\ b & \pm e \ell & \pm c\end{pmatrix} \begin{pmatrix} x & \\ y & \frac{1}{c^2 - e^2\ell }\end{pmatrix}, \\ r, x \in (\zed/p\zed)^\times, a, b, y \in \zed/p\zed, \\ (c,e) \neq (0,0) \in (\zed/p\zed)^2  \end{array} & 2s(3,3,1,0)\\
  \hline
  \sO_{Cs} & (w^2, vw) & \begin{array}{l}\begin{pmatrix} r &&\\ a &s& \\ b&c&t\end{pmatrix} \begin{pmatrix} \frac{1}{t^2} &\\ \frac{-c}{st^2} & \frac{1}{st} \end{pmatrix},\\ r,s,t \in (\zed/p\zed)^\times, a,b,c \in \zed/p\zed\end{array} & s(3,3,0,0)\\
  \hline
  \sO_{1^4} & (w^2, uw + v^2) & \begin{array}{l} \begin{pmatrix} s & &\\ -\frac{2as}{t} & t & \\ b & a & \frac{t^2}{s} \end{pmatrix} \begin{pmatrix} \frac{s^2}{t^4} &\\ -\left( a^2+\frac{bt^2}{s} \right) \frac{s^2}{t^6} & \frac{1}{t^2} \end{pmatrix}, \\ s, t \in (\zed/p\zed)^\times, a, b \in \zed/p\zed \end{array} & s(2,2,0,0)\\
  \hline
  \sO_{1^31} & (vw, uw + v^2) & \begin{array}{l}\begin{pmatrix} s &&\\ a & t &\\ && \frac{t^2}{s} \end{pmatrix} \begin{pmatrix} \frac{s}{t^3} &\\ \frac{-a}{t^3} & \frac{1}{t^2} \end{pmatrix},\\ s,t \in (\zed/p\zed)^\times, a \in \zed/p\zed \end{array} & s(2,1,0,0)\\
  \hline
  \sO_{1^21^2} & (w^2, uv) & \begin{array}{l} \begin{pmatrix} r &&\\ &s&\\ &&t \end{pmatrix}\begin{pmatrix} \frac{1}{t^2} &\\ & \frac{1}{rs} \end{pmatrix} \sqcup \begin{pmatrix} & r &\\ s &&\\ &&t \end{pmatrix} \begin{pmatrix} \frac{1}{t^2} &\\ & \frac{1}{rs} \end{pmatrix},\\ r,s,t \in (\zed/p\zed)^\times \end{array} &2 s(3,0,0,0)\\
  \hline
  \sO_{2^2} & (w^2, u^2 - \ell v^2) & \begin{array}{l}\begin{pmatrix} c & e &\\ \pm e \ell & \pm c &\\ && s\end{pmatrix} \begin{pmatrix} \frac{1}{s^2} &\\ & \frac{1}{c^2 -  e^2\ell}\end{pmatrix},\\ (c,e) \neq (0,0) \in (\zed/p\zed)^2, s \in (\zed/p\zed)^\times \end{array} & 2 s(2, 0, 1, 0)\\
  \hline
  \sO_{1^211} & (v^2-w^2, uw) & \begin{pmatrix} s && \\ & \pm t & \\ && t \end{pmatrix} \begin{pmatrix} \frac{1}{t^2} &\\ & \frac{1}{s t} \end{pmatrix}, s, t \in (\zed/p\zed)^\times & 2 s(2,0,0,0)
  \\
  \hline
  \sO_{1^2 2} & (v^2 - \ell w^2, uw) & \begin{pmatrix} s &&\\ & \pm t &\\ && t\end{pmatrix} \begin{pmatrix} \frac{1}{t^2} &\\ & \frac{1}{s t} \end{pmatrix}, s, t \in (\zed/p\zed)^\times & 2 s(2,0,0,0)\\
  \hline
 \end{array}
\end{equation}
The stabilizers may be checked by verifying that the claimed stabilizer does fix the form and comparing the size  with the orbit sizes in (\ref{orbit_table}), see \textbf{mod\_p\_orbit\_stabilizer.nb} for the verification.

\begin{lemma}
 For each standard orbit representative $x$ of an orbit \begin{equation}\sO_{D1^2}, \sO_{D11}, \sO_{D2}, \sO_{Cs}, \sO_{1^4}, \sO_{1^31}, \sO_{1^21^2}, \sO_{2^2}, \sO_{1^211}, \sO_{1^22}\end{equation} of $V(\zed/p\zed)$ listed in (\ref{stabilizer_table}), the stabilizer subgroup $G_x$ acts on $V(\zed/p\zed)/V_x$.  The orbits of $V(\zed/p^2\zed)$ under $G(\zed/p^2\zed)$ above $\sO_x \bmod p$ are in bijection with the orbits of $V(\zed/p\zed)/V_x$ under $G_x$. 
\end{lemma}

\begin{proof}
For each orbit representative, let $P_x$ be the algebraic group listed in the table below.  Note that $P_x(\zed/p\zed) \supset G_x$.
\begin{equation}\label{stabilizer_table}
 \begin{array}{|l|l|l|l|}
  \hline
  \text{Orbit} & \text{Repr.} & P_x & V_x\\
  \hline
  \sO_{D1^2} & (0, w^2) &  \begin{pmatrix} * &* &\\ * &* &\\ *&* &*\end{pmatrix}\begin{pmatrix} * &\\ * & * \end{pmatrix} & \begin{pmatrix} 0&0&0\\ &0&0\\&&*\end{pmatrix}\begin{pmatrix} 0&0&*\\&0&*\\&&*\end{pmatrix} \\
  \hline
  \sO_{D11} & (0, vw) & \begin{pmatrix} * &&\\ * & * &*\\ * &*& *\end{pmatrix}\begin{pmatrix} * &\\ * & * \end{pmatrix}  & \begin{pmatrix} 0&0&0\\ &0&*\\&&0\end{pmatrix}\begin{pmatrix} 0&*&*\\&*&*\\&&*\end{pmatrix}\\
  \hline
  \sO_{D2} & (0, v^2 - \ell w^2) &  \begin{pmatrix} * && \\ * & * & *\\ * & * & *\end{pmatrix} \begin{pmatrix} * & \\ * & *\end{pmatrix}& \begin{pmatrix} 0&0&0\\ &z&0\\&&- z\ell\end{pmatrix}\begin{pmatrix} 0&*&*\\&*&*\\&&*\end{pmatrix}\\
  \hline
  \sO_{Cs} & (w^2, vw) & \begin{pmatrix} * &&\\ * &*& \\ *&*&*\end{pmatrix} \begin{pmatrix} * &\\ * & * \end{pmatrix} & \begin{pmatrix} 0&0&z\\ &0&*\\&&*\end{pmatrix}\begin{pmatrix} 0&\frac{z}{2}&*\\&*&*\\&&*\end{pmatrix}\\
  \hline
  \sO_{1^4} & (w^2, uw + v^2) &  \begin{pmatrix} * & &\\ * & * & \\ *&*&* \end{pmatrix} \begin{pmatrix} * &\\ * & * \end{pmatrix} & \begin{pmatrix} 0&0&y+ \frac{z}{2}\\ &z&*\\&&*\end{pmatrix}\begin{pmatrix} y&*&*\\&*&*\\&&*\end{pmatrix}\\
  \hline
  \sO_{1^31} & (vw, uw + v^2) & \begin{pmatrix} * &&\\ * & * &\\ && * \end{pmatrix} \begin{pmatrix} * &\\ * & * \end{pmatrix} & \begin{pmatrix} 0&\frac{z}{2}&*\\ &*&*\\&&*\end{pmatrix}\begin{pmatrix} z&*&*\\&*&*\\&&*\end{pmatrix}\\
  \hline
  \sO_{1^21^2} & (w^2, uv) & \begin{pmatrix} * &*&\\ *&*&\\ &&* \end{pmatrix}\begin{pmatrix} * &\\ & * \end{pmatrix}  & \begin{pmatrix} 0&*&*\\ &0&*\\&&*\end{pmatrix}\begin{pmatrix} *&*&*\\&*&*\\&&*\end{pmatrix}\\
  \hline
  \sO_{2^2} & (w^2, u^2 - \ell v^2) & \begin{pmatrix} * & * &\\ * & * &\\ && *\end{pmatrix} \begin{pmatrix} * &\\ & *\end{pmatrix} & \begin{pmatrix} z&0&*\\ &- z\ell&*\\&&*\end{pmatrix}\begin{pmatrix} *&*&*\\&*&*\\&&*\end{pmatrix}\\
  \hline
  \sO_{1^211} & (v^2-w^2, uw) & \begin{pmatrix} * && \\ & * & \\ && * \end{pmatrix} \begin{pmatrix} * &\\ & * \end{pmatrix} & \begin{pmatrix} 0&*&*\\ &*&*\\&&*\end{pmatrix}\begin{pmatrix} *&*&*\\&*&*\\&&*\end{pmatrix}
  \\
  \hline
  \sO_{1^2 2} & (v^2 - \ell w^2, uw) & \begin{pmatrix} * &&\\ & * &\\ && *\end{pmatrix} \begin{pmatrix} * &\\ & * \end{pmatrix} & \begin{pmatrix} 0&*&*\\ &*&*\\&&*\end{pmatrix}\begin{pmatrix} *&*&*\\&*&*\\&&*\end{pmatrix}\\
  \hline
 \end{array}
\end{equation}
By examination, over $\bQ$, $P_x(\bQ) \cdot x \subset V_x(\bQ)$, see \textbf{action\_contained\_tangent\_space.nb}.

Over $\zed/p\zed$, $G_x$ preserves $V_x$ since if $g \in G_x$ and $y \in V_x$ then for some $m \in M_3(\zed/p\zed)\times M_2(\zed/p\zed)$,
\begin{equation}
 py = (I+pm) \cdot x - x
\end{equation}
and thus
\begin{equation}
 pg\cdot y = g \cdot (I +pm) \cdot x - g \cdot x = (I + p g m g^{-1}) (g \cdot x) - (g\cdot x) \in pV_{g \cdot x} = pV_x.
\end{equation}
Thus $G_x$ acts on $V/V_x$.

To obtain the orbits over $\zed/p^2\zed$, given a form $x_0$ act by $G$ to bring it to a form $x + px'$ where $x$ is one of the standard representative forms listed in the table, taken over $\zed$.    Next, by acting by $(I + pM_3(\zed/p\zed), I + pM_2(\zed/p\zed))$, consider this to be an element in $x + p \cdot (V(\zed/p\zed)/V_x)$ where $V_x$ is the annihilator subspace of $x$.  Given $g \in G(\zed/p^2\zed)$ with $g\cdot x \equiv x \bmod p$, then $g \in G_x < P_x \bmod p$, so write $g = g_0 + p g_1$ with $g_0 \in P_x(\zed)$.  Let $m \in M_3(\zed/p\zed)\times M_2(\zed/p\zed)$ be such that $g = g_0(I + pm)$.
%
Calculate
\begin{align}
 g \cdot (x + p x') &\equiv g_0 \cdot ((I+pm) \cdot x + p x') \bmod p V_x\\\notag&\equiv g_0 \cdot (x + p x' + p V_x)\bmod p V_x\\\notag& \equiv g_0 \cdot x + p g_0 \cdot x' \bmod pV_x.
\end{align}
Notice that $g_0 \cdot x \in V_x(\zed)$ and hence $g_0 \cdot x - x \in pV_x(\zed)$.  Hence
\begin{align}
  g \cdot (x + p x')&\equiv x + (g_0 \cdot x - x) + p g_0 \cdot x' \bmod p V_x \\\notag&\equiv x + p g_0 \cdot x' \bmod p V_x.
\end{align}
It follows that $x + px_1 \equiv x+ px_2$ under $G(\zed/p^2 \zed)$ if and only if $x_1 \equiv x_2 \bmod V_x$ under $G_x$.
\end{proof}

The following lemma determines the size of the $\mod p^2$ stabilizers in terms of the size of the stabilizer of the $G_x$ action on $V/V_x$.
\begin{lemma}\label{stabilizer_size_lemma} Let $x$ be a standard representative for one of the orbits
 \begin{equation}\sO_{D1^2}, \sO_{D11}, \sO_{D2}, \sO_{Cs}, \sO_{1^4}, \sO_{1^31}, \sO_{1^21^2}, \sO_{2^2}, \sO_{1^211}, \sO_{1^22}\end{equation} taken over $\zed$ and let $x' = x + px_1$.  Then the following relationship between stabilizers holds
 \begin{equation}
  \left|\Stab_{G(\zed/p^2\zed)}(x')\right| = p^{13 - \dim V_x} \left|\Stab_{G_x}(x_1)\right|.
 \end{equation}
In particular, the orbit sizes are related by
\begin{equation}
 \left| G(\zed/p^2\zed) \cdot x'\right| = p^{\dim V_x} \cdot \left|G(\zed/p\zed) \cdot x \bmod p\right| \cdot \left|G_x \cdot x_1 \bmod V_x\right|,
\end{equation}
and the density of the orbit $G(\zed/p^2\zed) \cdot x'$ within $\sO_x \bmod p$ is equal to the density of $G_x \cdot x_1 \bmod V_x$ in $V/V_x$,
\begin{equation}
 \frac{\left| G(\zed/p^2\zed) \cdot x'\right|}{p^{12} \left|G(\zed/p\zed)\cdot x \right|} = \frac{\left|G_x \cdot x_1 \bmod V_x\right|}{|V/V_x|}.
\end{equation}

\end{lemma}
\begin{proof}
 Let $g \in G(\zed/p^2\zed)$ stabilize $x'$, so $g \cdot x \equiv x \bmod p$ and write $g = g_0 + pg_1$ with $g_0 \in P_x(\zed)$.  Factor, on the left this time, $g = (I + pm) \cdot g_0$, and write
 \begin{align}
  (I + pm) g_0 \cdot (x + px_1)  &= (I + pm) (g_0 \cdot x + p g_0 \cdot x_1)\\
  \notag &\equiv (I + pm) \cdot (x + (g_0 \cdot x - x) + p g_0 \cdot x_1) \bmod p^2\\
  \notag &\equiv ((I+pm) \cdot x + (g_0 \cdot x - x) + p g_0 \cdot x_1) \bmod p^2.
 \end{align}
Since $(I+pm) \cdot x -x\in p V_x$ and $g_0 \cdot x  - x \in p V_x$, it follows that $g_0 \cdot x_1 \equiv x_1 \bmod V_x$.  As $m$ varies in the 13 dimensional vector space $M_3(\zed/p\zed) \times M_2(\zed/p\zed)$, the map $(I+pm) \cdot x - x \in pV_x$ is a linear map onto, with kernel of dimension $13 - \dim V_x$.  Hence after choosing $g_0$ in $\left|\Stab_{G_x}(x_1)\right|$ ways, there are $p^{13 - \dim V_x}$ ways to choose $m$.

To obtain the relations among the orbit sizes, calculate, using the Orbit-Stabilizer Theorem
\begin{align}
 \left|G(\zed/p^2\zed) \cdot x' \right| &= \frac{|G(\zed/p^2\zed)|}{\left|\Stab_{G(\zed/p^2\zed)}(x')\right|}\\
 \notag &= \frac{|G(\zed/p^2\zed)|p^{\dim V_x - 13}}{\left|\Stab_{G_x}(x_1)\right|}\\
 \notag &= p^{\dim V_x} \frac{|G(\zed/p\zed)|}{\left|\Stab_{G(\zed/p\zed)}(x)\right|} \frac{|G_x|}{\left|\Stab_{G_x}(x_1)\right|}\\
 \notag &= p^{\dim V_x} \cdot \left|G(\zed/p\zed) \cdot x \bmod p\right| \cdot \left|G_x \cdot x_1 \bmod V_x\right|.
\end{align}

The relation among densities is obtained by rearranging the relation among orbit sizes.
\end{proof}

\subsection{The $\mod p^2$ orbits relevant to maximality}
First the $\mod p^2$ orbits are considered above $\mod p$ orbits containing both maximal and non-maximal orbits.  Recall the density of maximal elements in orbits of this type,
 \begin{equation}
  \begin{array}{|l|l|l|}
   \hline
   \text{Orbit} & \text{Density} & \dim(V/V_x)\\
   \hline
   \sO_{1^4} & \frac{p-1}{p} & 3\\
   \hline
   \sO_{1^31} & \frac{p-1}{p} &2\\
   \hline
   \sO_{1^21^2} & \left(\frac{p-1}{p}\right)^2& 2\\
   \hline
   \sO_{2^2} & \frac{p^2-1}{p^2}& 2\\
   \hline
   \sO_{1^211} & \frac{p-1}{p}&1\\
   \hline
   \sO_{1^22} & \frac{p-1}{p}&1\\
   \hline
  \end{array}
 \end{equation}
  In the case of $\sO_{2^2}, \sO_{1^211}$ and $\sO_{1^22}$ the non-maximal orbits are most easily classified.
 \begin{lemma}\label{22_1211_122_lemma}
  The non-maximal elements above the orbits $\sO_{2^2}, \sO_{1^211}$ and $\sO_{1^22}$ consist in a single $G(\zed/p^2\zed)$ orbit corresponding to the singleton 0 orbit in $V/V_x$.
 \end{lemma}
  \begin{proof}
   In each case, the zero orbit has an element which satisfies the congruence conditions of Lemma \ref{maximal_ring_lemma}, so these orbits all consist of non-maximal elements.  By Lemma \ref{stabilizer_size_lemma}, these orbits already make up the full density of non-maximal elements (see the table above).
  \end{proof}

\subsubsection{Case of $\sO_{1^4}$}
The orthogonal complement to $V_x = \begin{pmatrix} 0 &0&y +\frac{z}{2}\\&z&*\\&&*\end{pmatrix}\begin{pmatrix}y&*&*\\&*&*\\&&*\end{pmatrix}$ is spanned by $\{v_1, v_2, v_3\}$,
\begin{equation}
  \begin{pmatrix} 1 &0&0\\&0&0\\&&0\end{pmatrix}\begin{pmatrix}0&0&0\\&0&0\\&&0\end{pmatrix}, \begin{pmatrix} 0&1&0\\&0&0\\&&0\end{pmatrix}\begin{pmatrix}0&0&0\\&0&0\\&&0\end{pmatrix},  \begin{pmatrix}0&0&-1\\&1&0\\&&0\end{pmatrix}\begin{pmatrix}2 &0&0\\&0&0\\&&0\end{pmatrix}.
\end{equation}
The stabilizer $G_x$ is given by
\begin{equation}
 G_x = \begin{pmatrix} s &&\\\frac{-2as}{t} & t&\\ b & a&\frac{t^2}{s}\end{pmatrix} \begin{pmatrix} \frac{s^2}{t^4} &\\ -\left(a^2+\frac{bt^2}{s} \right)\frac{s^2}{t^6} & \frac{1}{t^2}\end{pmatrix}.
\end{equation}

\begin{lemma}\label{14_lemma}
The orbits and stabilizers of $V/V_x$ under the $G_x$ action are summarized below.
\begin{equation}
 \begin{array}{|l|l|l|}
  \hline
  \text{Orbit representative}& \text{Orbit size} & \text{Stabilizer}\\
  \hline
  \begin{array}{l}\begin{pmatrix} \epsilon & 0 &0 \\&0 &0\\&&0\end{pmatrix}\begin{pmatrix}0&0&0\\&0&0\\&&0\end{pmatrix},\\ \epsilon\in \bF_p^\times / \{x^4: x \in \bF_p^\times\}\end{array} & \frac{(p-1)p^2}{\#\{\epsilon^4 = 1 \}} & \begin{array}{l}\begin{pmatrix} s &0&0\\0& s \varepsilon  &0\\ 0&0&s \varepsilon^2 \end{pmatrix} \begin{pmatrix} \frac{1}{s^2} &0\\0&\frac{1}{s^2\varepsilon^2 }\end{pmatrix},\\ \varepsilon \in \bF_p^\times / \{x^4: x \in \bF_p^\times\} \end{array} \\
  \hline
  \begin{array}{l}\begin{pmatrix} 0& \epsilon &0 \\&0 &0\\&&0\end{pmatrix}\begin{pmatrix}0&0&0\\&0&0\\&&0\end{pmatrix},\\ \epsilon \in \bF_p^\times / \{x^3: x \in \bF_p^\times\} \end{array} & \frac{(p-1)p}{\#\{\epsilon^3 = 1 \}} & \begin{array}{l}\begin{pmatrix} s &0&0\\0& s\varepsilon &0\\ b&0& s\varepsilon^2\end{pmatrix} \begin{pmatrix} \frac{1}{ s^2\varepsilon} &0\\-\frac{b}{ s^3\varepsilon}&\frac{1}{ s^2\varepsilon^2}\end{pmatrix},\\ \varepsilon \in \bF_p^\times / \{x^3: x \in \bF_p^\times\} \end{array} \\
  \hline
  \begin{pmatrix} 0& 0 &-1 \\&1 &0\\&&0\end{pmatrix}\begin{pmatrix}2&0&0\\&0&0\\&&0\end{pmatrix} & \frac{p-1}{2} & \begin{array}{l}\begin{pmatrix} s &0&0\\-2a\varepsilon& s\varepsilon &0\\ b&a& s\end{pmatrix} \begin{pmatrix} \frac{1}{s^2} &0\\-\frac{a^2}{s^4}- \frac{b}{s^3}&\frac{1}{ s^2}\end{pmatrix},\\ \varepsilon = \pm 1 \end{array} \\
  \hline
  \begin{pmatrix} 0& 0 &-\ell \\&\ell &0\\&&0\end{pmatrix}\begin{pmatrix}2\ell&0&0\\&0&0\\&&0\end{pmatrix} & \frac{p-1}{2} & \begin{array}{l}\begin{pmatrix} s &0&0\\-2a\varepsilon& s\varepsilon &0\\ b&a& s\end{pmatrix} \begin{pmatrix} \frac{1}{s^2} &0\\-\frac{a^2}{s^4}- \frac{b}{s^3}&\frac{1}{ s^2}\end{pmatrix},\\ \varepsilon = \pm 1 \end{array} \\
  \hline       
  \begin{pmatrix}0 & 0 &0 \\&0 &0\\&&0\end{pmatrix}\begin{pmatrix}0&0&0\\&0&0\\&&0\end{pmatrix} & 1 & G_x\\
  \hline
 \end{array}
\end{equation}
The orbits of the first type correspond to the maximal orbits under the $G(\zed/p^2\zed)$ action.
\end{lemma}

\begin{proof}
See \textbf{O14\_stabilizers.nb}.

 To check the stabilizer of the orbits of the first type, note that $a$ is forced to be 0 by considering the $uv$ coefficient, which then forces $b$ to be 0 by considering the $uw$ coefficient.  The $u^2$ coefficient becomes $\frac{s^4}{t^4}\epsilon$, which  obtains the relation $\frac{s^4}{t^4}=1$. The inequivalence of orbits having differing $\epsilon$ is checked similarly.  

To calculate the stabilizer of the second orbit, note that $a$ is forced to be 0, since it is not possible to produce a $u^2$ term which would be required to avoid a $v_3$ component.  Evidently $b$ is arbitrary, and the relation $\frac{s^3}{t^3}=1$ is obtained by considering the $uv$ coefficient.  This also proves the inequivalence of orbits having differing $\epsilon$. For the final two orbits, note that the span of $v_3$ is preserved by the action.  The square and non-square leading coefficients are inequivalent. Since the remaining orbits all have differing sizes, each of the orbits listed is inequivalent from the others.

Those orbits not of the first type make up a $\frac{1}{p}$ fraction of all of the elements of $V/V_x$.  Thus at least one, and thus all orbits of the first type are maximal, and these exhaust the maximal orbits by density. 
\end{proof}

\subsubsection{Case of $\sO_{1^31}$}
The orthogonal complement to $V_x = \begin{pmatrix} 0& \frac{z}{2} &*\\&*&*\\&&*\end{pmatrix} \begin{pmatrix} z &*&*\\&*&*\\&&*\end{pmatrix}$ is spanned by
\begin{equation}
 v_1 = \begin{pmatrix} 1 &0&0\\&0&0\\&&0\end{pmatrix}\begin{pmatrix}0&0&0\\&0&0\\&&0\end{pmatrix}, v_2 = \begin{pmatrix} 0 &1&0\\&0&0\\&&0\end{pmatrix}\begin{pmatrix}-1&0&0\\&0&0\\&&0\end{pmatrix}.
\end{equation}
The acting group is
\begin{equation}
 G_x = \begin{pmatrix} s &&\\a &t&\\ && \frac{t^2}{s}\end{pmatrix}\begin{pmatrix}\frac{s}{t^3} &\\ -\frac{a}{t^3} & \frac{1}{t^2}\end{pmatrix}.
\end{equation}
\begin{lemma}\label{131_lemma}
 The maximal orbits above $\sO_{1^31}$ in the $G(\zed/p^2\zed)$ correspond to the orbits of 
 \begin{equation}
  \begin{pmatrix} \epsilon &0&0\\&0&0\\&&0\end{pmatrix}\begin{pmatrix}0&0&0\\&0&0\\&&0\end{pmatrix}, \qquad \epsilon \in \bF_p^\times / \{x^3: x \in \bF_p^\times\}
 \end{equation}
in the action on $V/V_x$.  In the action on $V/V_x$, each of these orbits has size $\frac{(p-1)p}{\#\{\epsilon^3 = 1\}}$.
\end{lemma}

\begin{proof}
When $G_x$ acts it maps $v_1$ to
\begin{equation}
 \begin{pmatrix} \frac{s^3}{t^3} & \frac{as^2}{t^3} & *\\ &*&*\\&&*\end{pmatrix} \begin{pmatrix} -\frac{as^2}{t^3} &* &*\\&*&*\\&&*\end{pmatrix}
\end{equation}
and thus the stabilizer of $\lambda v_1$ satisfies $s^3 = t^3$ and $a = 0$, see \textbf{O131\_stabilizers.nb}.  Hence there are $\#\{\epsilon^3 =1\}$ orbits $\epsilon v_1$ with $\epsilon \in \bF_p^\times / \{x^3: x \in \bF_p^\times\}$, each of size $\frac{(p-1)p}{\#\{\epsilon^3 = 1\}}$.  These orbits together are maximal or non-maximal, and hence by density they are maximal and exhaust the maximal orbits.  
\end{proof}

\subsubsection{Case of $\sO_{1^21^2}$}
The annihilated subspace is $V_x =\begin{pmatrix} 0& * &*\\&0&*\\&&*\end{pmatrix} \begin{pmatrix} * &*&*\\&*&*\\&&*\end{pmatrix}$. The acting group is
\begin{equation}
 G_x = \begin{pmatrix} r &&\\ &s &\\ &&t\end{pmatrix}\begin{pmatrix} \frac{1}{t^2} &\\ & \frac{1}{rs}\end{pmatrix}  \sqcup \begin{pmatrix}  &r&\\ s& &\\ &&t\end{pmatrix}\begin{pmatrix} \frac{1}{t^2} &\\ & \frac{1}{rs}\end{pmatrix} .
\end{equation}
\begin{lemma}\label{1212_lemma}
 The orbits of the $G_x$ action on $V/V_x$ in the case $x \in \sO_{1^21^2}$ are given in the following table
 \begin{equation}
   \begin{array}{|l|l|l|l|}\hline
  \text{Orbit representative} & \text{Orbit size} & \text{Stabilizer size} & \text{Type}\\
  \hline 
  \begin{pmatrix} 0 &0 &0\\ &1 &0 \\ &&0\end{pmatrix} \begin{pmatrix} 0&0&0\\&0&0\\&&0\end{pmatrix} & p-1 & 2(p-1)^2 & \text{Non-maximal}\\
  \hline
  \begin{pmatrix} 0&0&0\\&\ell &0 \\&&0\end{pmatrix}\begin{pmatrix} 0&0&0\\&0&0\\&&0\end{pmatrix} & p-1 & 2(p-1)^2 & \text{Non-maximal}\\
  \hline
  \begin{pmatrix} 1& 0 &0\\&1&0\\&&0\end{pmatrix} \begin{pmatrix} 0 &0&0\\&0&0\\&&0\end{pmatrix} & \frac{(p-1)^2}{4} & 8 (p-1) & \text{Maximal}\\ 
  \hline
  \begin{pmatrix} \ell& 0 &0\\&\ell&0\\&&0\end{pmatrix} \begin{pmatrix} 0 &0&0\\&0&0\\&&0\end{pmatrix} & \frac{(p-1)^2}{4} & 8(p-1)& \text{Maximal}\\ 
  \hline
  \begin{pmatrix} 1& 0 &0\\&\ell&0\\&&0\end{pmatrix} \begin{pmatrix} 0 &0&0\\&0&0\\&&0\end{pmatrix} & \frac{(p-1)^2}{2} & 4(p-1)& \text{Maximal}\\
  \hline
 \end{array}
 \end{equation}

\end{lemma}

\begin{proof}
 According to the congruence condition, the orbits with representatives
\begin{equation}
 \begin{pmatrix} 0& 0 &0\\&1&0\\&&0\end{pmatrix} \begin{pmatrix} 0 &0&0\\&0&0\\&&0\end{pmatrix},  \begin{pmatrix} 0& 0 &0\\&\ell&0\\&&0\end{pmatrix} \begin{pmatrix} 0 &0&0\\&0&0\\&&0\end{pmatrix}
\end{equation}
are both non-maximal.  These orbits are inequivalent, since the non-zero elements scale by a square under the action.  Each has stabilizer
\begin{equation} 
\begin{pmatrix}
s &&\\ &t &\\ && \pm t 
\end{pmatrix}\begin{pmatrix} \frac{1}{t^2} &\\ & \frac{1}{st}\end{pmatrix}.
\end{equation}
Hence both orbits have size $p-1$.  Together with the zero orbit, these exhaust the non-maximal orbits, by density.

The stabilizer of the orbits with representatives 
\begin{equation}
  \begin{pmatrix} 1& 0 &0\\&1&0\\&&0\end{pmatrix} \begin{pmatrix} 0 &0&0\\&0&0\\&&0\end{pmatrix}, \qquad \begin{pmatrix} \ell& 0 &0\\&\ell&0\\&&0\end{pmatrix} \begin{pmatrix} 0 &0&0\\&0&0\\&&0\end{pmatrix}
\end{equation}
is
\begin{equation}
 \begin{pmatrix}
t &&\\ & t\epsilon_1 &\\ &&  t\epsilon_2 
\end{pmatrix}\begin{pmatrix} \frac{1}{t^2} &\\ & \frac{1}{ t^2\epsilon_1}\end{pmatrix} \sqcup  \begin{pmatrix}
 &t&\\  t\epsilon_1& &\\ &&  t\epsilon_2 
\end{pmatrix}\begin{pmatrix} \frac{1}{t^2} &\\ & \frac{1}{ t^2\epsilon_1}\end{pmatrix}, \qquad \epsilon_1, \epsilon_2 \in \{\pm 1\} .
\end{equation}
In the case of $\begin{pmatrix} 1& 0 &0\\&\ell&0\\&&0\end{pmatrix} \begin{pmatrix} 0 &0&0\\&0&0\\&&0\end{pmatrix}$ the stabilizer is
\begin{equation}
 \begin{pmatrix}
t &&\\ & t\epsilon_1 &\\ &&  t\epsilon_2 
\end{pmatrix}\begin{pmatrix} \frac{1}{t^2} &\\ & \frac{1}{ t^2\epsilon_1}\end{pmatrix}, \qquad \epsilon_1, \epsilon_2 \in \{\pm 1\}
\end{equation}
since the square and non-square coordinate cannot be exchanged.  These exhaust the maximal orbits, by density.
\end{proof}

\subsection{The $\mod p^2$ orbits which appear as frequencies}
The remaining orbits of $V/V_x$ are classified projectively, which suffices in calculating the Fourier transform, see Lemma \ref{dilation_lemma} below. Equivalently, the acting group is taken to be $G_x \times \GL_1$, in which $\GL_1$ acts by multiplication by a scalar.
\subsubsection{Case of $\sO_{Cs}$}
The acting group is 
\begin{equation}
 G_x \times \GL_1 = \begin{pmatrix}r &&\\ a &s &\\ b & c&t \end{pmatrix} \begin{pmatrix} \frac{\lambda}{t^2} &\\ \frac{-c\lambda}{st^2} & \frac{\lambda}{st}\end{pmatrix}.
\end{equation}
The action is  on $V/V_x$, 
\begin{equation}
V_x = \begin{pmatrix}
       0&0&z\\ &0&*\\&&*
      \end{pmatrix}\begin{pmatrix} 0 & \frac{z}{2} &*\\&*&*\\&&*\end{pmatrix}.
\end{equation}

\begin{lemma}
When $x$ is the standard representative of $\sO_{Cs} \bmod p$, there are 11 orbits of $ G_x \times \GL_1$ acting on $V/V_x$, listed in the following table, together with the orbit size, and the relations on $ G_x \times \GL_1$ describing each orbit stabilizer.
 \begin{equation}
  \begin{array}{|l|l|l|}
   \hline
   \text{Orbit representative} & \text{Orbit size} & \text{Stabilizer relations}\\
   \hline
   \begin{pmatrix} 1&0&0\\&0&0\\&&0\end{pmatrix} \begin{pmatrix} 0&0&0\\&0&0\\&&0\end{pmatrix} & (p-1)p^3 & a=b=c=0, r^2 \lambda = t^2 \\
   \hline
   \begin{pmatrix} 0&0&0\\&1&0\\&&0\end{pmatrix} \begin{pmatrix} 0&0&0\\&0&0\\&&0\end{pmatrix} & p-1 & s^2 \lambda = t^2 \\
   \hline
   \begin{pmatrix} 0&0&1\\&1&0\\&&0\end{pmatrix} \begin{pmatrix} 0&0&0\\&0&0\\&&0\end{pmatrix} & (p-1)^2 &  s^2 \lambda = t^2,  r \lambda = t \\
   \hline
   \begin{pmatrix} 0&0&0\\&1&0\\&&0\end{pmatrix} \begin{pmatrix} 1&0&0\\&0&0\\&&0\end{pmatrix} & (p-1)^2p & s^2 \lambda  = t^2, r^2 \lambda = st, a=0 \\
   \hline
   \begin{pmatrix} 1&\frac{1}{2}&0\\&0&0\\&&0\end{pmatrix} \begin{pmatrix} 0&0&0\\&0&0\\&&0\end{pmatrix} & \frac{ (p-1)^2p^3}{2} & \begin{array}{l} b=c=0, a(a+s)=0\\ \text{if $a=0$: }s=r, r^2\lambda  = t^2 \\ \text{if $a+s = 0$: } a=r, r^2  \lambda = t^2 \end{array}  \\
   \hline
   \begin{pmatrix} 0&1&0\\&0&0\\&&0\end{pmatrix} \begin{pmatrix} 0&0&0\\&0&0\\&&0\end{pmatrix} & (p-1)p^2 & a = 0, c=0,  rs \lambda = t^2 \\
   \hline
   \begin{pmatrix} 0&1&0\\&0&0\\&&0\end{pmatrix} \begin{pmatrix} 1&0&0\\&0&0\\&&0\end{pmatrix} & (p-1)^2p^2 & a=0, c=0,  rs\lambda = t^2,  r^2 \lambda = st \\
   \hline
   \begin{pmatrix} 1&0&0\\&-\ell&0\\&&0\end{pmatrix} \begin{pmatrix} 0&0&0\\&0&0\\&&0\end{pmatrix} & \frac{(p-1)^2 p^3}{2} & a=b=c=0, s = \pm r, r^2\lambda  = t^2 \\
   \hline
   \begin{pmatrix} 0&0&1\\&0&0\\&&0\end{pmatrix} \begin{pmatrix} 0&0&0\\&0&0\\&&0\end{pmatrix} & p-1 & r\lambda   = t \\
   \hline
   \begin{pmatrix} 0&0&0\\&0&0\\&&0\end{pmatrix} \begin{pmatrix} 1&0&0\\&0&0\\&&0\end{pmatrix} & (p-1)p & a=0, r^2\lambda  = st \\
   \hline
   \begin{pmatrix} 0&0&0\\&0&0\\&&0\end{pmatrix} \begin{pmatrix} 0&0&0\\&0&0\\&&0\end{pmatrix} & 1 & - \\
   \hline
  \end{array}
 \end{equation}

\end{lemma}
\begin{proof}
See \textbf{OCs\_stabilizers.nb}.

 Identify the upper left two by two elements of the first quadratic form with $\sym^2(\bF_p^2)$.  The group $ \begin{pmatrix} r &\\ a &s\end{pmatrix}\times \GL_1$ acts on $\sym^2(\bF_p^2)$ with six orbits, having representatives $0$, $u^2$, $v^2$, $u(u+v)$, $uv$, $u^2 - \ell v^2$, see Lemma \ref{sym_2_action_smaller_lemma}. The orbits on $V/V_x$ can be augmented with either a $uw$ term in the first quadratic form, or a $u^2$ term in the second.  
 
 The 0 form can be augmented by either term, but not both, since if the $u^2$ coefficient is non-zero in the second quadratic form, making the transformation $u \mapsto u + a v$ can reduce the pair of forms to one equivalent to $(0, u^2) \bmod V_x$.  
 
 The $u^2$ orbit cannot be augmented with either term, since subtracting a multiple of the first form from the second eliminates a $u^2$ term in the second quadratic form, and replacing $u \mapsto u + bw$ can eliminate the $uw$ term in the first.  
 
 The $v^2$ orbit can be augmented with either term, but not both by the same argument as in the case of 0. 
 
 $u(u+v)$ cannot be augmented for the same reason that $u^2$ cannot. 
 
 $uv$ can be augmented with the $u^2$ term in the second form, but not a $uw$ term, which could be eliminated by replacing $v \mapsto v + cw$. 
 
 $u^2 - \ell v^2$ cannot be augmented by the same argument as in the $u^2$ case.
 
 The above give the 11 orbits listed.  The stabilizer conditions have been verified to guarantee stabilization of the form. If any of the stabilizer groups were larger than the one listed, or if any of the 11 orbits listed were equivalent, then there would not be sufficiently many points to cover the space, so that it follows that the 11 orbits are inequivalent and the stabilizers are correct.
\end{proof}

\subsubsection{Case of $\sO_{D2}$}
The acting group is \begin{equation}G_x \times \GL_1= \begin{pmatrix} x&&\\ a& c & e\\ b&\pm  e\ell &\pm c\end{pmatrix} \begin{pmatrix} r &\\ s&t\end{pmatrix}.\end{equation} The action is on the space $V/V_x$ with
\begin{equation}
 V_x = \begin{pmatrix} 0 &0&0\\ & z &0\\ &&- z\ell\end{pmatrix} \begin{pmatrix} 0 &*&*\\&*&*\\&&*\end{pmatrix}.
\end{equation}
\begin{lemma}
In the case of $\sO_{D2}$, the action of $G_x \times \GL_1$ on $V/V_x$ has eleven orbits, listed with a representative, size and stabilizer in the following table.
\begin{equation}
 \begin{array}{|l|l|l|l|}
  \hline
 & \text{Orbit representative} & \text{Orbit size} & \text{Stabilizer}\\
  \hline
 1& \begin{pmatrix} 1 &0&0\\ &0&0\\&&0\end{pmatrix}\begin{pmatrix}0&0&0\\ &0&0\\&&0\end{pmatrix} & (p-1)p^3 & \begin{pmatrix} x &0&0\\ 0& c &e \\ 0 &\pm e\ell  &\pm c\end{pmatrix} \begin{pmatrix}\frac{1}{x^2} & \\ & t\end{pmatrix}\\
  \hline
 2& \begin{pmatrix} 1 &1&0\\ &0&0\\&&0\end{pmatrix}\begin{pmatrix}0&0&0\\ &0&0\\&&0\end{pmatrix} & \frac{(p-1)^2p^3(p+1)}{2} & \begin{array}{l} \begin{pmatrix} x &&\\ &x&\\ &&\pm x\end{pmatrix}\begin{pmatrix}\frac{1}{x^2} &\\ &t\end{pmatrix} \\\sqcup \begin{pmatrix} x &&\\2x &-x &\\&&\pm x\end{pmatrix}\begin{pmatrix} \frac{1}{x^2} &\\ & t\end{pmatrix}\end{array}\\
  \hline
 3a& \begin{array}{l} \begin{pmatrix} 1 &0&0\\ &0&1\\&&0\end{pmatrix}\begin{pmatrix}0&0&0\\ &0&0\\&&0\end{pmatrix}, \\ (-1 = \square) \end{array} & \frac{(p-1)^2p^3(p+1)}{2} & \begin{array}{l}  \begin{pmatrix} x&&\\&\pm x &\\ &&\pm x\end{pmatrix}\begin{pmatrix} \frac{1}{x^2} &\\ & t\end{pmatrix}\\ \sqcup \begin{pmatrix}x &&\\ & \pm \sqrt{-1}x &\\ &&\mp \sqrt{-1} x\end{pmatrix} \begin{pmatrix} \frac{1}{x^2} &\\&t\end{pmatrix}
  \end{array}\\ 
  \hline 
3b&  \begin{array}{l} \begin{pmatrix} 1 &0&0\\ &1&k\\&&0\end{pmatrix}\begin{pmatrix}0&0&0\\ &0&0\\&&0\end{pmatrix}, \\ \ell- 4 k^2 = \square, (-1 \neq \square) \end{array} & \frac{(p-1)^2p^3(p+1)}{2} & *
  \\ 
  \hline 
4&\begin{pmatrix} 0 &1&0\\ &0&0\\&&0\end{pmatrix}\begin{pmatrix}0&0&0\\ &0&0\\&&0\end{pmatrix} & (p-1)p^2(p+1) & \begin{pmatrix} x &&\\&c&\\&&\pm c\end{pmatrix} \begin{pmatrix} \frac{1}{xc} &\\ s &t \end{pmatrix}\\
  \hline
5&  \begin{pmatrix} 0 &1&0\\ &0&0\\&&0\end{pmatrix}\begin{pmatrix}1&0&0\\ &0&0\\&&0\end{pmatrix} & (p-1)^2p^2(p+1) & \begin{pmatrix} x &&\\&c&\\&&\pm c\end{pmatrix} \begin{pmatrix} \frac{1}{xc} &\\ s &\frac{1}{x^2} \end{pmatrix}\\
  \hline
  6& \begin{pmatrix} 0 &0&0\\ &1&0\\&&0\end{pmatrix}\begin{pmatrix}0&0&0\\ &0&0\\&&0\end{pmatrix} & \frac{(p-1)(p+1)}{2} & \begin{array}{l} \begin{pmatrix} x &&\\ a & c&\\b & &\pm c\end{pmatrix}\begin{pmatrix} \frac{1}{c^2} &\\ s &t \end{pmatrix} \\
  \sqcup \begin{pmatrix} x &&\\ a & &e\\b &\pm  e\ell &\end{pmatrix}\begin{pmatrix} \frac{1}{ e^2\ell} &\\ s &t \end{pmatrix}
  \end{array}
  \\
  \hline
  \end{array}
\end{equation}
\begin{equation*}
\begin{array}{|l|l|l|l|}
\hline 
7&  \begin{pmatrix} 0 &0&0\\ &1&0\\&&0\end{pmatrix}\begin{pmatrix}1&0&0\\ &0&0\\&&0\end{pmatrix} & \frac{(p-1)^2(p+1)}{2} & \begin{array}{l} \begin{pmatrix} x &&\\a & c&\\ b & &\pm c\end{pmatrix} \begin{pmatrix} \frac{1}{c^2} &\\ s & \frac{1}{x^2}\end{pmatrix}\\ \sqcup \begin{pmatrix} x &&\\ a &&e\\ b & \pm  e\ell &\end{pmatrix} \begin{pmatrix} \frac{1}{ e^2\ell} & \\ s & \frac{1}{x^2}\end{pmatrix}\end{array}\\
  \hline
  8a&  \begin{array}{l}\begin{pmatrix} 0 &0&0\\ &0&1\\&&0\end{pmatrix}\begin{pmatrix}0&0&0\\ &0&0\\&&0\end{pmatrix},\\ (-1 = \square)\end{array} & \frac{(p-1)(p+1)}{2} & \begin{array}{l} \begin{pmatrix} x &&\\ a &c &\\ b & &\pm c\end{pmatrix} \begin{pmatrix} \frac{\pm 1}{ c^2} &\\ s &t \end{pmatrix} \\ \sqcup \begin{pmatrix} x &&\\ a && e \\ b &\pm  e\ell  &\end{pmatrix} \begin{pmatrix} \frac{\pm 1}{ e^2\ell} &\\ s & t\end{pmatrix}\end{array} \\
  \hline
8b&  \begin{array}{l}\begin{pmatrix} 0 &0&0\\ &1&k\\&&0\end{pmatrix}\begin{pmatrix}0&0&0\\ &0&0\\&&0\end{pmatrix},\\ \ell - 4k^2 = \square, (-1 \neq \square)\end{array} & \frac{(p-1)(p+1)}{2} & *\\
  \hline
9a&  \begin{array}{l}\begin{pmatrix} 0 &0&0\\ &0&1\\&&0\end{pmatrix}\begin{pmatrix}1&0&0\\ &0&0\\&&0\end{pmatrix},\\ (-1 = \square)\end{array} & \frac{(p-1)^2(p+1)}{2} & \begin{array}{l} \begin{pmatrix} x &&\\ a &c &\\ b & &\pm c\end{pmatrix} \begin{pmatrix} \frac{\pm 1}{ c^2} &\\ s &\frac{1}{x^2} \end{pmatrix} \\ \sqcup \begin{pmatrix} x &&\\ a && e \\ b &\pm  e\ell  &\end{pmatrix} \begin{pmatrix} \frac{\pm 1}{e^2\ell } &\\ s & \frac{1}{x^2}\end{pmatrix}\end{array} \\
  \hline
9b&  \begin{array}{l}\begin{pmatrix} 0 &0&0\\ &1&k\\&&0\end{pmatrix}\begin{pmatrix}1&0&0\\ &0&0\\&&0\end{pmatrix},\\ \ell - 4k^2 = \square, (-1 \neq \square)\end{array} & \frac{(p-1)^2(p+1)}{2} & *\\
  \hline
10&  \begin{pmatrix} 0 &0&0\\ &0&0\\&&0\end{pmatrix}\begin{pmatrix}1&0&0\\ &0&0\\&&0\end{pmatrix} & p-1 & \begin{pmatrix} x &&\\ a & c & e\\ b & \pm e \ell & \pm c\end{pmatrix} \begin{pmatrix} r & \\ s & \frac{1}{x^2}\end{pmatrix}\\
  \hline
11&  \begin{pmatrix} 0 &0&0\\ &0&0\\&&0\end{pmatrix}\begin{pmatrix}0&0&0\\ &0&0\\&&0\end{pmatrix} & 1 & G_x \times \GL_1\\
  \hline
 \end{array}
\end{equation*}
Descriptions of the *'d cases where $-1 \neq \square$ are given below.
\end{lemma}
\begin{proof}
To prove correctness, it suffices to verify the stabilizers and prove that the orbits of the same size are inequivalent, since then all points are accounted for.  

\subsubsection*{Verification that equal sized orbits are inequivalent}
To prove that orbits 2 and 3 are inequivalent, consider the action  
\begin{equation}
 \begin{pmatrix} x &&\\ a & c&e\\ b &\pm e \ell & \pm c \end{pmatrix} \begin{pmatrix} r &\\s&t\end{pmatrix} \cdot (u^2 + 2uv,0) \bmod V_x.
\end{equation}
The first quadratic form becomes
\begin{equation} 
 r (xu + av + bw)(xu + (a + 2c)v + (b \pm 2 e\ell)w).
\end{equation}
To fix the $u^2$ coefficient, $r = \frac{1}{x^2}$.
For the $uv$ and $uw$ terms to vanish, $(a + 2c) = -a$ and $(b \pm 2 e\ell) = -b$.  Hence the $v^2$ and $w^2$ terms become 
\begin{equation} \frac{1}{x^2}(-a^2v^2 -b^2w^2) \equiv -\frac{1}{x^2}(a^2 + b^2\ell^{-1} )v^2 \bmod (v^2-\ell w^2).\end{equation} The $vw$ coefficient is $\frac{-2ab}{x^2}$.  When $-1 = \square$, $a^2 +  b^2\ell^{-1} = 0$  if and only if $a = b= 0$, but this causes the acting matrix to be singular, a contradiction. When $-1 \neq \square$ substitute $a = \frac{-a}{x}$, $b = \frac{b}{x}$ so that the system to be solved is $ab = k$, $a^2 +  b^2\ell^{-1} = -1$ or $a^2 +  \frac{k^2}{ a^2\ell} = -1$, or $a^4 + a^2 +  \frac{k^2}{\ell} = 0$.   This has no solution, since the discriminant $1 - \frac{4 k^2}{\ell} \neq \square \Leftrightarrow \ell - 4k^2 = \square$.

To check that orbits 6 and 8 are inequivalent, acting on the representative in 6 maps the first quadratic form to
\begin{equation}
 r \cdot (c^2 v^2 \pm 2 ce\ell vw +  e^2\ell^2 w^2) \equiv r \cdot ((c^2 +  e^2\ell) v^2 \pm 2 ce\ell vw) \bmod v^2 -\ell w^2.
\end{equation}
When $-1 =\square$, $c^2 + e^2\ell = 0$ if and only if $c= e=0$, which makes the acting matrix singular.  When $-1 \neq \square$, there is no solution to  
\begin{equation}
 \frac{\pm ce\ell}{c^2 + e^2\ell} = k \Leftrightarrow c^2 k\pm  ce\ell +  e^2\ell k= 0
\end{equation}
in $(c,e) \neq (0,0)$ since $\ell - 4k^2 = \square$ implies that the discriminant $\ell^2 - 4\ell k^2\neq \square$. The inequivalence of orbits 7 and 9 follows similarly.

\subsubsection*{Calculation of stabilizers} The following calculations are performed in \textbf{OD2\_stabilizers.nb}.
  Throughout the acting group is assumed given in coordinates as 
\begin{equation}
 \begin{pmatrix} x &&\\ a & c&e\\ b &\pm  e\ell & \pm c \end{pmatrix} \begin{pmatrix} r &\\s&t\end{pmatrix}
\end{equation}
and recall $V_x = \begin{pmatrix} 0 &0&0\\ & z &0\\ &&-z \ell \end{pmatrix} \begin{pmatrix} 0 &*&*\\&*&*\\&&*\end{pmatrix}.$
\begin{enumerate}
 \item [(1.)] The pair of forms are $(u^2, 0)$. The $u^2$ coefficients force $s = 0$, $r = \frac{1}{x^2}$.  The $uv$ and $uw$ coefficients force $a = b=0$.  
 \item [(2.)] The pair of forms are $(u^2+2uv, 0)$. The $u^2$ coefficients force $s = 0$, $r = \frac{1}{x^2}$.  The $\GL_3$ action carries 
 \begin{equation}
  u(u+2v) \mapsto (xu + av+bw)(xu + (a+2c)v + (b+2d)w), \qquad d = \pm 2 e\ell
 \end{equation}
 To make the $uw$ term 0, $x(2b + 2d) = 0$ so $b = -d$.  To make the $vw$ term 0, $a \cdot(2b + 2d) + 2bc = 0$ so $b = 0$ or $c = 0$.
 
 \begin{itemize}
  \item (Case 1, $b = 0$): In this case $d = 0$ so $(xu + av)(xu + (a + 2c)v) = s u(u+2v)$, and $s = \frac{1}{x^2}$.  If $a = 0$ then $c = x$ which gives the solutions
  \begin{equation}
   \begin{pmatrix} x &&\\ & x&\\ && \pm x\end{pmatrix} \begin{pmatrix} \frac{1}{x^2}&\\ & t\end{pmatrix}.
  \end{equation}
  If $a \neq 0$ then $a = 2x$ and $c = -x$ which obtains the solutions
\begin{equation}  \begin{pmatrix} x &&\\ 2x &-x &\\ && \pm x \end{pmatrix} \begin{pmatrix} \frac{1}{x^2} &\\ & t\end{pmatrix}. \end{equation}
  \item (Case 2, $b \neq 0$, $b = -d$, $c = 0$): In this case
  \begin{align} 
   (xu + av + bw)(xu + av - bw) &= x^2 u^2 + 2ax uv + a^2v^2 - b^2w^2 \\& \notag \equiv x^2 u^2 + 2ax uv + ( a^2\ell -b^2)w^2 \bmod v^2 - \ell w^2.
  \end{align}
 Since $b \neq 0$, $ a^2\ell - b^2 \neq 0$, so this does not produce solutions.
 \end{itemize}
 \item [(3a.)] The pair of forms are $(u^2+ 2vw, 0)$. The $u^2$ coefficients force $r = \frac{1}{x^2}$ and $s = 0$. The $uv$ and $uw$ coefficients force $a = b= 0$. The $\GL_3$ action carries 
 \begin{align}
  2vw \mapsto &2 (cv \pm  e\ell w) (ev \pm cw)\\&\notag  = 2ce v^2 \pm 2 ( c^2 +  e^2\ell) vw + 2 ce\ell w^2 \\& \notag \equiv 4ce v^2 \pm 2(c^2 +  e^2\ell) vw \bmod v^2 - \ell w^2 
 \end{align}
 Thus $ce=0$.  The case $c = 0$ may be ruled out, since the ratio between the $u^2$ and $2vw$ coefficient becomes a non-square in this case.  The case $e = 0$, $c \neq 0$ obtains the stabilizer claimed.
 \item [(3b.)] The pair of forms are $(u^2+ v^2 + 2kvw, 0)$. The $u^2$ coefficients force $r = \frac{1}{x^2}$ and $s = 0$. The $uv$ and $uw$ coefficients force $a = b= 0$.  Under the $\GL_3$ action,
\begin{align}
v(v + 2k w) \mapsto &(cv \pm  e\ell w) ((cv \pm  e\ell w) + 2k(ev \pm cw))\\
\notag &\equiv (c^2 + 4cek +  e^2\ell)v^2 \pm (2 c^2 k+ 2 ce\ell + 2 e^2\ell k) vw \bmod (v^2 - \ell w^2). 
\end{align}
To stabilize,
\begin{equation}
 k(c^2 + 4cek +  e^2\ell) = \pm ( c^2 k+  ce\ell +  e^2\ell k).
\end{equation}
The plus case leads to $(4k^2 - \ell) ce = 0$, or $ce = 0$.  Note that $c = 0$, $e \neq 0$ does not lead to a solution, since it is not possible that $e^2\ell = x^2$.  This obtains two solutions, with $e = 0$ and $c = \pm x$.  

Solving the quadratic equation, the minus case obtains
\begin{equation}
 \frac{c}{e} = \frac{-(4k^2 + \ell)\pm (4k^2 -\ell)}{4k}, \quad \Leftrightarrow \quad  \frac{c}{e} = -2k \; \text{or}\; -\frac{\ell}{2k}.
\end{equation}
 Subject to the quadratic condition $2c^2 k^2 + (4k^2 + \ell)ce + 2 e^2 \ell k = 0$, the form becomes
\begin{equation}
 -\frac{1}{2k}\left(\ell - 4k^2\right) ce v^2 - (\ell - 4k^2) ce vw.
\end{equation}
Thus the condition to be a stabilizer becomes, 
\begin{equation}
 -\frac{1}{2k}\left(\ell - 4k^2\right) ce = x^2.
\end{equation}
Since $\ell - 4k^2 = \square$, $\frac{c}{e} = -2k$ leads to a pair of solutions.  Meanwhile $\frac{c}{e} = -\frac{\ell}{2k}$ leads to none, since its product with $-2k$ is not a square.

\item[(4.)] The pair of forms are $(2uv, 0)$.  Under the $\GL_3$ action, 
\begin{equation}
 2uv \mapsto 2(xu + av + bw) (cv \pm  e\ell w).
\end{equation}
Considering the $uw$ coefficient forces $e = 0$.  Considering the $vw$ and $v^2$ coefficients forces $a = b= 0$.  The $uv$ coefficient gives $r = \frac{1}{xc}$.
\item[(5.)] The pair of forms are $(2uv, u^2)$. This is the same as (4.), except that now $t = \frac{1}{x^2}$.
\item[(6.)] The pair of forms are $(v^2, 0)$. The $\GL_3$ actions carries $v^2 \mapsto (cv \pm  e \ell w)^2$, so the $vw$ coefficient implies $ce = 0$.  This leads to the two cases given.
\item[(7.)] The pair of forms are $(v^2, u^2)$. This is the same as (6.), with the additional requirement that now $t = \frac{1}{x^2}$.
\item[(8a.)] The pair of forms are $(2vw, 0)$. Under the $\GL_3$ action, 
\begin{equation}
 2vw \mapsto 2 (cv \pm  e\ell w)(ev \pm cw) \equiv 4ce v^2 \pm 2 (c^2 +  e^2\ell) vw \bmod v^2 -\ell w^2.
\end{equation}
Hence $ce = 0$, which leads to the two solutions given.
\item[(8b.)] The pair of forms are $(v^2 + 2kvw, 0)$. Under the $\GL_3$ action,
\begin{align}
v(v + 2k w) \mapsto &(cv \pm  e\ell w) ((cv \pm  e\ell w) + 2k(ev \pm cw))\\
\notag &\equiv (c^2 + 4cek +  e^2\ell)v^2 \pm (2 c^2 k+ 2 ce\ell + 2 e^2\ell k) vw \bmod (v^2 - \ell w^2). 
\end{align}
To stabilize,
\begin{equation}
 k(c^2 + 4cek +  e^2\ell) = \pm ( c^2 k+  ce\ell +  e^2\ell k).
\end{equation}
The plus case leads to $(4k^2 - \ell) ce = 0$, or $ce = 0$.  Both $c \neq 0$ and $e \neq 0$ lead to solutions.  Solving the quadratic equation, the minus case obtains
\begin{equation}
 \frac{c}{e} = \frac{-(4k^2 + \ell)\pm (4k^2 -\ell)}{4k}.
\end{equation}
This gives two further solutions.
\item[(9a.)] The pair of forms are $(2vw, u^2)$. This case is the same as (8a.) with the further restriction $t = \frac{1}{x^2}$.
\item[(9b.)] The pair of forms are $(v^2 + 2kvw, u^2)$. This case is the same as (8b.) with the further restriction $t = \frac{1}{x^2}$.
\item[(10.)] The pair of forms are $(0, u^2)$. The only restriction is $t = \frac{1}{x^2}$.
\item[(11.)] The pair of forms are $(0,0)$. There is no constraint.

\end{enumerate}
\end{proof}
\subsubsection{Case of $\sO_{D11}$}
The acting group is
\begin{equation}
 G_x \times \GL_1= \begin{pmatrix} *&&\\ *& *&\\ * && *\end{pmatrix} \begin{pmatrix} *&\\ *& *\end{pmatrix} \sqcup \begin{pmatrix} *&&\\ *&& *\\ *& *&\end{pmatrix} \begin{pmatrix} *&\\ *& *\end{pmatrix}
\end{equation}
of size $2 (p-1)^5  p^3$.  This acts on $V/V_x$ with 
\begin{equation}
V_x = \begin{pmatrix}
       0&0&0\\&0&*\\&&0
      \end{pmatrix}\begin{pmatrix} 0&*&*\\ &*&*\\&&*\end{pmatrix}.
\end{equation}
\begin{lemma}
 In the case of $\sO_{D11}$ the space $V/V_x$ splits into 20 orbits with representatives, sizes and stabilizers given in the following table.
 \begin{equation}
  \begin{array}{|l|l|l|l|}
  \hline
   &\text{Orbit representative} & \text{Orbit size} & \text{Stabilizer}\\
   \hline
 1 & \begin{pmatrix} 1&0&0\\&0&0\\&&0\end{pmatrix}\begin{pmatrix} 0&0&0\\&0&0\\&&0\end{pmatrix}& (p-1)p^3 & \begin{array}{l}\begin{pmatrix}r &&\\ &*&\\ &&*\end{pmatrix} \begin{pmatrix} \frac{1}{r^2} &\\&*\end{pmatrix}\\ \sqcup \begin{pmatrix} r &&\\ &&* \\&*&\end{pmatrix}\begin{pmatrix} \frac{1}{r^2} &\\&*\end{pmatrix}\end{array}\\
 \hline
 2 & \begin{pmatrix} 1&1&0\\&0&0\\&&0\end{pmatrix}\begin{pmatrix} 0&0&0\\&0&0\\&&0\end{pmatrix}& (p-1)^2p^3 & \begin{array}{l} \begin{pmatrix} r &&\\&r&\\&&*\end{pmatrix}\begin{pmatrix}\frac{1}{r^2} &\\&*\end{pmatrix} \\ \sqcup \begin{pmatrix} r &&\\ 2r &-r &\\ &&*\end{pmatrix} \begin{pmatrix}\frac{1}{r^2} &\\&*\end{pmatrix}\end{array}\\
 \hline
 3 & \begin{pmatrix} 1&1&1\\&0&0\\&&0\end{pmatrix}\begin{pmatrix} 0&0&0\\&0&0\\&&0\end{pmatrix}& \frac{(p-1)^3p^3}{4} & \begin{array}{l} \begin{pmatrix} r &&\\&r&\\&&r\end{pmatrix} \begin{pmatrix}\frac{1}{r^2} &\\&*\end{pmatrix} \\ \sqcup \begin{pmatrix} r &&\\&&r\\&r&\end{pmatrix} \begin{pmatrix}\frac{1}{r^2} &\\&*\end{pmatrix} \\ \sqcup \begin{pmatrix} r &&\\2r&-r&\\2r&&-r\end{pmatrix} \begin{pmatrix}\frac{1}{r^2} &\\&*\end{pmatrix} \\ \sqcup \begin{pmatrix} r &&\\2r&&-r\\2r&-r&\end{pmatrix} \begin{pmatrix}\frac{1}{r^2} &\\&*\end{pmatrix}\\ \sqcup \begin{pmatrix} r &&\\2r&-r&\\&&r\end{pmatrix} \begin{pmatrix}\frac{1}{r^2} &\\&*\end{pmatrix}\\ \sqcup \begin{pmatrix} r &&\\2r&&-r\\&r&\end{pmatrix} \begin{pmatrix}\frac{1}{r^2} &\\&*\end{pmatrix} \\ \sqcup \begin{pmatrix} r &&\\&r&\\2r&&-r\end{pmatrix} \begin{pmatrix}\frac{1}{r^2} &\\&*\end{pmatrix}\\ \sqcup \begin{pmatrix} r &&\\&&r\\2r&-r&\end{pmatrix} \begin{pmatrix}\frac{1}{r^2} &\\&*\end{pmatrix} \end{array}\\
   \hline 
4 & \begin{pmatrix} 1&0&0\\&-\ell&0\\&&0\end{pmatrix}\begin{pmatrix} 0&0&0\\&0&0\\&&0\end{pmatrix}&(p-1)^2p^3 & \begin{pmatrix} r &&\\ &\pm r& \\ &&*\end{pmatrix} \begin{pmatrix}\frac{1}{r^2} &\\ &*\end{pmatrix}\\
 \hline
 5 & \begin{pmatrix} 1&0&1\\&-\ell&0\\&&0\end{pmatrix}\begin{pmatrix} 0&0&0\\&0&0\\&&0\end{pmatrix}& \frac{(p-1)^3p^3}{2} & \begin{array}{l} \begin{pmatrix} r &&\\ &\pm r& \\ &&r \end{pmatrix} \begin{pmatrix}\frac{1}{r^2} &\\&*\end{pmatrix}\\ \sqcup \begin{pmatrix}r &&\\  &\pm r & \\ 2r&&-r\end{pmatrix}\begin{pmatrix} \frac{1}{r^2} &\\ &*\end{pmatrix}\end{array}\\
 \hline
  \end{array}
 \end{equation}

 \begin{equation*}
  \begin{array}{|l|l|l|l|}
 \hline
 6 & \begin{pmatrix} 1&0&0\\&-\ell&0\\&&-\ell\end{pmatrix}\begin{pmatrix} 0&0&0\\&0&0\\&&0\end{pmatrix}& \frac{(p-1)^3p^3}{4} & \begin{array}{l}\begin{pmatrix} r &&\\ & r\epsilon_1  &\\ && r\epsilon_2 \end{pmatrix} \begin{pmatrix} \frac{1}{r^2} &\\ &*\end{pmatrix}\\ \sqcup \begin{pmatrix} r &&\\ && r\epsilon_1 \\ & r\epsilon_2  & \end{pmatrix}\begin{pmatrix} \frac{1}{r^2} &\\ &*\end{pmatrix} \end{array}\\
 \hline
 7 & \begin{pmatrix} 0&1&0\\&0&0\\&&0\end{pmatrix}\begin{pmatrix} 0&0&0\\&0&0\\&&0\end{pmatrix}& 2(p-1)p & \begin{pmatrix} r &&\\ &s&\\ *&&*\end{pmatrix}\begin{pmatrix} \frac{1}{rs} &\\ *&*\end{pmatrix}\\
 \hline
 8 & \begin{pmatrix} 0&1&1\\&0&0\\&&0\end{pmatrix}\begin{pmatrix} 0&0&0\\&0&0\\&&0\end{pmatrix}& (p-1)^2p^2 & \begin{array}{l} \begin{pmatrix} r &&\\ &s&\\ &&s\end{pmatrix}\begin{pmatrix}\frac{1}{rs}&\\ *&*\end{pmatrix}\\ \sqcup \begin{pmatrix} r &&\\&&s\\ &s&\end{pmatrix}\begin{pmatrix}\frac{1}{rs} &\\ *&*\end{pmatrix}\end{array}\\
 \hline
 9 & \begin{pmatrix} 0&1&0\\&0&0\\&&1\end{pmatrix}\begin{pmatrix} 0&0&0\\&0&0\\&&0\end{pmatrix}& 2(p-1)^2p & \begin{pmatrix} r &&\\ & \frac{s^2}{r}&\\ * &&s\end{pmatrix} \begin{pmatrix} \frac{1}{s^2} &\\ *&*\end{pmatrix}\\
 \hline
 10 & \begin{pmatrix} 0&1&0\\&0&0\\&&0\end{pmatrix}\begin{pmatrix} 1&0&0\\&0&0\\&&0\end{pmatrix}& 2(p-1)^2p & \begin{pmatrix}r && \\&s&\\ *&&*\end{pmatrix} \begin{pmatrix} \frac{1}{rs}&\\*&\frac{1}{r^2}\end{pmatrix}\\
 \hline
 11 & \begin{pmatrix} 0&1&1\\&0&0\\&&0\end{pmatrix}\begin{pmatrix} 1&0&0\\&0&0\\&&0\end{pmatrix}& (p-1)^3p^2 & \begin{array}{l} \begin{pmatrix} r&&\\&s&\\&&s\end{pmatrix}\begin{pmatrix}\frac{1}{rs} &\\ *&\frac{1}{r^2}\end{pmatrix}\\ \sqcup \begin{pmatrix} r &&\\&&s\\&s&\end{pmatrix}\begin{pmatrix}\frac{1}{rs} &\\ *&\frac{1}{r^2}\end{pmatrix}\end{array} \\
 \hline
 12 & \begin{pmatrix} 0&1&0\\&0&0\\&&1\end{pmatrix}\begin{pmatrix} 1&0&0\\&0&0\\&&0\end{pmatrix}& 2(p-1)^3p & \begin{pmatrix}r && \\& \frac{s^2}{r} &\\ *&&s\end{pmatrix}\begin{pmatrix}\frac{1}{s^2} &\\ *& \frac{1}{r^2}\end{pmatrix}\\
 \hline
 13 & \begin{pmatrix} 0&0&0\\&1&0\\&&0\end{pmatrix}\begin{pmatrix} 0&0&0\\&0&0\\&&0\end{pmatrix}& 2(p-1) & \begin{pmatrix} *&&\\ *&s&\\ *&&*\end{pmatrix}\begin{pmatrix}\frac{1}{s^2} &\\ *&*\end{pmatrix}\\
 \hline
 14 & \begin{pmatrix} 0&0&0\\&1&0\\&&1\end{pmatrix}\begin{pmatrix} 0&0&0\\&0&0\\&&0\end{pmatrix}& \frac{(p-1)^2}{2} & \begin{array}{l}\begin{pmatrix} *&&\\ *&s&\\ *&&\pm s\end{pmatrix}\begin{pmatrix}\frac{1}{s^2} &\\ *&*\end{pmatrix} \\ \sqcup \begin{pmatrix} *&&\\ * &&s\\ *&\pm s&\end{pmatrix}\begin{pmatrix}\frac{1}{s^2} &\\ *&*\end{pmatrix}\end{array}\\
 \hline
 15 & \begin{pmatrix} 0&0&0\\&1&0\\&&\ell\end{pmatrix}\begin{pmatrix} 0&0&0\\&0&0\\&&0\end{pmatrix}& \frac{(p-1)^2}{2} & \begin{array}{l} \begin{pmatrix} *&&\\ *&s&\\ *&&\pm s\end{pmatrix}\begin{pmatrix}\frac{1}{s^2} &\\ *&*\end{pmatrix}\\ \sqcup \begin{pmatrix} *&&\\ *&& \pm s\\ *& s\ell& \end{pmatrix}\begin{pmatrix}\frac{1}{ s^2\ell} &\\ *&*\end{pmatrix}\end{array}\\
 \hline
 16 & \begin{pmatrix} 0&0&0\\&1&0\\&&0\end{pmatrix}\begin{pmatrix} 1&0&0\\&0&0\\&&0\end{pmatrix}&2(p-1)^2& \begin{pmatrix} r &&\\ *&s&\\ *&&*\end{pmatrix}\begin{pmatrix}\frac{1}{s^2} &\\ *&\frac{1}{r^2}\end{pmatrix}\\
\hline
  \end{array}
 \end{equation*}
\begin{equation*}
 \begin{array}{|l|l|l|l|}
   \hline
 17 & \begin{pmatrix} 0&0&0\\&1&0\\&&1\end{pmatrix}\begin{pmatrix} 1&0&0\\&0&0\\&&0\end{pmatrix}& \frac{(p-1)^3}{2} & \begin{array}{l} \begin{pmatrix} r &&\\ *&s&\\ *&&\pm s\end{pmatrix}\begin{pmatrix} \frac{1}{s^2}&\\ * & \frac{1}{r^2}\end{pmatrix}\\ \sqcup \begin{pmatrix} r &&\\ *&&s\\ *& \pm s &\end{pmatrix} \begin{pmatrix} \frac{1}{s^2} &\\ *& \frac{1}{r^2}\end{pmatrix}\end{array}\\
 \hline
 18 & \begin{pmatrix} 0&0&0\\&1&0\\&&\ell\end{pmatrix}\begin{pmatrix} 1&0&0\\&0&0\\&&0\end{pmatrix}& \frac{(p-1)^3}{2} & \begin{array}{l} \begin{pmatrix} r &&\\ *&s&\\ *&&\pm s\end{pmatrix}\begin{pmatrix} \frac{1}{s^2} &\\* &\frac{1}{r^2}\end{pmatrix}\\ \sqcup \begin{pmatrix} r &&\\ *&&\pm s\\ *&s\ell  &\end{pmatrix} \begin{pmatrix} \frac{1}{ s^2\ell} &\\ *&\frac{1}{r^2}\end{pmatrix}\end{array}\\
 \hline
 19 & \begin{pmatrix} 0&0&0\\&0&0\\&&0\end{pmatrix}\begin{pmatrix} 1&0&0\\&0&0\\&&0\end{pmatrix}& p-1 & \begin{array}{l} \begin{pmatrix} r &&\\ *& *&\\ *&&*\end{pmatrix} \begin{pmatrix} *&\\*&\frac{1}{r^2}\end{pmatrix}\\ \sqcup \begin{pmatrix}r &&\\ *&&*\\ *&*&\end{pmatrix}\begin{pmatrix} * &\\ *&\frac{1}{r^2}\end{pmatrix}\end{array}\\
 \hline
 20 & \begin{pmatrix} 0&0&0\\&0&0\\&&0\end{pmatrix}\begin{pmatrix} 0&0&0\\&0&0\\&&0\end{pmatrix}& 1 &  G_x \times \GL_1\\
 \hline
 \end{array}
\end{equation*}

\end{lemma}

\begin{proof}
 To verify correctness, it suffices to check that the orbits of the same size are inequivalent, and that the stabilizers are correct, since the count of elements exhausts $V/V_x$.  
 \subsubsection*{Verification that equal sized orbits are inequivalent}
 
 The verification that 2. and 4. are inequivalent follows from the fact that $u^2 +2uv$ and $u^2 -\ell v^2$ are inequivalent under the action of $\GL_2 \times \GL_1$ on $\sym^2(\bF_p^2)$.  The same is true of 3. and 6. since both orbits are invariant under swapping $v$ and $w$. 9. and 10. are inequivalent since the second quadratic form contains a $u^2$ term. 14. and 15. are inequivalent because when acting on $v^2 + w^2$, the ratio between the coefficents on $v^2$ and $w^2$ changes by a square.  The same holds for 17. and 18.

 \subsubsection*{Calculation of stabilizers}
 The fact that each of the listed stabilizers does in fact stabilize the form is checked in \textbf{OD11\_stabilizers.nb}.  The following proofs guarantee that the actual stabilizing group is no larger.
 
 Assume throughout that the acting group is given in coordinates as 
 \begin{equation}
 G_x \times \GL_1= \begin{pmatrix} r&&\\ a& s&\\ b && t\end{pmatrix} \begin{pmatrix} x&\\ y& z\end{pmatrix} \sqcup \begin{pmatrix} r&&\\ a&& s\\ b& t&\end{pmatrix} \begin{pmatrix} x&\\ y& z\end{pmatrix} = G_1 \sqcup G_2.
\end{equation}
Recall that
\begin{equation}
V_x = \begin{pmatrix}
       0&0&0\\&0&*\\&&0
      \end{pmatrix}\begin{pmatrix} 0&*&*\\ &*&*\\&&*\end{pmatrix}.
\end{equation}
\begin{enumerate}
\item[(1.)] The pair of forms is $(u^2, 0)$. The $uv$ and $uw$ coefficients in the first form force $a = b= 0$.  The $u^2$ coefficients imply $x = \frac{1}{r^2}$ and $y = 0$.
\item[(2.)] The pair of forms is $(u^2 + 2uv, 0)$. By considering the $u^2$ coefficient of the second form, $y = 0$. Consider the action only on the first form where $u(u+2v) \bmod vw$ must be preserved.  The action either preserves or exchanges the two lines $u = 0$ and $u + 2v = 0$, which obtains the stabilizer claimed. 
\item[(3.)] The pair of forms is $(u^2 + 2uv + 2uw, 0)$. The $u^2$ coefficient in the second quadratic form forces $y = 0$.  Depending on the choice of $c$, there are two ways of factoring $u^2 + 2uv + 2uw + c vw$ into a pair of linear factors.  If $c = 0$ the factorization is $u(u+2v+2w)$.  If $c \neq 0$ then both linear factors contain a $u$ term, one contains a $v$ term and the other contains a $w$ term.  Up to scaling the two forms, this obtains $(u + 2v)(u+2w) = u^2 + 2uv + 2uw + 4vw$.  To preserve $u(u+2v + 2w) \bmod vw$, the $\GL_3$ action can first either swap $v$ and $w$, or not, then map $u = 0$ to any of the four lines $u = 0$, $u + 2v + 2w = 0$, $u+2v = 0$ or $u + 2w=0$, and map $u + 2v + 2w=0$ to the other line in the corresponding factorization.  This obtains the eight parts of the stabilizer given.
\item[(4.)] The pair of forms is $(u^2 - \ell v^2, 0)$. By considering the $u^2$ coefficient of the second form, $y = 0$. By considering the $uv$ and $uw$ coefficients of the first form, $a = b= 0$.  By considering the $w^2$ coefficient of the first form, the acting group is contained in $G_1$.  By considering the $u^2$ and $v^2$ coefficients of the first form, $s = \pm r$ and $x = \frac{1}{r^2}$.
\item[(5.)] The pair of forms is $(u^2  + 2uw- \ell v^2, 0)$. By considering the $u^2$ coefficient of the second form, $y = 0$, so it suffices to consider the action on the first form.  First consider the possibility that the action is from $G_2$.  This maps $u(u + 2w) - \ell v^2 \mapsto (ru + av + bw)(ru + av + bw + 2sv) -  t^2 \ell w^2 \bmod vw$.  To obtain that the $uv$ coefficient is 0, $2s + a = -a$, but this makes the $v^2$ coefficient $-a^2$ which is not the negative of a non-square.  Hence the acting matrix is from $G_1$.  This maps $u(u+2w) - \ell v^2 \mapsto (ru+ av + bw)(ru+av + bw + 2tw) -  s^2\ell v^2 \bmod vw$.  By considering the  $uv$ coordinate, $a = 0$.  It follows that the mapping either fixes $u = 0$ and $u + 2w = 0$ or exchanges them.  These two possibilities obtain the two parts of the stabilizers given.
\item[(6.)] The pair of forms is $(u^2 - \ell v^2 - \ell w^2, 0)$. By considering the $u^2$ coefficient in the second form, $y=0$, so it is sufficient to consider the action on the first form.  By considering the $uv$ and $uw$ coefficients, $a = b = 0$.  Either exchanging $v=0$ and $w = 0$ or not leads to the stabilizer claimed.
\item[(7.)] The pair of forms is $(2uv, 0)$. Since there is no $u^2$ coefficient in either form, $y$ and $z$ are arbitrary, and it suffices to consider stabilization of the first form.  To preserve $2uv \bmod vw$, note that $u \mapsto ru + av + bw$ with $r \neq 0$, and hence the $uw$ coefficient forces $v = 0$ to be preserved.  The $v^2$ coefficient forces $a = 0$.  Now $b$ is arbitrary, which obtains the stabilizer given.
\item[(8.)] The pair of forms is $(2uv + 2uw, 0)$. Since there is no $u^2$ coefficient in either form, $y$ and $z$ are arbitrary, and it suffices to consider stabilization of the first form.  To preserve $2u(v + w) \bmod vw$, note that the only way to factor $2uv + 2uw + c vw$ into a pair of linear forms is with $c = 0$.  Hence the lines $u = 0$ and $v+w = 0$ are preserved.  Note that $v$ and $w$ may be preserved or exchanged, which leads to the stabilizer given.
\item[(9.)] The pair of forms is $(2uv + w^2, 0)$. Since there is no $u^2$ coefficient in either form, $y$ and $z$ are arbitrary, and it suffices to consider stabilization of the first form.  Since there is no $uw$ coefficient,  $v$ must be mapped to a scalar times itself.  The same is then true of $w$.  Considering the $v^2$ coefficient fixes $a = 0$, while $b$ is arbitrary.
\item[(10.)] The pair of forms is $(2uv, u^2)$. This is the same as 7, except that the $u^2$ term in the second form fixes $z$.
\item[(11.)] The pair of forms is $(2uv + 2uw, u^2)$. This is the same as 8, except that the $u^2$ term in the second form fixes $z$.
\item[(12.)] The pair of forms is $(2uv + w^2, u^2)$. This is the same as 9, except that the $u^2$ term in the second form fixes $z$.

For the remainder of the forms, $r, a, b$ are arbitrary since the first form does not depend on $u$.
\item[(13.)] The pair of forms is $(v^2, 0)$. The only constraints are that $v = 0$ is preserved, and that $x$ is fixed.
\item[(14.)] The pair of forms is $(v^2 + w^2, 0)$. $v = 0$ and $w = 0$ may either be fixed or exchanged, the scaling of both variables has the same square.
\item[(15.)] The pair of forms is $(v^2 + \ell w^2, 0)$. This is the same as 14, except that if $v$ and $w$ are exchanged, a futher scaling is necessary to preserve the ratio of the $v^2$ and $w^2$ coefficients.
\item[(16.)] The pair of forms is $(v^2, u^2)$. This is the same as 13, except that the $u^2$ term in the second form fixes $z$.
\item[(17.)] The pair of forms is $(v^2 + w^2, u^2)$. This is the same as 14, except that the $u^2$ term in the second form fixes $z$.
\item[(18.)] The pair of forms is $(v^2 + \ell w^2, u^2)$. This is the same as 15, except that the $u^2$ term in the second form fixes $z$.
\item[(19.)] The pair of forms is $(0, u^2)$. The only constraint is that the $u^2$ term in the second form fixes $z$.
\item[(20.)] The pair of forms is $(0,0)$. No constraint.
\end{enumerate}
\end{proof}

\subsubsection{Case of $\sO_{D1^2}$}
The acting group is
\begin{equation}
 G_x \times \GL_1 = \begin{pmatrix} *&*&\\ *&*&\\ *&*&*\end{pmatrix}\begin{pmatrix} *&\\ *&*\end{pmatrix}.
\end{equation}
The action is on the space $V/V_x$,
\begin{equation}
 V_x = \begin{pmatrix} 0&0&0\\ &0&0\\ && *\end{pmatrix} \begin{pmatrix} 0&0& *\\ &0& *\\ && *\end{pmatrix}.
\end{equation}

\begin{lemma}\label{first_form_orbits_lemma}
 Treat $\sym^2(\bF_p^3)$ as ternary quadratic forms in variables $u, v, w$.  The group $\begin{pmatrix} *&*&\\ *&*&\\ *&*&*\end{pmatrix} \times \GL_1$ acts on $\sym^2(\bF_p^3)/(w^2)$, with the first factor acting by linear change of variable, and the $\GL_1$ factor acting by multiplying the form by a scalar.  Under this action, $\sym^2(\bF_p^3)/(w^2)$ has six orbits, with representatives
 \begin{enumerate}
 \item [a.] $\begin{pmatrix} 0 & 0 &0\\ &0 &0\\ &&0\end{pmatrix}, \begin{pmatrix} 0&0&1\\ &0&0\\&&0\end{pmatrix}$
 \item [b.] $\begin{pmatrix}1 &0 &0 \\ &0 &0 \\ &&0\end{pmatrix}, \begin{pmatrix} 1 &0 &0\\ &0&1\\ &&0\end{pmatrix}$
 \item [c.] $\begin{pmatrix} 0 &1&0 \\ &0&0\\&&0\end{pmatrix}$
 \item [d.] $\begin{pmatrix} 1 &0&0\\ &-\ell &0\\ &&0\end{pmatrix}$.
\end{enumerate}
\end{lemma}

\begin{proof}
 The change of variable preserves $(w^2)$, so the usual action of $\GL_3 \times \GL_1$ on $\sym^2(\bF_p^3)$ descends to the quotient.  Identify the quotient by $(uw, vw, w^2)$ with $\sym^2(\bF_p^2)$ acted on by $\GL_2(\bF_p)\times \GL_1(\bF_p)$.  By Lemma \ref{sym_2_action_lemma}, this action has four orbits with representatitives
 \begin{equation}
 \begin{pmatrix} 0 &0\\ &0\end{pmatrix}, \begin{pmatrix} 1 &0\\ &0\end{pmatrix}, \begin{pmatrix} 0&1\\ &0\end{pmatrix}, \begin{pmatrix} 1 &0\\ &-\ell\end{pmatrix}.
\end{equation}  Extend these orbits to the quotient by $(w^2)$ by fibering $uw$ and $vw$ over each $\sym^2(\bF_p^2)$ orbit. Above 0, $uw$ and $vw$ are equivalent.  This obtains the orbits listed a.  Above $u^2$, translating $u$ by $w$ can eliminate a $uw$ term, so that the possiblities are $u^2$ and $u^2 + vw$, which are inequivalent, obtaining the b. orbits. Above $uv$, translating $u$ and $v$ by $w$ can eliminate $uw$ and $vw$ terms, so the only c. orbit is $uv$.  The same applies to  $u^2 - \ell v^2$, so this is the only orbit in d.
\end{proof}

The orbits of $V/V_x$ under $G_x \times \GL_1$ are obtained by fibering the second quadratic form in $u, v$, treated as $\sym^2(\bF_p^2)$ over the orbits listed in Lemma \ref{first_form_orbits_lemma}.
\begin{lemma}
 In the case of $\sO_{D1^2}$, the action of $ G_x \times \GL_1$ on $V/V_x$ has 23 orbits, listed with a representative, size and stabilizer in the following table.
 
 \begin{equation}
 \begin{array}{|l|l|l|l|}
  \hline
  \text{Type a.} & \text{Orbit representative} & \text{Orbit size} & \text{Stabilizer}\\
  \hline
  1. & \begin{pmatrix}0 &0&0\\ &0&0\\ &&0\end{pmatrix} \begin{pmatrix}0&0&0\\&0&0\\&&0\end{pmatrix} &1 & G_x \times \GL_1\\
  \hline
  2. & \begin{pmatrix}0 &0&0\\ &0&0\\ &&0\end{pmatrix} \begin{pmatrix}1&0&0\\&0&0\\&&0\end{pmatrix} &(p-1)(p+1) & \begin{pmatrix}s & * &\\ &* &\\ * & * & *\end{pmatrix} \begin{pmatrix} *&\\ * & \frac{1}{s^2}\end{pmatrix}\\
  \hline
  3. & \begin{pmatrix}0 &0&0\\ &0&0\\ &&0\end{pmatrix} \begin{pmatrix}0&1&0\\&0&0\\&&0\end{pmatrix} &\frac{(p-1)p(p+1)}{2} & \begin{array}{l} \begin{pmatrix} s&&\\ &t&\\ *&*&*\end{pmatrix} \begin{pmatrix} *&\\ *&\frac{1}{st}\end{pmatrix}\\ \sqcup \begin{pmatrix} &s&\\ t &&\\ * & * &*\end{pmatrix}\begin{pmatrix} * &\\ * & \frac{1}{st}\end{pmatrix}\end{array}\\
  \hline
  4. & \begin{pmatrix}0 &0&0\\ &0&0\\ &&0\end{pmatrix} \begin{pmatrix}1&0&0\\&-\ell&0\\&&0\end{pmatrix} &\frac{(p-1)^2p}{2} & \begin{pmatrix} c & e &\\ \pm  e\ell & \pm c & \\ * & * & *\end{pmatrix} \begin{pmatrix} * &\\ * & \frac{1}{c^2 -  e^2\ell}\end{pmatrix}\\
  \hline
  5. & \begin{pmatrix}0 &0&1\\ &0&0\\ &&0\end{pmatrix} \begin{pmatrix}1&0&0\\&0&0\\&&0\end{pmatrix} &(p-1)^2(p+1) & \begin{pmatrix} s &* &\\ &*&\\ *&*&t\end{pmatrix} \begin{pmatrix} \frac{1}{st} &\\ *& \frac{1}{s^2}\end{pmatrix}\\
  \hline
  6. & \begin{pmatrix}0 &0&1\\ &0&0\\ &&0\end{pmatrix} \begin{pmatrix}0&0&0\\&1&0\\&&0\end{pmatrix} &(p-1)^2p (p+1) & \begin{pmatrix} r && \\ &s&\\ *&*&t\end{pmatrix}\begin{pmatrix}\frac{1}{rt} &\\ * & \frac{1}{s^2}\end{pmatrix}\\
  \hline
  7. & \begin{pmatrix}0 &0&1\\ &0&0\\ &&0\end{pmatrix} \begin{pmatrix}0&1&0\\&0&0\\&&0\end{pmatrix} &(p-1)^2p (p+1) & \begin{pmatrix} r && \\ &s&\\ *&*&t\end{pmatrix} \begin{pmatrix} \frac{1}{rt} &\\ * & \frac{1}{rs}\end{pmatrix}\\
  \hline
  8. & \begin{pmatrix}0 &0&1\\ &0&0\\ &&0\end{pmatrix} \begin{pmatrix}0&1&0\\&1&0\\&&0\end{pmatrix} &\frac{(p-1)^3p(p+1)}{2} & \begin{array}{l} \begin{pmatrix} s &&\\&s&\\ *&*&t\end{pmatrix}\begin{pmatrix} \frac{1}{st} &\\ *& \frac{1}{s^2}\end{pmatrix}\\ \sqcup \begin{pmatrix} -s & 2s &\\ &s&\\ *& *&t\end{pmatrix}\begin{pmatrix}\frac{-1}{st} &\\ *& \frac{1}{s^2}\end{pmatrix}\end{array}\\
  \hline
  9. & \begin{pmatrix}0 &0&1\\ &0&0\\ &&0\end{pmatrix} \begin{pmatrix}1&0&0\\&-\ell&0\\&&0\end{pmatrix} &\frac{(p-1)^3p(p+1)}{2} & \begin{pmatrix}s && \\ &\pm s &\\ *&*&t\end{pmatrix} \begin{pmatrix} \frac{1}{st} &\\ * & \frac{1}{s^2}\end{pmatrix}\\
  \hline
  10. & \begin{pmatrix}0 &0&1\\ &0&0\\ &&0\end{pmatrix} \begin{pmatrix}0&0&0\\&0&0\\&&0\end{pmatrix} &(p-1)(p+1) & \begin{pmatrix} s &*&\\ &*&\\ *&*&t\end{pmatrix}\begin{pmatrix} \frac{1}{st} &\\ *&*\end{pmatrix}\\
  \hline
 \end{array}
 \end{equation}
 \begin{equation}
 \begin{array}{|l|l|l|l|}
  \hline
  \text{Type b.} & \text{Orbit representative} & \text{Orbit size} & \text{Stabilizer}\\
  \hline
  11. & \begin{pmatrix}1 &0&0\\ &0&0\\ &&0\end{pmatrix} \begin{pmatrix}0&0&0\\&0&0\\&&0\end{pmatrix} &(p-1)p^2(p+1) & \begin{pmatrix}s &*&\\ & * &\\ &*&*\end{pmatrix}\begin{pmatrix} \frac{1}{s^2}&\\ &*\end{pmatrix}\\
  \hline
  12. & \begin{pmatrix}1 &0&0\\ &0&0\\ &&0\end{pmatrix} \begin{pmatrix}0&1&0\\&0&0\\&&0\end{pmatrix} &(p-1)^2p^2(p+1) & \begin{pmatrix}s &a&\\ & t &\\ &*&*\end{pmatrix}\begin{pmatrix} \frac{1}{s^2}&\\ -\frac{2a}{s^2 t}&\frac{1}{st}\end{pmatrix}\\
  \hline
  13. & \begin{pmatrix}1 &0&0\\ &0&0\\ &&0\end{pmatrix} \begin{pmatrix}0&0&0\\&1&0\\&&0\end{pmatrix} &(p-1)^2p^3(p+1) & \begin{pmatrix}s &&\\ & t &\\ &*&*\end{pmatrix}\begin{pmatrix} \frac{1}{s^2}&\\ &\frac{1}{t^2}\end{pmatrix}\\
  \hline
  14. & \begin{pmatrix}1 &0&0\\ &0&1\\ &&0\end{pmatrix} \begin{pmatrix}0&0&0\\&0&0\\&&0\end{pmatrix} &(p-1)^2p^2(p+1) & \begin{pmatrix}s &-\frac{at}{s}&\\ & t &\\ a&*&\frac{s^2}{t}\end{pmatrix}\begin{pmatrix} \frac{1}{s^2}&\\ &*\end{pmatrix}\\
  \hline
  15. & \begin{pmatrix}1 &0&0\\ &0&1\\ &&0\end{pmatrix} \begin{pmatrix}0&1&0\\&0&0\\&&0\end{pmatrix} &(p-1)^3p^2(p+1) & \begin{pmatrix}s &-\frac{at}{s}&\\ & t &\\ a& *&\frac{s^2}{t}\end{pmatrix}\begin{pmatrix} \frac{1}{s^2}&\\ \frac{2a}{s^3}&\frac{1}{st}\end{pmatrix}\\
  \hline
  16. & \begin{pmatrix}1 &0&0\\ &0&1\\ &&0\end{pmatrix} \begin{pmatrix}0&0&0\\&1&0\\&&0\end{pmatrix} &(p-1)^3p^3(p+1) & \begin{pmatrix}s &&\\ & t &\\ &*&\frac{s^2}{t}\end{pmatrix}\begin{pmatrix} \frac{1}{s^2}&\\ &\frac{1}{t^2}\end{pmatrix}\\
  \hline
 \end{array}
 \end{equation}
  \begin{equation}
 \begin{array}{|l|l|l|l|}
  \hline
  \text{Type c.} & \text{Orbit representative} & \text{Orbit size} & \text{Stabilizer}\\
  \hline
  17. & \begin{pmatrix}0 &1&0\\ &0&0\\ &&0\end{pmatrix} \begin{pmatrix}0&0&0\\&0&0\\&&0\end{pmatrix} &\frac{(p-1)p^4(p+1)}{2} & \begin{array}{l}\begin{pmatrix}s &&\\ & t &\\ &&*\end{pmatrix}\begin{pmatrix} \frac{1}{st}&\\ &*\end{pmatrix}\\ \sqcup \begin{pmatrix} &s&\\t &  &\\ &&*\end{pmatrix}\begin{pmatrix} \frac{1}{st}&\\ &*\end{pmatrix}\end{array}\\
  \hline
  18. & \begin{pmatrix}0 &1&0\\ &0&0\\ &&0\end{pmatrix} \begin{pmatrix}1&0&0\\&0&0\\&&0\end{pmatrix} &(p-1)^2p^4(p+1) & \begin{pmatrix}s &&\\ & t &\\ &&*\end{pmatrix}\begin{pmatrix} \frac{1}{st}&\\ &\frac{1}{s^2}\end{pmatrix}\\
  \hline
  19. & \begin{pmatrix}0 &1&0\\ &0&0\\ &&0\end{pmatrix} \begin{pmatrix}1&0&0\\&1&0\\&&0\end{pmatrix} &\frac{(p-1)^3p^4(p+1)}{4} & \begin{array}{l}\begin{pmatrix}s &&\\ & \pm s &\\ &&*\end{pmatrix}\begin{pmatrix} \frac{\pm 1}{s^2}&\\ &\frac{1}{s^2}\end{pmatrix}\\ \sqcup \begin{pmatrix} &s&\\ \pm s &  &\\ &&*\end{pmatrix}\begin{pmatrix} \frac{\pm 1}{s^2}&\\ &\frac{1}{s^2}\end{pmatrix}\end{array}\\
  \hline
  20. & \begin{pmatrix}0 &1&0\\ &0&0\\ &&0\end{pmatrix} \begin{pmatrix}1&0&0\\&\ell&0\\&&0\end{pmatrix} &\frac{(p-1)^3p^4(p+1)}{4} & \begin{array}{l}\begin{pmatrix}s &&\\ & \pm s &\\ &&*\end{pmatrix}\begin{pmatrix} \frac{\pm 1}{s^2}&\\ &\frac{1}{s^2}\end{pmatrix}\\ \sqcup \begin{pmatrix} &\pm s&\\ s\ell &  &\\ &&*\end{pmatrix}\begin{pmatrix} \frac{\pm 1}{ s^2\ell}&\\ &\frac{1}{ s^2\ell}\end{pmatrix}\end{array}\\
  \hline
 \end{array}
 \end{equation}
   \begin{equation}
 \begin{array}{|l|l|l|l|}
  \hline
  \text{Type d.} & \text{Orbit representative} & \text{Orbit size} & \text{Stabilizer}\\
  \hline
  21. & \begin{pmatrix}1 &0&0\\ &-\ell &0\\ &&0\end{pmatrix} \begin{pmatrix}0&0&0\\&0&0\\&&0\end{pmatrix} &\frac{(p-1)^2p^4}{2} & \begin{pmatrix}c & e &\\ \pm  e\ell & \pm c &\\ &&*\end{pmatrix} \begin{pmatrix} \frac{1}{c^2 -  e^2\ell} &\\ &*\end{pmatrix}\\
  \hline
  22. & \begin{pmatrix}1 &0&0\\ &-\ell &0\\ &&0\end{pmatrix} \begin{pmatrix}0&0&0\\&1&0\\&&0\end{pmatrix} &\frac{(p-1)^3p^4(p+1)}{4} & \begin{array}{l} \begin{pmatrix}s &  &\\  & \pm s &\\ &&*\end{pmatrix} \begin{pmatrix} \frac{1}{s^2} &\\ &\frac{1}{s^2}\end{pmatrix} \\ \sqcup  \begin{pmatrix} & \pm s &\\  s\ell &&\\ &&*\end{pmatrix} \begin{pmatrix} \frac{-1}{ s^2\ell} &\\ \frac{1}{ s^2\ell^2} & \frac{1}{ s^2\ell} \end{pmatrix}\end{array} \\
  \hline
  23a. & \begin{array}{l}\begin{pmatrix}1 &0&0\\ &-\ell &0\\ &&0\end{pmatrix} \begin{pmatrix}0&1&0\\&0&0\\&&0\end{pmatrix},\\ (-1 = \square)\end{array} &\frac{(p-1)^3p^4(p+1)}{4} & \begin{array}{l} \begin{pmatrix}s &  &\\ & \pm s &\\ &&*\end{pmatrix} \begin{pmatrix} \frac{1}{s^2 } &\\ &\pm \frac{1}{s^2}\end{pmatrix}\\ \sqcup \begin{pmatrix} & s &\\ \pm  s\ell & &\\ & & *\end{pmatrix}\begin{pmatrix}  \frac{-1}{ s^2\ell} &\\ & \pm \frac{1}{ s^2\ell}\end{pmatrix} \end{array}\\
  \hline
  23b. & \begin{array}{l}\begin{pmatrix}1 &0&0\\ &-\ell &0\\ &&0\end{pmatrix} \begin{pmatrix}0&k&0\\&1&0\\&&0\end{pmatrix},\\ 1-4\ell k^2 \neq \square, (-1 \neq \square)\end{array} &\frac{(p-1)^3p^4(p+1)}{4} & *\\
  \hline
 \end{array}
 \end{equation}
A description of the stabilizer $*$ is given in the proof.
 \end{lemma}
\begin{proof}
To classify the orbits of type a., when the first form is 0, the action on the second form is the same as $\GL_2\times \GL_1$ acting on $\sym^2(\bF_p^2)$ which has the four orbits claimed.  When the first form is $uw$, since $w = 0$ is preserved by the group action, modulo $w$, $u = 0$ is preserved, so that the action becomes $\begin{pmatrix} * & *\\ & *\end{pmatrix} \times \GL_1$ acting on $\sym^2(\bF_p^2)$.  By Lemma \ref{sym_2_action_smaller_lemma}, this action has six orbits, which gives the six orbits claimed.

To classify the orbits of type b., note that, by subtracting a multiple of the first form from the second, the $u^2$ term in the second form may be eliminated, so that it takes the form $c_1 uv + c_2 v^2$.  When the first form is $u^2$, to stabilize the first quadratic form, $u$ must be replaced by a multiple of itself.  Thus the action on the second quadratic form is by $\begin{pmatrix} * & *\\ &*\end{pmatrix} \times \GL_1$, and $uv$ is inequivalent with $v^2$ since the factor $u$ cannot be eliminated.  Meanwhile, if $c_2 \neq 0$ then the $uv$ term may be eliminated by replacing $v$ with $v + au$ for an appropriate $a$.  Thus  there are three inequivalent orbits with representatives $0, uv$ and $v^2$ under this action.  When the first form is $u^2 + 2vw$, the first form may be stabilized by adjusting $u$ by $w$ and $v$ by $u$ simultaneously to keep the $uw$ term 0.  Thus the action on the second quadratic form (which does not depend on $w$) is the same, so there are three orbits in this case, also.

To classify the orbits of type c., by subtracting a multiple of the first form from the second it may be assumed that the second form is of type $c_1 u^2 + c_2 v^2$.  To stabilize the first form, modulo $w$ the lines $u = 0$ and $v = 0$ are either fixed or exchanged.  If only one of $c_1, c_2$ is non-zero, it may be assumed that $c_1 \neq 0$ by exchanging $u,v$.  Since the lines are preserved or exchanged, this is inequivalent to the orbits with both $c_1$ and $c_2 \neq 0$, and the orbits where $c_1$ and $c_2$ are related by a square or non-square are inequivalent, since the scaling in individual coordinates is quadratic.

To classify the orbits of type d., by subtracting a multiple of the first form from the second it may be assumed that the second form is of the type $c_1 uv + c_2 v^2$.  To preserve the first quadratic form, the change of variables in $u, v$ takes the form $\begin{pmatrix} c & e\\ \pm  e\ell& \pm c\end{pmatrix}$.  This maps $v^2$ in the second form to 
\begin{equation}
 (eu \pm cv)^2 = e^2 u^2 \pm 2ce uv + c^2 v^2 \equiv \pm 2ce uv + (c^2 +  e^2\ell )v^2 \bmod u^2 - \ell v^2
\end{equation}
after subtracting a multiple of the first form to clear the $u^2$ term.  If $-1 \neq \square$ modulo $p$ then $c^2 +  e^2\ell \not \equiv 0$ so $uv$ and $v^2$ are inequivalent.  Otherwise, $v^2$ and $2k uv + v^2$ are inequivalent since $ce \equiv k(c^2 +  e^2\ell)$ leads to an irreducible quadratic.  These are all of the inequivalent orbits, since testing a non-zero form as equivalent to $v^2$ leads to a quadratic equation which is either soluble (equivalent) or not (inequivalent).

   Note that, since the orbits have already been classified, it suffices to check that each claimed stabilizer does in fact fix the form, since the orbits exactly cover the space.  This verification is performed in \textbf{OD12\_stabilizers.nb}, except for in the case 23b.  In the case of 23 b., the stabilizer is obtained as follows.  To preserve the top quadratic form, the stabilizer takes the form
  \begin{equation}
   \begin{pmatrix} c & e & \\ \pm  e\ell & \pm c &\\ 0&0& * \end{pmatrix}\begin{pmatrix} \frac{1}{c^2 - e^2\ell} & \\ \frac{s}{c^2 - e^2\ell} & t\end{pmatrix}.
  \end{equation}
The $\GL_3$ part maps
\begin{equation}
v(v + 2k u) \mapsto (e u \pm cv) ( (2ck + e)u \pm (2  e\ell k+ c)v) . 
\end{equation}
Subtracting down to eliminate the coefficient on $u^2$, which determines $s$, this becomes
\begin{equation}
 \pm\left(2 c^2 k+ 2ce +2 e^2\ell k\right) uv + (c^2 + 4 ce\ell k+  e^2\ell) v^2.
\end{equation}
To satisfy the stabilizer condition it is necessary and sufficient that the ratio of the $uv$ and $v^2$ coefficients satisfies
\begin{equation}
 \frac{ \pm\left(2 c^2 k+ 2ce +2 e^2\ell k\right)}{c^2 + 4 ce\ell k+  e^2\ell} = 2k.
\end{equation}
For the $+$ condition to be satisfied, it follows that
\begin{equation}
 (1 - 4  \ell k^2) ce = 0,
\end{equation}
which implies that $ce = 0$.  Both $c = 0$ and $e = 0$ lead to solutions fixing $t$.  The minus case obtains
\begin{equation}
 -(c^2 k+ ce +   e^2\ell k) =  c^2 k+ 4  ce\ell k^2+   e^2\ell k.
\end{equation}
This leads to a reducible quadratic equation in $c$ and $e$ with two non-zero projective solutions.
\end{proof}

\section{Evaluation of exponential sums}
By Lemma \ref{full_mod_p_orbit_lemma},  and the classification of maximal and non-maximal $\mod p^2$ orbits in Lemmas \ref{22_1211_122_lemma}, \ref{14_lemma}, \ref{131_lemma}, and \ref{1212_lemma}, in order to evaluate the Fourier transform of the maximal set at frequencies not divisible by $p$ it suffices to consider the following $\mod p^2$ orbits in time domain.
\begin{equation}\label{standard_mod_p_rep}
 \begin{array}{|l|l|l|}
  \hline
  \text{$\mod p$ orbit} & \text{$\mod p^2$ orbit representative} & \text{Justification}\\
  \hline
  \sO_{1^211} & \begin{pmatrix} 0 &0&0\\ &1&0\\ &&-1\end{pmatrix}\begin{pmatrix} 0 &0 & \frac{1}{2}\\ &0&0\\ &&0\end{pmatrix} & \text{This is the only non-maximal orbit}\\
  \hline
  \sO_{1^22} & \begin{pmatrix} 0 &0&0\\ &1&0\\ &&-\ell\end{pmatrix}\begin{pmatrix} 0&0&\frac{1}{2}\\ &0&0\\&&0\end{pmatrix} & \text{This is the only non-maximal orbit}\\
  \hline
  \sO_{2^2} & \begin{pmatrix} 0&0&0\\ &0&0\\&&1\end{pmatrix} \begin{pmatrix}1 &0&0\\ &-\ell &0\\ &&0\end{pmatrix} & \text{This is the only non-maximal orbit}\\
  \hline
  \sO_{1^4} & \begin{array}{l}\begin{pmatrix} \epsilon p &0&0\\&0&0\\&&1\end{pmatrix}\begin{pmatrix} 0 &0&\frac{1}{2}\\ &1&0\\&&0\end{pmatrix},\\ \epsilon \in \bF_p^\times/\{x^4: x \in \bF_p^\times\} \end{array} & \text{These are all of the maximal orbits}\\
  \hline
  \sO_{1^31} & \begin{array}{l} \begin{pmatrix} \epsilon p &0&0\\ &0&\frac{1}{2} \\&&0\end{pmatrix} \begin{pmatrix} 0&0&\frac{1}{2}\\ &1&0\\ &&0\end{pmatrix}, \\ \epsilon \in \bF_p^\times/\{x^3: x \in \bF_p^\times\} \end{array} & \text{These are all of the maximal orbits}\\
  \hline
  \sO_{1^21^2} & \begin{array}{l} \begin{pmatrix} p &0&0\\ &p&0\\ &&1\end{pmatrix} \begin{pmatrix} 0 & \frac{1}{2} &0\\ &0&0\\ &&0\end{pmatrix}, \\ \begin{pmatrix} \ell p &0 &0\\ &p &0\\ &&1\end{pmatrix} \begin{pmatrix} 0 & \frac{1}{2} &0 \\ &0 &0\\ &&0\end{pmatrix},\\ \begin{pmatrix} \ell p & 0&0\\ & \ell p &0\\ & & 1\end{pmatrix} \begin{pmatrix} 0 & \frac{1}{2} &0\\ &0&0\\&&0\end{pmatrix} \end{array} & \text{These are all of the maximal orbits}\\
  \hline
 \end{array}
\end{equation}
The size of the stabilizer in $G(\zed/p^2\zed)$ of each orbit listed above is given in the following table.  These were obtained by multiplying the size of the stabilizer of the corresponding orbit in the $G_x$ action on $V/V_x$ by $p^{\dim G- \dim V_x} = p^{13 - \dim V_x}$, see Lemma \ref{stabilizer_size_lemma}.
\begin{equation}\label{maximal_orbit_stab_size}
 \begin{array}{|l|l|l|}
  \hline
  \text{$\mod p$ orbit} & \text{$\mod p^2$ orbit representative} & \text{Stabilizer size}\\
  \hline
  \sO_{1^211} & \begin{pmatrix} 0 &0&0\\ &1&0\\ &&-1\end{pmatrix}\begin{pmatrix} 0 &0 & \frac{1}{2}\\ &0&0\\ &&0\end{pmatrix} & 2(p-1)^2p^2\\
  \hline
  \sO_{1^22} & \begin{pmatrix} 0 &0&0\\ &1&0\\ &&-\ell\end{pmatrix}\begin{pmatrix} 0&0&\frac{1}{2}\\ &0&0\\&&0\end{pmatrix} & 2(p-1)^2p^2\\
  \hline
  \sO_{2^2} & \begin{pmatrix} 0&0&0\\ &0&0\\&&1\end{pmatrix} \begin{pmatrix}1 &0&0\\ &-\ell &0\\ &&0\end{pmatrix} & 2(p-1)^2p^3(p+1)\\
  \hline
  \sO_{1^4} & \begin{array}{l}\begin{pmatrix} \epsilon p &0&0\\&0&0\\&&1\end{pmatrix}\begin{pmatrix} 0 &0&\frac{1}{2}\\ &1&0\\&&0\end{pmatrix},\\ \epsilon \in \bF_p^\times/\{x^4: x \in \bF_p^\times\} \end{array} & (p-1)p^4\#\{\epsilon^4 = 1\}\\
  \hline
  \sO_{1^31} & \begin{array}{l} \begin{pmatrix} \epsilon p &0&0\\ &0&\frac{1}{2} \\&&0\end{pmatrix} \begin{pmatrix} 0&0&\frac{1}{2}\\ &1&0\\ &&0\end{pmatrix}, \\ \epsilon \in \bF_p^\times/\{x^3: x \in \bF_p^\times\} \end{array} & (p-1)p^3\#\{\epsilon^3 = 1\}\\
  \hline
  \sO_{1^21^2} &\begin{pmatrix} p &0&0\\ &p&0\\ &&1\end{pmatrix} \begin{pmatrix} 0 & \frac{1}{2} &0\\ &0&0\\ &&0\end{pmatrix} & 8(p-1)p^3\\
  \hline
  \sO_{1^21^2} & \begin{pmatrix} \ell p &0 &0\\ &\ell p &0\\ &&1\end{pmatrix} \begin{pmatrix} 0 & \frac{1}{2} &0 \\ &0 &0\\ &&0\end{pmatrix} & 8(p-1)p^3\\
  \hline
  \sO_{1^21^2} &  \begin{pmatrix} \ell p & 0&0\\ & p &0\\ & & 1\end{pmatrix} \begin{pmatrix} 0 & \frac{1}{2} &0\\ &0&0\\&&0\end{pmatrix}  & 4(p-1)p^3\\
  \hline
 \end{array}
\end{equation}

In evaluating the exponential sums write the time domain variable $x = x_0 + px_1$ and frequency variable $\xi = \xi_0 + p \xi_1$ where $x_0$ and $\xi_0$ are standard orbit representatives defined over $\zed$. In this section it is useful to recall the formulas (see Lemma \ref{orb_exp_sum_lemma})
\begin{align}
 S_{p^2}(x, \xi) &=  p^{13} \sS(x,\xi),\\
 \notag \sS(x, \xi)&= \sum_{g \in G_{x, \xi}^t}e_p\left([x_1, g\cdot \xi_0] + [x_0, g\cdot \xi_1] \right)
 \end{align}
 and (see the beginning of Section \ref{orb_exp_section})
 \begin{align} \Sigma_{p^2}(x, \xi) &= \frac{p^{13}\sS(x, \xi)}{\left|\Stab_{G(\zed/p^2\zed)}(x)\right|}\\
 \notag \sM(x, \xi) &= \sum_{\substack{\sO_{x'}\bmod p^2: \sO_{x'} \subset \sO_x \bmod p\\ \text{maximal}}} \Sigma_{p^2}(x', \xi)\\
 \notag &=  -\sum_{\substack{\sO_{x'}\bmod p^2: \sO_{x'} \subset \sO_x \bmod p\\ \text{non-maximal}}} \Sigma_{p^2}(x', \xi).
\end{align}

The following lemma justifies classifying the $G_x \times \GL_1$ action on $V/V_x$ in the case of $\xi$.
\begin{lemma}\label{dilation_lemma}
 Let $x \in \sO_{1^211}, \sO_{1^22}, \sO_{2^2}, \sO_{1^4}, \sO_{1^31}$ or $\sO_{1^21^2} \bmod p$ and write $\xi = \xi_0 + p \xi_1$ satisfy $G_{x, \xi} \neq \emptyset$.  Let $\lambda \in \bF_p^\times$.  Then
 \begin{equation}
  \sM(x, \xi_0 + p\xi_1) = \sM(x, \xi_0 + p \lambda \xi_1).
 \end{equation}

\end{lemma}

\begin{proof}
 Let $x = x_0 + px_1$ be a $\mod p^2$ orbit representative from (\ref{standard_mod_p_rep}).  Note that, by its definition  the action set $G_{x, \xi}$ is invariant by multiplication by a scalar matrix in $\GL_2$.  By Lemma \ref{orb_exp_sum_lemma}, 
\begin{align}
\sS(x, \xi_0 + p\lambda \xi_1) &= \sum_{g \in G_{x, \xi}^t} e_p\left([x_0,\lambda g \cdot \xi_1] + [x_1, g\cdot \xi_0] \right)\\
\notag &= \sum_{g \in G_{x, \xi}^t} e_p\left([x_0, g \cdot \xi_1] + \left[\frac{1}{\lambda} x_1, g\cdot \xi_0\right] \right)\\
\notag&= \sS(x_0 + p\lambda^{-1}x_1, \xi_0 + p \xi_1). 
\end{align}
In the case of $\sO_{1^211}, \sO_{1^22}$ and $\sO_{2^2}$, $x_1 = 0$, so 
\begin{equation}
 \sM(x, \xi_0 + p\lambda \xi_1) = - \frac{p^{13} \sS(x, \xi_0 + p\lambda \xi_1)}{\left|\Stab_{G(\zed/p^2\zed)}(x)\right|} =  - \frac{p^{13} \sS(x, \xi_0 + p \xi_1)}{\left|\Stab_{G(\zed/p^2\zed)}(x)\right|} = \sM(x, \xi_0 + p\xi_1).
\end{equation}
In the case of $\sO_{1^4}$, $\sO_{1^31}$ and $\sO_{1^21^2}$, multiplying $x_1$ by $\lambda^{-1}$ permutes the maximal $\mod p^2$ orbits above the $\mod p$ orbit, while leaving the stabilizer size unchanged, see (\ref{maximal_orbit_stab_size}).  Since $\sM(x, \xi)$ is a sum of $\sS(x, \xi)$ weighted by $p^{13}$ times the inverse of the stabilizer, summed over these orbits, it follows that $\sM(x, \xi_0 + p \lambda \xi_1) = \sM(x, \xi_0 + p \xi_1)$ in this case, also.
\end{proof}

 \subsection{The exponential sums pair $(\sO_{1^211}, \sO_{D1^2})$}\label{O1211_OD12_section}
 In this case, 
 \begin{equation}x_0 =\begin{pmatrix} 0 & 0 &0\\ &1 &0\\ &&-1\end{pmatrix}\begin{pmatrix} 0 & 0 & \frac{1}{2}\\ &0 &0\\ &&0\end{pmatrix}, \qquad \xi_0 = \begin{pmatrix} 1 & 0 &0\\ &0&0\\&&0\end{pmatrix}\begin{pmatrix} 0&0&0\\ &0&0\\&&0\end{pmatrix} 
 \end{equation}
 and $x_1 = 0$.
 The action set is 
 \begin{equation}
G_{x,\xi}^t= \begin{pmatrix} a_{11} & a_{12} & a_{13}\\ & a & b\\ & c & d\end{pmatrix}\begin{pmatrix} b_{11} & b_{12}\\ & b_{22}\end{pmatrix}.
 \end{equation}
 The stabilizer has size
 \begin{equation}
  \left|\Stab_{G(\zed/p^2\zed)}(x)\right| = 2(p-1)^2p^2.
 \end{equation}

 The $\GL_2$ condition on $\begin{pmatrix} a & b\\c &d\end{pmatrix}$ may be fibered as follows
 \begin{align}
  \sum_{\begin{pmatrix} a & b\\ c & d\end{pmatrix} \in \GL_2} &= \sum_{\begin{pmatrix} a & b\\ c& d\end{pmatrix} \in \bF_p^4} - \sum_{\substack{\begin{pmatrix}  b\lambda & b\\  d\lambda & d\end{pmatrix}\\ \begin{pmatrix} b \\ d \end{pmatrix} \in \bF_p^2, \lambda \in \bF_p^\times}} - \sum_{\substack{\begin{pmatrix} a & 0\\ c & 0\end{pmatrix}\\ \begin{pmatrix} a\\ c\end{pmatrix} \in \bF_p^2}} - \sum_{\substack{ \begin{pmatrix} 0 & b\\ 0 & d \end{pmatrix} \\ \begin{pmatrix} b\\d\end{pmatrix} \in \bF_p^2}} + p \sum_{\begin{pmatrix} 0&0\\0&0\end{pmatrix}}\\
 \notag  &= \Sigma_1 - \Sigma_2 -\Sigma_3 -\Sigma_4 + \Sigma_5.
 \end{align}
 In the summations that follow, unless otherwise indicated summations over matrices are constrained by the fact that the pairs of matrices are contained in $\GL_3 \times \GL_2$.  If $\Sigma_1 - \Sigma_5$ is indicated then $a_{11}, b_{11}, b_{22} \in \bF_p^\times$, $a_{12}, a_{13}, b_{12} \in \bF_p$ and the remaining variables range according to the fibration above.
 
 Representatives and exponential sum pairings are given in the table below. Since this is the single non-maximal orbit,
 \begin{equation}
  \sM(x,\xi) = - \frac{p^{11} \sS(x,\xi)}{2(p-1)^2}.
 \end{equation}
See \textbf{exponential\_sums\_O1211\_OD12.nb}.
 {\tiny
 \begin{equation*}
  \begin{array}{|l|l|l|l|l|}
   \hline
   \text{Orbit} & \xi_1 & [ x_0, g \cdot \xi_1 ] & \sS(x,\xi) & \sM(x,\xi)\\
   \hline
   1. & \begin{pmatrix} 0 &0&0\\&0&0\\&&0\end{pmatrix}\begin{pmatrix} 0&0&0\\&0&0\\&&0\end{pmatrix} & 0 &(p-1)^5p^4(p+1) & -\frac{(p-1)^3 p^{15}(p+1)}{2}\\
   \hline
   2. & \begin{pmatrix} 0 & 0&0\\ &0&0\\ &&1\end{pmatrix}\begin{pmatrix} 0&0&0\\&0&0\\&&0\end{pmatrix} & b_{11} (b^2 - d^2)&  (p-1)^5p^4& - \frac{(p-1)^3 p^{15}}{2} \\
   \hline
   3. & \begin{pmatrix} 0 &0&0\\&0&\frac{1}{2}\\ &&0\end{pmatrix}\begin{pmatrix}0&0&0\\&0&0\\&&0\end{pmatrix} & b_{11}(ab - cd) &-2(p-1)^4p^4 & (p-1)^2 p^{15}\\
   \hline
   4. & \begin{pmatrix} 0 & 0 & 0\\ & -\ell &0\\ &&1\end{pmatrix}\begin{pmatrix} 0 &0&0\\ &0&0\\&&0\end{pmatrix} &  \begin{array}{l} b_{11} (-a^2\ell + b^2+c^2\ell - d^2)\end{array} & 0 &0\\
   \hline
   5. & \begin{pmatrix} 0 &0&0\\ &0&0\\&&1\end{pmatrix}\begin{pmatrix} 0&0&1\\ &0&0\\&&0\end{pmatrix} & \begin{array}{l} b_{11} (b^2-d^2) +b_{22}a_{11}d \end{array}  & -(p-1)^3p^4(2p-1) & \frac{(p-1)p^{15}(2p-1)}{2}\\
   \hline
   6. & \begin{pmatrix} 0 &0&0\\ &1 &0\\&&0\end{pmatrix}\begin{pmatrix}0 &0&1\\ &0&0\\&&0\end{pmatrix} & \begin{array}{l}b_{11}(a^2 - c^2) + b_{22}a_{11}d\end{array} & (p-1)^3p^4 & -\frac{(p-1)p^{15}}{2}\\
   \hline
   7. & \begin{pmatrix} 0 &0 &0\\ &0 & \frac{1}{2}\\ &&0\end{pmatrix}\begin{pmatrix}0 &0&1\\ &0&0\\&&0\end{pmatrix} & \begin{array}{l}b_{11}(ab-cd)+b_{22}a_{11} d \end{array} & 2(p-1)^3p^4 & - (p-1)p^{15}\\
   \hline
   8. & \begin{pmatrix} 0 &0&0\\ &1&1\\&&0\end{pmatrix} \begin{pmatrix} 0&0&1\\&0&0\\&&0\end{pmatrix} &  \begin{array}{l}b_{11} ( a (a + 2b) - c(c+2d))\\+b_{22}a_{11}d \end{array} & -2(p-1)^2p^4 & p^{15}\\
   \hline
   9. & \begin{pmatrix} 0 &0&0\\ &-\ell &0\\ &&1\end{pmatrix}\begin{pmatrix} 0&0&1\\ &0&0\\&&0\end{pmatrix} & \begin{array}{l}b_{11} (-a^2\ell+ b^2 + c^2\ell -d^2)\\+b_{22}a_{11}d  \end{array} &0 &0\\
   \hline
   10. & \begin{pmatrix} 0 &0&0\\ &0&0\\&&0\end{pmatrix}\begin{pmatrix}0 &0&1 \\ & 0&0\\ &&0\end{pmatrix} &   b_{22}a_{11}d &-(p-1)^4p^4 & \frac{(p-1)^2p^{15}}{2}\\
   \hline

   \end{array}
 \end{equation*}

 \begin{equation*}
   \begin{array}{|l|l|l|l|l|}
   \hline
   \text{Orbit} & \xi_1 & [ x_0, g \cdot \xi_1 ] & \sS(x,\xi) &\sM(x,\xi)\\
   \hline   11. & \begin{pmatrix} 0&0&0\\&0&0\\&&0\end{pmatrix}\begin{pmatrix}0&0&0\\&0&0\\&&1\end{pmatrix} &  b_{12}(b^2 -d^2)+b_{22}a_{13}d & 0 &0\\
   \hline
   12. & \begin{pmatrix} 0&0&0\\&0&1\\&&0\end{pmatrix}\begin{pmatrix}0&0&0\\&0&0\\&&1\end{pmatrix} & \begin{array}{l} b_{11}(2ab - 2cd) + b_{12}(b^2-d^2)\\+b_{22}a_{13}d\end{array} &0&0\\
   \hline
   13. & \begin{pmatrix} 0&0&0\\&1&0\\&&0\end{pmatrix}\begin{pmatrix}0&0&0\\&0&0\\&&1\end{pmatrix}& \begin{array}{l} b_{11}(a^2-c^2) + b_{12}(b^2 -d^2) \\+b_{22}a_{13}d \end{array}&0&0\\
   \hline
   14. & \begin{pmatrix} 0&0&0\\&0&0\\&&0\end{pmatrix}\begin{pmatrix}0&1&0\\&0&0\\&&1\end{pmatrix}& \begin{array}{l} b_{12}(b^2 - d^2)\\+b_{22}(a_{11}c+a_{13}d) \end{array}&0&0\\
   \hline
   15. & \begin{pmatrix} 0&0&0\\&0&1\\&&0\end{pmatrix}\begin{pmatrix}0&1&0\\&0&0\\&&1\end{pmatrix}& \begin{array}{l} b_{11}(2ab - 2cd)+ b_{12}(b^2 - d^2)\\+b_{22}(a_{11}c+a_{13}d)  \end{array} & 0&0\\
   \hline
   16. &\begin{pmatrix} 0&0&0\\&1&0\\&&0\end{pmatrix}\begin{pmatrix}0&1&0\\&0&0\\&&1\end{pmatrix} & \begin{array}{l}   b_{11}(a^2 -c^2) + b_{12}(b^2 -d^2)\\+b_{22}(a_{11}c+a_{13}d) \end{array}& 0&0\\
   \hline
   17. & \begin{pmatrix} 0&0&0\\&0&0\\&&0\end{pmatrix}\begin{pmatrix}0&0&0\\&0&1\\&&0\end{pmatrix}&b_{12}(2ab - 2cd)+ b_{22}(a_{12}d + a_{13}c)  &0&0\\
   \hline
   18. &\begin{pmatrix} 0&0&0\\&0&0\\&&1\end{pmatrix}\begin{pmatrix}0&0&0\\&0&1\\&&0\end{pmatrix} & \begin{array}{l}b_{11}(b^2 - d^2)+ b_{12}(2ab - 2cd) \\+b_{22}(a_{12}d + a_{13}c)  \end{array}&0&0\\
   \hline
   19. & \begin{pmatrix} 0&0&0\\&1&0\\&&1\end{pmatrix}\begin{pmatrix}0&0&0\\&0&1\\&&0\end{pmatrix}& \begin{array}{l}b_{11}(a^2 + b^2 - c^2 -d^2)\\+ b_{12}(2ab -2cd)+b_{22}(a_{12}d + a_{13}c)  \end{array}&0&0\\
   \hline
   20. &\begin{pmatrix} 0&0&0\\&\ell&0\\&&1\end{pmatrix}\begin{pmatrix}0&0&0\\&0&1\\&&0\end{pmatrix} & \begin{array}{l}b_{11}(a^2\ell +b^2 - c^2 \ell - d^2 )\\+ b_{12}(2ab-2cd)+b_{22}(a_{12}d + a_{13}c) \end{array} &0&0\\
   \hline
   21. & \begin{pmatrix} 0 &0&0\\&0&0\\&&0\end{pmatrix}\begin{pmatrix}0&0&0\\ &-\ell &0\\&&1\end{pmatrix} & \begin{array}{l}b_{12}(- a^2\ell + b^2 +  c^2\ell-d^2)\\+b_{22}(- a_{12}c\ell + a_{13}d)\end{array} &0&0\\
   \hline
   22. & \begin{pmatrix} 0 &0&0\\&1&0\\&&0\end{pmatrix}\begin{pmatrix}0&0&0\\ &-\ell &0\\&&1\end{pmatrix} & \begin{array}{l}b_{11}(a^2 - c^2)\\ +b_{12}(- a^2\ell + b^2 +  c^2\ell - d^2)\\+b_{22}(- a_{12}c\ell + a_{13}d) \end{array}&0 &0\\
   \hline
   23a. &\begin{pmatrix} 0 &0&0\\&0&1\\&&0\end{pmatrix}\begin{pmatrix}0&0&0\\ &-\ell &0\\&&1\end{pmatrix} & \begin{array}{l} b_{11}(2ab - 2cd)\\ +b_{12}(- a^2\ell + b^2 +  c^2\ell - d^2)\\+ b_{22}(- a_{12}c\ell + a_{13}d) \end{array} &0&0\\
   \hline
   23b. &\begin{pmatrix} 0 &0&0\\&1&k\\&&0\end{pmatrix}\begin{pmatrix}0&0&0\\ &-\ell &0\\&&1\end{pmatrix} & \begin{array}{l} b_{11}(a^2 + 2ab k-c^2 -2cdk)\\ +b_{12}(- a^2\ell + b^2 +  c^2\ell - d^2)\\ +b_{22}(- a_{12}c\ell + a_{13}d) \end{array} &0&0\\
   \hline
  \end{array}  
 \end{equation*}
}

\begin{enumerate}
 \item [1.]  In this case $[ x_0, g \cdot \xi_1 ] = 0$, so
 \begin{equation}
  \sS(x, \xi) = |G_{x, \xi}| = (p-1)^5p^4 (p+1).
 \end{equation}
 \item [2.]  In this case
 \begin{equation}
  [ x_0, g\cdot \xi_1] = b_{11} (b^2 - d^2).
 \end{equation}
The fibration gives
\begin{align*}
 \Sigma_1 &= \sum_{\begin{pmatrix} * & * & *\\ & a & b\\ & c & d\end{pmatrix}\begin{pmatrix} b_{11} & *\\ & *\end{pmatrix}} e_p(b_{11}(b^2-d^2)).
 \end{align*}
 Substitute $b:= b+d$, $d:= b-d$ to obtain
 \begin{align*}
 \Sigma_1 = (p-1)^2p^5 \sum_{b_{11} \in \bF_p^\times, b,d \in \bF_p} e_p(b_{11}bd).
 \end{align*}
 The sum over $d$ vanishes unless $b = 0$, so $\Sigma_1 = (p-1)^3 p^6$.  The remaining summations may be performed similarly,
 \begin{align*}
 \Sigma_2 &= \sum_{\begin{pmatrix} * & * & *\\ & b \lambda  & b\\ & d \lambda  & d\end{pmatrix}\begin{pmatrix} b_{11} & *\\ & *\end{pmatrix}} e_p(b_{11}(b^2-d^2))\\
 \notag &= (p-1)^3 p^3 \sum_{b_{11} \in \bF_p^\times, b,d \in \bF_p} e_p(b_{11}bd)\\
 \notag &= (p-1)^4 p^4\\
 \Sigma_3 &= \sum_{\begin{pmatrix} * & * & *\\ & a & 0\\ & c & 0\end{pmatrix} \begin{pmatrix} b_{11} & *\\ & *\end{pmatrix}} 1 = (p-1)^3p^5\\
 \Sigma_4 &= \sum_{\begin{pmatrix} * & * & *\\ & 0 & b\\ & 0 & d\end{pmatrix} \begin{pmatrix} b_{11} & *\\ & *\end{pmatrix}} e_p(b_{11}(b^2-d^2))\\
 \notag &= (p-1)^2p^3 \sum_{b_{11} \in \bF_p^\times, b, d \in \bF_p} e_p(b_{11}bd)\\
 \notag &= (p-1)^3p^4\\
 \Sigma_5 &= p\sum_{\begin{pmatrix} * & * & *\\ & 0 & 0\\ & 0 & 0\end{pmatrix} \begin{pmatrix} b_{11} & *\\ & *\end{pmatrix}}1=(p-1)^3 p^4.
\end{align*}
Thus
\begin{equation}
 \sS(x, \xi) = \Sigma_1 - \Sigma_2 - \Sigma_3 - \Sigma_4 + \Sigma_5 = (p-1)^5 p^4.
\end{equation}
\item [3.]   In this case
\begin{equation}
 [ x_0, g \cdot \xi_1 ] = b_{11}(ab - cd).
\end{equation}
The fibration gives
\begin{align*}
 \Sigma_1 &= \sum_{\begin{pmatrix} * & * & *\\ & a & b\\ & c & d\end{pmatrix}\begin{pmatrix} b_{11} & *\\ & *\end{pmatrix}} e_p(b_{11}(ab-cd))
 \end{align*}
 Summation in $b, d$ vanish unless $a = c = 0$.  Thus $\Sigma_1 = (p-1)^3p^5.$
 \begin{align*}
 \Sigma_2 &= \sum_{\begin{pmatrix} * & * & *\\ & b \lambda  & b\\ &d \lambda  & d\end{pmatrix} \begin{pmatrix} b_{11} & *\\ & *\end{pmatrix}} e_p(b_{11} (b^2-d^2)\lambda)
\end{align*}
Substitute $b:= b+d$, $d:= b-d$ so $\Sigma_2  = (p-1)^4p^4$.
\begin{align*}
 \Sigma_3 &= \sum_{\begin{pmatrix} * & * & *\\ & a & 0\\ & c &0\end{pmatrix}\begin{pmatrix} * & *\\ & *\end{pmatrix}} 1 = (p-1)^3p^5\\
 \Sigma_4 &= \Sigma_3 = (p-1)^3p^5\\
 \Sigma_5 &= (p-1)^3p^4.
\end{align*}
Thus
\begin{equation}
 \sS(x,\xi) = (p-1)^3p^4 [p - (p-1) -2p + 1] = -2(p-1)^4p^4.
\end{equation}
\item [4.]  In this case
\begin{equation}
 [ x_0, g\cdot \xi_1 ] = b_{11}(- a^2\ell + b^2 + c^2\ell - d^2) = b_{11} (-(a^2-c^2)\ell + b^2 - d^2).
\end{equation}
The fibration gives
\begin{align*}
 \Sigma_1 &= \sum_{\begin{pmatrix} * & * & *\\ & a & b\\ & c & d \end{pmatrix}\begin{pmatrix} b_{11} & *\\ & *\end{pmatrix}}e_p(b_{11}(-(a^2 - c^2)\ell + (b^2 - d^2)))\\
 &= (p-1)^2 p^3 \sum_{a,b,c,d \in \bF_p, b_{11} \in \bF_p^\times} e_p(b_{11}(- ac\ell + bd))\\
 &= (p-1)^3 p^5\\
 \Sigma_2 &= \sum_{\begin{pmatrix} * & * & *\\ & b \lambda  & b\\ & d \lambda  & d \end{pmatrix} \begin{pmatrix} b_{11} & *\\ & *\end{pmatrix}} e_p(b_{11}(b^2-d^2)(- \lambda^2\ell + 1))
\end{align*}
  Since $-\lambda^2 \ell + 1 \neq 0$ this may be absorbed into summation over $b_{11}$.  Substitute $b := b+d$, $d:= b-d$ to obtain $\Sigma_2 = (p-1)^4 p^4$.  Similarly,
\begin{align*}
 \Sigma_3 &= \sum_{\begin{pmatrix} * & * & *\\ & a & 0\\ & c & 0\end{pmatrix} \begin{pmatrix} b_{11} & * \\ & *\end{pmatrix}} e_p(-b_{11}(a^2 -c^2)\ell) = (p-1)^3p^4\\
 \Sigma_4 &= \sum_{\begin{pmatrix} * & * & *\\ & 0 & b\\ & 0 & d\end{pmatrix} \begin{pmatrix} b_{11} & * \\ & *\end{pmatrix}} e_p(b_{11}(b^2-d^2)) = (p-1)^3p^4\\
 \Sigma_5 &= (p-1)^3p^4.
\end{align*}
Thus
\begin{equation}
 \sS(x,\xi) = (p-1)^3p^4 (p - (p-1) - 1) = 0.
\end{equation}
\item[5.]   In this case
\begin{equation}
 [ x_0, g \cdot \xi_1] = b_{11} (b^2-d^2) +b_{22}a_{11}d .
\end{equation}
Summing in $b_{22}$ gives $-1$ if $d \neq 0$ and $p-1$ if $d = 0$.  Thus, 
\begin{align}
 \sS(x,\xi) &= - \frac{1}{p-1} \sum_{g \in G_{x,\xi}^t} e_p(b_{11}(b^2-d^2)) + \frac{p}{p-1} \sum_{\begin{pmatrix} * & * & *\\ & a & b\\ & c &\end{pmatrix} \begin{pmatrix} b_{11} & * \\ & *\end{pmatrix}} e_p(b_{11} b^2)\\
 \notag &= - \frac{ 2.}{p-1} - (p-1)^3 p^5\\
 \notag &= - (p-1)^4p^4 - (p-1)^3 p^5 \\ \notag &= - (p-1)^3p^4 (2p-1).
\end{align}
\item[6.]  In this case
\begin{equation}
 [ x_0, g \cdot \xi_1 ] = b_{11}(a^2 - c^2) + b_{22}a_{11}d.
\end{equation}
Sum over $b_{22}$ as before to obtain
\begin{equation}
 \sS(x,\xi) = -\frac{1}{p-1} \sum_{g \in G_{x,\xi}^t}e_p(b_{11}(a^2-c^2)) + \frac{p}{p-1}\sum_{\begin{pmatrix} * & * & *\\ & a&b\\ & c &\end{pmatrix} \begin{pmatrix} b_{11} & *\\ & *\end{pmatrix}} e_p(b_{11}(a^2-c^2)).
\end{equation}
The $c = 0$ term in the second sum may be added without altering the sum, since \begin{equation}\sum_{b_{11} \in \bF_p^\times, a \in \bF_p} e_p(b_{11}a^2) = 0.\end{equation}  Thus, substituting $a+c$ and $a-c$ for $a$ and $c$ as before,
\begin{align}
\sS(x, \xi) &= - \frac{2.}{p-1} + (p-1)^3 p^5\\
\notag &= -(p-1)^4p^4 + (p-1)^3p^5 \\ \notag &= (p-1)^3p^4.
\end{align}
\item[7.]  In this case
\begin{equation}
 [ x_0, g \cdot \xi_1 ] = b_{11}(ab-cd)+b_{22}a_{11} d .
\end{equation}
Sum in $b_{22}$ to find
\begin{align*}
 \sS(x,\xi) &= - \frac{1}{p-1} \sum_{g \in G_{x, \xi}^t} e_p(b_{11}(ab-cd)) + \frac{p}{p-1} \sum_{\begin{pmatrix} * & * & *\\ & a & b\\ & c &0\end{pmatrix} \begin{pmatrix} b_{11} & *\\ & *\end{pmatrix}} e_p(b_{11}ab)\\
  &= -\frac{3.}{p-1} = 2 (p-1)^3p^4.
\end{align*}
\item[8.]   In this case
\begin{equation}
 [ x_0, g\cdot \xi_1 ] = b_{11} ( a (a + 2b) - c(c+2d))+b_{22}a_{11}d .
\end{equation}
Sum in $b_{22}$ to find
\begin{align*}
 \sS(x,\xi) &= - \frac{1}{p-1} \sum_{g \in G_{x,\xi}^t} e_p(b_{11}(a(a+2b) - c(c+2d)))\\ & \qquad + \frac{p}{p-1} \sum_{\begin{pmatrix} * & * & *\\ & a& b\\ & c& \end{pmatrix} \begin{pmatrix} b_{11} & *\\ & *\end{pmatrix}}e_p(b_{11}(a(a+2b) - c^2))
\end{align*}
The first sum is the sum from 3.  In the second sum, replace $a:= a+b$ to obtain
\begin{align*}
 \sS(x,\xi) &= - \frac{3.}{p-1}  + (p-1)p^4 \sum_{b_{11}, b, c \in \bF_p^\times, a \in \bF_p} e_p(b_{11}(a^2 - b^2 -c^2))
 \end{align*}
 The sum may be evaluated in terms of the Gauss sum \begin{equation}\tau = \sum_{n \in \bF_p} e_p\left(n^2 \right).\end{equation}  Write $\left(\frac{\cdot}{p}\right)$ for the Legendre symbol which is multiplicative and satisfies $\left(\frac{\ell}{p}\right) = -1$.  When $b_{11} \neq 0$, \begin{equation}\sum_{n \in \bF_p} e_p\left(b_{11}n^2\right) = \tau \left(\frac{b_{11}}{p}\right).\end{equation}  Also, \begin{equation}\sum_{b_{11} \in \bF_p^\times} \left(\frac{b_{11}}{p}\right) = 0.\end{equation}  The Gauss sum satisfies \begin{equation}\tau^2 = \left(\frac{-1}{p}\right)p.\end{equation}  With these facts in hand, 
 \begin{align*}
 \sS(x, \xi)&= 2 (p-1)^3 p^4  + (p-1)p^4 \sum_{b_{11} \in \bF_p^\times} \tau \left( \frac{b_{11}}{p}\right)\left(\tau \left(\frac{-b_{11}}{p}\right) - 1\right)^2\\
 &= 2(p-1)^3p^4 - 2 (p-1)^2 p^5\\
 &= -2(p-1)^2p^4.
\end{align*}
\item[9.]   In this case
\begin{equation}
 [ x_0, g \cdot \xi_1 ] = b_{11} (-(a^2 - c^2)\ell+ (b^2-d^2))+b_{22}a_{11}d .
\end{equation}
Sum in $b_{22}$ to find
\begin{align*}
 \sS(x,\xi) &= - \frac{4.}{p-1} + \frac{p}{p-1}\sum_{\begin{pmatrix} * & * & *\\ & a & b\\ & c &0\end{pmatrix} \begin{pmatrix} b_{11} & *\\ & *\end{pmatrix}}e_p(b_{11}(- a^2\ell + b^2 +  c^2\ell))\\
 &= (p-1)p^4 \sum_{b_{11} \in \bF_p^\times} \tau\left(\frac{- b_{11}\ell}{p}\right) \left(\tau \left(\frac{b_{11}}{p}\right) -1\right) \left(\tau \left(\frac{b_{11}\ell}{p}\right) -1\right)\\ &=0.
\end{align*}
\item[10.] In this case
\begin{equation}
 [ x_0, g \cdot \xi_1 ] =  b_{22}a_{11}d.
\end{equation}
Summing in $b_{22}$,
\begin{align*}
 \sS(x, \xi) &= -\frac{|G_{x, \xi}|}{p-1} + p \# \begin{pmatrix} * & * & *\\ & * & *\\ & * &\end{pmatrix}\begin{pmatrix} * & *\\ & 1\end{pmatrix}\\
 & = - (p-1)^4 p^4 (p+1) + (p-1)^4 p^5\\ &=- (p-1)^4 p^4.
\end{align*}
\item[11.-23.] These sums vanish on summing in $a_{12}$, $a_{13}$ and $b_{12}$.
\end{enumerate}

\subsection{The exponential sums pair $(\sO_{1^22}, \sO_{D1^2})$}\label{O122_OD12_section}
 In this case, 
 \begin{equation}x_0 =\begin{pmatrix} 0 & 0 &0\\ &1 &0\\ &&-\ell\end{pmatrix}\begin{pmatrix} 0 & 0 & \frac{1}{2}\\ &0 &0\\ &&0\end{pmatrix}, \qquad \xi_0 = \begin{pmatrix} 1 & 0 &0\\ &0&0\\&&0\end{pmatrix}\begin{pmatrix} 0&0&0\\ &0&0\\&&0\end{pmatrix} 
 \end{equation}
 and $x_1 = 0$.
 The action set is 
 \begin{equation}
G_{x,\xi}^t= \begin{pmatrix} a_{11} & a_{12} & a_{13}\\ & a & b\\ & c & d\end{pmatrix}\begin{pmatrix} b_{11} & b_{12}\\ & b_{22}\end{pmatrix},
 \end{equation}
 which is fibered as before.
The stabilizer has size $\left|\Stab_{G(\zed/p^2\zed)}(x)\right| = 2(p-1)^2p^2.$
 
 Representatives and exponential sum pairings are given in the table below. Since this is the single non-maximal orbit,
 \begin{equation}
  \sM(x,\xi) = -\frac{p^{11}\sS(x,\xi)}{2(p-1)^2}.
 \end{equation}

 See \textbf{exponential\_sums\_O122\_OD12.nb}.
 {\tiny
 \begin{equation*}
  \begin{array}{|l|l|l|l|l|}
   \hline
   \text{Orbit} & \xi_1 & [ x_0, g \cdot \xi_1 ] &\sS(x,\xi)& \sM(x,\xi)\\
   \hline
   1. & \begin{pmatrix} 0 &0&0\\&0&0\\&&0\end{pmatrix}\begin{pmatrix} 0&0&0\\&0&0\\&&0\end{pmatrix} & 0 &(p-1)^5p^4(p+1) & -\frac{(p-1)^3p^{15}(p+1)}{2}\\
   \hline
   2. & \begin{pmatrix} 0 & 0&0\\ &0&0\\ &&1\end{pmatrix}\begin{pmatrix} 0&0&0\\&0&0\\&&0\end{pmatrix} & b_{11} (b^2 - d^2\ell)& -(p-1)^4p^4(p+1)& \frac{(p-1)^2p^{15}(p+1)}{2} \\
   \hline
   3. & \begin{pmatrix} 0 &0&0\\&0&\frac{1}{2}\\ &&0\end{pmatrix}\begin{pmatrix}0&0&0\\&0&0\\&&0\end{pmatrix} & b_{11}(ab - cd\ell) &0 &0\\
   \hline
   4. & \begin{pmatrix} 0 & 0 & 0\\ & -\ell &0\\ &&1\end{pmatrix}\begin{pmatrix} 0 &0&0\\ &0&0\\&&0\end{pmatrix} &  \begin{array}{l} b_{11} (- a^2\ell + b^2 +  c^2\ell^2 -  d^2\ell)\end{array} & 2(p-1)^3p^4(p+1) & -(p-1)p^{15}(p+1) \\
   \hline
   5. & \begin{pmatrix} 0 &0&0\\ &0&0\\&&1\end{pmatrix}\begin{pmatrix} 0&0&1\\ &0&0\\&&0\end{pmatrix} & \begin{array}{l}b_{11} (b^2- d^2\ell)+b_{22}a_{11} d  \end{array} & (p-1)^3p^4 & -\frac{(p-1)p^{15}}{2} \\
   \hline
   6. & \begin{pmatrix} 0 &0&0\\ &1 &0\\&&0\end{pmatrix}\begin{pmatrix}0 &0&1\\ &0&0\\&&0\end{pmatrix} &b_{11}(a^2 -  c^2\ell)+b_{22}a_{11}d  & (p-1)^3p^4 & -\frac{(p-1)p^{15}}{2}\\
   \hline
   7. & \begin{pmatrix} 0 &0 &0\\ &0 & \frac{1}{2}\\ &&0\end{pmatrix}\begin{pmatrix}0 &0&1\\ &0&0\\&&0\end{pmatrix} & b_{11}(ab- cd\ell)+b_{22}a_{11}d  & 0&0\\
   \hline
   8. & \begin{pmatrix} 0 &0&0\\ &1&1\\&&0\end{pmatrix} \begin{pmatrix} 0&0&1\\&0&0\\&&0\end{pmatrix} &  \begin{array}{l}b_{11} ( a (a + 2b) -  c(c+2d)\ell)\\+b_{22}a_{11}d\end{array} &0&0\\
   \hline
   9. & \begin{pmatrix} 0 &0&0\\ &-\ell &0\\ &&1\end{pmatrix}\begin{pmatrix} 0&0&1\\ &0&0\\&&0\end{pmatrix} & \begin{array}{l}b_{11} (- a^2\ell + b^2 +  c^2\ell^2- d^2\ell)\\+b_{22}a_{11}d \end{array} & -2(p-1)^2p^4 & p^{15}\\
   \hline
   10. & \begin{pmatrix} 0 &0&0\\ &0&0\\&&0\end{pmatrix}\begin{pmatrix}0 &0&1 \\ & 0&0\\ &&0\end{pmatrix} &  b_{22}a_{11}d &-(p-1)^4p^4 & \frac{(p-1)^2p^{15}}{2}\\
   \hline
   \end{array}
 \end{equation*}

 \begin{equation*}
   \begin{array}{|l|l|l|l|l|}
   \hline
   \text{Orbit} & \xi_1 & [ x_0, g \cdot \xi_1 ] & \sS(x, \xi) & \sM(x,\xi)\\
   \hline   11. & \begin{pmatrix} 0&0&0\\&0&0\\&&0\end{pmatrix}\begin{pmatrix}0&0&0\\&0&0\\&&1\end{pmatrix} &  b_{12}(b^2 - d^2\ell)+b_{22}a_{13}d &0 &0\\
   \hline
   12. & \begin{pmatrix} 0&0&0\\&0&1\\&&0\end{pmatrix}\begin{pmatrix}0&0&0\\&0&0\\&&1\end{pmatrix} & \begin{array}{l}  b_{11}(2ab - 2 cd\ell) + b_{12}(b^2- d^2\ell)\\+b_{22}a_{13}d\end{array}&0&0\\
   \hline
   13. & \begin{pmatrix} 0&0&0\\&1&0\\&&0\end{pmatrix}\begin{pmatrix}0&0&0\\&0&0\\&&1\end{pmatrix}& \begin{array}{l} b_{11}(a^2- c^2\ell) + b_{12}(b^2 - d^2\ell)\\+b_{22}a_{13}d\end{array}&0&0\\
   \hline
   14. & \begin{pmatrix} 0&0&0\\&0&0\\&&0\end{pmatrix}\begin{pmatrix}0&1&0\\&0&0\\&&1\end{pmatrix}& \begin{array}{l} b_{12}(b^2 -  d^2\ell)\\+b_{22}(a_{11}c+ a_{13}d)\end{array}&0&0\\
   \hline
   15. & \begin{pmatrix} 0&0&0\\&0&1\\&&0\end{pmatrix}\begin{pmatrix}0&1&0\\&0&0\\&&1\end{pmatrix}& \begin{array}{l} b_{11}(2ab - 2 cd\ell)+ b_{12}(b^2 -  d^2\ell)\\+b_{22}(a_{11}c+a_{13}d)  \end{array}&0&0\\
   \hline
   16. &\begin{pmatrix} 0&0&0\\&1&0\\&&0\end{pmatrix}\begin{pmatrix}0&1&0\\&0&0\\&&1\end{pmatrix} & \begin{array}{l} b_{11}(a^2 - c^2\ell)+ b_{12}(b^2 - d^2\ell)\\+b_{22}(a_{11}c+a_{13}d) \end{array} &0&0\\
   \hline
   17. & \begin{pmatrix} 0&0&0\\&0&0\\&&0\end{pmatrix}\begin{pmatrix}0&0&0\\&0&1\\&&0\end{pmatrix}& \begin{array}{l}b_{12}(2ab - 2 cd\ell)\\+b_{22}(a_{12}d + a_{13}c)\end{array} &0&0\\
   \hline
   18. &\begin{pmatrix} 0&0&0\\&0&0\\&&1\end{pmatrix}\begin{pmatrix}0&0&0\\&0&1\\&&0\end{pmatrix} & \begin{array}{l}b_{11}(b^2 -  d^2\ell)+ b_{12}(2ab - 2 cd\ell)\\ +b_{22}(a_{12}d + a_{13}c) \end{array} &0&0\\
   \hline
   19. & \begin{pmatrix} 0&0&0\\&1&0\\&&1\end{pmatrix}\begin{pmatrix}0&0&0\\&0&1\\&&0\end{pmatrix}& \begin{array}{l}b_{11}(a^2 + b^2 - c^2\ell - d^2\ell) \\+ b_{12}(2ab -2 cd\ell)+b_{22}(a_{12}d + a_{13}c) \end{array}&0&0\\
   \hline
   20. &\begin{pmatrix} 0&0&0\\&\ell&0\\&&1\end{pmatrix}\begin{pmatrix}0&0&0\\&0&1\\&&0\end{pmatrix} & \begin{array}{l}b_{11}(a^2\ell +b^2 - c^2 \ell^2 - d^2 \ell )\\+ b_{12}(2ab-2 cd\ell) +b_{22}(a_{12}d + a_{13}c) \end{array}&0&0\\
   \hline
   21. & \begin{pmatrix} 0 &0&0\\&0&0\\&&0\end{pmatrix}\begin{pmatrix}0&0&0\\ &-\ell &0\\&&1\end{pmatrix} & \begin{array}{l} b_{12}(- a^2\ell + b^2 +  c^2\ell^2- d^2\ell)\\+b_{22}(- a_{12}c\ell + a_{13}d)\end{array}&0&0\\
   \hline
   22. & \begin{pmatrix} 0 &0&0\\&1&0\\&&0\end{pmatrix}\begin{pmatrix}0&0&0\\ &-\ell &0\\&&1\end{pmatrix} & \begin{array}{l}b_{11}(a^2 -  c^2\ell)\\+b_{12}(- a^2\ell + b^2 +  c^2\ell^2 -  d^2\ell)\\+b_{22}(- a_{12}c\ell + a_{13}d) \end{array}&0 &0\\
   \hline
   23a. &\begin{pmatrix} 0 &0&0\\&0&1\\&&0\end{pmatrix}\begin{pmatrix}0&0&0\\ &-\ell &0\\&&1\end{pmatrix} & \begin{array}{l}b_{11}(2ab - 2 cd\ell)\\ +b_{12}(- a^2\ell + b^2 + c^2\ell^2  -  d^2\ell)\\+ b_{22}(- a_{12}c\ell + a_{13}d) \end{array}&0&0\\
   \hline
   23b. &\begin{pmatrix} 0 &0&0\\&1&k\\&&0\end{pmatrix}\begin{pmatrix}0&0&0\\ &-\ell &0\\&&1\end{pmatrix} & \begin{array}{l} b_{11}(a^2 + 2ab k- c^2\ell -2 cd\ell k) \\ +b_{12}(- a^2\ell + b^2 +  c^2\ell^2 -  d^2\ell)\\ +b_{22}(- a_{12}c\ell + a_{13}d)\end{array}&0&0\\
   \hline
  \end{array}  
 \end{equation*}
}
 
 The exponential sums are as follows.
\begin{enumerate}
 \item [1.]  In this case $[ x_0, g\cdot \xi_1] = 0$ so that
 \begin{equation}
  \sS(x, \xi) = |G_{x, \xi}| = (p-1)^5p^4 (p+1).
 \end{equation}
 \item [2.]  In this case
 \begin{equation}
  [ x_0, g\cdot \xi_1] = b_{11} (b^2 -  d^2\ell).
 \end{equation}
 Since $b, d$ are not 0 simultaneously, $b^2 -  d^2\ell \neq 0$, and the sum over $b_{11}$ is $-1$, so
 \begin{equation}
  \sS(x,\xi) = -\frac{|G_{x, \xi}|}{p-1} = -(p-1)^4 p^4 (p+1).
 \end{equation}

\item [3.]   In this case
\begin{equation}
 [ x_0, g \cdot \xi_1 ] = b_{11}(ab -   cd\ell).
\end{equation}
The fibration gives
\begin{align*}
 \Sigma_1 &= \sum_{\begin{pmatrix} * & * & *\\ & a & b\\ & c & d\end{pmatrix}\begin{pmatrix} b_{11} & *\\ & *\end{pmatrix}} e_p(b_{11}(ab-  cd\ell))= (p-1)^3 p^5\\
 \Sigma_2 &= \sum_{\begin{pmatrix} * & * & *\\ & b \lambda  & b\\ &d \lambda  & d\end{pmatrix} \begin{pmatrix} b_{11} & *\\ & *\end{pmatrix}} e_p(b_{11} (b^2- d^2\ell)\lambda)
\end{align*}
Summation in $b$ and $d$ gives $\tau^2 \left( \frac{-\ell}{p}\right) = -p$, independent of $b_{11}\lambda$.  Thus $\Sigma_2 = -(p-1)^4p^4.$
\begin{align*}
 \Sigma_3 &= \sum_{\begin{pmatrix} * & * & *\\ & a & 0\\ & c &0\end{pmatrix}\begin{pmatrix} * & *\\ & *\end{pmatrix}} 1 = (p-1)^3p^5\\
 \Sigma_4 &= \Sigma_3 = (p-1)^3p^5\\
 \Sigma_5 &= (p-1)^3p^4
\end{align*}
Thus
\begin{equation}
 \sS(x,\xi) = 0.
\end{equation}
\item [4.]  In this case
\begin{equation}
 [ x_0, g\cdot \xi_1 ] = b_{11}(- a^2\ell + b^2 +  c^2\ell^2 -  d^2\ell) .
\end{equation}
The fibration gives
\begin{align*}
 \Sigma_1 &= \sum_{\begin{pmatrix} * & * & *\\ & a & b\\ & c & d \end{pmatrix}\begin{pmatrix} b_{11} & *\\ & *\end{pmatrix}}e_p(b_{11}(- a^2\ell + b^2 +  c^2\ell^2 -  d^2\ell))
\end{align*}
Expanding in Gauss sums, summation in $a,b,c,d$ gives $p^2$ independent of $b_{11}$, so $\Sigma_1 = (p-1)^3p^5.$
\begin{align*}
 \Sigma_2 &= \sum_{\begin{pmatrix} * & * & *\\ & b \lambda  & b\\ & d \lambda  & d \end{pmatrix} \begin{pmatrix} b_{11} & *\\ & *\end{pmatrix}} e_p(b_{11}(b^2- d^2\ell)(- \lambda^2\ell + 1))
 \end{align*}
 Since $-\lambda^2 \ell + 1 \neq 0$, summation in $b^2$ and $d^2$ obtains $-p$ as in $\Sigma_2$ of 3., so that $\Sigma_2 = -(p-1)^4p^4$. Similarly,
 \begin{align*}
 \Sigma_3 &= \sum_{\begin{pmatrix} * & * & *\\ & a & 0\\ & c & 0\end{pmatrix} \begin{pmatrix} b_{11} & * \\ & *\end{pmatrix}} e_p(b_{11}(- a^2\ell + c^2\ell^2)) = -(p-1)^3p^4\\
 \Sigma_4 &= \sum_{\begin{pmatrix} * & * & *\\ & 0 & b\\ & 0 & d\end{pmatrix} \begin{pmatrix} b_{11} & * \\ & *\end{pmatrix}} e_p(b_{11}(b^2- d^2\ell)) = -(p-1)^3p^4\\
 \Sigma_5 &= (p-1)^3p^4.
\end{align*}
Thus
\begin{equation}
 \sS(x,\xi) = 2(p-1)^3p^4(p+1).
\end{equation}
\item[5.]   In this case
\begin{equation}
 [ x_0, g \cdot \xi_1] = b_{11} (b^2- d^2\ell)+b_{22}a_{11} d .
\end{equation}
The sum over $b_{11}$ is $-1$ since $b^2 -  d^2\ell \neq 0$.  Sum in $b_{22}$ to find
\begin{align*}
 \sS(x,\xi) &= \frac{|G_{x,\xi}|}{(p-1)^2} - \frac{p}{(p-1)^2} \#\begin{pmatrix} * & * & *\\ & * & *\\ & * &\end{pmatrix}\begin{pmatrix} * & *\\ & *\end{pmatrix}\\
 &= (p-1)^3 p^4(p+1) - (p-1)^3p^5 \\
 &= (p-1)^3 p^4.
\end{align*}

\item[6.]  In this case
\begin{equation}
 [ x_0, g \cdot \xi_1 ] = b_{11}(a^2 -  c^2\ell)+b_{22}a_{11}d.
\end{equation}
As for 5., $\sS(x, \xi) = (p-1)^3p^4$.

\item[7.]   In this case
\begin{equation}
 [ x_0, g \cdot \xi_1 ] =b_{11}(ab- cd\ell)+b_{22}a_{11}d .
\end{equation}
Sum in $b_{22}$ to find
\begin{align*}
 \sS(x,\xi) &= - \frac{1}{p-1} \sum_{g \in G_{x, \xi}^t} e_p(b_{11}(ab- cd\ell)) + \frac{p}{p-1} \sum_{\begin{pmatrix} * & * & *\\ & a & b\\ & c &0\end{pmatrix} \begin{pmatrix} b_{11} & *\\ & *\end{pmatrix}} e_p(b_{11}ab)
 \end{align*}
 The first sum is the same as for 3., which is 0, while in the second sum, $b_{11}b \neq 0$ so that the sum over $a$, which now runs over $\bF_p$, is 0.  Thus $\sS(x, \xi) = 0$.

\item[8.]  In this case
\begin{equation}
 [ x_0, g\cdot \xi_1 ] =b_{11} ( a (a + 2b) -  c(c+2d)\ell)+b_{22}a_{11}d .
\end{equation}
Sum in $b_{22}$ to find
\begin{align*}
 \sS(x,\xi) &= - \frac{1}{p-1} \sum_{g \in G_{x,\xi}^t} e_p(b_{11}(a(a+2b) -  c(c+2d)\ell))\\ & \qquad + \frac{p}{p-1} \sum_{\begin{pmatrix} * & * & *\\ & a& b\\ & c& \end{pmatrix} \begin{pmatrix} b_{11} & *\\ & *\end{pmatrix}}e_p(b_{11}(a(a+2b) -  c^2\ell))
 \end{align*}
 After making a column operation on $\GL_2$, the first sum is equal to the sum from 3., which is 0.  The second sum may be evaluated by replacing $a:=a+b$, which is permissible since $a$ ranges in $\bF_p$, so that $a(a+2b)$ becomes $a^2-b^2$.  This obtains
\begin{align*}
 \sS(x,\xi)&=  (p-1)p^4 \sum_{b_{11}, b, c \in \bF_p^\times, a \in \bF_p} e_p(b_{11}(a^2 - b^2 - c^2\ell))\\
 &=  (p-1)p^4 \sum_{b_{11} \in \bF_p^\times} \tau \left( \frac{b_{11}}{p}\right)\left(\tau \left(\frac{-b_{11}}{p}\right) - 1\right)\left(\tau \left(\frac{-b_{11}\ell}{p}\right) - 1\right)\\
 &= 0.
\end{align*}
\item[9.]   In this case
\begin{equation}
 [ x_0, g \cdot \xi_1 ] = b_{11} (- a^2\ell + b^2 +  c^2\ell^2- d^2\ell)+b_{22}a_{11}d .
\end{equation}
Sum in $b_{22}$ to find
\begin{align*}
 \sS(x,\xi) &= - \frac{4.}{p-1} + \frac{p}{p-1}\sum_{\begin{pmatrix} * & * & *\\ & a & b\\ & c &\end{pmatrix} \begin{pmatrix} b_{11} & *\\ & *\end{pmatrix}}e_p(b_{11}(- a^2\ell + b^2 +  c^2\ell^2))\\
 &=-2(p-1)^2p^4(p+1) +  (p-1)p^4 \sum_{b_{11} \in \bF_p^\times} \tau\left(\frac{- b_{11}\ell}{p}\right) \left(\tau \left(\frac{b_{11}}{p}\right) -1\right)^2\\ &=-2(p-1)^2p^4(p+1) + 2(p-1)^2p^5 \\&= -2(p-1)^2p^4.
\end{align*}

\item[10.] In this case
\begin{equation}
 [ x_0, g \cdot \xi_1 ] = b_{22}a_{11}d.
\end{equation}
Summing in $b_{22}$,
\begin{align*}
 \sS(x, \xi) &= -\frac{|G_{x, \xi}|}{p-1} + p \# \begin{pmatrix} * & * & *\\ & * & *\\ & * &\end{pmatrix}\begin{pmatrix} * & *\\ & 1\end{pmatrix}\\
 & = - (p-1)^4 p^4 (p+1) + (p-1)^4 p^5\\ &=- (p-1)^4 p^4.
\end{align*}
\item[11.-23.] These sums vanish on summing in $a_{12}$, $a_{13}$ and $b_{12}$.
\end{enumerate}

\subsection{The exponential sums pair $(\sO_{2^2}, \sO_{D11})$}\label{O22_OD11_section}
For this pair, the standard representatives are
\begin{equation}
 x_0 = \begin{pmatrix} 0 &0&0\\ &0&0\\&&1\end{pmatrix}\begin{pmatrix} 1 &0&0\\&-\ell &0\\&&0\end{pmatrix}, \qquad \xi_0 = \begin{pmatrix} 0 &\frac{1}{2} &0\\ &0&0\\&&0\end{pmatrix}\begin{pmatrix}0&0&0\\&0&0\\&&0\end{pmatrix}.
\end{equation}
The acting set is
\begin{equation}
 G_{x, \xi}^t =  \left\{\begin{pmatrix} a &c \lambda \ell  & a_{13}\\ c & a\lambda & a_{23}\\ && a_{33}\end{pmatrix} \begin{pmatrix} b_{11} & b_{12}\\ & b_{22}\end{pmatrix}\right\}.
\end{equation}
The range of summation is $(a,c) \in \bF_p^2 \setminus \{(0,0)\}$, $\lambda, a_{33}, b_{11}, b_{22} \in \bF_p^\times$, and $a_{13}, a_{23}, b_{12} \in \bF_p$.
Thus $|G_{x, \xi}| = (p-1)^5 p^3 (p+1)$. The stabilizer has size
\begin{equation}
 \left|\Stab_{G(\zed/p^2\zed)}(x)\right| = 2(p-1)^2p^3(p+1).
\end{equation}

Representatives and exponential sum pairings are given in the table below. Since this is the single non-maximal orbit,
\begin{equation}
 \sM(x,\xi) =- \frac{p^{10}\sS(x,\xi)}{2(p-1)^2(p+1)}.
\end{equation}

See \textbf{exponential\_sums\_O22\_OD11.nb}.

{\tiny
 \begin{equation*}
  \begin{array}{|l|l|l|l|l|}
   \hline
   \text{Orbit} & \xi_1 & [ x_0, g \cdot \xi_1 ] & \sS(x,\xi)& \sM(x,\xi)\\
   \hline
   1. & \begin{pmatrix} 0 &0&0\\&0&0\\&&0\end{pmatrix}\begin{pmatrix} 0&0&0\\&0&0\\&&1\end{pmatrix} & b_{12}a_{33}^2+ b_{22}(a_{13}^2  - a_{23}^2\ell)  &0&0\\
   \hline
   2. & \begin{pmatrix} 0 &0&0\\&0&0\\&&0\end{pmatrix}\begin{pmatrix} 0&0&0\\&0&1\\&&1\end{pmatrix} &  \begin{array}{l}b_{12}a_{33}^2+\\b_{22}( a_{13}^2  - a_{23}^2\ell - 2a a_{23}\lambda \ell + 2 a_{13}c \lambda \ell) \end{array}&0&0\\
   \hline
   3. &  \begin{pmatrix} 0 &0&0\\&0&0\\&&0\end{pmatrix}\begin{pmatrix} 0&0&1\\&0&1\\&&1\end{pmatrix} & \begin{array}{l}b_{12}a_{33}^2 +\\ b_{22}( 2 a a_{13} + a_{13}^2  - a_{23}^2  \ell - 2 a_{23}c \ell\\ - 2 a a_{23}  \lambda \ell+ 2 a_{13} c  \lambda\ell) \end{array}&0&0\\
   \hline
   4. & \begin{pmatrix} 0 &0&0\\&0&0\\&&0\end{pmatrix}\begin{pmatrix} 0&0&0\\&-\ell&0\\&&1\end{pmatrix} &  \begin{array}{l} b_{12}a_{33}^2 +\\b_{22}(a_{13}^2 - a_{23}^2 \ell + a^2 \lambda^2\ell^2 - c^2  \lambda^2\ell^3)\end{array}&0&0\\
   \hline
   5. & \begin{pmatrix} 0 &0&0\\&0&0\\&&0\end{pmatrix}\begin{pmatrix} 0&0&1\\&-\ell&0\\&&1\end{pmatrix} &  \begin{array}{l} b_{12}a_{33}^2 +\\b_{22}(2a a_{13} + a_{13}^2 - a_{23}^2\ell-2a_{23}c\ell\\ + a^2  \lambda^2\ell^2 - c^2  \lambda^2\ell^3)\end{array}&0&0\\
   \hline
   6. & \begin{pmatrix} 0 &0&0\\&0&0\\&&0\end{pmatrix}\begin{pmatrix} -\ell&0&0\\&-\ell&0\\&&1\end{pmatrix} & \begin{array}{l} b_{12}a_{33}^2 +\\ b_{22}(a_{13}^2 - a^2 \ell - a_{23}^2 \ell + c^2 \ell^2 \\+ a^2  \lambda^2\ell^2 - c^2  \lambda^2\ell^3) \end{array}&0&0 \\
   \hline
   7. & \begin{pmatrix} 0 &0&0\\&0&0\\&&0\end{pmatrix}\begin{pmatrix} 0&0&0\\&0&1\\&&0\end{pmatrix} & b_{22}(-2aa_{23}\lambda \ell + 2a_{13}c \lambda \ell)&0&0\\
   \hline
   8. & \begin{pmatrix} 0 &0&0\\&0&0\\&&0\end{pmatrix}\begin{pmatrix} 0&0&1\\&0&1\\&&0\end{pmatrix}& b_{22}(a_{13}(2a +2c\lambda \ell)-a_{23}(2a\lambda \ell + 2c\ell))&0&0\\
   \hline
   9. & \begin{pmatrix} 0 &0&0\\&0&0\\&&0\end{pmatrix}\begin{pmatrix} 1&0&0\\&0&1\\&&0\end{pmatrix}& b_{22}(a^2 - c^2 \ell -2a a_{23}\lambda \ell + 2 a_{13}c\lambda \ell)&0&0\\
   \hline
   10. & \begin{pmatrix} 0 &0&0\\&0&0\\&&1\end{pmatrix}\begin{pmatrix} 0&0&0\\&0&1\\&&0\end{pmatrix}& b_{11}a_{33}^2 + b_{22}(-2aa_{23}\lambda \ell + 2a_{13}c\lambda \ell) &0&0\\
   \hline
   11. & \begin{pmatrix} 0 &0&0\\&0&0\\&&1\end{pmatrix}\begin{pmatrix} 0&0&1\\&0&1\\&&0\end{pmatrix}& \begin{array}{l}b_{11}a_{33}^2+\\b_{22}(a_{13}(2a +2c\lambda \ell)-a_{23}(2a\lambda \ell + 2c\ell))\end{array} &0&0\\
   \hline
   12. & \begin{pmatrix} 0 &0&0\\&0&0\\&&1\end{pmatrix}\begin{pmatrix} 1&0&0\\&0&1\\&&0\end{pmatrix}& \begin{array}{l}b_{11}a_{33}^2+\\b_{22}(a^2 - c^2 \ell -2a a_{23} \lambda \ell+ 2 a_{13}c  \lambda\ell) \end{array}&0&0 \\
   \hline
   13. & \begin{pmatrix} 0 &0&0\\&0&0\\&&0\end{pmatrix}\begin{pmatrix} 0&0&0\\&1&0\\&&0\end{pmatrix}& b_{22}(-a^2  + c^2 \ell) \lambda^2\ell & -(p-1)^4p^3(p+1) & \frac{(p-1)^2p^{13}}{2}\\
   \hline
   14. & \begin{pmatrix} 0 &0&0\\&0&0\\&&0\end{pmatrix}\begin{pmatrix} 1&0&0\\&1&0\\&&0\end{pmatrix}& b_{22}(a^2 - c^2 \ell)(1-\lambda^2 \ell) & -(p-1)^4p^3(p+1) & \frac{(p-1)^2p^{13}}{2}\\
   \hline
   15. & \begin{pmatrix} 0 &0&0\\&0&0\\&&0\end{pmatrix}\begin{pmatrix} \ell&0&0\\&1&0\\&&0\end{pmatrix} & b_{22}(a^2  - c^2 \ell)(1- \lambda^2)\ell & (p-1)^3p^3 (p+1)^2& - \frac{(p-1)p^{13}(p+1)}{2}\\
   \hline
   16. & \begin{pmatrix} 0 &0&0\\&0&0\\&&1\end{pmatrix}\begin{pmatrix} 0&0&0\\&1&0\\&&0\end{pmatrix} & b_{11}a_{33}^2  + b_{22}(-a^2  +c^2\ell)\lambda^2 \ell&(p-1)^3p^3(p+1) & -\frac{(p-1)p^{13}}{2}\\
   \hline
   17. & \begin{pmatrix} 0 &0&0\\&0&0\\&&1\end{pmatrix}\begin{pmatrix} 1&0&0\\&1&0\\&&0\end{pmatrix} & b_{11}a_{33}^2  + b_{22}(a^2-c^2\ell)(1-\lambda^2 \ell)&(p-1)^3p^3(p+1)& - \frac{(p-1)p^{13}}{2}\\
   \hline
   18. & \begin{pmatrix} 0 &0&0\\&0&0\\&&1\end{pmatrix}\begin{pmatrix} \ell&0&0\\&1&0\\&&0\end{pmatrix}& b_{11}a_{33}^2 + b_{22}(a^2  -c^2\ell)(1-\lambda^2)\ell&-(p-1)^2p^3(p+1)^2 & \frac{p^{13}(p+1)}{2}\\
   \hline
   19. & \begin{pmatrix} 0 &0&0\\&0&0\\&&1\end{pmatrix}\begin{pmatrix} 0&0&0\\&0&0\\&&0\end{pmatrix}& b_{11}a_{33}^2 &-(p-1)^4p^3(p+1) & \frac{(p-1)^2p^{13}}{2}\\
   \hline
   20. & \begin{pmatrix} 0 &0&0\\&0&0\\&&0\end{pmatrix}\begin{pmatrix} 0&0&0\\&0&0\\&&0\end{pmatrix} &0 & (p-1)^5p^3(p+1) & -\frac{(p-1)^3p^{13}}{2}\\
   \hline
   \end{array}
   \end{equation*}
 }
 
The exponential sums are as follows.
\begin{enumerate}
 \item [1.-12.] These sums vanish.  For sums 1.-6., sum in $b_{12}$ to force $a_{33}=0$.  For sums 7.-12. sum in $a_{13}$ and $a_{23}$ to force $a = c= 0$.  In either case, this makes the acting matrix singular.
 
\item [13.] In this case \begin{equation}[ x_0, g \cdot \xi_1 ] = b_{22}(-a^2  + c^2 \ell) \lambda^2\ell.\end{equation}  Since $(-a^2 +c^2\ell) \neq 0$, sum in $b_{22}$ to obtain
\begin{equation}
 \sS(x,\xi) = -\frac{|G_{x,\xi}|}{p-1} = -(p-1)^4p^3(p+1).
\end{equation}
\item [14.] In this case \begin{equation}[ x_0, g \cdot \xi_1 ] = b_{22}(a^2 - c^2\ell)(1-\lambda^2 \ell).\end{equation} Since $(a^2-c^2\ell)(1-\lambda^2 \ell) \neq 0$, sum in $b_{22}$ to obtain as for 13,
\begin{equation}
 \sS(x,\xi) = -(p-1)^4 p^3 (p+1).
\end{equation}
\item[15.] In this case \begin{equation}[ x_0, g \cdot \xi_1 ] =b_{22}(a^2  - c^2 \ell)(1- \lambda^2)\ell.\end{equation} Changing variables in $b_{22}$,
\begin{align*}
 \sS(x,\xi) &= \sum_{\begin{pmatrix} a &c \lambda \ell  & a_{13}\\ c &a \lambda  & a_{23} \\ && a_{33}\end{pmatrix}\begin{pmatrix} b_{11} & b_{12}\\ & b_{22}\end{pmatrix}} e_p(b_{22}(a^2  - c^2 \ell)(1- \lambda^2)\ell)\\
 &= (p-1)^3 p^3 (p+1) \sum_{b_{22}, \lambda \in (\zed/p\zed)^\times} e_p(b_{22}(1-\lambda^2))\\
 &= (p-1)^3 p^3 (p+1)^2.
\end{align*}
\item[16.] In this case \begin{equation}[ x_0, g \cdot \xi_1 ] = b_{11}a_{33}^2  + b_{22}(-a^2  +c^2\ell)\lambda^2 \ell.\end{equation} Sum in $b_{11}$ to obtain $\sS(x, \xi) = \frac{-13.}{p-1} = (p-1)^3 p^3 (p+1).$
\item[17.] In this case \begin{equation}[ x_0, g \cdot \xi_1 ] = b_{11}a_{33}^2  + b_{22}(a^2-c^2\ell)(1-\lambda^2 \ell).\end{equation} Sum in $b_{11}$ to obtain $\sS(x, \xi) = \frac{-14.}{p-1} = (p-1)^3 p^3 (p+1).$
\item[18.] In this case \begin{equation}[ x_0, g \cdot \xi_1 ] = b_{11}a_{33}^2 + b_{22}(a^2  -c^2\ell)(1-\lambda^2)\ell.\end{equation} Sum in $b_{11}$ to obtain $\sS(x, \xi) = \frac{-15.}{p-1} = -(p-1)^2 p^3 (p+1)^2.$
\item[19.] In this case \begin{equation}[ x_0, g \cdot \xi_1 ] = b_{11}a_{33}^2.\end{equation} Sum in $b_{11}$ to obtain $\sS(x, \xi) = \frac{-|G_{x,\xi}|}{p-1} = -(p-1)^4 p^3 (p+1).$
\item[20.] In this case \begin{equation}[ x_0, g \cdot \xi_1 ] = 0\end{equation} so $
 \sS(x, \xi) = |G_{x, \xi}| = (p-1)^5 p^3 (p+1).$

\end{enumerate}

\subsection{The exponential sums pair $(\sO_{2^2}, \sO_{D2})$}\label{O22_OD2_section}
Take standard representatives
\begin{equation}
 x_0 = \begin{pmatrix}0&0&0\\ &0&0\\ &&1\end{pmatrix}\begin{pmatrix}1 &0&0\\ &-\ell &0\\ &&0\end{pmatrix}, \qquad \xi_0  = \begin{pmatrix} \ell & \beta & 0\\ &1 &0\\ &&0\end{pmatrix}\begin{pmatrix} 0 &0&0\\&0&0\\&&0\end{pmatrix}
\end{equation}
with $\ell u^2 + 2\beta uv + v^2$ irreducible. The acting set is $(g = \begin{pmatrix} a & b\\ c & d\end{pmatrix})$
\begin{equation}
 G_{x,\xi}^t  =\left\{ g=\begin{pmatrix} a & b & a_{13}\\ c & d & a_{23}\\ && a_{33} \end{pmatrix}\begin{pmatrix} b_{11} & b_{12}\\ & b_{22}\end{pmatrix}, [ u^2 - \ell v^2, g\cdot (\ell u^2 + 2 \beta uv + v^2) ] = 0\right\}.
\end{equation}
As determined in Appendix \ref{action_set_appendix}, $|G_{x, \xi}| = (p-1)^4p^3(p+1)^2$.
The stabilizer has size
\begin{equation}
 \left|\Stab_{G(\zed/p^2 \zed)}(x)\right| = 2(p-1)^2p^3(p+1).
\end{equation}

Representatives and exponential sum pairings are given in the table below. Since this is the single non-maximal orbit,
\begin{equation}
 \sM(x,\xi) = -\frac{p^{10} \sS(x,\xi)}{2(p-1)^2(p+1)}.
\end{equation}

See \textbf{exponential\_sums\_O22\_OD2.nb}.

{\tiny 
 \begin{equation*}
  \begin{array}{|l|l|l|l|l|}
   \hline
   \text{Orbit} & \xi_1 & [ x_0, g \cdot \xi_1 ]& \sS(x,\xi) & \sM(x,\xi)\\
   \hline
   1. & \begin{pmatrix} 0 &0&0\\&0&0\\&&0\end{pmatrix}\begin{pmatrix} 0&0&0\\&0&0\\&&1\end{pmatrix} &b_{12}a_{33}^2 +  b_{22}(a_{13}^2 - a_{23}^2\ell)&0 &0\\
   \hline
   2. & \begin{pmatrix} 0 &0&0\\&0&0\\&&0\end{pmatrix}\begin{pmatrix} 0&0&0\\&0&1\\&&1\end{pmatrix} &  \begin{array}{l}b_{12}a_{33}^2\\+b_{22}(a_{13}^2 + 2a_{13}b -a_{23}^2 - 2a_{23}d\ell)\end{array} &0&0\\
   \hline
   3. & \begin{pmatrix} 0 &0&0\\&0&0\\&&0\end{pmatrix}\begin{pmatrix} 0&1&0\\&0&0\\&&1\end{pmatrix} & \begin{array}{l}b_{12}a_{33}^2\\+b_{22}(a_{13}^2 + 2ab - a_{23}^2 \ell -2cd\ell) \end{array} &0&0\\
   \hline
   4. & \begin{pmatrix} 0 &0&0\\&0&0\\&&0\end{pmatrix}\begin{pmatrix} 0&0&0\\&0&1\\&&0\end{pmatrix} & 2b_{22}(a_{13}b - a_{23}d\ell) &0&0\\
   \hline
   5. & \begin{pmatrix} 0 &0&0\\&0&0\\&&1\end{pmatrix}\begin{pmatrix} 0&0&0\\&0&1\\&&0\end{pmatrix} & b_{11}a_{33}^2+ 2b_{22}(a_{13}b - a_{23}d\ell) &0&0\\
   \hline
   6. & \begin{pmatrix} 0 &0&0\\&0&0\\&&0\end{pmatrix}\begin{pmatrix} 0&0&0\\&1&0\\&&0\end{pmatrix} & b_{22}(b^2 - d^2 \ell) &-(p-1)^3p^3(p+1)^2 & \frac{(p-1)p^{13}(p+1)}{2}\\
   \hline
   7. & \begin{pmatrix} 0 &0&0\\&0&0\\&&1\end{pmatrix}\begin{pmatrix} 0&0&0\\&1&0\\&&0\end{pmatrix} &  b_{11}a_{33}^2+ b_{22}( b^2 - d^2\ell) & (p-1)^2p^3(p+1)^2 & -\frac{p^{13}(p+1)}{2}\\
   \hline
   8a. & \begin{pmatrix} 0 &0&0\\&0&0\\&&0\end{pmatrix}\begin{pmatrix} 0&1&0\\&0&0\\&&0\end{pmatrix} & 2b_{22}(ab - cd \ell) & (p-1)^4p^3(p+1) & -\frac{(p-1)^2 p^{13}}{2}\\
   \hline
   8b. & \begin{pmatrix} 0 &0&0\\&0&0\\&&0\end{pmatrix}\begin{pmatrix} 0&k&0\\&1&0\\&&0\end{pmatrix} & b_{22}(b^2 + 2abk -d^2 \ell -2cd\ell k)& (p-1)^4p^3(p+1) & - \frac{(p-1)^2p^{13}}{2}\\
   \hline
   9a. & \begin{pmatrix} 0 &0&0\\&0&0\\&&1\end{pmatrix}\begin{pmatrix} 0&1&0\\&0&0\\&&0\end{pmatrix} & b_{11}a_{33}^2+b_{22}(2ab - 2cd\ell)& -(p-1)^3p^3(p+1) & \frac{(p-1)p^{13}}{2}\\
   \hline
   9b. & \begin{pmatrix} 0 &0&0\\&0&0\\&&1\end{pmatrix}\begin{pmatrix} 0&k&0\\&1&0\\&&0\end{pmatrix} &  \begin{array}{l}b_{11}a_{33}^2+\\b_{22}(b^2 + 2abk -d^2\ell -2cd\ell k) \end{array} & -(p-1)^3p^3(p+1) & \frac{(p-1)p^{13}}{2}\\
   \hline
   10 & \begin{pmatrix} 0 &0&0\\&0&0\\&&1\end{pmatrix}\begin{pmatrix} 0&0&0\\&0&0\\&&0\end{pmatrix} & b_{11}a_{33}^2 & -(p-1)^3p^3(p+1)^2 & \frac{(p-1)p^{13}(p+1)}{2}\\
   \hline
   11 & \begin{pmatrix} 0 &0&0\\&0&0\\&&0\end{pmatrix}\begin{pmatrix} 0&0&0\\&0&0\\&&0\end{pmatrix} & 0 & (p-1)^4p^3(p+1)^2 & - \frac{(p-1)^2p^{13}(p+1)}{2}\\
   \hline
   \end{array}
 \end{equation*}
 }
 The exponential sums are as follows.
\begin{enumerate}
 \item [1.-5.] These sums vanish on summing in $b_{12}$, $a_{13}$ and $a_{23}$.
\item [6.] In this case,
\begin{equation}
 [ x_0, g \cdot \xi_1 ] = b_{22}(b^2 - d^2\ell).
\end{equation}
Since $b^2 - d^2 \ell \neq 0$, summing in $b_{22}$ gives $-1$, so
\begin{equation}
 \sS(x, \xi_1) = \frac{-|G_{x, \xi_1}|}{p-1} = - (p-1)^3 p^3 (p+1)^2.
\end{equation}
\item [7.] In this case,
\begin{equation}
 [ x_0, g \cdot \xi_1 ] = b_{11}a_{33}^2+ b_{22}( b^2 - d^2\ell).
\end{equation}
Summing in $b_{11}$ and $b_{22}$ each give $-1$, so
\begin{equation}
 \sS(x, \xi) = \frac{|G_{x, \xi}|}{(p-1)^2} = (p-1)^2p^3 (p+1)^2.
\end{equation}
\item [8a.] $p \equiv 1 \bmod 4$. In this case,
\begin{equation}
 [ x_0, g \cdot \xi_1 ] = 2b_{22}(ab-cd\ell).
\end{equation}
By summing in $b_{22}$,
\begin{align*}
 \sS(x, \xi) &= - \frac{|G_{x, \xi}|}{p-1}\\
 & \;+ \frac{p}{p-1} \# \left\{ g \in G_{x, \xi}^t: [ u^2 - \ell v^2, g \cdot uv] = 0, [ u^2 - \ell v^2, g\cdot (\ell u^2 + v^2)] = 0\right\}\\
  &= - \frac{|G_{x, \xi}|}{p-1}\\
 & \;+ \frac{p}{p-1} \# \left\{ g \in G_{x, \xi}: [ g\cdot(u^2 - \ell v^2), uv] = 0, [ g\cdot (u^2 - \ell v^2), \ell u^2 + v^2] = 0\right\}
 \end{align*}
 The conditions on $g$ are equivalent to $g$ acting on $u^2 - \ell v^2$ by a scalar.  Thus
 \begin{align*}
 \sS(x, \xi)
 &= -\frac{|G_{x, \xi}|}{p-1} +\frac{p}{p-1}\#\left\{ \begin{pmatrix} c &  e \ell& *\\ \pm e & \pm c & *\\ && *\end{pmatrix} \begin{pmatrix} * & *\\ & *\end{pmatrix}\right\}\\
 &= - (p-1)^3 p^3 (p+1)^2 + 2(p-1)^3p^4(p+1)\\
 &= (p-1)^4 p^3 (p+1).
\end{align*}
\item [9a.] $p \equiv 1 \bmod 4$.  In this case,
\begin{equation}
 [ x_0, g \cdot \xi_1 ] = b_{11}a_{33}^2+b_{22}(2ab - 2cd\ell).
\end{equation} The sum in $b_{11}$ is $-1$, so
\begin{equation}
 \sS(x, \xi) = -\frac{8a.}{p-1} = - (p-1)^3p^3 (p+1).
\end{equation}
\item [10.]  In this case,
\begin{equation}
 [ x_0, g \cdot \xi_1 ] = b_{11}a_{33}^2.
\end{equation} The sum in $b_{11}$ is $-1$, so
\begin{equation}
 \sS(x, \xi) = -\frac{|G_{x, \xi}|}{p-1} = -(p-1)^3p^3(p+1)^2.
\end{equation}
\item [11.] $\xi_1 = 0$.  In this case $\sS(x, \xi) = |G_{x, \xi}| = (p-1)^4p^3(p+1)^2.$

The relation $9b. = -\frac{8b.}{p-1}$ holds as in the a case.  Since the sum of $e_p([x, \xi])$ over $\xi$ in a full $\mod p$ orbit vanishes by Lemma \ref{full_mod_p_orbit_lemma},  $8a. = 8b.$ and $9a. = 9b.$
\end{enumerate}
\subsection{The exponential sums pair $(\sO_{1^21^2}, \sO_{D1^2})$}\label{O1212_OD12_section}
For this pair,
\begin{equation}
 x_0 = \begin{pmatrix}0&0&0\\&0&0\\&&1\end{pmatrix}\begin{pmatrix} 0&\frac{1}{2}&0\\&0&0\\&&0\end{pmatrix}, \qquad \xi_0= \begin{pmatrix} 1&0&0\\&0&0\\&&0\end{pmatrix}\begin{pmatrix} 0&0&0\\&0&0\\&&0\end{pmatrix}
\end{equation}
and
\begin{equation}
 x_1 = \begin{pmatrix}x &0 &0\\&y&0\\&&0\end{pmatrix}\begin{pmatrix}0&0&0\\&0&0\\&&0\end{pmatrix}
\end{equation}
with $x, y$ equal to either $1$ or $\ell$.  The acting set is
\begin{equation}
 G_{x,\xi}^t = \begin{pmatrix} a_{11} & a_{12} & a_{13}\\ & a & b\\ & c & d\end{pmatrix}\begin{pmatrix} b_{11} & b_{12} \\ & b_{22}\end{pmatrix} \sqcup \begin{pmatrix} & a & b\\ a_{21} & a_{22} & a_{23}\\ & c & d\end{pmatrix}\begin{pmatrix} b_{11} & b_{12}\\ & b_{22}\end{pmatrix}.
\end{equation}
Write this as $G_{x, \xi, 1}^t \sqcup G_{x,\xi, 2}^t$. 
This obtains
\begin{equation}
 [ x_1, g \cdot \xi_0 ] = b_{11}(x a_{11}^2 + y a_{21}^2).
\end{equation}

 Note that
\begin{equation}
 G_{x, \xi, 1} = G_{x, \xi, 2} \cdot \begin{pmatrix} 0 & 1 &0\\ 1&0&0\\0&0&1\end{pmatrix} \begin{pmatrix}1 &0\\0&1\end{pmatrix}.
\end{equation}
Since $x_0$ in invariant under $\begin{pmatrix} 0&1&0\\1&0&0\\0&0&1\end{pmatrix}$ while in $x_1$ this exchanges $x$ and $y$, the exponential sums may be written by setting
\begin{equation}
 x_1' = \begin{pmatrix} \epsilon &0&0\\&0&0\\ &&0\end{pmatrix} \begin{pmatrix}0&0&0\\&0&0\\&&0\end{pmatrix},
\end{equation}
$[ x_1', g\cdot \xi_0 ] = b_{11}a_{11}^2  \epsilon$ and
\begin{equation}
 \sS(x,\xi) = \sum_{\epsilon = x, y} \sum_{g \in G_{x, \xi,1}^t} e_p\left([ x_1', g\cdot \xi_0 ] + [ x_0, g\cdot \xi_1] \right).
\end{equation}
The sum over $G_{x, \xi, 1}^t$ is fibered as for the pairs $(\sO_{1^211}, \sO_{D1^2})$ and $(\sO_{1^22}, \sO_{D1^2})$.

The stabilizer has size
\begin{equation}
 \left|\Stab_{G(\zed/p^2\zed)}(x)\right| = \left\{\begin{array}{lll}8(p-1)p^3 && x = y\\ 4(p-1)p^3 && x\neq y\end{array}\right..
\end{equation}

 Representatives and exponential sum pairings are given in the table below. Summing over the three maximal orbits
 \begin{equation}
  \sM(x,\xi) = \sum_{(x,y) \in \{(1,1),(1,\ell), (\ell,\ell)\}} \frac{2^{\delta(x \neq y)}p^{10}\sS(x, \xi)}{8(p-1)}.
 \end{equation}

 See \textbf{exponential\_sums\_O1212\_OD12.nb}.
 {\tiny
 \begin{equation*}
  \begin{array}{|l|l|l|l|l|}
   \hline
   \text{Orbit} & \xi_1 & [ x_0, g \cdot \xi_1 ] + [ x_1', g \cdot \xi_0 ]& \sS(x,\xi)& \sM(x,\xi)\\
   \hline
   1. & \begin{pmatrix} 0 &0&0\\&0&0\\&&0\end{pmatrix}\begin{pmatrix} 0&0&0\\&0&0\\&&0\end{pmatrix} &  b_{11} a_{11}^2 \epsilon & -2(p-1)^4p^4(p+1) & -(p-1)^3 p^{14}(p+1)\\
   \hline
   2. & \begin{pmatrix} 0 & 0&0\\ &0&0\\ &&1\end{pmatrix}\begin{pmatrix} 0&0&0\\&0&0\\&&0\end{pmatrix} & b_{11} (a_{11}^2 \epsilon+d^2) & \begin{array}{l} (p-1)^3p^6\left(\left(\frac{-x}{p}\right) + \left(\frac{-y}{p}\right) \right)\\ + 2 (p-1)^3 p^4 \end{array} & (p-1)^2 p^{14} \\
   \hline
   3. & \begin{pmatrix} 0 &0&0\\&0&\frac{1}{2}\\ &&0\end{pmatrix}\begin{pmatrix}0&0&0\\&0&0\\&&0\end{pmatrix} & b_{11}(a_{11}^2 \epsilon+cd ) & -2(p-1)^4p^4 & -(p-1)^3 p^{14}\\
   \hline
   4. & \begin{pmatrix} 0 & 0 & 0\\ & -\ell &0\\ &&1\end{pmatrix}\begin{pmatrix} 0 &0&0\\ &0&0\\&&0\end{pmatrix} &   b_{11} (a_{11}^2 \epsilon- c^2\ell + d^2 ) & 2(p-1)^3p^4(p+1) & (p-1)^2p^{14}(p+1) \\
   \hline
   5. & \begin{pmatrix} 0 &0&0\\ &0&0\\&&1\end{pmatrix}\begin{pmatrix} 0&0&1\\ &0&0\\&&0\end{pmatrix} & \begin{array}{l}b_{11}(a_{11}^2 \epsilon+d^2 ) +b_{22}a_{11}b \end{array} & 2(p-1)^3p^4 & (p-1)^2 p^{14}\\
   \hline
   6. & \begin{pmatrix} 0 &0&0\\ &1 &0\\&&0\end{pmatrix}\begin{pmatrix}0 &0&1\\ &0&0\\&&0\end{pmatrix} & \begin{array}{l}b_{11}(a_{11}^2 \epsilon+c^2) +b_{22}a_{11}b\end{array} & \begin{array}{l}-2(p-1)^2p^4\\ -(p-1)^2p^5 \left( \left(\frac{-x}{p} \right) + \left(\frac{-y}{p} \right)\right) \end{array}  & -(p-1)p^{14}\\
   \hline
   7. & \begin{pmatrix} 0 &0 &0\\ &0 & \frac{1}{2}\\ &&0\end{pmatrix}\begin{pmatrix}0 &0&1\\ &0&0\\&&0\end{pmatrix} & \begin{array}{l}b_{11}(a_{11}^2 \epsilon+cd )\\ + b_{22}a_{11}b\end{array} & 2(p-1)^3p^4 & (p-1)^2p^{14}\\
   \hline
   8. & \begin{pmatrix} 0 &0&0\\ &1&1\\&&0\end{pmatrix} \begin{pmatrix} 0&0&1\\&0&0\\&&0\end{pmatrix} &  \begin{array}{l}b_{11}(a_{11}^2\epsilon+c(c+2d) )\\ + b_{22}a_{11}b\end{array}& \begin{array}{l} -2(p-1)^2 p^4 \\ -(p-1)^2p^5 \left(\left(\frac{-x}{p}\right) + \left( \frac{-y}{p}\right) \right) \end{array}& -(p-1)p^{14}\\
   \hline
   9. & \begin{pmatrix} 0 &0&0\\ &-\ell &0\\ &&1\end{pmatrix}\begin{pmatrix} 0&0&1\\ &0&0\\&&0\end{pmatrix} & \begin{array}{l}b_{11}(a_{11}^2\epsilon - c^2\ell +d^2)\\ +b_{22}a_{11}b \end{array} & \begin{array}{l} -2(p-1)^2p^4 \\ + (p-1)^2 p^5\left(\left(\frac{x}{p} \right) + \left(\frac{y}{p} \right) \right) \end{array}&-(p-1)p^{14}\\
   \hline
   10. & \begin{pmatrix} 0 &0&0\\ &0&0\\&&0\end{pmatrix}\begin{pmatrix}0 &0&1 \\ & 0&0\\ &&0\end{pmatrix} &  b_{11}a_{11}^2\epsilon + b_{22}a_{11}b& 2(p-1)^3p^4 & (p-1)^2 p^{14}\\
    \hline   
    11. & \begin{pmatrix} 0&0&0\\&0&0\\&&0\end{pmatrix}\begin{pmatrix}0&0&0\\&0&0\\&&1\end{pmatrix} & b_{11}a_{11}^2\epsilon + b_{12}d^2+ b_{22}a_{13}b  &0&0 \\
   \hline
   12. & \begin{pmatrix} 0&0&0\\&0&1\\&&0\end{pmatrix}\begin{pmatrix}0&0&0\\&0&0\\&&1\end{pmatrix} & b_{11}(a_{11}^2\epsilon + 2cd) + b_{12}d^2 + b_{22}a_{13}b &0&0\\
   \hline
   13. & \begin{pmatrix} 0&0&0\\&1&0\\&&0\end{pmatrix}\begin{pmatrix}0&0&0\\&0&0\\&&1\end{pmatrix}& b_{11}(a_{11}^2\epsilon + c^2)+b_{12}d^2+ b_{22}a_{13}b  &0&0\\
   \hline
   14. & \begin{pmatrix} 0&0&0\\&0&0\\&&0\end{pmatrix}\begin{pmatrix}0&1&0\\&0&0\\&&1\end{pmatrix}& b_{11}a_{11}^2\epsilon + b_{12}d^2 + b_{22}(a a_{11} + a_{13}b) &0&0\\
   \hline
   15. & \begin{pmatrix} 0&0&0\\&0&1\\&&0\end{pmatrix}\begin{pmatrix}0&1&0\\&0&0\\&&1\end{pmatrix}& \begin{array}{l}b_{11}(a_{11}^2\epsilon+2cd) + b_{12}d^2 \\+ b_{22}(a_{11}a + a_{13}b)\end{array}&0&0\\
   \hline
   16. &\begin{pmatrix} 0&0&0\\&1&0\\&&0\end{pmatrix}\begin{pmatrix}0&1&0\\&0&0\\&&1\end{pmatrix} & \begin{array}{l}b_{11}(a_{11}^2\epsilon + c^2) + b_{12}d^2 \\+ b_{22}(a_{11}a + a_{13}b)\end{array}&0&0\\
   \hline
   17. & \begin{pmatrix} 0&0&0\\&0&0\\&&0\end{pmatrix}\begin{pmatrix}0&0&0\\&0&1\\&&0\end{pmatrix}& b_{11}a_{11}^2\epsilon + 2b_{12}cd + b_{22}(a_{12}b + a_{13}a)&0&0\\
   \hline
   18. &\begin{pmatrix} 0&0&0\\&0&0\\&&1\end{pmatrix}\begin{pmatrix}0&0&0\\&0&1\\&&0\end{pmatrix} & \begin{array}{l}b_{11}(a_{11}^2\epsilon + d^2) + 2b_{12}cd\\ + b_{22}(a_{12}b + a_{13}a) \end{array}&0&0\\
   \hline
   19. & \begin{pmatrix} 0&0&0\\&1&0\\&&1\end{pmatrix}\begin{pmatrix}0&0&0\\&0&1\\&&0\end{pmatrix}& \begin{array}{l}b_{11}(a_{11}^2 \epsilon + c^2 +d^2) + 2b_{12}cd\\ + b_{22}(a_{12}b + a_{13}a)\end{array}&0&0\\
   \hline
   20. &\begin{pmatrix} 0&0&0\\&\ell&0\\&&1\end{pmatrix}\begin{pmatrix}0&0&0\\&0&1\\&&0\end{pmatrix} & \begin{array}{l}b_{11}(a_{11}^2 \epsilon+c^2 \ell + d^2) + 2 b_{12}cd\\ + b_{22}(a_{12}b + a_{13}a)\end{array}&0&0\\
   \hline
   21. & \begin{pmatrix} 0 &0&0\\&0&0\\&&0\end{pmatrix}\begin{pmatrix}0&0&0\\ &-\ell &0\\&&1\end{pmatrix} & \begin{array}{l}b_{11}a_{11}^2\epsilon + b_{12}(-c^2\ell + d^2)\\ + b_{22}(-a_{12}a\ell + a_{13}b)\end{array}&0&0\\
   \hline
   22. & \begin{pmatrix} 0 &0&0\\&1&0\\&&0\end{pmatrix}\begin{pmatrix}0&0&0\\ &-\ell &0\\&&1\end{pmatrix} & \begin{array}{l} b_{11}(a_{11}^2\epsilon + c^2) + b_{12}(-c^2 \ell + d^2)\\ + b_{22}(-a_{12}a\ell + a_{13}b)\end{array}&0&0 \\
   \hline
   23a. &\begin{pmatrix} 0 &0&0\\&0&1\\&&0\end{pmatrix}\begin{pmatrix}0&0&0\\ &-\ell &0\\&&1\end{pmatrix} & \begin{array}{l} b_{11}(a_{11}^2 \epsilon + 2cd) \\ + b_{12}(-c^2 \ell + d^2)\\ + b_{22}(-a_{12}a\ell + a_{13}b)  \end{array}&0&0\\
   \hline
   23b. &\begin{pmatrix} 0 &0&0\\&1&k\\&&0\end{pmatrix}\begin{pmatrix}0&0&0\\ &-\ell &0\\&&1\end{pmatrix} & \begin{array}{l} b_{11}(a_{11}^2\epsilon + c^2 + 2cdk)\\ +b_{12}(-c^2\ell + d^2)\\ + b_{22}(-a_{12}a\ell + a_{13}b) \end{array}&0&0\\
   \hline
  \end{array}  
 \end{equation*}
}

 The exponential sums are as follows.

\begin{enumerate}
 \item [1.]  In this case \begin{equation}[ x_0, g \cdot \xi_1] + [ x_1', g \cdot \xi_0] = b_{11} a_{11}^2 \epsilon.\end{equation}  The sum in $b_{11}$ gives $-1$, so
 \begin{equation}
  \sS(x,\xi) = -\frac{|G_{x, \xi}|}{p-1} = -2(p-1)^4p^4(p+1).
 \end{equation}
 \item [2.]   In this case \begin{equation}[ x_0, g \cdot \xi_1] + [ x_1', g \cdot \xi_0] = b_{11}(a_{11}^2\epsilon+d^2 ).\end{equation}  The following sums are evaluated by first summing the quadratic terms to obtain Gauss sums.
 \begin{align*}
  \Sigma_1 &= \sum_{\epsilon = x, y}\sum_{\begin{pmatrix} a_{11} & * & *\\ & a & b\\ & c & d\end{pmatrix}\begin{pmatrix} b_{11} & *\\  & *\end{pmatrix}} e_p(b_{11}(a_{11}^2\epsilon  + d^2))\\
  &= (p-1)p^6 \sum_{\epsilon = x, y}\sum_{b_{11} \in \bF_p^\times}\left( \tau \left(\frac{b_{11}\epsilon}{p}\right) - 1\right)\tau \left(\frac{b_{11}}{p}\right)\\
  &= (p-1)^2p^7\left(\left(\frac{-x}{p}\right) + \left(\frac{-y}{p}\right) \right)\\
  \Sigma_2 &= \sum_{\epsilon = x, y}\sum_{\begin{pmatrix} a_{11} & * & *\\ & b \lambda  & b\\ & d \lambda  & d\end{pmatrix} \begin{pmatrix} b_{11} & * \\ & *\end{pmatrix}}e_p(b_{11}(a_{11}^2\epsilon  + d^2))\\
  &= (p-1)^2 p^4\sum_{\epsilon = x, y}  \sum_{b_{11} \in \bF_p^\times} \left( \tau \left( \frac{b_{11}\epsilon}{p}\right)-1\right)\tau \left(\frac{b_{11}}{p}\right)\\& = (p-1)^3p^5 \left(\left(\frac{-x}{p}\right) + \left(\frac{-y}{p}\right) \right)\\
  \Sigma_3 &= \sum_{\epsilon = x, y} \sum_{\begin{pmatrix} a_{11} & * & * \\ & a & 0\\ & c &0\end{pmatrix} \begin{pmatrix} b_{11} & * \\ & *\end{pmatrix}} e_p(b_{11} a_{11}^2\epsilon)\\& = -2(p-1)^2 p^5\\
  \Sigma_4 &= \sum_{\epsilon = x, y} \sum_{\begin{pmatrix} a_{11} & * & *\\ & 0 & b \\ & 0 & d\end{pmatrix}\begin{pmatrix} b_{11} & * \\ & *\end{pmatrix}}e_p(b_{11} (a_{11}^2 \epsilon + d^2))\\
  &= (p-1)p^4\sum_{\epsilon = x, y} \sum_{b_{11} \in \bF_p^\times}  \left( \tau \left( \frac{b_{11}\epsilon}{p}\right)-1\right)\tau \left(\frac{b_{11}}{p}\right)\\
  &= (p-1)^2 p^5 \left(\left(\frac{-x}{p}\right) + \left(\frac{-y}{p}\right) \right)\\
  \Sigma_5 &= p\sum_{\epsilon = x, y}\sum_{\begin{pmatrix} a_{11} & * & *\\ & 0 & 0 \\ & 0 & 0\end{pmatrix}\begin{pmatrix} b_{11} & * \\ & *\end{pmatrix}}e_p(b_{11}a_{11}^2 \epsilon  )\\
  &= -2 (p-1)^2 p^4.
 \end{align*}
Thus
\begin{equation}
 \sS(x,\xi) = (p-1)^3p^6 \left(\left(\frac{-x}{p}\right) + \left(\frac{-y}{p}\right) \right) + 2 (p-1)^3p^4.
\end{equation}
 \item [3.] In this case \begin{equation}[ x_0, g \cdot \xi_1] + [ x_1', g \cdot \xi_0] = b_{11}(a_{11}^2\epsilon+cd ).\end{equation} 
\begin{align*}
 \Sigma_1 &= \sum_{\epsilon = x, y} \sum_{\begin{pmatrix} a_{11} & * & * \\ & a & b\\ & c & d\end{pmatrix}\begin{pmatrix}b_{11} & * \\ & * \end{pmatrix}}e_p(b_{11}(  a_{11}^2\epsilon + cd )) 
\end{align*} 
 First sum in $d$ to force $c = 0$.  The sum in $b_{11}$ is now $-1$.  Thus $\Sigma_1 = - 2 (p-1)^2p^6$.
 \begin{align*}
 \Sigma_2 &= \sum_{\epsilon = x, y}\sum_{\begin{pmatrix} a_{11} & * & *\\ & b \lambda  & b\\ & d \lambda  & d\end{pmatrix}\begin{pmatrix} b_{11} & * \\ & *\end{pmatrix}} e_p(b_{11}(a_{11}^2\epsilon  +  d^2\lambda))
 \end{align*}
 This sum vanishes by summing over $d$ and $\lambda$.
 
 \begin{align*}
 \Sigma_3 &= \sum_{\epsilon = x, y}\sum_{\begin{pmatrix} a_{11} & * & *\\ & a & 0\\ & c &0\end{pmatrix}\begin{pmatrix} b_{11} & *\\ & *\end{pmatrix}} e_p(b_{11}a_{11}^2\epsilon  )\\
 &= -2(p-1)^2p^5\\
 \Sigma_4 &= \sum_{\epsilon = x, y}\sum_{\begin{pmatrix} a_{11} & * & *\\ & 0 & b\\ &0 & d\end{pmatrix}\begin{pmatrix} b_{11} & *\\ & *\end{pmatrix}} e_p( b_{11}a_{11}^2\epsilon )\\
 &= -2(p-1)^2p^5\\
 \Sigma_5 &= p \sum_{\epsilon = x, y} \sum_{\begin{pmatrix} a_{11} & * & *\\ &0&0\\&0&0\end{pmatrix}\begin{pmatrix} b_{11} & *\\ & *\end{pmatrix}} e_p( b_{11}a_{11}^2\epsilon) = -2(p-1)^2p^4.
\end{align*}
Thus
\begin{equation}
 \sS(x, \xi) = -2 (p-1)^4p^4.
\end{equation}
 \item [4.]   In this case \begin{equation}[ x_0, g \cdot \xi_1] + [ x_1', g \cdot \xi_0] = b_{11}(a_{11}^2\epsilon-c^2 \ell + d^2 ).\end{equation} 
\begin{align*}
 \Sigma_1 &= \sum_{\epsilon = x, y} \sum_{\begin{pmatrix} a_{11} & * & *\\ & a & b\\ &c & d\end{pmatrix}\begin{pmatrix} b_{11} & * \\ & *\end{pmatrix}} e_p(b_{11}( a_{11}^2\epsilon -  c^2\ell + d^2))\\
  &= (p-1) p^5 \sum_{\epsilon = x, y}\sum_{b_{11} \in \bF_p^\times} \left(\tau \left(\frac{b_{11}\epsilon }{p}\right) - 1\right)\tau^2 \left( \frac{-\ell}{p}\right)\\
  &= 2(p-1)^2 p^6\\
  \Sigma_2 &= \sum_{\epsilon = x, y}\sum_{\begin{pmatrix} a_{11} & * & *\\  & b \lambda  & b\\ & d \lambda  & d\end{pmatrix}\begin{pmatrix} b_{11} & * \\ & *\end{pmatrix}} e_p(b_{11}(a_{11}^2\epsilon  + d^2(-\lambda^2 \ell + 1) ))
  \end{align*}
  Set apart the $d = 0$ term, which evaluates to $-2(p-1)^3p^4$.  When $d \neq 0$, replace $\lambda := d \lambda $ to obtain
  \begin{align*}
 \Sigma_2 &= -2(p-1)^3p^4 + (p-1)p^4\sum_{\epsilon = x, y} \sum_{b_{11}, a_{11}, \lambda, d \in \bF_p^\times} e_p(b_{11}( a_{11}^2 \epsilon- \lambda^2 \ell + d^2))\\
  &= -2(p-1)^3p^4 \\ &+ (p-1)p^4\sum_{\epsilon = x, y} \sum_{b_{11} \in \bF_p^\times} \left(\tau \left(\frac{ b_{11}\epsilon}{p}\right)-1\right)\left(\tau \left(-\frac{ b_{11}\ell}{p}\right) -1\right)\left(\tau \left(\frac{b_{11}}{p}\right)-1\right)\\
  \end{align*}
  The sum over $b_{11}$ vanishes unless 0 or two factors of $\left(\frac{b_{11}}{p}\right)$ are taken.  Using $\left(\frac{\ell}{p}\right) = -1$ obtains
  \begin{align*}
  \Sigma_2&= -2(p-1)^3p^4 - 2(p-1)^2p^4 + 2(p-1)^2p^5\\
  & \qquad + (p-1)^2p^5 \left(\left(\frac{x}{p}\right) - \left(\frac{-x}{p}\right) + \left(\frac{y}{p}\right) - \left(\frac{-y}{p}\right) \right)\\
  &= (p-1)^2p^5 \left(\left(\frac{x}{p}\right) - \left(\frac{-x}{p}\right) + \left(\frac{y}{p}\right) - \left(\frac{-y}{p}\right) \right)\\
  \Sigma_3 &= \sum_{\epsilon = x, y} \sum_{\begin{pmatrix} a_{11} & * & *\\ & a & 0\\ & c & 0\end{pmatrix}\begin{pmatrix} b_{11} & * \\ & *\end{pmatrix}} e_p(b_{11}( a_{11}^2\epsilon -  c^2\ell))\\
  &=(p-1)p^4 \sum_{\epsilon = x, y}\sum_{b_{11} \in \bF_p^\times}\left( \tau \left(\frac{ b_{11}\epsilon}{p}\right)-1\right)\tau \left(\frac{- b_{11}\ell}{p}\right) \\&= -(p-1)^2p^5 \left(\left(\frac{x}{p}\right) + \left(\frac{y}{p}\right) \right)\\
  \Sigma_4 &= \sum_{\epsilon = x, y} \sum_{\begin{pmatrix} a_{11} & * & *\\ & 0 & b\\ & 0 & d\end{pmatrix}\begin{pmatrix} b_{11} & * \\ & *\end{pmatrix}} e_p(b_{11}( a_{11}^2\epsilon +d^2))\\
  &= (p-1)^2p^5 \left( \left(\frac{-x}{p}\right) + \left(\frac{-y}{p}\right)\right)\\
  \Sigma_5 &= p \sum_{\epsilon = x, y} \sum_{\begin{pmatrix} a_{11} & * & *\\ & 0 & 0\\ & 0 & 0\end{pmatrix}\begin{pmatrix} b_{11} & * \\ & *\end{pmatrix}} e_p(b_{11} a_{11}^2\epsilon)=-2(p-1)^2p^4.
\end{align*}
Thus
\begin{equation}
\sS(x,\xi) =  2 (p-1)^3 p^4 (p+1).
\end{equation}
\item [5.]  In this case \begin{equation}[ x_0, g \cdot \xi_1] + [ x_1', g \cdot \xi_0] = b_{11}(a_{11}^2 \epsilon+d^2 ) +b_{22}a_{11}b.\end{equation} In $\Sigma_1, \Sigma_2$ and $\Sigma_4$ below, sum in $b$ to find that the sum vanishes.  In the remaining sums, sum in $b_{11}$.  This obtains
\begin{align*}
 \Sigma_1 &= \sum_{\epsilon = x, y} \sum_{\begin{pmatrix}a_{11} & * & *\\ & a & b\\ &c &d\end{pmatrix} \begin{pmatrix} b_{11} & *\\ & *\end{pmatrix}}e_p(b_{11}(a_{11}^2\epsilon  + d^2) + b_{22}a_{11}b) = 0\\
 \Sigma_2 &= \sum_{\epsilon = x, y} \sum_{\begin{pmatrix}a_{11} & * & *\\ & b \lambda  & b\\ &d \lambda  &d\end{pmatrix} \begin{pmatrix} b_{11} & *\\ & *\end{pmatrix}}e_p(b_{11}( a_{11}^2\epsilon + d^2) + b_{22}a_{11}b) = 0\\
 \Sigma_3 &= \sum_{\epsilon = x, y} \sum_{\begin{pmatrix}a_{11} & * & *\\ & a & 0\\ &c &0\end{pmatrix} \begin{pmatrix} b_{11} & *\\ & *\end{pmatrix}}e_p( b_{11}a_{11}^2\epsilon ) = -2(p-1)^2p^5\\
 \Sigma_4 &= \sum_{\epsilon = x, y} \sum_{\begin{pmatrix}a_{11} & * & *\\ & 0 & b\\ &0 &d\end{pmatrix} \begin{pmatrix} b_{11} & *\\ & *\end{pmatrix}}e_p(b_{11}(a_{11}^2\epsilon  + d^2) + b_{22}a_{11}b) = 0\\
 \Sigma_5 &= p\sum_{\epsilon = x, y} \sum_{\begin{pmatrix}a_{11} & * & *\\ & 0 & 0\\ &0 &0\end{pmatrix} \begin{pmatrix} b_{11} & *\\ & *\end{pmatrix}}e_p(b_{11}a_{11}^2\epsilon  ) = - 2(p-1)^2p^4.
\end{align*}
Thus
\begin{equation}
 \sS(x, \xi) = 2 (p-1)^3p^4.
\end{equation}
\item[6.] In this case \begin{equation}[ x_0, g \cdot \xi_1] + [ x_1', g \cdot \xi_0] = b_{11}(a_{11}^2\epsilon+c^2)+b_{22}a_{11}b.\end{equation} Here in $\Sigma_1, \Sigma_2$ and $\Sigma_4$, sum in $b$ to conclude that the sums vanish.  In $\Sigma_3$, sum first in $a_{11}$ and $c$.  In $\Sigma_5$, sum in $b_{11}$.  This obtains
\begin{align*}
 \Sigma_1 &= \sum_{\epsilon = x, y} \sum_{\begin{pmatrix}a_{11} & * & *\\ & a & b\\ &c &d\end{pmatrix} \begin{pmatrix} b_{11} & *\\ & *\end{pmatrix}}e_p(b_{11}( a_{11}^2\epsilon + c^2) + b_{22}a_{11}b)= 0\\
 \Sigma_2 &=\sum_{\epsilon = x, y} \sum_{\begin{pmatrix}a_{11} & * & *\\ & b \lambda  & b\\ &d \lambda  &d\end{pmatrix} \begin{pmatrix} b_{11} & *\\ & *\end{pmatrix}}e_p(b_{11}(a_{11}^2 \epsilon +  d^2\lambda^2) + b_{22}a_{11}b) = 0\\
 \Sigma_3 &= \sum_{\epsilon = x, y} \sum_{\begin{pmatrix}a_{11} & * & *\\ & a & 0\\ &c &0\end{pmatrix} \begin{pmatrix} b_{11} & *\\ & *\end{pmatrix}}e_p(b_{11}(a_{11}^2\epsilon  + c^2) )\\&= (p-1)p^4\sum_{\epsilon = x, y} \sum_{b_{11} \in \bF_p^\times} \tau\left(\frac{b_{11}}{p}\right)\left( \tau \left(\frac{b_{11}\epsilon}{p}\right)-1\right)\\& = (p-1)^2p^5 \left( \left(\frac{-x}{p}\right) + \left(\frac{-y}{p}\right)\right)\\ 
  \Sigma_4 &=\sum_{\epsilon = x, y} \sum_{\begin{pmatrix}a_{11} & * & *\\ & 0 & b\\ &0 &d\end{pmatrix} \begin{pmatrix} b_{11} & *\\ & *\end{pmatrix}}e_p( b_{11}a_{11}^2\epsilon  + b_{22}a_{11}b) = 0\\
  \Sigma_5 &= p\sum_{\epsilon = x, y} \sum_{\begin{pmatrix}a_{11} & * & *\\ & 0 & 0\\ &0 &0\end{pmatrix} \begin{pmatrix} b_{11} & *\\ & *\end{pmatrix}}e_p(b_{11}a_{11}^2\epsilon  ) = -2(p-1)^2p^4.
\end{align*}
Thus
\begin{equation}
\sS(x, \xi)= -2(p-1)^2p^4 - (p-1)^2p^5 \left( \left(\frac{-x}{p}\right) + \left(\frac{-y}{p}\right)\right).
\end{equation}

\item[7.]  In this case \begin{equation}[ x_0, g \cdot \xi_1] + [ x_1', g \cdot \xi_0] = b_{11}(a_{11}^2\epsilon+cd) + b_{22}a_{11}b.\end{equation} 
As in 5. and 6., sum in $b$ to show that $\Sigma_1, \Sigma_2, \Sigma_4 = 0$.  In the remaining sums, sum in $b_{11}$.  This obtains,
\begin{align*}
 \Sigma_1 &=  \sum_{\epsilon = x, y} \sum_{\begin{pmatrix}a_{11} & * & *\\ & a & b\\ &c &d\end{pmatrix} \begin{pmatrix} b_{11} & *\\ & *\end{pmatrix}}e_p(b_{11}(a_{11}^2\epsilon  + cd) + b_{22}a_{11}b) =0\\
 \Sigma_2 &=\sum_{\epsilon = x, y} \sum_{\begin{pmatrix}a_{11} & * & *\\ & b \lambda  & b\\ &d \lambda  &d\end{pmatrix} \begin{pmatrix} b_{11} & *\\ & *\end{pmatrix}}e_p(b_{11}(a_{11}^2\epsilon  +  d^2\lambda) + b_{22}a_{11}b) = 0\\
 \Sigma_3 &= \sum_{\epsilon = x, y} \sum_{\begin{pmatrix}a_{11} & * & *\\ & a & 0\\ &c &0\end{pmatrix} \begin{pmatrix} b_{11} & *\\ & *\end{pmatrix}}e_p(b_{11}a_{11}^2\epsilon   ) = - 2(p-1)^2p^5\\
 \Sigma_4 &=\sum_{\epsilon = x, y} \sum_{\begin{pmatrix}a_{11} & * & *\\ & 0 & b\\ &0 &d\end{pmatrix} \begin{pmatrix} b_{11} & *\\ & *\end{pmatrix}}e_p(b_{11}a_{11}^2\epsilon   + b_{22}a_{11}b) = 0\\
 \Sigma_5 &= p\sum_{\epsilon = x, y} \sum_{\begin{pmatrix}a_{11} & * & *\\ & 0 & 0\\ &0 &0\end{pmatrix} \begin{pmatrix} b_{11} & *\\ & *\end{pmatrix}}e_p(b_{11}a_{11}^2\epsilon  ) = -2(p-1)^2p^4.
\end{align*}
Thus
\begin{equation}
 \sS(x,\xi) = 2(p-1)^3p^4.
\end{equation}

\item[8.]  In this case \begin{equation}[ x_0, g \cdot \xi_1] + [ x_1', g \cdot \xi_0] = b_{11}(a_{11}^2\epsilon+c(c+2d) )+b_{22}a_{11}b.\end{equation} Sum in $b$ to show that $\Sigma_1, \Sigma_2$ and $\Sigma_4$ vanish.  In $\Sigma_3$, sum first in $a_{11}$ and $c$, while in $\Sigma_5$ sum in $b_{11}$ first.  This obtains,
\begin{align*}
 \Sigma_1 &= \sum_{\epsilon = x, y} \sum_{\begin{pmatrix}a_{11} & * & *\\ & a & b\\ &c &d\end{pmatrix} \begin{pmatrix} b_{11} & *\\ & *\end{pmatrix}} e_p(b_{11}(a_{11}^2\epsilon  + c(c+2d)) + b_{22}a_{11}b) = 0\\
 \Sigma_2 &=\sum_{\epsilon = x, y} \sum_{\begin{pmatrix}a_{11} & * & *\\ & b \lambda  & b\\ &d \lambda  &d\end{pmatrix} \begin{pmatrix} b_{11} & *\\ & *\end{pmatrix}}e_p(b_{11}( a_{11}^2\epsilon + d^2\lambda(\lambda+2)) + b_{22}a_{11}b) = 0\\
 \Sigma_3 &= \sum_{\epsilon = x, y} \sum_{\begin{pmatrix}a_{11} & * & *\\ & a & 0\\ &c &0\end{pmatrix} \begin{pmatrix} b_{11} & *\\ & *\end{pmatrix}}e_p(b_{11} (a_{11}^2 \epsilon + c^2))\\
 &= (p-1)p^4\sum_{\epsilon = x, y} \sum_{b_{11} \in \bF_p^\times} \tau\left(\frac{b_{11}}{p}\right)\left( \tau \left(\frac{b_{11}\epsilon}{p}\right)-1\right)\\
 &= (p-1)^2p^5 \left(\left(\frac{-x}{p}\right) + \left(\frac{-y}{p}\right) \right)\\
 \Sigma_4 &= \sum_{\epsilon = x, y}\sum_{\begin{pmatrix}a_{11} & * & *\\ & 0 & b\\ &0 &d\end{pmatrix} \begin{pmatrix} b_{11} & *\\ & *\end{pmatrix}}e_p(b_{11}a_{11}^2\epsilon  + b_{22}a_{11}b) = 0\\
 \Sigma_5 &= p\sum_{\epsilon = x, y}\sum_{\begin{pmatrix}a_{11} & * & *\\ & 0 & 0\\ &0 &0\end{pmatrix} \begin{pmatrix} b_{11} & *\\ & *\end{pmatrix}}e_p(b_{11}a_{11}^2\epsilon ) = -2(p-1)^2 p^4.
\end{align*}
Thus
\begin{equation}
 \sS(x, \xi) = - 2(p-1)^2 p^4 - (p-1)^2p^5 \left(\left(\frac{-x}{p}\right) + \left(\frac{-y}{p}\right) \right).
\end{equation}
\item[9.]  In this case \begin{equation}[ x_0, g \cdot \xi_1] + [ x_1', g \cdot \xi_0] = b_{11}(a_{11}^2\epsilon-c^2\ell+d^2)+ b_{22}a_{11}b.\end{equation} Sum in $b$ to show that $\Sigma_1, \Sigma_2$ and $\Sigma_4$ vanish.  In $\Sigma_3$ sum in $a_{11}$ and $c$ first, while in $\Sigma_5$ sum in $b_{11}$ first. This obtains,
\begin{align*}
 \Sigma_1 &= \sum_{\epsilon = x, y} \sum_{\begin{pmatrix}a_{11} & * & *\\ & a & b\\ &c &d\end{pmatrix} \begin{pmatrix} b_{11} & *\\ & *\end{pmatrix}} e_p(b_{11}(a_{11}^2\epsilon -c^2 \ell + d^2) + b_{22}a_{11}b) = 0\\
 \Sigma_2 &= \sum_{\epsilon = x, y} \sum_{\begin{pmatrix}a_{11} & * & *\\ & b \lambda  & b\\ &d \lambda  &d\end{pmatrix} \begin{pmatrix} b_{11} & *\\ & *\end{pmatrix}} e_p(b_{11}(a_{11}^2\epsilon  +  d^2(-\ell\lambda^2 + 1)) + b_{22}a_{11}b) = 0\\
  \Sigma_3 &= \sum_{\epsilon = x, y} \sum_{\begin{pmatrix}a_{11} & * & *\\ & a & 0\\ &c &0\end{pmatrix} \begin{pmatrix} b_{11} & *\\ & *\end{pmatrix}}e_p(b_{11} (a_{11}^2\epsilon  - c^2\ell))\\
  &= (p-1)p^4 \sum_{\epsilon = x, y} \sum_{b_{11} \in \bF_p^\times}\tau\left(\frac{- b_{11}\ell}{p}\right)\left(\tau \left(\frac{b_{11}\epsilon}{p}\right)-1\right)\\
  &= -(p-1)^2p^5 \left(\left(\frac{x}{p}\right) + \left(\frac{y}{p}\right) \right)\\
  \Sigma_4 &= \sum_{\epsilon = x, y} \sum_{\begin{pmatrix}a_{11} & * & *\\ & 0 & b\\ &0 &d\end{pmatrix} \begin{pmatrix} b_{11} & *\\ & *\end{pmatrix}}e_p(b_{11} (a_{11}^2\epsilon   + d^2) + b_{22}a_{11}b) = 0\\
  \Sigma_5&= p\sum_{\epsilon = x, y}\sum_{\begin{pmatrix}a_{11} & * & *\\ & 0 & 0\\ &0 &0\end{pmatrix} \begin{pmatrix} b_{11} & *\\ & *\end{pmatrix}}e_p(b_{11}a_{11}^2\epsilon ) = -2(p-1)^2 p^4.
\end{align*}
Thus
\begin{equation}
 \sS(x, \xi) = -2(p-1)^2 p^4 +  (p-1)^2p^5 \left(\left(\frac{x}{p}\right) + \left(\frac{y}{p}\right) \right).
\end{equation}
\item[10.] In this case \begin{equation}[ x_0, g \cdot \xi_1] + [ x_1', g \cdot \xi_0] = b_{11}a_{11}^2\epsilon+b_{22}a_{11}b.\end{equation} Sum in $b$ in $\Sigma_1, \Sigma_2$ and $\Sigma_4$ to show that these sums vanish. In $\Sigma_3$ and $\Sigma_5$ sum in $b_{11}$ first.  This obtains
\begin{align*}
 \Sigma_1 &= \sum_{\epsilon = x, y} \sum_{\begin{pmatrix}a_{11} & * & *\\ & a & b\\ &c &d\end{pmatrix} \begin{pmatrix} b_{11} & *\\ & *\end{pmatrix}} e_p(b_{11}a_{11}^2\epsilon  + b_{22}a_{11}b) = 0\\
 \Sigma_2 &= \sum_{\epsilon = x, y} \sum_{\begin{pmatrix}a_{11} & * & *\\ & b \lambda  & b\\ &d \lambda  &d\end{pmatrix} \begin{pmatrix} b_{11} & *\\ & *\end{pmatrix}} e_p(b_{11}a_{11}^2\epsilon  + b_{22}a_{11}b) = 0\\
 \Sigma_3 &= \sum_{\epsilon = x, y} \sum_{\begin{pmatrix}a_{11} & * & *\\ & a & 0\\ &c &0\end{pmatrix} \begin{pmatrix} b_{11} & *\\ & *\end{pmatrix}}e_p(b_{11}a_{11}^2 \epsilon  )= -2(p-1)^2p^5\\
  \Sigma_4 &= \sum_{\epsilon = x, y} \sum_{\begin{pmatrix}a_{11} & * & *\\ & 0 & b\\ &0 &d\end{pmatrix} \begin{pmatrix} b_{11} & *\\ & *\end{pmatrix}}e_p(b_{11}a_{11}^2 \epsilon    + b_{22}a_{11}b) = 0\\
  \Sigma_5 &=  p\sum_{\epsilon = x, y}\sum_{\begin{pmatrix}a_{11} & * & *\\ & 0 & 0\\ &0 &0\end{pmatrix} \begin{pmatrix} b_{11} & *\\ & *\end{pmatrix}}e_p(b_{11}a_{11}^2\epsilon ) = -2(p-1)^2 p^4.
\end{align*}
Thus
\begin{equation}
 \sS(x, \xi) = 2(p-1)^3 p^4.
\end{equation}
\item[11.-23.] These sums vanish on summing in $a_{12}$, $a_{13}$ and $b_{12}$. 
\end{enumerate}
\subsection{The exponential sums pair $(\sO_{1^21^2}, \sO_{D11})$}\label{O1212_OD11_section}
The standard representatives in this case are
\begin{equation}
 x_0 = \begin{pmatrix} 0&0&0\\ &0&0\\&&1\end{pmatrix}\begin{pmatrix} 0 & \frac{1}{2} &0\\&0&0\\&&0\end{pmatrix}, \qquad \xi_0 = \begin{pmatrix} 1 &0&0\\&-1&0\\&&0\end{pmatrix}\begin{pmatrix}0&0&0\\&0&0\\&&0\end{pmatrix},
\end{equation}
and
\begin{equation}
 x_1 = \begin{pmatrix} x & 0 &0\\ & y &0\\&&0\end{pmatrix}\begin{pmatrix}0&0&0\\&0&0\\&&0\end{pmatrix}.
\end{equation}
The acting set is the subset 
\begin{equation}
 G_{x, \xi}^t\subset \begin{pmatrix} * & * & a_{13} \\ * & * & a_{23}\\ && a_{33}\end{pmatrix}\begin{pmatrix} b_{11} & b_{12} \\ & b_{22}\end{pmatrix}
\end{equation}
in which the $\GL_3$ factor acts by
\begin{align*}
 u+v & \mapsto \alpha u + \beta v\\
 u-v & \mapsto \lambda(\alpha u - \beta v)
\end{align*}
thus
\begin{equation}
 (u+v)(u-v) \mapsto  \alpha^2\lambda u^2 -  \beta^2\lambda v^2.
\end{equation}
Thus
\begin{equation}
 [ x_1, g \cdot \xi_0 ] = b_{11} (\alpha^2 x - \beta^2 y)\lambda.
\end{equation}

The stabilizer has size
\begin{equation}
 \left|\Stab_{G(\zed/p^2\zed)}(x)\right| = \left\{\begin{array}{lll}8(p-1)p^3 && x = y\\ 4(p-1)p^3 && x\neq y\end{array}\right..
\end{equation}

 Representatives and exponential sum pairings are given in the table below. Summing over the three maximal orbits
 \begin{equation}
  \sM(x,\xi) = \sum_{(x,y) \in \{(1,1),(1,\ell), (\ell,\ell)\}} \frac{2^{\delta(x \neq y)}p^{10}\sS(x,\xi)}{8(p-1)}.
 \end{equation}

 See \textbf{exponential\_sums\_O1212\_OD11.nb}.
 
{\tiny
 \begin{equation*}
  \begin{array}{|l|l|l|l|l|}
   \hline
   \text{Orbit} & \xi_1 & [ x_0, g \cdot \xi_1 ] + [ x_1, g \cdot \xi_0 ] & \sS(x,\xi) &\sM(x,\xi)\\
   \hline
   1. & \begin{pmatrix} 0 &0&0\\&0&0\\&&0\end{pmatrix}\begin{pmatrix} 0&0&0\\&0&0\\&&1\end{pmatrix} & b_{11}(\alpha^2 x - \beta^2 y)\lambda + b_{12}a_{33}^2 + b_{22}a_{13}a_{23} & 0&0\\
   \hline
   2. & \begin{pmatrix} 0 &0&0\\&0&0\\&&0\end{pmatrix}\begin{pmatrix} 0&0&1\\&0&-1\\&&1\end{pmatrix} & \begin{array}{l}b_{11}(\alpha^2 x - \beta^2 y)\lambda + b_{12}a_{33}^2 \\+b_{22}(a_{13}a_{23} - a_{13}\beta\lambda + a_{23}\alpha \lambda)\end{array} & 0&0\\
   \hline
   3. &  \begin{pmatrix} 0 &0&0\\&0&0\\&&0\end{pmatrix}\begin{pmatrix} 0&0&2\\&0&0\\&&1\end{pmatrix} & \begin{array}{l} b_{11}(\alpha^2 x - \beta^2 y)\lambda + b_{12}a_{33}^2 \\+ b_{22}(a_{23}\alpha(1+\lambda) \\+ a_{13}\beta(1-\lambda) + a_{13}a_{23})\end{array}&0&0\\
   \hline
   4. & \begin{pmatrix} 0 &0&0\\&0&0\\&&0\end{pmatrix}\begin{pmatrix} -\ell&\ell&0\\&-\ell&0\\&&1\end{pmatrix} &\begin{array}{l}b_{11}(\alpha^2 x - \beta^2 y)\lambda + b_{12}a_{33}^2 \\+ b_{22}(a_{13}a_{23} + \alpha \beta \lambda^2 \ell) \end{array}&0&0\\
   \hline
   5. & \begin{pmatrix} 0 &0&0\\&0&0\\&&0\end{pmatrix}\begin{pmatrix} -\ell&\ell&1\\&-\ell&1\\&&1\end{pmatrix} &  \begin{array}{l}b_{11}(\alpha^2 x - \beta^2 y)\lambda +b_{12}a_{33}^2 \\ + b_{22}(a_{13}a_{23} + a_{13}\beta + a_{23}\alpha + \alpha\beta \lambda^2 \ell ) \end{array}&0&0\\
   \hline
   6. & \begin{pmatrix} 0 &0&0\\&0&0\\&&0\end{pmatrix}\begin{pmatrix} -2\ell&0&0\\&-2\ell&0\\&&1\end{pmatrix} &  \begin{array}{l}b_{11}(\alpha^2 x - \beta^2 y)\lambda + b_{12}a_{33}^2 \\+ b_{22}(\alpha \beta(\lambda^2-1) \ell  + a_{13}a_{23}) \end{array}&0&0\\
   \hline
   7. & \begin{pmatrix} 0 &0&0\\&0&0\\&&0\end{pmatrix}\begin{pmatrix} 0&0&1\\&0&-1\\&&0\end{pmatrix} & \begin{array}{l}b_{11}(\alpha^2 x - \beta^2 y)\lambda\\ + b_{22}(a_{23}\alpha - a_{13}\beta)\lambda \end{array}&0&0\\
   \hline
   8. & \begin{pmatrix} 0 &0&0\\&0&0\\&&0\end{pmatrix}\begin{pmatrix} 0&0&2\\&0&0\\&&0\end{pmatrix}&\begin{array}{l}b_{11}(\alpha^2 x - \beta^2 y)\lambda \\+ b_{22}(a_{13}\beta(1-\lambda) + a_{23}\alpha(1 + \lambda)) \end{array} &0&0\\
   \hline
   9. & \begin{pmatrix} 0 &0&0\\&0&0\\&&0\end{pmatrix}\begin{pmatrix} 1&1&1\\&1&-1\\&&0\end{pmatrix}&\begin{array}{l}b_{11}(\alpha^2 x - \beta^2 y)\lambda \\+ b_{22}( - a_{13}\beta \lambda + a_{23}\alpha \lambda+\alpha \beta) \end{array} &0 &0\\
   \hline
   10. & \begin{pmatrix} 0 &0&0\\&0&0\\&&1\end{pmatrix}\begin{pmatrix} 1&1&1\\&1&-1\\&&0\end{pmatrix}& \begin{array}{l}b_{11}(a_{33}^2+\alpha^2\lambda x - \beta^2\lambda y ) \\+ b_{22}( -a_{13}\beta \lambda + a_{23}\alpha\lambda+\alpha \beta) \end{array} &0&0\\
   \hline
   11. & \begin{pmatrix} 0 &0&0\\&0&0\\&&1\end{pmatrix}\begin{pmatrix} 0&0&2\\&0&0\\&&0\end{pmatrix}& \begin{array}{l}b_{11}( a_{33}^2+\alpha^2\lambda x - \beta^2\lambda y ) \\+ b_{22}(a_{13}\beta(1 - \lambda)+a_{23}\alpha(1 + \lambda)) \end{array} &0&0\\
   \hline
   12. & \begin{pmatrix} 0 &0&0\\&0&0\\&&1\end{pmatrix}\begin{pmatrix} 1&1&1\\&1&-1\\&&0\end{pmatrix}& \begin{array}{l}b_{11}(a_{33}^2+\alpha^2\lambda x - \beta^2\lambda y  ) \\+ b_{22}( - a_{13}\beta \lambda + a_{23}\alpha \lambda+\alpha\beta) \end{array} &0&0\\
   \hline
   13. & \begin{pmatrix} 0 &0&0\\&0&0\\&&0\end{pmatrix}\begin{pmatrix} 1&-1&0\\&1&0\\&&0\end{pmatrix}& \begin{array}{l}b_{11}(\alpha^2 x - \beta^2 y)\lambda  - b_{22}\alpha \beta \lambda^2 \end{array} & \begin{array}{l} -(p-1)^3p^3\\ -(p-1)^3p^4 \left(\frac{xy}{p}\right)\end{array} & -\frac{(p-1)^2p^{13}}{2}\\
   \hline
   14. & \begin{pmatrix} 0 &0&0\\&0&0\\&&0\end{pmatrix}\begin{pmatrix} 2&0&0\\&2&0\\&&0\end{pmatrix}& \begin{array}{l}b_{11}(\alpha^2 x - \beta^2 y)\lambda  +b_{22} \alpha \beta(1 - \lambda^2) \end{array} & \begin{array}{l} (p-1)^2p^3(p+1)\\ (p-1)^2p^4(p+1)\left(\frac{xy}{p}\right)\end{array} & \frac{(p-1)p^{13}(p+1)}{2}\\
   \hline
   15. & \begin{pmatrix} 0 &0&0\\&0&0\\&&0\end{pmatrix}\begin{pmatrix} 1+\ell&-1+\ell&0\\&1+\ell&0\\&&0\end{pmatrix} & \begin{array}{l}b_{11}(\alpha^2 x - \beta^2 y)\lambda + b_{22}\alpha \beta (\ell - \lambda^2) \end{array} & \begin{array}{l} -(p-1)^3p^3\\ -(p-1)^3p^4 \left(\frac{xy}{p}\right)\end{array} & - \frac{(p-1)^2p^{13}}{2}\\
   \hline
   16. & \begin{pmatrix} 0 &0&0\\&0&0\\&&1\end{pmatrix}\begin{pmatrix} 1&-1&0\\&1&0\\&&0\end{pmatrix} & \begin{array}{l}b_{11}(a_{33}^2+\alpha^2\lambda x - \beta^2\lambda y ) - b_{22}\alpha \beta \lambda^2 \end{array}& \begin{array}{l} (p-1)^2 p^3 \\ + (p-1)^2p^4\left(\frac{xy}{p}\right)\end{array} & \frac{(p-1)p^{13}}{2}\\
   \hline
   17. & \begin{pmatrix} 0 &0&0\\&0&0\\&&1\end{pmatrix}\begin{pmatrix} 2&0&0\\&2&0\\&&0\end{pmatrix} & \begin{array}{l}b_{11}(a_{33}^2+\alpha^2\lambda x - \beta^2\lambda y  )\\ + b_{22}\alpha \beta (1-\lambda^2) \end{array}& \begin{array}{l} -(p-1)p^3(p+1)\\ - (p-1)p^4(p+1)\left(\frac{xy}{p}\right)\\ - (p-1)p^5\\\times \Bigl(\left(\frac{x}{p}\right) + \left(\frac{-x}{p}\right)\\ + \left(\frac{y}{p}\right) + \left(\frac{-y}{p}\right) \Bigr)\end{array}& -\frac{p^{13}(p+1)}{2}\\
   \hline
   18. & \begin{pmatrix} 0 &0&0\\&0&0\\&&1\end{pmatrix}\begin{pmatrix} 1+\ell&-1+\ell&0\\&1+\ell&0\\&&0\end{pmatrix}& \begin{array}{l}b_{11}( a_{33}^2+\alpha^2\lambda x - \beta^2\lambda y ) \\+ b_{22}\alpha \beta (\ell-\lambda^2) \end{array} & \begin{array}{l} (p-1)^2p^3 \\ + (p-1)^2p^4\left(\frac{xy}{p}\right)\end{array} & \frac{(p-1)p^{13}}{2}\\
   \hline
   19. & \begin{pmatrix} 0 &0&0\\&0&0\\&&1\end{pmatrix}\begin{pmatrix} 0&0&0\\&0&0\\&&0\end{pmatrix}& b_{11}(a_{33}^2+\alpha^2\lambda x - \beta^2\lambda y ) & \begin{array}{l} - (p-1)^3p^3 \\ - (p-1)^3p^4 \left(\frac{xy}{p}\right)\end{array}& \frac{-(p-1)^2p^{13}}{2} \\
   \hline
   20. & \begin{pmatrix} 0 &0&0\\&0&0\\&&0\end{pmatrix}\begin{pmatrix} 0&0&0\\&0&0\\&&0\end{pmatrix} & b_{11}(\alpha^2 x - \beta^2 y)\lambda & \begin{array}{l} (p-1)^4 p^3 \\ +(p-1)^4 p^4\left(\frac{xy}{p}\right)\end{array} & \frac{(p-1)^3p^{13}}{2}\\
   \hline
   \end{array}
   \end{equation*}
}

The orbital exponential sums are as follows.
\begin{enumerate}
 \item [1.-12.] These sums vanish on summing over $a_{13}, a_{23}, b_{12}$.

 \item [13.] Here 
 \begin{equation}
  [ x_0, g\cdot \xi_1] + [ x_1, g \cdot \xi_0] = b_{11}(\alpha^2 x - \beta^2 y)\lambda  - b_{22}\alpha \beta \lambda^2.
 \end{equation}
The sum over $b_{22}$ is $-1$ since $\alpha, \beta, \lambda \neq 0$.  Hence,
 \begin{align*}
  \sS(x, \xi)& = \sum_{\begin{pmatrix}\alpha, \beta, \lambda & *\\& *\\& *\end{pmatrix}\begin{pmatrix} b_{11} & * \\ & b_{22} \end{pmatrix}} e_p(b_{11}(\alpha^2 x - \beta^2 y)\lambda  - b_{22}\alpha \beta \lambda^2) \\
  &=-(p-1)^2p^3 \sum_{\alpha, \beta, b_{11} \in \bF_p^\times} e_p(b_{11}(\alpha^2 x - \beta^2 y))\\
  &= -(p-1)^2p^3 \sum_{b_{11} \in \bF_p^\times} \left(\tau \left(\frac{b_{11}x}{p} \right)-1 \right)\left(\tau \left(\frac{-b_{11}y}{p} \right)-1 \right)\\
  &= -(p-1)^3p^3 - (p-1)^3p^4 \left(\frac{xy}{p}\right).
 \end{align*}
 \item [14.] Here 
 \begin{equation}
  [ x_0, g\cdot \xi_1] + [ x_1, g \cdot \xi_0] =b_{11}(\alpha^2 x - \beta^2 y)\lambda  +b_{22} \alpha \beta(1 - \lambda^2) .
 \end{equation}
 Thus,
\begin{align*}
\sS(x, \xi)= &\sum_{\begin{pmatrix}\alpha, \beta, \lambda & *\\& *\\& *\end{pmatrix}\begin{pmatrix} b_{11} & * \\ & b_{22} \end{pmatrix}} e_p(b_{11}(\alpha^2 x - \beta^2 y)\lambda  +b_{22} \alpha \beta(1 - \lambda^2) ).
\end{align*}
Sum in $b_{22}$ and split in to $\lambda = \pm 1$ and otherwise to obtain
\begin{align*}
 \sS(x, \xi) &= 2(p-1)^2 p^3 \sum_{\alpha, \beta, b_{11} \in \bF_p^\times} e_p(b_{11}(\alpha^2 x - \beta^2 y))\\ & \qquad - (p-3)(p-1)p^3 \sum_{\alpha, \beta, b_{11} \in \bF_p^\times} e_p(b_{11}(\alpha^2 x - \beta^2 y))\\
 &= (p-1)p^3(p+1) \sum_{b_{11} \in \bF_p^\times} \left(\tau \left(\frac{b_{11}x}{p} \right) -1\right) \left(\tau \left(\frac{-b_{11}y}{p} \right) -1\right)\\
 &= (p-1)^2p^4(p+1) \left( \frac{xy}{p}\right) + (p-1)^2 p^3 (p+1).
\end{align*}
 \item[15.] Here 
 \begin{equation}
  [ x_0, g\cdot \xi_1] + [ x_1, g \cdot \xi_0] =b_{11}(\alpha^2 x - \beta^2 y)\lambda + b_{22}\alpha \beta (\ell - \lambda^2).
 \end{equation}
 Since $\ell- \lambda^2 \neq 0$, the sum in $b_{22}$ is $-1$.  Hence, as in 13., 
\begin{align*}
 \sS(x, \xi)&= \sum_{\begin{pmatrix}\alpha, \beta, \lambda & *\\& *\\& *\end{pmatrix}\begin{pmatrix} b_{11} & * \\ & b_{22} \end{pmatrix}} e_p(b_{11}(\alpha^2 x - \beta^2 y)\lambda + b_{22}\alpha \beta (\ell - \lambda^2))\\
 &= - (p-1)^2 p^3 \sum_{\alpha, \beta, b_{11} \in \bF_p^\times} e_p(b_{11}(\alpha^2 x - \beta^2 y))\\
 &= - (p-1)^3 p^3 - (p-1)^3p^4 \left(\frac{xy}{p}\right).
\end{align*}
\item[16.] Here 
 \begin{equation}
  [ x_0, g\cdot \xi_1] + [ x_1, g \cdot \xi_0] = b_{11}(a_{33}^2+\alpha^2\lambda x - \beta^2\lambda y ) - b_{22}\alpha \beta \lambda^2.
 \end{equation}
 The sum in $b_{22}$ is $-1$.  Then replace $b_{11}:= b_{11}\lambda$ and sum in $\lambda$ to obtain
\begin{align*}
  \sS(x, \xi)& = \sum_{\begin{pmatrix}\alpha, \beta, \lambda & *\\& *\\& *\end{pmatrix}\begin{pmatrix} b_{11} & * \\ & b_{22} \end{pmatrix}} e_p(b_{11}(a_{33}^2+\alpha^2\lambda x - \beta^2\lambda y ) - b_{22}\alpha \beta \lambda^2) \\
  &= (p-1)p^3 \sum_{\alpha, \beta, b_{11} \in \bF_p^\times} e_p(b_{11}(\alpha^2 x - \beta^2 y))\\
  &= (p-1)^2p^3+ (p-1)^2 p^4 \left( \frac{xy}{p}\right) .
\end{align*}
\item[17.] Here 
 \begin{equation}
  [ x_0, g\cdot \xi_1] + [ x_1, g \cdot \xi_0] = b_{11}(\alpha^2\lambda x - \beta^2\lambda y + a_{33}^2) - b_{22}\alpha \beta (\lambda^2-1).
 \end{equation}
 Sum in $b_{22}$ which obtains $-1$ unless $\lambda = \pm 1$.  Thus
\begin{align*}
 \sS(x, \xi)& = \sum_{\begin{pmatrix}\alpha, \beta, \lambda & *\\& *\\& *\end{pmatrix}\begin{pmatrix} b_{11} & * \\ & b_{22} \end{pmatrix}} e_p(b_{11}( a_{33}^2+\alpha^2\lambda x - \beta^2\lambda y ) + b_{22}\alpha \beta (1-\lambda^2))\\
 &= - \frac{1}{p-1} \sum_{\begin{pmatrix}\alpha, \beta, \lambda & *\\& *\\& *\end{pmatrix}\begin{pmatrix} b_{11} & * \\ & b_{22} \end{pmatrix}} e_p(b_{11} (a_{33}^2+\alpha^2\lambda x - \beta^2\lambda y))\\ & \qquad + \frac{p}{p-1}\sum_{\lambda = \pm } \sum_{\begin{pmatrix}\alpha, \beta & *\\& *\\& *\end{pmatrix}\begin{pmatrix} b_{11} & * \\ & b_{22} \end{pmatrix}} e_p(b_{11} ( a_{33}^2+\alpha^2\lambda x - \beta^2\lambda y ))
 \end{align*}
 In the first sum, substitute $b_{11}:= b_{11}\lambda $, then sum in $\lambda$.  In the second sum, sum in $\alpha$, $\beta$ and $a_{33}$ to obtain Gauss sums.  Hence,
 \begin{align*}
\sS(x, \xi) &= (p-1)p^3  \sum_{\alpha, \beta, b_{11} \in \bF_p^\times} e_p(b_{11}(\alpha^2 x - \beta^2 y))\\
 & \qquad + p^4 \sum_{\lambda = \pm }\sum_{b_{11} \in \bF_p^\times} \left(\tau \left(\frac{b_{11}\lambda x }{p} \right) -1\right)\left(\tau\left(\frac{-b_{11}\lambda y }{p} \right) -1\right)\left(\tau\left(\frac{b_{11}}{p} \right) -1\right)\\
 &= (p-1)^2p^4 \left(\frac{xy}{p}\right) + (p-1)^2p^3 - 2(p-1)p^4 - 2(p-1)p^5 \left(\frac{xy}{p} \right) \\& \qquad - (p-1)p^5\left( \left(\frac{x}{p}\right) + \left(\frac{-x}{p}\right) + \left(\frac{y}{p}\right) + \left(\frac{-y}{p}\right) \right)\\
 &= - (p-1)p^4(p+1)\left(\frac{xy}{p}\right)  - (p-1)p^3(p+1)\\& \qquad -(p-1)p^5\left( \left(\frac{x}{p}\right) + \left(\frac{-x}{p}\right) + \left(\frac{y}{p}\right) + \left(\frac{-y}{p}\right) \right). 
\end{align*}
\item[18.]  Here 
 \begin{equation}
  [ x_0, g\cdot \xi_1] + [ x_1, g \cdot \xi_0] = b_{11}( a_{33}^2+\alpha^2\lambda x - \beta^2\lambda y ) + b_{22}\alpha \beta (\ell-\lambda^2).
 \end{equation}
 The sum in $b_{22}$ is $-1$.  Then replace $b_{11} := b_{11}\lambda$ and sum in $\lambda$.  This obtains
\begin{align*}
 \sS(x, \xi) &= \sum_{\begin{pmatrix}\alpha, \beta, \lambda & *\\& *\\& *\end{pmatrix}\begin{pmatrix} b_{11} & * \\ & b_{22} \end{pmatrix}} e_p(b_{11}( a_{33}^2+\alpha^2\lambda x - \beta^2\lambda y ) + b_{22}\alpha \beta (\ell-\lambda^2))\\
 &= (p-1)p^3 \sum_{\alpha, \beta, b_{11} \in \bF_p^\times} e_p(b_{11}(\alpha^2 x - \beta^2 y))\\&=(p-1)^2p^3+ (p-1)^2p^4 \left(\frac{xy}{p}\right) .
\end{align*}
\item[19.]  Here 
 \begin{equation}
  [ x_0, g\cdot \xi_1] + [ x_1, g \cdot \xi_0] = b_{11}( a_{33}^2+\alpha^2\lambda x - \beta^2\lambda y ) .
 \end{equation}
 This is the same as 18 without the sum over $b_{22}$.  Hence,
\begin{align*}
 \sS(x, \xi)& = \sum_{\begin{pmatrix}\alpha, \beta, \lambda & *\\& *\\& *\end{pmatrix}\begin{pmatrix} b_{11} & * \\ & b_{22} \end{pmatrix}}  e_p(b_{11}( a_{33}^2+\alpha^2\lambda x - \beta^2\lambda y ))\\
 &= - (p-1)^2 p^3 \sum_{\alpha, \beta, b_{11} \in \bF_p^\times} e_p(b_{11}(\alpha^2 x - \beta^2 y))\\
 &= - (p-1)^3 p^3- (p-1)^3p^4 \left(\frac{xy}{p}\right) .
\end{align*}
\item[20.] Here 
 \begin{equation}
  [ x_0, g\cdot \xi_1] + [ x_1, g \cdot \xi_0] = b_{11}(\alpha^2 x- \beta^2 y )\lambda .
 \end{equation}
 Replace $b_{11}:= b_{11}\lambda$.  Hence,
\begin{align*}
 \sS(x, \xi)& = \sum_{\begin{pmatrix}\alpha, \beta, \lambda & *\\& *\\& *\end{pmatrix}\begin{pmatrix} b_{11} & * \\ & b_{22} \end{pmatrix}}  e_p(b_{11}(\alpha^2 x - \beta^2 y)\lambda )\\
 &= (p-1)^4 p^3+(p-1)^4p^4\left(\frac{xy}{p}\right) .
\end{align*}

 \end{enumerate}

\subsection{The exponential sums pair $(\sO_{1^21^2}, \sO_{D2})$}\label{O1212_OD2_section}
Standard representatives are
\begin{equation}
 x_0 = \begin{pmatrix} 0&0&0\\ &0&0\\ &&1\end{pmatrix}\begin{pmatrix} 0 & \frac{1}{2} &0\\&0&0\\&&0\end{pmatrix}, \qquad \xi_0 = \begin{pmatrix} -\ell &0&0\\ &1&0\\&&0\end{pmatrix}\begin{pmatrix}0&0&0\\&0&0\\&&0\end{pmatrix}, 
\end{equation}
and
\begin{equation}
 x_1 = \begin{pmatrix} x &0&0\\&y&0\\&&0\end{pmatrix}\begin{pmatrix}0&0&0\\&0&0\\&&0\end{pmatrix}.
\end{equation}
The acting set is
\begin{equation}
 G_{x, \xi}^t = \left\{\begin{pmatrix} a &  c\ell & a_{13} \\ c\lambda & a\lambda & a_{23}\\ && a_{33}\end{pmatrix}\begin{pmatrix} b_{11} & b_{12} \\ & b_{22}\end{pmatrix}\right\}.
\end{equation}
The range of summation is $(a,c) \in \bF_p^2 \setminus \{(0,0)\}$, $\lambda, a_{33}, b_{11}, b_{22} \in \bF_p^\times$, and $a_{13}, a_{23}, b_{12} \in \bF_p$.

The $\GL_3$ action maps
\begin{align*}
 -\ell u^2 + v^2 \mapsto &-\ell(a u + c\lambda v)^2 + ( c\ell u + a\lambda v)^2\\
 &= (- a^2\ell +  c^2\ell^2) u^2 + (a^2\lambda^2 -  c^2 \lambda^2\ell)v^2 \bmod uv.
\end{align*}
Thus
\begin{equation}
 [ x_1, g\cdot \xi_0] = b_{11}(a^2 -  c^2\ell)(-\ell x + \lambda^2 y).
\end{equation}
The stabilizer has size
\begin{equation}
 \left|\Stab_{G(\zed/p^2\zed)}(x)\right| = \left\{\begin{array}{lll}8(p-1)p^3 && x = y\\ 4(p-1)p^3 && x\neq y\end{array}\right..
\end{equation}

 Representatives and exponential sum pairings are given in the table below. Summing over the three maximal orbits
 \begin{equation}
  \sM(x,\xi) = \sum_{(x,y) \in \{(1,1),(1,\ell), (\ell,\ell)\}} \frac{2^{\delta(x \neq y)}p^{10}\sS(x,\xi)}{8(p-1)}.
 \end{equation}

 See \textbf{exponential\_sums\_O1212\_OD2.nb}.
 
{\tiny
 \begin{equation*}
  \begin{array}{|l|l|l|l|l|}
   \hline
   \text{Orbit} & \xi_1 & [ x_0, g \cdot \xi_1 ] + [ x_1, g \cdot \xi_0 ] & \sS(x,\xi) &\sM(x,\xi)\\
   \hline
   1. & \begin{pmatrix} 0 &0&0\\&0&0\\&&0\end{pmatrix}\begin{pmatrix} 0&0&0\\&0&0\\&&1\end{pmatrix} & \begin{array}{l}-b_{11}(a^2-c^2 \ell)(\ell x - \lambda^2 y)\\ + b_{12}a_{33}^2 + b_{22}a_{13}a_{23}\end{array} &  0 &0\\
   \hline
   2. & \begin{pmatrix} 0 &0&0\\&0&0\\&&0\end{pmatrix}\begin{pmatrix} 0&0&0\\&0&1\\&&1\end{pmatrix} & \begin{array}{l}-b_{11}(a^2-c^2 \ell)(\ell x - \lambda^2 y) \\+ b_{12}a_{33}^2 \\+ b_{22}(a_{13}a_{23} + a_{13} a\lambda + a_{23}c \ell) \end{array} & 0 &0\\
   \hline
   3. & \begin{pmatrix} 0 &0&0\\&0&0\\&&0\end{pmatrix}\begin{pmatrix} 0&1&0\\&0&0\\&&1\end{pmatrix} & \begin{array}{l}-b_{11}(a^2-c^2 \ell)(\ell x - \lambda^2 y)\\+ b_{12}a_{33}^2\\ + b_{22}(a_{13}a_{23} +  a^2\lambda +  c^2 \lambda\ell)\end{array} & 0 &0\\
   \hline
   4. & \begin{pmatrix} 0 &0&0\\&0&0\\&&0\end{pmatrix}\begin{pmatrix} 0&0&0\\&0&1\\&&0\end{pmatrix} & \begin{array}{l}-b_{11}(a^2-c^2 \ell)(\ell x - \lambda^2 y)\\ + b_{22}(a_{13}a \lambda  + a_{23}c\ell)\end{array} &0 &0\\
   \hline
   5. & \begin{pmatrix} 0 &0&0\\&0&0\\&&1\end{pmatrix}\begin{pmatrix} 0&0&0\\&0&1\\&&0\end{pmatrix} & \begin{array}{l}-b_{11}(a^2-c^2 \ell)(\ell x - \lambda^2 y)\\ + b_{11}a_{33}^2 + b_{22}(a_{13} a \lambda + a_{23}c\ell)\end{array} & 0&0\\
   \hline
   6. & \begin{pmatrix} 0 &0&0\\&0&0\\&&0\end{pmatrix}\begin{pmatrix} 0&0&0\\&1&0\\&&0\end{pmatrix} & \begin{array}{l} -b_{11}(a^2-c^2 \ell)(\ell x - \lambda^2 y)\\ + b_{22} a c\lambda \ell \end{array} & \begin{array}{l} (p-1)^3p^3\\ - (p-1)^3p^4\left(\frac{xy}{p}\right)\end{array} & \frac{(p-1)^2p^{13}}{2}\\
   \hline
   7. & \begin{pmatrix} 0 &0&0\\&0&0\\&&1\end{pmatrix}\begin{pmatrix} 0&0&0\\&1&0\\&&0\end{pmatrix} &\begin{array}{l} -b_{11}(a^2-c^2 \ell)(\ell x - \lambda^2 y)\\ +b_{11}a_{33}^2 + b_{22}  a c\lambda\ell\end{array} & \begin{array}{l} -(p-1)^2p^3\\ + (p-1)^2p^4 \left(\frac{xy}{p}\right)\\ + (p-1)^2p^5 \Bigl(\left(\frac{x}{p}\right) - \left(\frac{-x}{p}\right)\\ + \left(\frac{y}{p}\right) - \left(\frac{-y}{p}\right)\Bigr)\end{array} & -\frac{(p-1)p^{13}}{2}\\
   \hline
   8a. & \begin{pmatrix} 0 &0&0\\&0&0\\&&0\end{pmatrix}\begin{pmatrix} 0&1&0\\&0&0\\&&0\end{pmatrix} & \begin{array}{l} -b_{11}(a^2-c^2 \ell)(\ell x - \lambda^2 y)\\ + b_{22}(a^2\lambda  +  c^2\lambda \ell)\end{array} & \begin{array}{l} - (p-1)^2p^3 (p+1)\\ + (p-1)^2p^4(p+1)\left(\frac{xy}{p}\right)\end{array} & -\frac{(p-1)p^{13}(p+1)}{2}\\
   \hline
   8b. & \begin{pmatrix} 0 &0&0\\&0&0\\&&0\end{pmatrix}\begin{pmatrix} 0&k&0\\&1&0\\&&0\end{pmatrix} &\begin{array}{l} -b_{11}(a^2-c^2 \ell)(\ell x - \lambda^2 y) \\+ b_{22}( a c \lambda \ell +  a^2 \lambda k +  c^2\lambda \ell k)\end{array} & \begin{array}{l} -(p-1)^2 p^3 (p+1)\\ + (p-1)^2 p^4 (p+1)\left(\frac{xy}{p}\right)\end{array}& -\frac{(p-1)p^{13}(p+1)}{2}\\
   \hline
   9a. & \begin{pmatrix} 0 &0&0\\&0&0\\&&1\end{pmatrix}\begin{pmatrix} 0&1&0\\&0&0\\&&0\end{pmatrix} &\begin{array}{l} -b_{11}(a^2-c^2 \ell)(\ell x - \lambda^2 y)\\ + b_{11}a_{33}^2 + b_{22}( a^2\lambda +  c^2\lambda \ell)\end{array} & \begin{array}{l} (p-1)p^3(p+1)\\ - (p-1)p^4(p+1)\left(\frac{xy}{p}\right)\end{array} & \frac{p^{13}(p+1)}{2}\\
   \hline
   9b. & \begin{pmatrix} 0 &0&0\\&0&0\\&&1\end{pmatrix}\begin{pmatrix} 0&k&0\\&1&0\\&&0\end{pmatrix} & \begin{array}{l}-b_{11}(a^2-c^2 \ell)(\ell x - \lambda^2 y)\\ + b_{11}a_{33}^2\\ + b_{22}( a c\lambda \ell +  a^2\lambda k + c^2\lambda  \ell k) \end{array}& \begin{array}{l} (p-1)p^3(p+1)\\ - (p-1)p^4(p+1)\left(\frac{xy}{p}\right)\end{array}& \frac{p^{13}(p+1)}{2}\\
   \hline
   10. & \begin{pmatrix} 0 &0&0\\&0&0\\&&1\end{pmatrix}\begin{pmatrix} 0&0&0\\&0&0\\&&0\end{pmatrix} &\begin{array}{l} -b_{11}(a^2-c^2 \ell)(\ell x - \lambda^2 y)\\ + b_{11}a_{33}^2\end{array}& \begin{array}{l} -(p-1)^2p^3(p+1)\\ + (p-1)^2 p^4 (p+1)\left(\frac{xy}{p}\right)\end{array} &-\frac{(p-1)p^{13}(p+1)}{2}\\
   \hline
   11. & \begin{pmatrix} 0 &0&0\\&0&0\\&&0\end{pmatrix}\begin{pmatrix} 0&0&0\\&0&0\\&&0\end{pmatrix} & -b_{11}(a^2-c^2 \ell)(\ell x - \lambda^2 y) & \begin{array}{l} (p-1)^3p^3(p+1)\\ - (p-1)^3p^4(p+1)\left(\frac{xy}{p}\right)\end{array}& \frac{(p-1)^2p^{13}(p+1)}{2} \\
   \hline
   \end{array}
 \end{equation*}
 }
\begin{enumerate}
 \item [1.-5.] These sums vanish on summing in $b_{12}, a_{13}$ and $a_{23}$.

 \item [6.] Here
 \begin{equation}
  [ x_0, g \cdot \xi_1] + [ x_1, g\cdot \xi_0] = -b_{11}(a^2 - c^2\ell)(\ell x - \lambda^2 y) +  b_{22}a c\lambda\ell .
 \end{equation}

 Note that $a^2 - c^2 \ell \neq 0$.  Thus summation over $b_{11}$ obtains $-1$ if $\ell x - \lambda^2 y \neq 0$ and $p-1$ otherwise.  Thus, 
 \begin{align*}
  \sS(x, \xi) &= \sum_{\begin{pmatrix} a &c \ell  & * \\ c\lambda & a \lambda & * \\ && *\end{pmatrix}\begin{pmatrix} b_{11} & *\\ & b_{22}\end{pmatrix}} e_p(b_{11}(a^2-c^2 \ell )(-\ell x + \lambda^2 y) +  b_{22}  a c\lambda\ell)\\
  &= - (p-1)p^3 \sum_{(a,c) \neq (0,0), \lambda, b_{22} \in \bF_p^\times} e_p( b_{22} a c\lambda \ell)\\
  & \qquad +(p-1)p^4 \sum_{\substack{\lambda^2 \equiv \ell x y^{-1} \\ (a, c) \neq (0,0), b_{22} \in \bF_p^\times}} e_p( b_{22}  a c\lambda\ell).
  \end{align*}
  Since the number of solutions to $\ell x - \lambda^2 y = 0$ is $1 - \left(\frac{xy}{p}\right)$, and since the sum over $a$ is 0 unless $c = 0$, in which case it is $p-1$, it follows that 
  
  \begin{align*}
   \sS(x, \xi) &=  -(p-1)^4p^3 + (p-1)^3p^4 \left(1 - \left(\frac{xy}{p} \right) \right)\\ &=(p-1)^3p^3 - (p-1)^3p^4 \left(\frac{xy}{p}\right).
 \end{align*}

 \item[7.] Here
 \begin{equation}
  [ x_0, g \cdot \xi_1] + [ x_1, g\cdot \xi_0] = -b_{11}(a^2 - c^2\ell)(\ell x - \lambda^2 y) +b_{11}a_{33}^2+ b_{22} a c\lambda\ell.
 \end{equation}
  Thus
 \begin{align*}
  \sS(x, \xi) &= \sum_{\begin{pmatrix} a & c \ell & * \\ c \lambda & a \lambda & * \\ && *\end{pmatrix}\begin{pmatrix} b_{11} & *\\ & b_{22}\end{pmatrix}} e_p(b_{11}((a^2- c^2\ell )(-\ell x + \lambda^2 y) + a_{33}^2) +  b_{22} a c\lambda\ell)
  \end{align*}
  Sum in $b_{22}$ to obtain 
  \begin{align*}
   \sS(x,\xi) &= -\frac{1}{p-1}\sum_{\begin{pmatrix} a & c \ell & * \\ c \lambda & a\lambda & * \\ && *\end{pmatrix}\begin{pmatrix} b_{11} & *\\ & b_{22}\end{pmatrix}} e_p(b_{11}(( a^2- c^2\ell )(-\ell x + \lambda^2 y) + a_{33}^2) )\\
   & \qquad + \frac{p}{p-1}\sum_{\begin{pmatrix} a & c \ell & * \\ c\lambda & a\lambda & * \\ && *\end{pmatrix}\begin{pmatrix} b_{11} & *\\ & b_{22}\end{pmatrix}, a c = 0} e_p(b_{11}(( a^2-c^2 \ell )(-\ell x + \lambda^2 y) + a_{33}^2) )\\
   &= - \frac{1}{p-1}\Sigma_1 + \frac{p}{p-1}\Sigma_2.
  \end{align*}
  Here, setting apart the case $-\ell x +\lambda^2 y = 0$,  and adding and subtracting the $(a,c) = (0,0)$ term,  
  \begin{align*}
   \Sigma_1 &= (p-1)^2p^3(p+1) \sum_{\substack{\lambda^2 \equiv \ell x y^{-1}\\ b_{11}, a_{33} \in \bF_p^\times}}e_p(b_{11}a_{33}^2)\\ & \qquad + (p-1)p^3 \sum_{\substack{\lambda^2 \not \equiv \ell x y^{-1}\\ b_{11}, a_{33} \in \bF_p^\times \\ a, c \in \bF_p}} e_p(b_{11}(( a^2- c^2\ell )(-\ell x + \lambda^2 y) + a_{33}^2))\\
    &\qquad - (p-1)p^3 \sum_{\substack{\lambda^2 \not \equiv \ell x y^{-1} \\ b_{11}, a_{33} \in \bF_p^\times}}e_p(b_{11}a_{33}^2)
  \end{align*}  
   In the first and third sums, summation over $b_{11}$ gives $-1$, while in the second sum, summation over $a, c$ gives $-p$ via the Gauss sums, followed by summation over $b_{11}$ which gives $-1$ again once the term involving $a$ and $c$ has been removed.  This obtains
   \begin{align*}
    &= -(p-1)^3p^3(p+1) \left(1 - \left(\frac{xy}{p}\right)\right) + (p-1)^2 p^4 \left(p-2 +\left(\frac{xy}{p}\right)\right)\\ &\qquad + (p-1)^2p^3\left(p-2 + \left(\frac{xy}{p}\right)\right)\\
    &= -(p-1)^2p^3(p+1) + (p-1)^2 p^4 (p+1) \left(\frac{xy}{p}\right).
  \end{align*}
  To evaluate $\Sigma_2$, set apart the $c = 0$ and $a = 0$ terms to obtain
  \begin{align*}
   \Sigma_2 &= (p-1)p^3\sum_{\lambda,a,b_{11}, a_{33} \in \bF_p^\times} e_p(b_{11}(- a^2\ell x +  a^2 \lambda^2 y + a_{33}^2)) \\& \qquad + (p-1)p^3 \sum_{\lambda,c,b_{11}, a_{33} \in \bF_p^\times} e_p(b_{11}(  c^2 \ell^2 x - c^2\lambda^2 \ell y + a_{33}^2)).
   \end{align*}
   In these sums, replace $\lambda:= \lambda a$ and $\lambda := \lambda c$, then sum the squared variables to obtain Gauss sums.  Thus,
   \begin{align*}
   \Sigma_2&= (p-1)p^3\sum_{b_{11} \in \bF_p^\times}\left(\tau \left(\frac{-b_{11}\ell x }{p}\right) - 1 \right)\left(\tau \left(\frac{b_{11}y }{p}\right) -1\right)\left(\tau \left(\frac{b_{11}}{p}\right) -1\right)\\
   &\qquad + (p-1)p^3 \sum_{b_{11}\in \bF_p^\times} \left(\tau \left(\frac{-b_{11}\ell y }{p}\right) - 1 \right)\left(\tau \left(\frac{b_{11}x }{p}\right) -1\right)\left(\tau \left(\frac{b_{11}}{p}\right) -1\right)\\
   &= -2(p-1)^2p^3 + (p-1)^2p^4\left\{2\left(\frac{xy}{p}\right) + \left(\frac{x}{p}\right) - \left(\frac{-x}{p}\right) + \left(\frac{y}{p}\right) - \left(\frac{-y}{p}\right) \right\}
  \end{align*}
  Hence
  \begin{align*}
   \sS(x, \xi) &= -(p-1)^2p^3 + (p-1)^2p^4 \left(\frac{xy}{p}\right) \\& \qquad + (p-1)p^5\left\{\left(\frac{x}{p}\right)- \left(\frac{-x}{p}\right) + \left(\frac{y}{p}\right) - \left(\frac{-y}{p}\right)\right\}.
  \end{align*}

\item[8a.]  $-1 = \square$.  Here
 \begin{equation}
  [ x_0, g \cdot \xi_1] + [ x_1, g\cdot \xi_0] = -b_{11}(a^2 - c^2\ell)(\ell x - \lambda^2 y) +b_{22}( a^2\lambda +  c^2\lambda \ell).
 \end{equation}
Here $ a^2\lambda +  c^2\lambda \ell \neq 0$ so the sum in $b_{22}$ is $-1$.  Since the sum in $b_{11}$ is $-1$ if $\ell x- \lambda^2 y \neq 0$ and $p-1$ otherwise, it follows that
\begin{align*}
 \sS(x, \xi) &= \sum_{\begin{pmatrix} a & c \ell & *\\ c\lambda & a\lambda & *\\ & & * \end{pmatrix} \begin{pmatrix} b_{11} & * \\ & b_{22}\end{pmatrix}}e_p(b_{11}(a^2-c^2 \ell )(-\ell x + \lambda^2 y)  + b_{22} (a^2\lambda +  c^2 \lambda\ell))\\
 &= - (p-1)p^3 \sum_{(a,c) \neq (0,0),\lambda,b_{22} \in \bF_p^\times} e_p(b_{22}(a^2 \lambda +  c^2 \lambda\ell))\\
 & \qquad + (p-1)p^4 \sum_{\substack{\lambda^2 \equiv \ell xy^{-1}\\ (a, c) \neq (0,0), b_{22} \in \bF_p^\times}} e_p(b_{22}(a^2 \lambda +  c^2 \lambda\ell))\\
 &= (p-1)^3p^3(p+1) - (p-1)^2p^4(p+1) \left(1 - \left(\frac{xy}{p}\right)\right)\\
 &= -(p-1)^2p^3(p+1) + (p-1)^2p^4(p+1) \left(\frac{xy}{p}\right).
\end{align*}
\item[8b.]  $-1 \neq \square$, $\ell - 4k^2 = \square$. Here
\begin{equation}
 [ x_0, g \cdot \xi_1] + [ x_1, g \cdot \xi_0 ]=-b_{11}(a^2 - c^2\ell)(\ell x - \lambda^2 y)+ b_{22}( (a^2 \lambda +  c^2 \lambda\ell)k +   a c\lambda\ell).
\end{equation}
Note that as a quadratic form in $a, c$, the discriminant is $\lambda^2\ell (\ell - 4k^2) \neq \square$ so that this form is irreducible. Thus   the sum in $b_{22}$ is $-1$ as in 8a, which reduces to that case,
\begin{align*}
\sS(x, \xi)= -(p-1)^2p^3(p+1) + (p-1)^2p^4(p+1)\left(\frac{xy}{p}\right).
\end{align*}

\item[9a.]  $-1 = \square$.  Here
\begin{equation}
 [ x_0, g \cdot \xi_1 ]+ [ x_1, g\cdot \xi_0] = -b_{11}(a^2 - c^2\ell)(\ell x - \lambda^2 y)+b_{11} a_{33}^2 + b_{22}(a^2\lambda +  c^2 \lambda\ell).
\end{equation}
Since $a^2 \lambda +  c^2 \lambda\ell \neq 0$, summing in $b_{22}$ obtains $-1$, so
\begin{align*}
 \sS(x, \xi) &= -\frac{1}{p-1} \sum_{\begin{pmatrix} a & c \ell & *\\ c\lambda & a\lambda & *\\ & & * \end{pmatrix} \begin{pmatrix} b_{11} & * \\ & b_{22}\end{pmatrix}} e_p(b_{11}((a^2- c^2\ell )(-\ell x + \lambda^2 y)+a_{33}^2))\\
 &=(p-1)p^3(p+1) - (p-1)p^4(p+1)\left(\frac{xy}{p}\right),
\end{align*}
see $\Sigma_1$ of 7.
\item[9b.]  Here
\begin{equation}
 [ x_0, g \cdot \xi_1 ]+ [ x_1, g\cdot \xi_0] = -b_{11}(a^2 - c^2\ell)(\ell x - \lambda^2 y)+b_{11} a_{33}^2 + b_{22}( a c\lambda\ell +  a^2\lambda k +  c^2 \lambda\ell k).
\end{equation} This is equal to 9a., see the calculation in 8b.
\item[10.]  Here
\begin{equation}
 [ x_0, g \cdot \xi_1 ]+ [ x_1, g\cdot \xi_0] = -b_{11}(a^2 - c^2\ell)(\ell x - \lambda^2 y)+b_{11} a_{33}^2 .
\end{equation}

This is $\Sigma_1$ of 7., so
\begin{align*}
 \sS(x, \xi) &=  \sum_{\begin{pmatrix} a & c \ell & *\\ c\lambda & a\lambda & *\\ & & * \end{pmatrix} \begin{pmatrix} b_{11} & * \\ & b_{22}\end{pmatrix}} e_p(b_{11}((a^2-c^2 \ell )(-\ell x + \lambda^2 y)+a_{33}^2))\\
 &= -(p-1)^2p^3(p+1) + (p-1)^2p^4(p+1)\left(\frac{xy}{p}\right).
\end{align*}
\item[11.]  Here
\begin{equation}
 [ x_0, g \cdot \xi_1 ]+ [ x_1, g\cdot \xi_0] = -b_{11}(a^2 - c^2\ell)(\ell x - \lambda^2 y).
\end{equation} Thus, splitting on $\ell x - \lambda^2 y = 0$ or not, 
\begin{align*}
 \sS(x,\xi) &= \sum_{\begin{pmatrix} a & c \ell & *\\ c\lambda & a\lambda & *\\ & & * \end{pmatrix} \begin{pmatrix} b_{11} & * \\ & b_{22}\end{pmatrix}} e_p(b_{11}(a^2-c^2 \ell )(-\ell x + \lambda^2 y))\\
 &= -(p-1)^4 p^3(p+1) + (p-1)^3p^4(p+1)\left(1 - \left(\frac{xy}{p}\right)\right)\\
 &= (p-1)^3p^3(p+1) - (p-1)^3p^4(p+1)\left(\frac{xy}{p}\right).
\end{align*}

\end{enumerate}

\subsection{The exponential sums pair $(\sO_{1^31}, \sO_{D1^2})$}\label{O131_OD12_section}
Standard representatives are
\begin{equation}
 x_0 = \begin{pmatrix} 0&0&0\\&0&\frac{1}{2}\\ &&0\end{pmatrix}\begin{pmatrix} 0&0&\frac{1}{2}\\ &1&0\\&&0\end{pmatrix}, \qquad \xi_0 = \begin{pmatrix} 1 &0&0\\&0&0\\&&0\end{pmatrix}\begin{pmatrix}0&0&0\\&0&0\\&&0\end{pmatrix},
\end{equation}
and
\begin{equation}
 x_1 = \begin{pmatrix} \epsilon &0&0\\&0&0\\&&0\end{pmatrix}\begin{pmatrix}0&0&0\\&0&0\\&&0\end{pmatrix}, \qquad \epsilon \in \bF_p^\times/\{x^3: x \in \bF_p^\times\}.
\end{equation}
The acting set is
\begin{equation}
 G_{x,\xi}^t = \begin{pmatrix} a_{11} & a_{12} & a_{13}\\ & a & b\\ & c & d\end{pmatrix}\begin{pmatrix} b_{11} & b_{12} \\ & b_{22}\end{pmatrix}.
\end{equation}
Thus
\begin{equation}
 [ x_1, g\cdot \xi_0] =  b_{11} a_{11}^2\epsilon.
\end{equation}
$\GL_2$ is fibered as for the pairs $(\sO_{1^211}, \sO_{D1^2})$, $(\sO_{1^22}, \sO_{D1^2})$ and $(\sO_{1^21^2}, \sO_{D1^2})$.

The stabilizer has size
\begin{equation}
\left|\Stab_{G(\zed/p^2\zed)}(x)\right| = (p-1)p^3\#\{\epsilon^3=1\}. 
\end{equation}

Representatives and exponential sum pairings are given in the table below. Summing over the maximal orbits,
\begin{equation}
 \sM(x,\xi) = \sum_{\epsilon \in \bF_p^\times/\{x^3: x \in \bF_p^\times\}} \frac{p^{10}\sS(x,\xi)}{(p-1)\#\{\epsilon^3 =1 \}}.
\end{equation}

See \textbf{exponential\_sums\_O131\_OD12.nb}.

{\tiny
 \begin{equation*}
  \begin{array}{|l|l|l|l|l|}
   \hline
   \text{Orbit} & \xi_1 & [ x_0, g \cdot \xi_1 ]+ [ x_1, g \cdot \xi_0 ] & \sS(x,\xi) & \sM(x,\xi)\\
   \hline
   1. & \begin{pmatrix} 0 &0&0\\&0&0\\&&0\end{pmatrix}\begin{pmatrix} 0&0&0\\&0&0\\&&0\end{pmatrix} & b_{11}a_{11}^2 \epsilon & -(p-1)^4p^4(p+1)  & -(p-1)^3 p^{14}(p+1)\\
   \hline
   2. & \begin{pmatrix} 0 & 0&0\\ &0&0\\ &&1\end{pmatrix}\begin{pmatrix} 0&0&0\\&0&0\\&&0\end{pmatrix} &  b_{11}(a_{11}^2 \epsilon+ bd) & -(p-1)^4p^4 & -(p-1)^3 p^{14}\\
   \hline
   3. & \begin{pmatrix} 0 &0&0\\&0&\frac{1}{2}\\ &&0\end{pmatrix}\begin{pmatrix}0&0&0\\&0&0\\&&0\end{pmatrix} & b_{11}\left(a_{11}^2 \epsilon+\frac{1}{2}(ad+bc )\right)  & 2(p-1)^3p^4 & 2(p-1)^2p^{14}\\
   \hline
   4. & \begin{pmatrix} 0 & 0 & 0\\ & -\ell &0\\ &&1\end{pmatrix}\begin{pmatrix} 0 &0&0\\ &0&0\\&&0\end{pmatrix} & b_{11}(a_{11}^2 \epsilon - ac\ell+bd) &0 &0 \\
   \hline
   5. & \begin{pmatrix} 0 &0&0\\ &0&0\\&&1\end{pmatrix}\begin{pmatrix} 0&0&1\\ &0&0\\&&0\end{pmatrix} & b_{11}(a_{11}^2\epsilon+ bd)  + b_{22}a_{11}d & -(p-1)^4p^4 & -(p-1)^3p^{14}\\
   \hline
   6. & \begin{pmatrix} 0 &0&0\\ &1 &0\\&&0\end{pmatrix}\begin{pmatrix}0 &0&1\\ &0&0\\&&0\end{pmatrix} & b_{11}(a_{11}^2\epsilon+ac) + b_{22}a_{11}d  & (p-1)^3p^4 & (p-1)^2p^{14} \\
   \hline
   7. & \begin{pmatrix} 0 &0 &0\\ &0 & \frac{1}{2}\\ &&0\end{pmatrix}\begin{pmatrix}0 &0&1\\ &0&0\\&&0\end{pmatrix} & b_{11}\left(a_{11}^2\epsilon + \frac{1}{2}(ad+bc)\right) + b_{22}a_{11}d  & (p-1)^2p^4(p-2) & (p-1)p^{14}(p-2)\\
   \hline
   8. & \begin{pmatrix} 0 &0&0\\ &1&1\\&&0\end{pmatrix} \begin{pmatrix} 0&0&1\\&0&0\\&&0\end{pmatrix} & \begin{array}{l}b_{11}(a_{11}^2\epsilon+ac  + ad+ bc) \\+ b_{22}a_{11}d \end{array} &-2(p-1)^2p^4 & -2(p-1)p^{14}\\
   \hline
   9. & \begin{pmatrix} 0 &0&0\\ &-\ell &0\\ &&1\end{pmatrix}\begin{pmatrix} 0&0&1\\ &0&0\\&&0\end{pmatrix} & b_{11}(a_{11}^2\epsilon- ac\ell+bd ) + b_{22}a_{11}d  &0&0\\
   \hline
   10. & \begin{pmatrix} 0 &0&0\\ &0&0\\&&0\end{pmatrix}\begin{pmatrix}0 &0&1 \\ & 0&0\\ &&0\end{pmatrix} & b_{11}a_{11}^2 \epsilon+b_{22}a_{11}d  &(p-1)^3p^4 & (p-1)^2 p^{14}\\
%
%
   \hline   11. & \begin{pmatrix} 0&0&0\\&0&0\\&&0\end{pmatrix}\begin{pmatrix}0&0&0\\&0&0\\&&1\end{pmatrix} &\begin{array}{l}b_{11}a_{11}^2 \epsilon+ b_{12}bd \\+ b_{22}( a_{13}d+b^2 )\end{array} &(p-1)^3p^4 &(p-1)^2p^{14}\\
   \hline
   12. & \begin{pmatrix} 0&0&0\\&0&1\\&&0\end{pmatrix}\begin{pmatrix}0&0&0\\&0&0\\&&1\end{pmatrix} & \begin{array}{l}b_{11}(a_{11}^2\epsilon+ad+bc) + b_{12}bd \\+ b_{22}( a_{13}d+b^2 ) \end{array}& -(p-1)^2p^4 & -(p-1)p^{14}\\
   \hline
   13. & \begin{pmatrix} 0&0&0\\&1&0\\&&0\end{pmatrix}\begin{pmatrix}0&0&0\\&0&0\\&&1\end{pmatrix}& \begin{array}{l}b_{11}(a_{11}^2\epsilon+ac) + b_{12}bd\\ + b_{22}( a_{13}d+b^2 )\end{array} &0&0\\
   \hline
   14. & \begin{pmatrix} 0&0&0\\&0&0\\&&0\end{pmatrix}\begin{pmatrix}0&1&0\\&0&0\\&&1\end{pmatrix}& \begin{array}{l}b_{11}a_{11}^2 \epsilon+b_{12}bd \\+ b_{22}(a_{11}c + a_{13}d+b^2  ) \end{array} &-(p-1)^2p^4 & -(p-1)p^{14}\\
   \hline
   15. & \begin{pmatrix} 0&0&0\\&0&1\\&&0\end{pmatrix}\begin{pmatrix}0&1&0\\&0&0\\&&1\end{pmatrix}& \begin{array}{l}b_{11}(a_{11}^2 \epsilon+ad+ bc ) + b_{12}bd \\+ b_{22}(a_{11}c+ a_{13}d + b^2 ) \end{array} & \begin{array}{l} - (p-1)^2p^4(p+1)\\ + (p-1)p^6\\\times \#\{x \in \bF_p^\times: x^3 \equiv \epsilon\}\end{array}& p^{14}\\
   \hline
   16. &\begin{pmatrix} 0&0&0\\&1&0\\&&0\end{pmatrix}\begin{pmatrix}0&1&0\\&0&0\\&&1\end{pmatrix} & \begin{array}{l}b_{11}(a_{11}^2\epsilon+ac) + b_{12}bd\\ + b_{22}(a_{11}c + a_{13}d+b^2 )\end{array}&0&0\\
   \hline
   17. & \begin{pmatrix} 0&0&0\\&0&0\\&&0\end{pmatrix}\begin{pmatrix}0&0&0\\&0&1\\&&0\end{pmatrix}& \begin{array}{l}b_{11}a_{11}^2\epsilon+b_{12}(ad+bc)\\ + b_{22}( a_{12}d+ a_{13}c +2ab )\end{array} &0&0\\
   \hline
   18. &\begin{pmatrix} 0&0&0\\&0&0\\&&1\end{pmatrix}\begin{pmatrix}0&0&0\\&0&1\\&&0\end{pmatrix} & \begin{array}{l}b_{11}(a_{11}^2\epsilon+bd) + b_{12}(ad+bc ) \\+ b_{22}(a_{12}d+a_{13}c + 2ab  ) \end{array} &0&0\\
   \hline
   19. & \begin{pmatrix} 0&0&0\\&1&0\\&&1\end{pmatrix}\begin{pmatrix}0&0&0\\&0&1\\&&0\end{pmatrix}& \begin{array}{l}b_{11}(a_{11}^2\epsilon+ac+bd) + b_{12}(ad+bc)\\ + b_{22}(a_{12}d+a_{13}c + 2ab) \end{array}&0&0\\
   \hline
   20. &\begin{pmatrix} 0&0&0\\&\ell&0\\&&1\end{pmatrix}\begin{pmatrix}0&0&0\\&0&1\\&&0\end{pmatrix} &\begin{array}{l}b_{11}(a_{11}^2\epsilon+ ac\ell+bd) + b_{12}(ad+bc)\\ + b_{22}(a_{12}d +a_{13}c +  2ab) \end{array}&0&0 \\
   \hline
   21. & \begin{pmatrix} 0 &0&0\\&0&0\\&&0\end{pmatrix}\begin{pmatrix}0&0&0\\ &-\ell &0\\&&1\end{pmatrix} & \begin{array}{l}b_{11}a_{11}^2\epsilon+b_{12}(-ac\ell+bd)\\ + b_{22}(-a_{12}c\ell+a_{13}d-a^2\ell+b^2)\end{array}&0&0\\
   \hline
   22. & \begin{pmatrix} 0 &0&0\\&1&0\\&&0\end{pmatrix}\begin{pmatrix}0&0&0\\ &-\ell &0\\&&1\end{pmatrix} & \begin{array}{l} b_{11}(a_{11}^2\epsilon+ac) +b_{12}(-ac\ell+bd)\\ + b_{22}(-a_{12}c\ell+a_{13}d-a^2\ell+b^2)\end{array}&0&0\\
   \hline
   23a. &\begin{pmatrix} 0 &0&0\\&0&1\\&&0\end{pmatrix}\begin{pmatrix}0&0&0\\ &-\ell &0\\&&1\end{pmatrix} & \begin{array}{l}b_{11}(a_{11}^2\epsilon+ad + bc) + b_{12}(-ac\ell+bd)\\ + b_{22}(-a_{12}c\ell+a_{13}d-a^2\ell+b^2) \end{array}&0&0\\
   \hline
   23b. &\begin{pmatrix} 0 &0&0\\&1&k\\&&0\end{pmatrix}\begin{pmatrix}0&0&0\\ &-\ell &0\\&&1\end{pmatrix} & \begin{array}{l} b_{11}(a_{11}^2\epsilon+ac+ adk + bck )\\ + b_{12}(-ac\ell+bd) \\+ b_{22}(-a_{12}c\ell+a_{13}d-a^2\ell+b^2)\end{array}&0&0\\
   \hline
  \end{array}  
 \end{equation*}
}
 
The exponential sums are as follows.
\begin{enumerate}
 \item [1.] Here
 \begin{equation}
  [ x_0, g\cdot \xi_1] + [ x_1, g \cdot \xi_0] = b_{11}a_{11}^2 \epsilon.
 \end{equation}
Thus
 \begin{equation}
  \sum_{\begin{pmatrix} * & * & *\\ & * & *\\  & * & *\end{pmatrix}\begin{pmatrix} * & *\\ & *\end{pmatrix}} e_p( b_{11}a_{11}^2\epsilon) = -\frac{|G_{x,\xi}|}{p-1} = - (p-1)^4 p^4 (p+1).
 \end{equation}
 \item [2.] Here
 \begin{equation}
  [ x_0, g\cdot \xi_1] + [ x_1, g \cdot \xi_0] = b_{11}(a_{11}^2 \epsilon+ bd).
 \end{equation}
\begin{align*}
 \Sigma_1 &= \sum_{\begin{pmatrix} * & * & * \\ & a & b\\ & c & d\end{pmatrix} \begin{pmatrix} b_{11} & * \\ & *\end{pmatrix}} e_p(b_{11}(a_{11}^2 \epsilon + bd)) = - (p-1)^2p^6\\
 \Sigma_2 &= \sum_{\begin{pmatrix} * & * & * \\ & b \lambda  & b\\ & d \lambda  & d\end{pmatrix} \begin{pmatrix} b_{11} & *\\& *\end{pmatrix}}e_p(b_{11}(a_{11}^2 \epsilon + bd)) = -(p-1)^3p^4\\
 \Sigma_3 &= \sum_{\begin{pmatrix} * & * & * \\ & a & 0\\ & c &0\end{pmatrix} \begin{pmatrix} b_{11} & * \\ & * \end{pmatrix}} e_p(b_{11} a_{11}^2 \epsilon) = -(p-1)^2 p^5\\
 \Sigma_4 &= \sum_{\begin{pmatrix} * & * & * \\ & 0 & b\\ & 0 &d\end{pmatrix} \begin{pmatrix} b_{11} & * \\ & * \end{pmatrix}} e_p(b_{11}(a_{11}^2 \epsilon + bd)) = -(p-1)^2 p^4\\
 \Sigma_5 &= p \sum_{\begin{pmatrix} * & * & *\\ &0&0\\&0&0\end{pmatrix}\begin{pmatrix} b_{11} & *\\ & *\end{pmatrix}} e_p(b_{11} a_{11}^2 \epsilon) = - (p-1)^2p^4.
\end{align*}
Thus
\begin{equation}
 \sS(x,\xi) = \Sigma_1 - \Sigma_2 -\Sigma_3 - \Sigma_4 +\Sigma_5 = -(p-1)^4p^4.
\end{equation}
\item [3.]  Here
 \begin{equation}
  [ x_0, g\cdot \xi_1] + [ x_1, g \cdot \xi_0] = b_{11}\left(a_{11}^2 \epsilon+\frac{1}{2}(ad+bc )\right).
 \end{equation}
\begin{align*}
 \Sigma_1 &= \sum_{\begin{pmatrix} * & * & * \\ & a & b\\ & c & d\end{pmatrix} \begin{pmatrix} b_{11} & * \\ & *\end{pmatrix}} e_p\left(b_{11}\left(a_{11}^2 \epsilon+\frac{1}{2}(ad+bc )\right)\right)\\
 &= -(p-1)^2p^5\\
 \Sigma_2 &= \sum_{\begin{pmatrix} * & * & *\\ & b \lambda  & b\\ & d \lambda  & d\end{pmatrix} \begin{pmatrix} b_{11} & * \\ & *\end{pmatrix}} e_p(b_{11}(a_{11}^2\epsilon  + bd \lambda) ) \\&= -(p-1)^3p^4\\
 \Sigma_3 &= \sum_{\begin{pmatrix} * & * & *\\ & a& 0\\ &c &0\end{pmatrix}\begin{pmatrix} b_{11} & *\\& *\end{pmatrix}} e_p(b_{11} a_{11}^2 \epsilon) = - (p-1)^2 p^5\\
 \Sigma_4 &= \sum_{\begin{pmatrix} * & * & *\\ & 0 &b\\ &0&d \end{pmatrix}\begin{pmatrix} b_{11} & *\\& *\end{pmatrix}} e_p(b_{11} a_{11}^2 \epsilon) = - (p-1)^2 p^5\\
 \Sigma_5 &= p \sum_{\begin{pmatrix} * & * & * \\ & 0 &0\\&0&0\end{pmatrix}\begin{pmatrix} b_{11} & *\\ & *\end{pmatrix}} e_p(b_{11} a_{11}^2 \epsilon) = -(p-1)^2p^4.
\end{align*}
Thus
\begin{equation}
 \sS(x,\xi) = 2 (p-1)^3 p^4.
\end{equation}
\item [4.]   Here
 \begin{equation}
  [ x_0, g\cdot \xi_1] + [ x_1, g \cdot \xi_0] =  b_{11}(a_{11}^2 \epsilon - ac\ell+bd) .
 \end{equation}
\begin{align*}
 \Sigma_1 &= \sum_{\begin{pmatrix} * & * & *\\ & a & b\\ & c & d\end{pmatrix}\begin{pmatrix} b_{11} & *\\ & *\end{pmatrix}} e_p(b_{11}( a_{11}^2 \epsilon -ac \ell  + bd))= -(p-1)^2 p^5\\
 \Sigma_2 &= \sum_{\begin{pmatrix} * & * & *\\ &b \lambda  & b\\ & d \lambda  & d\end{pmatrix} \begin{pmatrix} b_{11} & *\\ & *\end{pmatrix}}e_p(b_{11} (a_{11}^2 \epsilon + bd(-\lambda^2 \ell + 1))) = -(p-1)^3p^4\\
 \Sigma_3 &= \sum_{\begin{pmatrix} * & * & *\\ & a & 0\\ & c &0\end{pmatrix}\begin{pmatrix} b_{11} & *\\ & *\end{pmatrix}} e_p(b_{11}( a_{11}^2 \epsilon  -ac\ell)) = -(p-1)^2p^4\\
 \Sigma_4 &= \sum_{\begin{pmatrix} * & * & *\\ & 0 & b\\ & 0 &d\end{pmatrix}\begin{pmatrix} b_{11} & *\\ & *\end{pmatrix}} e_p(b_{11} (a_{11}^2 \epsilon +  bd)) = -(p-1)^2p^4\\
 \Sigma_5 &= p\sum_{\begin{pmatrix} * & * & *\\ & 0& 0 \\ &0 &0\end{pmatrix} \begin{pmatrix} b_{11} & * \\ & * \end{pmatrix}} e_p(b_{11} a_{11}^2 \epsilon ) = -(p-1)^2 p^4.
\end{align*}
Thus
\begin{equation*}
 \sS(x, \xi) =  0.
\end{equation*}
\item [5.]  Here
 \begin{equation}
  [ x_0, g\cdot \xi_1] + [ x_1, g \cdot \xi_0] = b_{11}(a_{11}^2\epsilon+ bd)  + b_{22}a_{11}d.
 \end{equation}
In $\Sigma_1, \Sigma_2$ and $\Sigma_4$, summation in $b$ vanishes unless $d=0$.  Thus,
\begin{align*}
 \Sigma_1 &= \sum_{\begin{pmatrix} * & * & *\\ & a & b\\ & c & d\end{pmatrix}\begin{pmatrix} b_{11} & *\\ & *\end{pmatrix}} e_p(b_{11}(a_{11}^2\epsilon+ bd)  + b_{22}a_{11}d) = -(p-1)^2p^6\\
 \Sigma_2 &= \sum_{\begin{pmatrix} * & * & *\\ &b \lambda  & b\\ & d \lambda  & d\end{pmatrix} \begin{pmatrix}b_{11} & *\\ & b_{22}\end{pmatrix}}e_p(b_{11}(a_{11}^2\epsilon+ bd)  + b_{22}a_{11}d) = -(p-1)^3p^4\\
 \Sigma_3 &= \sum_{\begin{pmatrix} * & * & *\\ & a &0\\ & c &0\end{pmatrix}\begin{pmatrix} b_{11} & * \\ & *\end{pmatrix}} e_p(b_{11} a_{11}^2 \epsilon) = -(p-1)^2 p^5\\
 \Sigma_4 &= \sum_{\begin{pmatrix} * & * & *\\ & 0 & b\\ &0 & d\end{pmatrix}\begin{pmatrix} b_{11} & * \\ & b_{22}\end{pmatrix}} e_p(b_{11}( a_{11}^2 \epsilon +  bd) + b_{22} a_{11}d) = -(p-1)^2p^4\\
 \Sigma_5 &= p\sum_{\begin{pmatrix} * & * & *\\ &0&0\\&0&0\end{pmatrix}\begin{pmatrix} * & *\\ & *\end{pmatrix}} e_p(b_{11} a_{11}^2 \epsilon) = - (p-1)^2p^4.
\end{align*}
Thus
\begin{equation}
 \sS(x, \xi) = -(p-1)^4p^4.
\end{equation}
\item[6.] Here
 \begin{equation}
  [ x_0, g\cdot \xi_1] + [ x_1, g \cdot \xi_0] = b_{11}(a_{11}^2\epsilon+ac) + b_{22}a_{11}d .
 \end{equation}
\begin{align*}
 \Sigma_1 &= \sum_{\begin{pmatrix} * & * & *\\ & a&b\\ &c&d\end{pmatrix}\begin{pmatrix}b_{11} & * \\ & b_{22}\end{pmatrix}}e_p(b_{11}(a_{11}^2\epsilon+ac) + b_{22}a_{11}d ) = 0\\
 \Sigma_2 &= \sum_{\begin{pmatrix} * & * & *\\ & b \lambda  & b\\ & d \lambda  & d\end{pmatrix}\begin{pmatrix}b_{11} & * \\ & b_{22}\end{pmatrix}} e_p( b_{11}(a_{11}^2\epsilon+bd\lambda^2) + b_{22}a_{11}d ) = -(p-1)^3 p^4\\
 \Sigma_3 &=\sum_{\begin{pmatrix} * & * & *\\ & a &0\\ & c &0\end{pmatrix}\begin{pmatrix} b_{11} & * \\ & *\end{pmatrix}} e_p(b_{11} (a_{11}^2 \epsilon + ac)) = -(p-1)^2p^4\\
 \Sigma_4 &=\sum_{\begin{pmatrix} * & * & *\\ & 0 & b\\ &0 & d\end{pmatrix}\begin{pmatrix} b_{11} & * \\ & b_{22}\end{pmatrix}}e_p( b_{11}a_{11}^2\epsilon + b_{22}a_{11}d) = 0\\
 \Sigma_5 &= p\sum_{\begin{pmatrix} * & * & *\\ &0&0\\&0&0\end{pmatrix}\begin{pmatrix} * & *\\ & *\end{pmatrix}} e_p(b_{11} a_{11}^2 \epsilon) = - (p-1)^2p^4.
\end{align*}
Thus
\begin{equation}
 \sS(x, \xi) = (p-1)^3p^4.
\end{equation}
\item[7.]  Here
 \begin{equation}
  [ x_0, g\cdot \xi_1] + [ x_1, g \cdot \xi_0] = b_{11}\left(a_{11}^2\epsilon + \frac{1}{2}(ad+bc)\right) + b_{22}a_{11}d.
 \end{equation}
Here
\begin{align*}
 \Sigma_1 &= \sum_{\begin{pmatrix} * & * & *\\ & a & b\\ & c & d\end{pmatrix}\begin{pmatrix}b_{11} & *\\ & b_{22}\end{pmatrix}} e_p\left( b_{11}\left(a_{11}^2\epsilon + \frac{1}{2}(ad+bc)\right) + b_{22}a_{11}d\right) = -(p-1)^2p^5\\
 \Sigma_2 &= \sum_{\begin{pmatrix} * & * & *\\ & b \lambda  & b\\ & d \lambda  & d\end{pmatrix}\begin{pmatrix}b_{11} & * \\ & b_{22}\end{pmatrix}} e_p( b_{11}\left(a_{11}^2\epsilon + bd\lambda\right) + b_{22}a_{11}d) = -(p-1)^3p^4\\
 \Sigma_3 &=\sum_{\begin{pmatrix} * & * & *\\ & a &0\\ & c &0\end{pmatrix}\begin{pmatrix} b_{11} & * \\ & *\end{pmatrix}}e_p(b_{11} a_{11}^2 \epsilon) = -(p-1)^2p^5\\
 \Sigma_4 &=\sum_{\begin{pmatrix} * & * & *\\ & 0 & b\\ &0 & d\end{pmatrix}\begin{pmatrix} b_{11} & * \\ & b_{22}\end{pmatrix}} e_p(b_{11} a_{11}^2 \epsilon + b_{22}a_{11}d) = 0\\
 \Sigma_5 &= p\sum_{\begin{pmatrix} * & * & *\\ &0&0\\&0&0\end{pmatrix}\begin{pmatrix} * & *\\ & *\end{pmatrix}} e_p(b_{11} a_{11}^2 \epsilon) = - (p-1)^2p^4.
\end{align*}
Thus
\begin{equation}
 \sS(x,\xi) = (p-1)^2p^4(p-2).
\end{equation}

\item[8.] Here
 \begin{equation}
  [ x_0, g\cdot \xi_1] + [ x_1, g \cdot \xi_0] =b_{11}(a_{11}^2 \epsilon +ac+ad + bc)+  b_{22}a_{11}d.
 \end{equation}
Thus (in $\Sigma_2$, summation in $d$ vanishes if $\lambda = -2$)
\begin{align*}
 \Sigma_1 &= \sum_{\begin{pmatrix} a_{11} & * & *\\ & a&b\\&c&d\end{pmatrix}\begin{pmatrix} b_{11} & *\\ & b_{22}\end{pmatrix}} e_p(b_{11}(a_{11}^2 \epsilon + ac+ ad + bc ) + b_{22}a_{11}d)\\
 &= -(p-1)^2p^5\\
 \Sigma_2 &= \sum_{\begin{pmatrix} * & * & *\\ & b \lambda  & b\\ & d \lambda  & d\end{pmatrix}\begin{pmatrix}b_{11} & * \\ & b_{22}\end{pmatrix}} e_p(b_{11}(a_{11}^2 \epsilon + bd(\lambda^2 + 2\lambda)) + b_{22}a_{11}d )\\
 &= -(p-2)p^4(p-1)^2\\
 \Sigma_3 &= \sum_{\begin{pmatrix} * & * & *\\ & a &0\\ & c &0\end{pmatrix}\begin{pmatrix} b_{11} & * \\ & *\end{pmatrix}} e_p(b_{11}(a_{11}^2 \epsilon + ac)) = -(p-1)^2p^4\\
 \Sigma_4 &= \sum_{\begin{pmatrix} * & * & *\\ & 0 & b\\ &0 & d\end{pmatrix}\begin{pmatrix} b_{11} & * \\ & b_{22}\end{pmatrix}}e_p(b_{11}a_{11}^2 \epsilon + b_{22}a_{11}d) = 0\\
 \Sigma_5 &= p\sum_{\begin{pmatrix} * & * & *\\ &0&0\\&0&0\end{pmatrix}\begin{pmatrix} * & *\\ & *\end{pmatrix}} e_p(b_{11} a_{11}^2 \epsilon) = - (p-1)^2p^4.
\end{align*}
Thus
\begin{align*}
 \sS(x,\xi) = -2(p-1)^2p^4.
\end{align*}
\item[9.] Here
 \begin{equation}
  [ x_0, g\cdot \xi_1] + [ x_1, g \cdot \xi_0] =b_{11}(a_{11}^2\epsilon- ac\ell+bd ) + b_{22}a_{11}d.
 \end{equation}
Thus
\begin{align*}
 \Sigma_1 &= \sum_{\begin{pmatrix} a_{11} & * & *\\ & a&b\\&c&d\end{pmatrix}\begin{pmatrix} b_{11} & *\\ & b_{22}\end{pmatrix}}e_p(b_{11}(a_{11}^2\epsilon- ac\ell+bd ) + b_{22}a_{11}d) = -(p-1)^2p^5\\
 \Sigma_2 &= \sum_{\begin{pmatrix} * & * & *\\ & b \lambda  & b\\ & d \lambda  & d\end{pmatrix}\begin{pmatrix}b_{11} & * \\ & b_{22}\end{pmatrix}}e_p(b_{11}(a_{11}^2 \epsilon + bd(-\lambda^2 \ell + 1)) + b_{22}a_{11}d) = -(p-1)^3p^4\\
 \Sigma_3 &= \sum_{\begin{pmatrix} * & * & *\\ & a &0\\ & c &0\end{pmatrix}\begin{pmatrix} b_{11} & * \\ & *\end{pmatrix}}e_p(b_{11}(a_{11}^2 \epsilon - ac\ell))= -(p-1)^2p^4\\
 \Sigma_4 &= \sum_{\begin{pmatrix} * & * & *\\ & 0 & b\\ &0 & d\end{pmatrix}\begin{pmatrix} b_{11} & * \\ & b_{22}\end{pmatrix}} e_p(b_{11}(a_{11}^2 \epsilon + bd)+ b_{22}a_{11}d) = -(p-1)^2p^4\\
 \Sigma_5 &= p\sum_{\begin{pmatrix} * & * & *\\ &0&0\\&0&0\end{pmatrix}\begin{pmatrix} * & *\\ & *\end{pmatrix}} e_p(b_{11} a_{11}^2 \epsilon) = - (p-1)^2p^4.
\end{align*}
Thus $\sS(x,\xi) = 0$.
\item[10.] Here
 \begin{equation}
  [ x_0, g\cdot \xi_1] + [ x_1, g \cdot \xi_0] = b_{11}a_{11}^2 \epsilon + b_{22}a_{11}d.
 \end{equation}
Thus 
\begin{align*}
 \Sigma_1 &= \sum_{\begin{pmatrix} a_{11} & * & *\\ & a&b\\&c&d\end{pmatrix}\begin{pmatrix} b_{11} & *\\ & b_{22}\end{pmatrix}} e_p(b_{11}a_{11}^2 \epsilon + b_{22}a_{11}d) = 0\\
 \Sigma_2 &= \sum_{\begin{pmatrix} * & * & *\\ & b \lambda  & b\\ & d \lambda  & d\end{pmatrix}\begin{pmatrix}b_{11} & * \\ & b_{22}\end{pmatrix}} e_p(b_{11}a_{11}^2 \epsilon + b_{22}a_{11}d) = 0\\
 \Sigma_3 &= \sum_{\begin{pmatrix} * & * & *\\ & a &0\\ & c &0\end{pmatrix}\begin{pmatrix} b_{11} & * \\ & *\end{pmatrix}}e_p(b_{11}a_{11}^2 \epsilon) = -(p-1)^2 p^5\\
 \Sigma_4 &= \sum_{\begin{pmatrix} * & * & *\\ & 0 & b\\ &0 & d\end{pmatrix}\begin{pmatrix} b_{11} & * \\ & b_{22}\end{pmatrix}}e_p(b_{11}a_{11}^2 \epsilon + b_{22}a_{11}d) = 0\\
 \Sigma_5 &= p\sum_{\begin{pmatrix} * & * & *\\ &0&0\\&0&0\end{pmatrix}\begin{pmatrix} * & *\\ & *\end{pmatrix}} e_p(b_{11} a_{11}^2 \epsilon) = - (p-1)^2p^4.
\end{align*}
Thus
\begin{equation}
 \sS(x, \xi) = (p-1)^3p^4.
\end{equation}
\item[11.] Here
\begin{equation}
[ x_0, g\cdot \xi_1 ] + [ x_1, g \cdot \xi_0]=b_{11}a_{11}^2 \epsilon+ b_{12}bd + b_{22}( a_{13}d+b^2 ) . 
\end{equation}
Summing in $a_{13}$ forces $d = 0$.  Thus, summing in $b_{11}$ and $b_{22}$,
\begin{align*}
 \sS(x, \xi) &= \sum_{\begin{pmatrix} * & *& * \\ & a & b\\ & c & \end{pmatrix}\begin{pmatrix} * & *\\ & *\end{pmatrix}} e_p(b_{11} a_{11}^2 \epsilon + b_{22}b^2)\\
 &= (p-1)^3p^4.
\end{align*}
\item[12.] Here
\begin{equation}
 [ x_0, g\cdot \xi_1 ] + [ x_1, g \cdot \xi_0]= b_{11}(a_{11}^2\epsilon+ad+bc) + b_{12}bd + b_{22}( a_{13}d+b^2 ) . 
\end{equation}
Summing in $a_{13}$ forces $d = 0$.  Thus, summing in $b_{22}, b,$ then $b_{11}$,
\begin{align*}
 \sS(x,\xi) &= \sum_{\begin{pmatrix} * & * & *\\ & a & b\\ &c &\end{pmatrix}\begin{pmatrix}b_{11} & b_{12}\\ & b_{22}\end{pmatrix}} e_p(b_{11}( a_{11}^2 \epsilon+bc) + b_{22}b^2 )\\
 &= - (p-1)^2p^4.
\end{align*}
\item[13.] Here
\begin{equation}
 [ x_0, g \cdot \xi_1 ] + [ x_1, g \cdot \xi_0] = b_{11}(a_{11}^2\epsilon+ac) + b_{12}bd + b_{22}( a_{13}d+b^2 )  .
\end{equation}
Sum in $a_{13}$ to force $d = 0$.  Now sum in $a$ to force $c = 0$.  Thus the sum vanishes.
\item[14.] Here
\begin{equation}
 [ x_0, g \cdot \xi_1 ]+ [ x_1, g \cdot \xi_0] = b_{11}a_{11}^2 \epsilon+b_{12}bd + b_{22}(a_{11}c + a_{13}d+b^2  ) .
\end{equation}
Sum in $a_{13}$ to force $d = 0$.  Hence, summing in $b_{11}, c$ and then $b_{22}$,
\begin{equation}
 \sS(x, \xi) = \sum_{\begin{pmatrix} * & * & *\\ &a & b\\ &c & \end{pmatrix}\begin{pmatrix}* & *\\ & *\end{pmatrix}} e_p(b_{11} a_{11}^2 \epsilon + b_{22}(a_{11}c+b^2)) = -(p-1)^2p^4.
\end{equation}
\item[15.] Here
\begin{equation}
 [ x_0, g \cdot \xi_1 ]+ [ x_1, g \cdot \xi_0] =b_{11}(a_{11}^2 \epsilon+ad+ bc ) + b_{12}bd + b_{22}(a_{11}c+ a_{13}d + b^2 ) .
\end{equation}
Sum in $a_{13}$ to force $d = 0$.  Thus
\begin{align*}
 \sS(x,\xi) &= \sum_{\begin{pmatrix} * & * & *\\ & a & b\\ &c &\end{pmatrix}\begin{pmatrix} b_{11} & * \\ & * \end{pmatrix}} e_p(b_{11} (a_{11}^2 \epsilon+bc) + b_{22} (a_{11}c+b^2)).
 \end{align*}
 Sum in $b_{22}$; when $b^2\neq -a_{11}c $ the sum is $-1$, otherwise $p-1$.  Hence,
 \begin{align*}
 \sS(x,\xi)&= -\frac{1}{p-1} \sum_{\begin{pmatrix} * & * & *\\ & a & b\\ &c &\end{pmatrix}\begin{pmatrix} b_{11} & * \\ & * \end{pmatrix}} e_p(b_{11} a_{11}^2 \epsilon +  b_{11}bc )\\ 
 & \qquad + \frac{p}{p-1} \sum_{\begin{pmatrix} * & * & *\\ & a & b\\ &c &\end{pmatrix}\begin{pmatrix} b_{11} & * \\ & * \end{pmatrix}, b^2 = -a_{11}c}e_p(b_{11} a_{11}^2 \epsilon +  b_{11}bc )
 \end{align*}
 In the first sum, sum first in $b$ then in $b_{11}$ obtaining $-1$ in both sums.  In the second sum, replace $c = -\frac{b^2}{a_{11}}$, then replace $b_{11}:= \frac{b_{11}}{a_{11}}$ and sum in $b_{11}$ to obtain
 \begin{align*}
 \sS(x, \xi)&= -(p-1)^2p^4 + \frac{p}{p-1} \sum_{\begin{pmatrix} * & * & *\\ & a & b\\ &c &\end{pmatrix}\begin{pmatrix} b_{11} & * \\ & * \end{pmatrix}, b^2 =- a_{11}c}e_p(b_{11}(\epsilon a_{11}^3 - b^3))\\
 &= -(p-1)^2p^4 - (p-1)^2p^5 + \#\{(a_{11}, b) \in (\bF_p^\times)^2: \epsilon a_{11}^3 \equiv  b^3 \} p^6\\
 &= -(p-1)^2p^4(p+1) + (p-1)p^6 \#\{x \in \bF_p^\times: x^3 \equiv \epsilon\}.
\end{align*}
\item[16.] Here
\begin{equation}
 [ x_0, g \cdot \xi_1 ] + [ x_1, g \cdot \xi_0]= b_{11}(a_{11}^2\epsilon+ac) + b_{12}bd + b_{22}(a_{11}c + a_{13}d+b^2 ).
\end{equation}
Sum in $a_{13}$ to force $d = 0$.  Sum in $a$ to force $c = 0$.  Thus the sum vanishes.
\item[17.-23.] 
Summing in $a_{12}, a_{13}$ forces $c = d= 0$ so these sums vanish.
\end{enumerate}

\subsection{The exponential sums pair $(\sO_{1^31}, \sO_{Cs})$}\label{O131_OCs_section}
Standard representatives are
\begin{equation}
 x_0 = \begin{pmatrix} 0&0&0\\&0&\frac{1}{2}\\ &&0\end{pmatrix}\begin{pmatrix}0&0&\frac{1}{2}\\ &1&0\\&&0\end{pmatrix}, \qquad \xi_0 = \begin{pmatrix}0 & -1&0\\&0&0\\&&0\end{pmatrix}\begin{pmatrix}1&0&0\\&0&0\\&&0\end{pmatrix}
\end{equation}
and
\begin{equation}
 x_1 = \begin{pmatrix} \epsilon &0&0\\&0&0\\&&0\end{pmatrix}\begin{pmatrix}0&0&0\\&0&0\\&&0\end{pmatrix}, \epsilon \in \bF_p^\times/\{x^3: x \in \bF_p^\times\}.
\end{equation}
The acting set is
\begin{equation}
 G_{x,\xi}^t = \begin{pmatrix} a & a_{12} & a_{13}\\ &b & a_{23}\\ & & a_{33} \end{pmatrix}\begin{pmatrix} c  &b_{12}\\ & \frac{bc}{a}\end{pmatrix}, 
\end{equation}
and
\begin{equation}
  [ x_1, g \cdot \xi_0] = (-2ac a_{12} + b_{12}a^2)\epsilon.
\end{equation}

The stabilizer has size
\begin{equation}
\left|\Stab_{G(\zed/p^2\zed)}(x)\right| = (p-1)p^3\#\{\epsilon^3=1\}. 
\end{equation}

Representatives and exponential sum pairings are given in the table below. Summing over the maximal orbits,
\begin{equation}
 \sM(x,\xi) = \sum_{\epsilon \in \bF_p^\times/\{x^3: x \in \bF_p^\times\}} \frac{p^{10}\sS(x,\xi)}{(p-1)\#\{\epsilon^3 =1 \}}.
\end{equation}

See \textbf{exponential\_sums\_O131\_OCs.nb}.

 {\tiny 
 \begin{equation*}
  \begin{array}{|l|l|l|l|l|}
   \hline
   \text{Orbit} & \xi_1 & [ x_0, g \cdot \xi_1 ]+ [ x_1, g \cdot \xi_0 ] & \sS(x,\xi)& \sM(x,\xi)\\
   \hline
   1. & \begin{pmatrix} 0 &0&0\\&0&0\\&&0\end{pmatrix}\begin{pmatrix} 0&0&0\\&0&0\\&&1\end{pmatrix} & \begin{array}{l}b_{12}a_{23}a_{33} + \frac{bc}{a}(a_{23}^2 + a_{13}a_{33}) \\+(-2aca_{12} + b_{12}a^2 )\epsilon \end{array} & 0&0 \\
   \hline
   2. & \begin{pmatrix} 0 &0&0\\&0&0\\&&0\end{pmatrix}\begin{pmatrix} 0&0&0\\&1&0\\&&0\end{pmatrix} & \frac{b^3c}{a} +(-2aca_{12} + b_{12}a^2 )\epsilon &0 &0\\
   \hline
   3. & \begin{pmatrix} 0 &0&0\\&0&0\\&&0\end{pmatrix}\begin{pmatrix} 0&0&1\\&1&0\\&&0\end{pmatrix} & bc\left(\frac{b^2}{a} + a_{33} \right) +(-2aca_{12} + b_{12}a^2 )\epsilon &0 &0\\
   \hline
   4. & \begin{pmatrix} 0 &0&0\\&0&0\\&&1\end{pmatrix}\begin{pmatrix} 0&0&0\\&1&0\\&&0\end{pmatrix} &c\left(\frac{b^3}{a} + a_{23}a_{33} \right) +(-2aca_{12} + b_{12}a^2 )\epsilon &0&0 \\
   \hline
   5. & \begin{pmatrix} 0 &0&0\\&0&0\\&&0\end{pmatrix}\begin{pmatrix} 0&0&0\\&0&\frac{1}{2}\\&&1\end{pmatrix} & \begin{array}{l}\frac{bc}{a}\left(ba_{23} + \frac{a_{12}a_{33}}{2} + a_{13}a_{33} + a_{23}^2 \right)\\ + b_{12}\left(a_{23}a_{33} +\frac{a_{33}b}{2}\right)\\ +(-2aca_{12} + b_{12}a^2 )\epsilon \end{array} &0&0\\
   \hline
   6. & \begin{pmatrix} 0 &0&0\\&0&0\\&&0\end{pmatrix}\begin{pmatrix} 0&0&0\\&0&1\\&&0\end{pmatrix} &\begin{array}{l}\frac{bc}{a}(2ba_{23} + a_{12}a_{33}) + bb_{12}a_{33}\\+(-2aca_{12} + b_{12}a^2 )\epsilon\end{array} & 0&0 \\
   \hline
   7. & \begin{pmatrix} 0 &0&0\\&0&0\\&&1\end{pmatrix}\begin{pmatrix} 0&0&0\\&0&1\\&&0\end{pmatrix} & \begin{array}{l}\frac{bc}{a}(2ba_{23} + a_{12}a_{33}) + bb_{12}a_{33} + ca_{23}a_{33}\\+(-2aca_{12} + b_{12}a^2 )\epsilon\end{array} & 0 &0\\
   \hline
   8. & \begin{pmatrix} 0 &0&0\\&0&0\\&&0\end{pmatrix}\begin{pmatrix} 0&0&0\\&-\ell&0\\&&1\end{pmatrix} & \begin{array}{l}\frac{bc}{a}(-b^2\ell + a_{13}a_{33} + a_{23}^2) + b_{12}a_{23}a_{33}\\  +(-2aca_{12} + b_{12}a^2 )\epsilon\end{array} & 0 &0\\
   \hline
   9. & \begin{pmatrix} 0 &0&0\\&0&0\\&&0\end{pmatrix}\begin{pmatrix} 0&0&1\\&0&0\\&&0\end{pmatrix} & bca_{33} +(-2aca_{12} + b_{12}a^2 )\epsilon &0&0\\
   \hline
   10. & \begin{pmatrix} 0 &0&0\\&0&0\\&&1\end{pmatrix}\begin{pmatrix} 0&0&0\\&0&0\\&&0\end{pmatrix} & c a_{23}a_{33} +(-2aca_{12} + b_{12}a^2 )\epsilon &0&0\\
   \hline
   11. & \begin{pmatrix} 0 &0&0\\&0&0\\&&0\end{pmatrix}\begin{pmatrix} 0&0&0\\&0&0\\&&0\end{pmatrix} &  (-2aca_{12} + b_{12}a^2 )\epsilon &0&0\\
   \hline
   \end{array}
 \end{equation*}
}

Each of these sums vanishes on summing in the indicated variables.
\begin{enumerate}
 \item [1.-4.] $a_{12}$
\item [5.]  $a_{13}$
\item [6.]  $a_{23}$
\item[7.] $a_{12}, b_{12}$.  The dependence on these variables is $a_{12}\left(\frac{bc}{a} a_{33} - 2ac\epsilon\right)$ and $b_{12}\left(ba_{33} + a^2 \epsilon \right)$.  The corresponding linear forms are independent since $p>3$ is assumed.
\item[8.]  $a_{13}$
\item[9.-11.] $a_{12}$
\end{enumerate}

\subsection{The exponential sums pair $(\sO_{1^4}, \sO_{D1^2})$}\label{O14_OD12_section}
Standard representatives
\begin{equation}
 x_0 = \begin{pmatrix} 0&0&0\\&0&0\\&&1\end{pmatrix}\begin{pmatrix}0&0&\frac{1}{2}\\ &1&0\\&&0\end{pmatrix}, \qquad \xi_0 = \begin{pmatrix}1 &0&0\\&0&0\\&&0\end{pmatrix}\begin{pmatrix}0&0&0\\&0&0\\&&0\end{pmatrix},
\end{equation}
and
\begin{equation}
 x_1 = \begin{pmatrix} \epsilon &0&0\\&0&0\\&&0\end{pmatrix}\begin{pmatrix}0&0&0\\&0&0\\&&0\end{pmatrix}, \epsilon \in \bF_p^\times/\{x^4: x \in \bF_p^\times\}.
\end{equation}
The acting set is
\begin{equation}
 G_{x,\xi}^t = \begin{pmatrix} a_{11} & a_{12} & a_{13}\\ & a& b\\ & c & d\end{pmatrix}\begin{pmatrix} b_{11} & b_{12}\\ & b_{22}\end{pmatrix},
\end{equation}
which is fibered as for the other exponential sums involving $\sO_{D1^2}$, 
and
\begin{equation}
 [ x_1, g\cdot \xi_0 ] = b_{11} a_{11}^2 \epsilon.
\end{equation}

The stabilizer has size
\begin{equation}
 \left|\Stab_{G(\zed/p^2\zed)}(x)\right| = (p-1)p^4 \#\{\epsilon^4 = 1\}.
\end{equation}

Representatives and exponential sum pairings are given in the table below. Summing over the maximal orbits,
\begin{equation}
 \sM(x,\xi) = \sum_{\epsilon \in \bF_p^\times/\{x^4: x \in \bF_p^\times\}} \frac{p^9 \sS(x,\xi)}{(p-1)\#\{\epsilon^4 = 1\}}.
\end{equation}

See \textbf{exponential\_sums\_O14\_OD12.nb}.

{\tiny
 \begin{equation*}
  \begin{array}{|l|l|l|l|l|}
   \hline
   \text{Orbit} & \xi_1 & [ x_0, g \cdot \xi_1 ]+ [ x_1, g \cdot \xi_0 ]& \sS(x,\xi) & \sM(x,\xi)\\
   \hline
   1. & \begin{pmatrix} 0 &0&0\\&0&0\\&&0\end{pmatrix}\begin{pmatrix} 0&0&0\\&0&0\\&&0\end{pmatrix} & b_{11}a_{11}^2\epsilon & - (p-1)^4p^4(p+1) & -(p-1)^3p^{13}(p+1)\\
   \hline
   2. & \begin{pmatrix} 0 & 0&0\\ &0&0\\ &&1\end{pmatrix}\begin{pmatrix} 0&0&0\\&0&0\\&&0\end{pmatrix} &  b_{11}( a_{11}^2\epsilon+d^2) &\begin{array}{l} (p-1)^3p^4 \\ + \left(\frac{-\epsilon}{p}\right)(p-1)^3p^6\end{array} & (p-1)^2p^{13}\\
   \hline
   3. & \begin{pmatrix} 0 &0&0\\&0&\frac{1}{2}\\ &&0\end{pmatrix}\begin{pmatrix}0&0&0\\&0&0\\&&0\end{pmatrix} & b_{11}( a_{11}^2\epsilon+cd) & - (p-1)^4p^4 & -(p-1)^3 p^{13}\\
   \hline
   4. & \begin{pmatrix} 0 & 0 & 0\\ & -\ell &0\\ &&1\end{pmatrix}\begin{pmatrix} 0 &0&0\\ &0&0\\&&0\end{pmatrix} & b_{11}( a_{11}^2\epsilon-c^2\ell + d^2)& (p-1)^3p^4(p+1) & (p-1)^2 p^{13}(p+1)\\
   \hline
   5. & \begin{pmatrix} 0 &0&0\\ &0&0\\&&1\end{pmatrix}\begin{pmatrix} 0&0&1\\ &0&0\\&&0\end{pmatrix} & b_{11}(a_{11}^2\epsilon+d^2)+b_{22}a_{11}d  & \begin{array}{l}-(p-1)^2p^4(p^2-p+1)\\ -(p-1)^2p^6\left(\frac{-\epsilon}{p}\right) \end{array} &-(p-1)p^{13}(p^2-p+1)\\
   \hline
   6. & \begin{pmatrix} 0 &0&0\\ &1 &0\\&&0\end{pmatrix}\begin{pmatrix}0 &0&1\\ &0&0\\&&0\end{pmatrix} & b_{11}( a_{11}^2\epsilon+ c^2)+b_{22}a_{11}d  & (p-1)^3p^4 & (p-1)^2p^{13}\\
   \hline
   7. & \begin{pmatrix} 0 &0 &0\\ &0 & \frac{1}{2}\\ &&0\end{pmatrix}\begin{pmatrix}0 &0&1\\ &0&0\\&&0\end{pmatrix} & b_{11}(a_{11}^2\epsilon+cd)+b_{22}a_{11}d &-(p-1)^4p^4 & - (p-1)^3 p^{13}\\
   \hline
   8. & \begin{pmatrix} 0 &0&0\\ &1&1\\&&0\end{pmatrix} \begin{pmatrix} 0&0&1\\&0&0\\&&0\end{pmatrix} &  b_{11}(a_{11}^2\epsilon+c^2 + 2cd)+ b_{22}a_{11}d  & \begin{array}{l} -(p-1)^4p^4\\ \left(1 + \left(\frac{-\epsilon}{p} \right) \right)(p-1)^2p^6 \end{array} & (p-1)p^{13}(2p-1)\\
   \hline
   9. & \begin{pmatrix} 0 &0&0\\ &-\ell &0\\ &&1\end{pmatrix}\begin{pmatrix} 0&0&1\\ &0&0\\&&0\end{pmatrix} &  b_{11}( a_{11}^2\epsilon-c^2\ell+d^2)+b_{22}a_{11}d & \begin{array}{l} -(p-1)^2p^4\\
   -(p-1)^2p^6\left(\frac{\epsilon}{p}\right)\end{array}&-(p-1)p^{13}\\
   \hline
   10. & \begin{pmatrix} 0 &0&0\\ &0&0\\&&0\end{pmatrix}\begin{pmatrix}0 &0&1 \\ & 0&0\\ &&0\end{pmatrix} & b_{11}a_{11}^2\epsilon+b_{22}a_{11}d  & (p-1)^3p^4&(p-1)^2p^{13}\\
   \hline
   \end{array}
 \end{equation*}

 \begin{equation*}
   \begin{array}{|l|l|l|l|l|}
   \hline
   \text{Orbit} & \xi_1 & [ x_0, g \cdot \xi_1 ] + [ x_1, g \cdot \xi_0 ] & \sS(x,\xi)& \sM(x,\xi)\\
   \hline   11. & \begin{pmatrix} 0&0&0\\&0&0\\&&0\end{pmatrix}\begin{pmatrix}0&0&0\\&0&0\\&&1\end{pmatrix} &\begin{array}{l} b_{11}a_{11}^2\epsilon   + b_{12}d^2\\ +b_{22}( a_{13}d+b^2)\end{array}& (p-1)^3p^4 & (p-1)^2p^{13}\\
   \hline
   12. & \begin{pmatrix} 0&0&0\\&0&1\\&&0\end{pmatrix}\begin{pmatrix}0&0&0\\&0&0\\&&1\end{pmatrix} & \begin{array}{l}b_{11}( a_{11}^2\epsilon+2cd) +b_{12}d^2\\ + b_{22}(a_{13}d+b^2 )\end{array} & (p-1)^3p^4 & (p-1)^2p^{13}\\
   \hline
   13. & \begin{pmatrix} 0&0&0\\&1&0\\&&0\end{pmatrix}\begin{pmatrix}0&0&0\\&0&0\\&&1\end{pmatrix}& \begin{array}{l}b_{11}( a_{11}^2\epsilon+c^2) + b_{12}d^2 \\+b_{22}(a_{13}d+b^2 ) \end{array}& \begin{array}{l} - (p-1)^2p^4\\ - (p-1)^2p^5 \left(\frac{-\epsilon}{p}\right)\end{array}& -(p-1)p^{13}\\
   \hline
   14. & \begin{pmatrix} 0&0&0\\&0&0\\&&0\end{pmatrix}\begin{pmatrix}0&1&0\\&0&0\\&&1\end{pmatrix}& \begin{array}{l}b_{11}a_{11}^2\epsilon + b_{12}d^2 \\+b_{22}( a_{11}c + a_{13}d+b^2)  \end{array} & -(p-1)^2p^4 & -(p-1)p^{13}\\
   \hline
   15. & \begin{pmatrix} 0&0&0\\&0&1\\&&0\end{pmatrix}\begin{pmatrix}0&1&0\\&0&0\\&&1\end{pmatrix}& \begin{array}{l}b_{11}( a_{11}^2\epsilon + 2cd) + b_{12}d^2 \\+b_{22}( a_{11}c + a_{13}d+ b^2) \end{array} & -(p-1)^2p^4 & -(p-1)p^{13}\\
   \hline
   16. &\begin{pmatrix} 0&0&0\\&1&0\\&&0\end{pmatrix}\begin{pmatrix}0&1&0\\&0&0\\&&1\end{pmatrix} & \begin{array}{l}b_{11}( a_{11}^2\epsilon+c^2) + b_{12}d^2\\ +b_{22}( a_{11}c + a_{13}d+b^2) \end{array} & \begin{array}{l} -(p-1)^2p^4(p+1)\\ - (p-1)^2p^5\left(\frac{-\epsilon}{p}\right)\\ + (p-1)p^6 \\ \times \#\{x:x^4 \equiv -\epsilon \bmod p\}                                                                                                                                                                             \end{array} & p^{13}
\\
   \hline
   17. & \begin{pmatrix} 0&0&0\\&0&0\\&&0\end{pmatrix}\begin{pmatrix}0&0&0\\&0&1\\&&0\end{pmatrix}& \begin{array}{l}  b_{11}a_{11}^2\epsilon+2b_{12}cd\\ +b_{22}(a_{12}d + a_{13}c +2ab)\end{array} &0&0\\
   \hline
   18. &\begin{pmatrix} 0&0&0\\&0&0\\&&1\end{pmatrix}\begin{pmatrix}0&0&0\\&0&1\\&&0\end{pmatrix} & \begin{array}{l}b_{11}( a_{11}^2\epsilon + d^2)+ 2b_{12}cd \\+b_{22}( a_{12}d + a_{13}c+2ab )\end{array} &0&0\\
   \hline
   19. & \begin{pmatrix} 0&0&0\\&1&0\\&&1\end{pmatrix}\begin{pmatrix}0&0&0\\&0&1\\&&0\end{pmatrix}& \begin{array}{l}b_{11}( a_{11}^2\epsilon+c^2 + d^2)+2b_{12}cd \\ +b_{22}( a_{12}d+a_{13}c +2ab )\end{array}&0&0\\
   \hline
   20. &\begin{pmatrix} 0&0&0\\&\ell&0\\&&1\end{pmatrix}\begin{pmatrix}0&0&0\\&0&1\\&&0\end{pmatrix} & \begin{array}{l} b_{11}( a_{11}^2\epsilon+c^2\ell + d^2)+ 2b_{12}cd\\+b_{22}(a_{12}d+ a_{13}c + 2ab)  \end{array} &0&0\\
   \hline
   21. & \begin{pmatrix} 0 &0&0\\&0&0\\&&0\end{pmatrix}\begin{pmatrix}0&0&0\\ &-\ell &0\\&&1\end{pmatrix} & \begin{array}{l}b_{11}a_{11}^2\epsilon+b_{12}(-c^2\ell + d^2)  \\ +b_{22}( - a_{12}c\ell+ a_{13}d - a^2\ell+b^2 )\end{array}&0&0\\
   \hline
   22. & \begin{pmatrix} 0 &0&0\\&1&0\\&&0\end{pmatrix}\begin{pmatrix}0&0&0\\ &-\ell &0\\&&1\end{pmatrix} & \begin{array}{l} b_{11}( a_{11}^2\epsilon+c^2) + b_{12}(-c^2\ell + d^2)\\+b_{22}(- a_{12}c\ell + a_{13}d - a^2\ell+b^2 ) \end{array}&0&0\\
   \hline
   23a. &\begin{pmatrix} 0 &0&0\\&0&1\\&&0\end{pmatrix}\begin{pmatrix}0&0&0\\ &-\ell &0\\&&1\end{pmatrix} & \begin{array}{l}b_{11}( a_{11}^2\epsilon + 2cd)\\ + b_{12}(-c^2\ell + d^2) \\+b_{22}(- a_{12}c\ell + a_{13}d  -a^2\ell+b^2 ) \end{array}&0&0\\
   \hline
   23b. &\begin{pmatrix} 0 &0&0\\&1&k\\&&0\end{pmatrix}\begin{pmatrix}0&0&0\\ &-\ell &0\\&&1\end{pmatrix} &\begin{array}{l} b_{11}( a_{11}^2\epsilon+c(c+2dk))\\ + b_{12}(-c^2\ell + d^2)\\+ b_{22}(- a_{12}c\ell+a_{13}d-a^2\ell+b^2 )\end{array}&0&0\\
   \hline
  \end{array}  
 \end{equation*}
}

The exponential sums are as follows.
\begin{enumerate}
 \item [1.] Here
 \begin{equation}
   [ x_0, g \cdot \xi_1 ]+ [ x_1, g \cdot \xi_0 ] = b_{11} a_{11}^2 \epsilon.
 \end{equation}
Thus, summing in $b_{11}$,
 \begin{equation}
  \sS(x, \xi) = \sum_{g \in G_{x, \xi}^t} e_p(b_{11} a_{11}^2 \epsilon) = -\frac{|G_{x,\xi}|}{p-1} = -(p-1)^4p^4(p+1).
 \end{equation}
 \item [2.] Here
 \begin{equation}
   [ x_0, g \cdot \xi_1 ]+ [ x_1, g \cdot \xi_0 ] = b_{11} (a_{11}^2 \epsilon+d^2).
 \end{equation}
Thus
\begin{align*}
  \Sigma_1 &= \sum_{\begin{pmatrix} * & * & *\\ & a & b\\ & c & d\end{pmatrix} \begin{pmatrix} b_{11} & *\\ & *\end{pmatrix}} e_p(b_{11}(a_{11}^2 \epsilon + d^2))\\
  &= (p-1)p^6 \sum_{b_{11} \in \bF_p^\times}\left(\tau \left(\frac{b_{11}\epsilon}{p}\right)-1\right)\tau \left(\frac{b_{11}}{p}\right)\\& = (p-1)^2p^7 \left(\frac{-\epsilon}{p}\right)\\
  \Sigma_2 &= \sum_{\begin{pmatrix} * & * & * \\ & b \lambda  & b\\ & d \lambda  & d\end{pmatrix}\begin{pmatrix} b_{11} & *\\ & *\end{pmatrix}} e_p(b_{11}(a_{11}^2 \epsilon + d^2))\\
  &= (p-1)^2p^4 \sum_{b_{11} \in \bF_p^\times} \left(\tau \left(\frac{b_{11}\epsilon}{p}\right)-1\right)\tau\left(\frac{b_{11}}{p}\right)= (p-1)^3p^5\left(\frac{-\epsilon}{p}\right)\\
  \Sigma_3 &= \sum_{\begin{pmatrix} * & * & *\\ & a & 0\\ & c & 0\end{pmatrix} \begin{pmatrix} b_{11} & * \\ & * \end{pmatrix}} e_p(b_{11}a_{11}^2 \epsilon)\\
  &= - (p-1)^2p^5\\
  \Sigma_4 &= \sum_{\begin{pmatrix} * & * & * \\ & 0 & b\\ & 0 & d\end{pmatrix} \begin{pmatrix} b_{11} & * \\ & *\end{pmatrix}} e_p(b_{11}(a_{11}^2 \epsilon + d^2))\\
  &= (p-1)p^4 \sum_{b_{11} \in \bF_p^\times}\left(\tau \left(\frac{b_{11}\epsilon}{p}\right) - 1\right)\tau\left(\frac{b_{11}}{p}\right) = (p-1)^2 p^5\left(\frac{-\epsilon}{p}\right)\\
  \Sigma_5 &= p \sum_{\begin{pmatrix} * & * & *\\ & 0 &0\\ &0&0\end{pmatrix}\begin{pmatrix} b_{11} & *\\ & * \end{pmatrix}} e_p(b_{11} a_{11}^2 \epsilon) = -(p-1)^2p^4.
\end{align*}
Thus
\begin{equation}
 \sS(x, \xi) = (p-1)^3 p^4 + (p-1)^3 p^6\left(\frac{-\epsilon}{p}\right).
\end{equation}

\item [3.] Here
\begin{equation}
 [ x_0, g\cdot \xi_1 ] + [ x_1, g\cdot \xi_0] = b_{11}(a_{11}^2 \epsilon + cd).
\end{equation}
\begin{align*}
 \Sigma_1 &= \sum_{\begin{pmatrix} * & * & *\\ & a & b\\ & c & d\end{pmatrix} \begin{pmatrix} b_{11} & *\\ & *\end{pmatrix}} e_p(b_{11}(a_{11}^2 \epsilon + cd))\\
 &=-(p-1)^2 p^6\\
 \Sigma_2 &= \sum_{\begin{pmatrix} * & * & * \\ & b \lambda  & b\\ & d \lambda  & d\end{pmatrix}\begin{pmatrix} b_{11} & *\\ & *\end{pmatrix}} e_p(b_{11}(a_{11}^2 \epsilon +  d^2\lambda)) = 0\\
 \Sigma_3 &= \sum_{\begin{pmatrix} * & * & *\\ & a & 0\\ & c & 0\end{pmatrix} \begin{pmatrix} b_{11} & * \\ & * \end{pmatrix}} e_p(b_{11} a_{11}^2 \epsilon) = -(p-1)^2p^5\\
 \Sigma_4 &= \sum_{\begin{pmatrix} * & * & *\\ & 0 & b\\ & 0 & d\end{pmatrix} \begin{pmatrix} b_{11} & * \\ & * \end{pmatrix}} e_p(b_{11} a_{11}^2 \epsilon) = -(p-1)^2p^5\\
 \Sigma_5 &= p \sum_{\begin{pmatrix} * & * & *\\ & 0 &0\\ &0&0\end{pmatrix}\begin{pmatrix} b_{11} & *\\ & * \end{pmatrix}} e_p(b_{11} a_{11}^2 \epsilon) = -(p-1)^2p^4.
\end{align*}
Thus
\begin{equation}
 \sS(x, \xi) = -(p-1)^4p^4.
\end{equation}
\item[4.] Here 
\begin{equation}
 [ x_1, g \cdot \xi_0 ] + [ x_0, g \cdot \xi_1 ] = b_{11} (a_{11}^2 \epsilon -c^2 \ell + d^2).
\end{equation}
\begin{align*}
 \Sigma_1 &= \sum_{\begin{pmatrix} * & * & *\\ & a & b\\ & c & d\end{pmatrix} \begin{pmatrix} b_{11} & *\\ & *\end{pmatrix}}e_p(b_{11}(a_{11}^2 \epsilon -c^2 \ell + d^2))\\
 &= (p-1)p^5 \sum_{b_{11} \in \bF_p^\times} \left( \tau \left(\frac{ b_{11}\epsilon}{p}\right)-1\right)\tau^2 \left(\frac{-\ell}{p}\right)\\
 &= (p-1)^2p^6\\
 \Sigma_2 &= \sum_{\begin{pmatrix} * & * & * \\ & b \lambda  & b\\ & d \lambda  & d\end{pmatrix}\begin{pmatrix} b_{11} & *\\ & *\end{pmatrix}} e_p(b_{11}(a_{11}^2 \epsilon + d^2(1 - \lambda^2 \ell) ))
 \end{align*}
 Split off the $d = 0$ term.  In the $d \neq 0$ terms replace $\lambda:= d \lambda $.  This obtains
 \begin{align*}
\Sigma_2 &=(p-1)^2p^4 \sum_{a_{11}, b_{11} \in \bF_p^\times}e_p(b_{11}a_{11}^2 \epsilon) + (p-1)p^4 \sum_{\lambda, a_{11}, b_{11}, d \in \bF_p^\times} e_p(b_{11}(a_{11}^2 \epsilon + d^2 - \lambda^2 \ell))\\
 &= -(p-1)^3p^4 +(p-1)p^4\sum_{b_{11} \in \bF_p^\times} \left(\tau \left(\frac{ b_{11}\epsilon}{p}\right)-1\right)\left(\tau \left(\frac{b_{11}}{p}\right)-1\right)\left(\tau \left(\frac{- b_{11}\ell}{p}\right)-1\right)\\
 &= -(p-1)^3p^4 - (p-1)^2p^4 -(p-1)^2 p^5 \left\{ \left(\frac{-\epsilon}{p}\right) - \left(\frac{\epsilon}{p}\right) -1\right\}\\
 &= -(p-1)^2 p^5 \left\{ \left(\frac{-\epsilon}{p}\right) - \left(\frac{\epsilon}{p}\right) \right\}\\
 \Sigma_3 &= \sum_{\begin{pmatrix} * & * & *\\ & a & 0\\ & c & 0\end{pmatrix} \begin{pmatrix} b_{11} & * \\ & * \end{pmatrix}} e_p(b_{11}(a_{11}^2 \epsilon - c^2 \ell))\\
 &= (p-1)p^4 \sum_{b_{11} \in \bF_p^\times} \left(\tau \left(\frac{ b_{11}\epsilon}{p}\right) - 1\right) \tau \left(\frac{- b_{11}\ell}{p}\right) \\&= -(p-1)^2p^5\left(\frac{\epsilon}{p}\right)\\
 \Sigma_4 &= \sum_{\begin{pmatrix} * & * & *\\ & 0 & b\\ & 0 & d\end{pmatrix} \begin{pmatrix} b_{11} & * \\ & * \end{pmatrix}} e_p( b_{11}(a_{11}^2\epsilon +d^2)) \\
 &= (p-1)^2 p^5 \left(\frac{-\epsilon}{p}\right)\\
 \Sigma_5 &= p \sum_{\begin{pmatrix} * & * & *\\ & 0 &0\\ &0&0\end{pmatrix}\begin{pmatrix} b_{11} & *\\ & * \end{pmatrix}} e_p( b_{11}a_{11}^2\epsilon) = -(p-1)^2p^4.
\end{align*}
Thus
\begin{equation}
 \sS(x, \xi) = (p-1)^3p^4(p+1).
\end{equation}
\item[5.] Here  
\begin{equation}
 [ x_0, g \cdot \xi_1 ] + [ x_1, g \cdot \xi_0 ] =b_{11}(a_{11}^2\epsilon+d^2)+b_{22}a_{11}d.
\end{equation}
Summing in $b_{22}$ obtains $-1$ if $d \neq 0$, $p-1$ if $d = 0$, so that 
\begin{align*}
 \sS(x,\xi) &= -\frac{2.}{p-1} + \frac{p}{p-1}\sum_{\begin{pmatrix} * & * & *\\ & * & *\\ &* &\end{pmatrix} \begin{pmatrix} * & * \\ & *\end{pmatrix}} e_p(b_{11}a_{11}^2 \epsilon)\\
 &= -(p-1)^2 p^4 - (p-1)^2 p^6 \left(\frac{-\epsilon}{p}\right) -(p-1)^3p^5\\
 &= - (p-1)^2p^4 (p^2-p+1) - (p-1)^2p^6 \left(\frac{-\epsilon}{p}\right).
\end{align*}
\item [6.] Here
\begin{equation}
 [ x_0, g\cdot \xi_1 ] + [ x_1, g \cdot \xi_0 ] = b_{11}( a_{11}^2\epsilon+ c^2)+b_{22}a_{11}d.
\end{equation}
Summing in $b_{22}$ obtains $-1$ if $d \neq 0$ and $p-1$ if $d = 0$, so that,
\begin{align*}
 \sS(x, \xi) &= - \frac{2.}{p-1} + \frac{p}{p-1}\sum_{\begin{pmatrix} * & * & *\\ & a & b\\ & c &\end{pmatrix}\begin{pmatrix} * & *\\ & *\end{pmatrix}}e_p(b_{11}(a_{11}^2 \epsilon + c^2))\\
 &= -\frac{2.}{p-1} + (p-1)p^5 \sum_{b_{11} \in \bF_p^\times} \left( \tau \left(\frac{ b_{11}\epsilon}{p}\right)-1\right)\left(\tau\left(\frac{b_{11}}{p}\right)-1\right)\\
 &= - \frac{2.}{p-1} + (p-1)^2 p^5 + (p-1)^2 p^6 \left(\frac{-\epsilon}{p}\right)\\
 &=(p-1)^3p^4.
\end{align*}
\item [7.] Here
\begin{equation}
 [ x_0, g \cdot \xi_1] + [ x_1, g \cdot \xi_0 ] =b_{11}(a_{11}^2\epsilon+cd)+b_{22}a_{11}d.
\end{equation}
Thus, summing in $b_{22}$ as above,
\begin{align*}
 \sS(x,\xi) &= - \frac{3.}{p-1} + \frac{p}{p-1}\sum_{\begin{pmatrix} * & * & *\\ & * & *\\ & * & \end{pmatrix}\begin{pmatrix} * & *\\ & *\end{pmatrix}}e_p(b_{11} a_{11}^2 \epsilon)\\
 &= -\frac{3.}{p-1}- (p-1)^3p^5\\
 &= -(p-1)^4p^4.
\end{align*}
\item[8.] Here
\begin{equation}
 [ x_0, g \cdot \xi_1 ] + [ x_1, g \cdot \xi_0 ] = b_{11}(a_{11}^2\epsilon+c^2 + 2cd)+ b_{22}a_{11}d 
\end{equation}
Sum in $b_{22}$ to find
\begin{align*}
 \sS(x, \xi) &= -\frac{3.}{p-1} + \frac{p}{p-1}\sum_{\begin{pmatrix} * & * & *\\ & a & b\\ & c &\end{pmatrix}\begin{pmatrix} * & *\\ & *\end{pmatrix}} e_p(b_{11}(a_{11}^2 \epsilon + c^2))\\
 &= -\frac{3.}{p-1} + (p-1)p^5 \sum_{b_{11} \in \bF_p^\times} \left(\tau\left(\frac{b_{11}\epsilon}{p}\right) -1\right)\left(\tau \left(\frac{b_{11}}{p}\right)-1\right)\\
 &= (p-1)^3p^4 + (p-1)^2p^5 + (p-1)^2 p^6 \left( \frac{-\epsilon}{p}\right)\\
 &= -(p-1)^4 p^4 + \left(1 + \left(\frac{-\epsilon}{p}\right)\right)(p-1)^2p^6.
\end{align*}
\item[9.] Here
\begin{equation}
 [ x_0, g \cdot \xi_1 ] + [ x_1, g \cdot \xi_0 ] =b_{11}( a_{11}^2\epsilon-c^2\ell+d^2)+b_{22}a_{11}d.
\end{equation}

Sum in $b_{22}$ to find
\begin{align*}
 \sS(x, \xi) &= -\frac{4.}{p-1} + \frac{p}{p-1}\sum_{\begin{pmatrix} * & * & *\\ & a & b\\ & c &\end{pmatrix}\begin{pmatrix} * & *\\ & *\end{pmatrix}} e_p(b_{11}(a_{11}^2 \epsilon - c^2 \ell))\\
 &= -(p-1)^2 p^4 (p+1)\\ &\qquad + (p-1)p^5 \sum_{b_{11} \in \bF_p^\times}\left(\tau \left(\frac{ b_{11}\epsilon}{p}\right)-1\right)\left(\tau \left(\frac{- b_{11}\ell}{p}\right)-1\right)\\
 &= -(p-1)^2 p^4 (p+1) + (p-1)^2p^5 - (p-1)^2 p^6 \left(\frac{\epsilon}{p}\right)\\
 &= -(p-1)^2 p^4 - (p-1)^2 p^6 \left(\frac{\epsilon}{p}\right).
\end{align*}
\item[10.] Here
\begin{equation}
[ x_0, g \cdot \xi_1 ] + [ x_1, g \cdot \xi_0 ] = b_{11}a_{11}^2\epsilon+b_{22}a_{11}d.
\end{equation}

Sum in $b_{22}$ to find
\begin{align*}
 \sS(x, \xi) &= - \frac{1.}{p-1} + \frac{p}{p-1} \sum_{\begin{pmatrix} * & * & *\\ & * & *\\ & * &\end{pmatrix}\begin{pmatrix} * & *\\ & *\end{pmatrix}} e_p(b_{11} a_{11}^2 \epsilon)\\
 &= (p-1)^3 p^4(p+1)-(p-1)^3p^5 = (p-1)^3p^4.
\end{align*}

In sums 11-16 sum over $b_{12}$ to force $d = 0$.
\item[11.] Here  
\begin{equation}
 [ x_0, g \cdot \xi_1 ] + [ x_1, g \cdot \xi_0] =  b_{11}a_{11}^2\epsilon   + b_{12}d^2 +b_{22}( a_{13}d+b^2).
\end{equation}
Thus
\begin{equation}
 \sS(x,\xi) = \sum_{\begin{pmatrix} * & * & *\\ & a & b\\ & c & \end{pmatrix} \begin{pmatrix} * & *\\ & *\end{pmatrix}} e_p(b_{11} a_{11}^2 \epsilon + b_{22}b^2) = (p-1)^3p^4,
\end{equation}
since summation in $b_{11}$ and $b_{22}$ each give $-1$.
\item[12.] Here  
\begin{equation}
 [ x_0, g \cdot \xi_1 ] + [ x_1, g \cdot \xi_0] = b_{11}( a_{11}^2\epsilon+2cd) +b_{12}d^2 + b_{22}(a_{13}d+b^2 ).
\end{equation}
Thus
\begin{equation}
 \sS(x, \xi) = \sum_{\begin{pmatrix} * & * & *\\ & a & b\\ & c & \end{pmatrix} \begin{pmatrix} * & *\\ & *\end{pmatrix}} e_p(b_{11} a_{11}^2 \epsilon + b_{22}b^2) = (p-1)^3p^4,
\end{equation}
since summation in $b_{11}$ and $b_{22}$ each give $-1$.
\item[13.] Here  
\begin{equation}
 [ x_0, g \cdot \xi_1 ] + [ x_1, g \cdot \xi_0] = b_{11}( a_{11}^2\epsilon+c^2) + b_{12}d^2 +b_{22}(a_{13}d+b^2 ).
\end{equation}
Now sum in $b_{22}$, obtaining $-1$, and sum in $a_{11}$ and $c$ to obtain Gauss sums, to find 
\begin{align*}
 \sS(x, \xi) &= \sum_{\begin{pmatrix} * & * & *\\ & a & b\\ & c & \end{pmatrix} \begin{pmatrix} * & *\\ & *\end{pmatrix}} e_p(b_{11} (a_{11}^2 \epsilon + c^2) + b_{22}b^2)\\
 &= -(p-1)p^4 \sum_{b_{11} \in \bF_p^\times}\left(\tau\left(\frac{b_{11}\epsilon}{p}\right)-1\right)\left(\tau\left(\frac{b_{11}}{p}\right)-1\right)\\
 &= - (p-1)^2p^4 - (p-1)^2p^5\left(\frac{-\epsilon}{p}\right).
\end{align*}
\item[14.] Here  
\begin{equation}
 [ x_0, g \cdot \xi_1 ] + [ x_1, g \cdot \xi_0] =b_{11}a_{11}^2\epsilon + b_{12}d^2 +b_{22}( a_{11}c + a_{13}d+b^2).
\end{equation}
Now sum in $b_{11}$, obtaining $-1$, then sum in $a_{11}$, obtaining $-1$, and finally sum in $b_{22}$, obtaining $-1$, to obtain
\begin{align*}
 \sS(x, \xi) &= \sum_{\begin{pmatrix} * & * & *\\ & a & b\\ & c & \end{pmatrix} \begin{pmatrix} * & *\\ & *\end{pmatrix}} e_p(b_{11} a_{11}^2 \epsilon + b_{22}(b^2 + a_{11}c)) \\
 &= -(p-1)^2 p^4.
\end{align*}
\item[15.] Here  
\begin{equation}
 [ x_0, g \cdot \xi_1 ] + [ x_1, g \cdot \xi_0] = b_{11}( a_{11}^2\epsilon + 2cd) + b_{12}d^2 \\+b_{22}( a_{11}c + a_{13}d+ b^2).
\end{equation}
After setting $d = 0$ this is the same sum as in 14., so 
\begin{equation}
 \sS(x, \xi) = 14. = -(p-1)^2 p^4.
\end{equation}
\item[16.] Here  
\begin{equation}
 [ x_0, g \cdot \xi_1 ] + [ x_1, g \cdot \xi_0] = b_{22}(b^2 + a_{11}c + a_{13}d) + b_{12}d^2 + b_{11}(c^2 + a_{11}^2\epsilon).
\end{equation}
Thus
\begin{align*}
 \sS(x, \xi) &= \sum_{\begin{pmatrix} * & *& *\\ & a & b\\ & c &\end{pmatrix}\begin{pmatrix} * & *\\ & *\end{pmatrix}} e_p(b_{11}(a_{11}^2 \epsilon + c^2) + b_{22}(b^2 + a_{11}c))\\
 &= p^4 \sum_{b_{11}, b_{22}, a_{11}, b, c \in \bF_p^\times} e_p(b_{11} (a_{11}^2 \epsilon + c^2) + b_{22}(b^2 + a_{11}c)).
 \end{align*}
 Summing in $b_{22}$ obtains $-1$ if $b^2 + a_{11}c \neq 0$ and $p-1$ otherwise.  Splitting the sum according to this and substituting $b_{11} := \frac{b_{11}}{c^2}$ in the second sum obtains
 \begin{align*}
 \sS(x, \xi)&= - p^4 \sum_{b_{11},  a_{11}, b, c \in \bF_p^\times} e_p(b_{11} (a_{11}^2 \epsilon + c^2))\\
  & \qquad + p^5 \sum_{b_{11}, b, c \in \bF_p^\times} e_p(b_{11}(\epsilon b^4 + c^4))\\
 &= - (p-1)p^4 \sum_{b_{11} \in \bF_p^\times} \left(\tau \left(\frac{ b_{11}\epsilon}{p}\right)-1\right)\left( \tau \left(\frac{b_{11}}{p}\right)-1\right)\\
 & \qquad - (p-1)^2p^5 + (p-1)p^6 \#\{x: x^4 \equiv -\epsilon \bmod p\}\\
 &= -(p-1)^2 p^4 (p+1) - (p-1)^2 p^5 \left(\frac{-\epsilon}{p}\right) + (p-1)p^6 \#\{x: x^4 \equiv -\epsilon \bmod p\}.
\end{align*}
\item[17.-23.] These sums vanish since summing in $a_{12}$ and $a_{13}$ forces $c = d= 0$.
\end{enumerate}

\subsection{The exponential sums pair $(\sO_{1^4}, \sO_{D11})$}\label{O14_OD11_section}
Standard representatives are
\begin{equation}
 x_0 = \begin{pmatrix} 0&0&0\\ &0&0\\&&1\end{pmatrix}\begin{pmatrix}0&0&\frac{1}{2}\\&1&0\\&&0\end{pmatrix}, \qquad \xi_0=\begin{pmatrix}0&\frac{1}{2}&0\\&0&0\\&&0\end{pmatrix}\begin{pmatrix}0&0&0\\&0&0\\&&0\end{pmatrix},
\end{equation}
and
\begin{equation}
 x_1 = \begin{pmatrix} \epsilon &0&0\\&0&0\\&&0\end{pmatrix}\begin{pmatrix}0&0&0\\&0&0\\&&0\end{pmatrix}, \epsilon \in \bF_p^\times/\{x^4: x \in \bF_p^\times\}.
\end{equation}
The acting set is
\begin{equation}
 G_{x,\xi}^t = \begin{pmatrix} a_{11} & a_{12} & a_{13}\\ & a_{22} & a_{23}\\ && a_{33}\end{pmatrix}\begin{pmatrix} b_{11} & b_{12} \\ & b_{22}\end{pmatrix} \sqcup \begin{pmatrix} a_{11} & a_{12} & a_{13}\\ a_{21} &  & a_{23}\\&& a_{33}\end{pmatrix}\begin{pmatrix} b_{11} & b_{12}\\ & b_{22}\end{pmatrix}.
\end{equation}
Here
\begin{equation}
 [ x_1, g \cdot \xi_0 ] =  b_{11}a_{11}a_{12}\epsilon.
\end{equation}

The stabilizer has size
\begin{equation}
 \left|\Stab_{G(\zed/p^2\zed)}(x)\right| = (p-1)p^4 \#\{\epsilon^4 = 1\}.
\end{equation}

Representatives and exponential sum pairings are given in the table below. Summing over the maximal orbits,
\begin{equation}
 \sM(x,\xi) = \sum_{\epsilon \in \bF_p^\times/\{x^4: x \in \bF_p^\times\}} \frac{p^9 \sS(x,\xi)}{(p-1)\#\{\epsilon^4 = 1\}}.
\end{equation}

See \textbf{exponential\_sums\_O14\_OD11.nb}.
{\tiny
 \begin{equation*}
  \begin{array}{|l|l|l|l|l|}
   \hline
   \text{Orbit} & \xi_1 & [ x_0, g \cdot \xi_1 ] + [ x_1, g\cdot x_0] & \sS(x,\xi)& \sM(x,\xi)\\
   \hline
   1. & \begin{pmatrix} 0 &0&0\\&0&0\\&&0\end{pmatrix}\begin{pmatrix} 0&0&0\\&0&0\\&&1\end{pmatrix} &b_{11}a_{11}a_{12}\epsilon+ b_{12}a_{33}^2 + b_{22}( a_{13}a_{33}+a_{23}^2)  &0&0\\
   \hline
   2. & \begin{pmatrix} 0 &0&0\\&0&0\\&&0\end{pmatrix}\begin{pmatrix} 0&0&0\\&0&1\\&&1\end{pmatrix} & \begin{array}{l}b_{11}a_{11}a_{12}\epsilon+  b_{12}a_{33}^2\\ +b_{22}(a_{12}a_{33} + a_{13}a_{33} + 2a_{22}a_{23} + a_{23}^2) \end{array} &0&0\\
   \hline
   3. &  \begin{pmatrix} 0 &0&0\\&0&0\\&&0\end{pmatrix}\begin{pmatrix} 0&0&1\\&0&1\\&&1\end{pmatrix} & \begin{array}{l}  b_{11}a_{11}a_{12}\epsilon+b_{12}a_{33}^2  \\+b_{22}(a_{23}(2a_{21} + 2a_{22} + a_{23}) + a_{33}(a_{11} + a_{12} + a_{13}))\\ \end{array}&0&0\\
   \hline
   4. & \begin{pmatrix} 0 &0&0\\&0&0\\&&0\end{pmatrix}\begin{pmatrix} 0&0&0\\&-\ell&0\\&&1\end{pmatrix} & b_{11}a_{11}a_{12}\epsilon+ b_{12}a_{33}^2 + b_{22}(a_{13}a_{33}- a_{22}^2\ell + a_{23}^2 )  &0&0\\
   \hline
   5. & \begin{pmatrix} 0 &0&0\\&0&0\\&&0\end{pmatrix}\begin{pmatrix} 0&0&1\\&-\ell&0\\&&1\end{pmatrix} & \begin{array}{l} b_{11}a_{11}a_{12}\epsilon+ b_{12}a_{33}^2 \\ +b_{22}(a_{11}a_{33}+a_{13}a_{33} + 2a_{21}a_{23}- a_{22}^2\ell + a_{23}^2 ) \end{array}&0&0\\
   \hline
   6. & \begin{pmatrix} 0 &0&0\\&0&0\\&&0\end{pmatrix}\begin{pmatrix} -\ell&0&0\\&-\ell&0\\&&1\end{pmatrix} & \begin{array}{l}b_{11}a_{11}a_{12}\epsilon+ b_{12}a_{33}^2 \\+ b_{22}(a_{13}a_{33}  -(a_{21}^2 + a_{22}^2)\ell+ a_{23}^2) \end{array}&0&0 \\
   \hline
   7. & \begin{pmatrix} 0 &0&0\\&0&0\\&&0\end{pmatrix}\begin{pmatrix} 0&0&0\\&0&1\\&&0\end{pmatrix} & b_{11}a_{11}a_{12}\epsilon+ b_{22}( a_{12}a_{33}+2a_{22}a_{23})  &0&0\\
   \hline
   8. & \begin{pmatrix} 0 &0&0\\&0&0\\&&0\end{pmatrix}\begin{pmatrix} 0&0&1\\&0&1\\&&0\end{pmatrix}& \begin{array}{l}b_{11}a_{11}a_{12}\epsilon\\ +b_{22}(2a_{23}(a_{21} + a_{22}) + a_{33}(a_{11} + a_{12})) \end{array} & 0&0 \\
   \hline
   9. & \begin{pmatrix} 0 &0&0\\&0&0\\&&0\end{pmatrix}\begin{pmatrix} 1&0&0\\&0&1\\&&0\end{pmatrix}& b_{11}a_{11}a_{12}\epsilon+b_{22}(a_{12}a_{33}+a_{21}^2 + 2a_{22}a_{23} )  &0 &0\\
   \hline
   10. & \begin{pmatrix} 0 &0&0\\&0&0\\&&1\end{pmatrix}\begin{pmatrix} 0&0&0\\&0&1\\&&0\end{pmatrix}& b_{11}( a_{11}a_{12}\epsilon+a_{33}^2)+b_{22}(a_{12}a_{33}+2a_{22}a_{23} )   &0&0\\
   \hline
   11. & \begin{pmatrix} 0 &0&0\\&0&0\\&&1\end{pmatrix}\begin{pmatrix} 0&0&1\\&0&1\\&&0\end{pmatrix}& \begin{array}{l}b_{11}( a_{11}a_{12}\epsilon+a_{33}^2)\\+b_{22}(2a_{23}(a_{21} + a_{22}) + a_{33}(a_{11} + a_{12})\end{array} &0&0\\
   \hline
   12. & \begin{pmatrix} 0 &0&0\\&0&0\\&&1\end{pmatrix}\begin{pmatrix} 1&0&0\\&0&1\\&&0\end{pmatrix}& b_{11}( a_{11}a_{12}\epsilon+a_{33}^2)+b_{22}(a_{12}a_{33}+a_{21}^2 + 2a_{22}a_{23}  ) &0&0\\
   \hline
   13. & \begin{pmatrix} 0 &0&0\\&0&0\\&&0\end{pmatrix}\begin{pmatrix} 0&0&0\\&1&0\\&&0\end{pmatrix}&b_{11}a_{11}a_{12}\epsilon+ b_{22}a_{22}^2 &0&0\\
   \hline
   14. & \begin{pmatrix} 0 &0&0\\&0&0\\&&0\end{pmatrix}\begin{pmatrix} 1&0&0\\&1&0\\&&0\end{pmatrix}&b_{11}a_{11}a_{12}\epsilon+ b_{22}(a_{21}^2 + a_{22}^2) &0&0\\
   \hline
   15. & \begin{pmatrix} 0 &0&0\\&0&0\\&&0\end{pmatrix}\begin{pmatrix} \ell&0&0\\&1&0\\&&0\end{pmatrix} &b_{11}a_{11}a_{12}\epsilon+ b_{22}( a_{21}^2\ell+a_{22}^2) &0&0\\
   \hline
   16. & \begin{pmatrix} 0 &0&0\\&0&0\\&&1\end{pmatrix}\begin{pmatrix} 0&0&0\\&1&0\\&&0\end{pmatrix} & b_{11}( a_{11}a_{12}\epsilon+a_{33}^2)+b_{22}a_{22}^2 &0&0\\
   \hline
   17. & \begin{pmatrix} 0 &0&0\\&0&0\\&&1\end{pmatrix}\begin{pmatrix} 1&0&0\\&1&0\\&&0\end{pmatrix} & b_{11}( a_{11}a_{12}\epsilon+a_{33}^2)+b_{22}(a_{21}^2 + a_{22}^2) &0&0\\
   \hline
   18. & \begin{pmatrix} 0 &0&0\\&0&0\\&&1\end{pmatrix}\begin{pmatrix} \ell&0&0\\&1&0\\&&0\end{pmatrix}& b_{11}( a_{11}a_{12}\epsilon+a_{33}^2)+b_{22}( a_{21}^2\ell+a_{22}^2 ) &0&0\\
   \hline
   19. & \begin{pmatrix} 0 &0&0\\&0&0\\&&1\end{pmatrix}\begin{pmatrix} 0&0&0\\&0&0\\&&0\end{pmatrix}& b_{11}(a_{11}a_{12}\epsilon+a_{33}^2 )&0&0\\
   \hline
   20. & \begin{pmatrix} 0 &0&0\\&0&0\\&&0\end{pmatrix}\begin{pmatrix} 0&0&0\\&0&0\\&&0\end{pmatrix} & b_{11}a_{11}a_{12}\epsilon&0&0\\
   \hline
   \end{array}
   \end{equation*}
}

All of the orbital exponential sums vanish.  This is checked as follows.
\begin{enumerate}
 \item[1.-6.] These sums vanish on summing over $b_{12}$.
 \item[7.-12.] These sums vanish on summing over $a_{23}$ and $a_{11}$. The summation in $a_{23}$ forces $a_{22} = 0$ in each case, so that $a_{11}$ ranges in $\bF_p$.
\item[13.-20.] These sums vanish on summing over $a_{11}$ or $a_{12}$. 
\end{enumerate}

\subsection{The exponential sums pair $(\sO_{1^4}, \sO_{1^4})$}\label{O14_O14_section}
Standard representatives are
\begin{equation}
 x_0 = \begin{pmatrix} 0&0&0\\ &0&0\\&&1\end{pmatrix}\begin{pmatrix}0&0&\frac{1}{2}\\ &1&0\\&&0\end{pmatrix}, \qquad \xi_0 = \begin{pmatrix}0&0&-1\\&1&0\\&&0\end{pmatrix}\begin{pmatrix}2 &0&0\\&0&0\\&&0\end{pmatrix}
\end{equation}
and
\begin{equation}
 x_1 = \begin{pmatrix} \epsilon &0&0\\&0&0\\&&0\end{pmatrix}\begin{pmatrix}0&0&0\\&0&0\\&&0\end{pmatrix}, \epsilon \in \bF_p^\times/\{x^4: x \in \bF_p^\times\}.
\end{equation}
The acting set is
\begin{equation}
 G_{x, \xi}^t = \begin{pmatrix} a & a_{12} & a_{13}\\ & b & a_{23}\\ &&c \end{pmatrix}\begin{pmatrix} d & b_{12}\\ & e \end{pmatrix}, b^2 = ac, cd = ae.
\end{equation}
Here
\begin{equation}
 [ x_1, g\cdot \xi_0 ] = (d (a_{12}^2 - 2 a a_{13}) + 2 b_{12}a^2)\epsilon.
\end{equation}
The stabilizer has size
\begin{equation}
 \left|\Stab_{G(\zed/p^2\zed)}(x)\right| = (p-1)p^4 \#\{\epsilon^4 = 1\}.
\end{equation}

Representatives and exponential sum pairings are given in the table below. Summing over the maximal orbits,
\begin{equation}
 \sM(x,\xi) = \sum_{\epsilon \in \bF_p^\times/\{x^4: x \in \bF_p^\times\}} \frac{p^9 \sS(x,\xi)}{(p-1)\#\{\epsilon^4 = 1\}}.
\end{equation}

See \textbf{exponential\_sums\_O14\_O14.nb}.
{\tiny
 \begin{equation*}
  \begin{array}{|l|l|l|l|l|}
   \hline
   \text{Orbit} & \xi_1 & [ x_0, g \cdot \xi_1 ] + [ x_1, g\cdot x_0] & \sS(x,\xi)& \sM(x,\xi)\\
   \hline
   1. & \begin{pmatrix} 0 &0&0\\&0&0\\&&0\end{pmatrix}\begin{pmatrix} 0&0&0\\&0&0\\&&1\end{pmatrix} & b_{12}c^2 + e(a_{13}c + a_{23}^2) +  (d (a_{12}^2 - 2 a a_{13}) + 2 b_{12}a^2)\epsilon &0&0\\
   \hline
   2. & \begin{pmatrix} 0 &0&0\\&0&0\\&&0\end{pmatrix}\begin{pmatrix} 0&0&0\\&0&1\\&&0\end{pmatrix} & e(2a_{23}b + a_{12}c) + (d (a_{12}^2 - 2 a a_{13}) + 2 b_{12}a^2)\epsilon &0&0\\
   \hline
   3. &  \begin{pmatrix} 0 &0&0\\&0&0\\&&1\end{pmatrix}\begin{pmatrix} 0&0&-1\\&1&0\\&&0\end{pmatrix} & c^2d + e(-ac+b^2) + (d (a_{12}^2 - 2 a a_{13}) + 2 b_{12}a^2)\epsilon &0&0\\
   \hline
   4. & \begin{pmatrix} 0 &0&0\\&0&0\\&&0\end{pmatrix}\begin{pmatrix} 0&0&0\\&0&0\\&&0\end{pmatrix} & (d (a_{12}^2 - 2 a a_{13}) + 2 b_{12}a^2)\epsilon &0&0\\
   \hline
   \end{array}
   \end{equation*}
}

All of the orbital exponential sums vanish, as is checked below.
\begin{enumerate}
 \item [1.] $\xi_1 = \begin{pmatrix} 0&0&0\\ &0&0\\&&0\end{pmatrix}\begin{pmatrix} 0&0&0\\ &0&0\\&&1 \end{pmatrix}$.  Here
 \begin{align*}
  [ x_0, g \cdot \xi_1 ] &= b_{12}c^2 + e(a_{13}c + a_{23}^2)\\
  [ x_0, g \cdot \xi_1 ] + [ x_1, g \cdot \xi_0 ] &=   b_{12}c^2 + e(a_{13}c + a_{23}^2)+(d (a_{12}^2 - 2 a a_{13}) + 2 b_{12}a^2)\epsilon.
 \end{align*}
 Sum in $b_{12}$ to find $2 a^2\epsilon + c^2 = 0$ or $2 a^2\epsilon + \left(\frac{ae}{d}\right)^2 = 0$.  Sum in $a_{13}$ to find $2 ad\epsilon = ce$ or $2 ad\epsilon = \frac{ae^2}{d}$.  These two conditions are inconsistent, as the first implies $2 d^2\epsilon = -e^2$ and the second implies $2 d^2\epsilon = e^2$, so the sum vanishes.
 \item[2.] $\xi_1 = \begin{pmatrix} 0&0&0\\&0&0\\&&0\end{pmatrix}\begin{pmatrix}0&0&0\\&0&1\\ &&0\end{pmatrix}$.
 The sum vanishes on summing over $b_{12}$.
 \item[3.] $\xi_1 = \begin{pmatrix} 0&0&0\\&0&0\\&&1\end{pmatrix}\begin{pmatrix}0&0&-1\\&1&0\\&&0\end{pmatrix}$.  The sum vanishes on summing over $b_{12}$.
 \item[4.] $\xi_1 = 0$.  The sum vanishes on summing over $b_{12}$.
\end{enumerate}

\section{Summation}\label{summation_section}
Combining the previous sections proves the following theorem. 

\begin{theorem}
 The Fourier transform of the maximal set is supported on the mod $p$ orbits $\sO_0, \sO_{D1^2}, \sO_{D11}$ and $\sO_{D2}$.  It is given explicitly in the following tables.
 \begin{enumerate} 
  \item Case $\sO_0$, $\xi = p\xi_0$.
 \begin{equation}
 \begin{array}{|l|l|l|}
  \hline
  \text{Orbit} & p^{-12}\widehat{\one_{\max}}(p \xi_0) & \text{Orbit size}\\
  \hline
    \sO_0 & (p-1)^4p(p+1)^2(p^5 + 2p^4+4p^3+4p^2+3p+1) & 1\\
  \hline
  \sO_{D1^2} & -(p-1)^3p(p+1)^4 & (p-1)(p+1)(p^2+p+1)\\
  \hline
  \sO_{D11} & -(p-1)^3p(2p^3 + 6p^2 + 4p + 1) & (p-1)p(p+1)^2(p^2+p+1)/2\\
  \hline
  \sO_{D2} & (p-1)^2p(2p^2 + 3p+1)& (p-1)^2p(p+1)(p^2+p+1)/2\\
  \hline
  \sO_{Dns} & (p-1)^2p(2p^2+3p+1)  & (p-1)^2p^2(p+1)(p^2+p+1)\\
  \hline
  \sO_{Cs} & -p^7+5p^5-3p^4-3p^3+p^2 + p & (p-1)^2p(p+1)^2(p^2+p+1)\\
  \hline
  \sO_{Cns} & (p-1)^2p(2p^2 + 3p + 1) & (p-1)^2p^3(p+1)(p^2+p+1)\\
  \hline
  \sO_{B11} & (p-1)^2p(2p^2 + 3p + 1)& (p-1)^2p^2(p+1)^2(p^2+p+1)/2\\
  \hline
  \sO_{B2} & (p-1)^2p(2p^2 + 3p + 1) & (p-1)^3p^2(p+1)(p^2+p+1)/2\\
  \hline
  \sO_{1^4} & p(p^3 -3p^2 +p+1) & (p-1)^3p^2(p+1)^2(p^2+p+1)\\
  \hline
  \sO_{1^31} & p(p^3 -3p^2 +p+1) & (p-1)^3p^3(p+1)^2(p^2+p+1)\\
  \hline
  \sO_{1^21^2} & (p-1)^2p(3p+1) & (p-1)^2p^4(p+1)^2(p^2 + p+ 1)/2\\
  \hline
  \sO_{2^2} & -(p-1)p(p+1)^2 & (p-1)^3p^4(p+1)(p^2+p+1)/2\\
  \hline
  \sO_{1^211} & p(p^3-3p^2 + p+1) & (p-1)^3p^4(p+1)^2(p^2 + p + 1)/2\\
  \hline
  \sO_{1^2 2} & p(p^3 - 3p^2 +p+1) & (p-1)^3p^4(p+1)^2(p^2+p+1)/2\\
  \hline
  \sO_{1111} & -p^3+p^2 + p & (p-1)^4p^4(p+1)^2(p^2 + p + 1)/24\\
  \hline
  \sO_{112} & -p^3 + p^2 + p& (p-1)^4p^4(p+1)^2(p^2 + p + 1)/4\\
  \hline
  \sO_{22} & -p^3 + p^2 + p & (p-1)^4p^4(p+1)^2(p^2 + p + 1)/8\\
  \hline
  \sO_{13} & -p^3 + p^2 + p& (p-1)^4p^4(p+1)^2(p^2 + p + 1)/3\\
  \hline
  \sO_{4} & -p^3 + p^2 + p& (p-1)^4p^4(p+1)^2(p^2 + p + 1)/4\\
  \hline
 \end{array}
\end{equation}
  \item Case $\sO_{D1^2}$.
   \begin{equation}
 \begin{array}{|l|l|l|}
  \hline
  \text{Orbit} & p^{-12}\widehat{\one_{\max}}( \xi)& \text{Orbit size}\\
  \hline
    1. & -(p-1)^3p(p+1)^3 & (p-1)p^4(p+1)(p^2+p+1)\\
    \hline
    2. & (p-1)^2 p (2p+1) & (p-1)^2p^4(p+1)^2(p^2+p+1)\\
    \hline
    3. & (p-1)^2 p(2p+1) & (p-1)^2p^5(p+1)^2(p^2+p+1)/2\\
    \hline
    4. & -(p-1)p(p+1)& (p-1)^3p^5(p+1)(p^2+p+1)/2\\
    \hline
    5. & p (p^3 - 2p^2 + 1) & (p-1)^3p^4(p+1)^2(p^2+p+1)\\
    \hline
    6. & -(p-1)p(p+1) & (p-1)^3p^5(p+1)^2(p^2+p+1)\\
    \hline
    7. & -(p-1)p(p+1) & (p-1)^3p^5(p+1)^2(p^2+p+1)\\
    \hline
    8. & p & (p-1)^4p^5(p+1)^2(p^2+p+1)/2\\
    \hline
    9. & p & (p-1)^4p^5(p+1)^2(p^2+p+1)/2\\
    \hline
    10. & (p-1)^2 p (p+1)^2 & (p-1)^2p^4(p+1)^2(p^2+p+1)\\
    \hline
    11. &(p-1)^2 p (p+1) & (p-1)^2p^6(p+1)^2(p^2+p+1)\\
    \hline
    12. &- (p-1)p & (p-1)^3p^6(p+1)^2(p^2+p+1)\\
    \hline
    13. & -(p-1)p & (p-1)^3p^7(p+1)^2(p^2+p+1)\\
    \hline
    14. & -(p-1)p(p+1) & (p-1)^3p^6(p+1)^2(p^2+p+1)\\
    \hline
    15. & p & (p-1)^4p^6(p+1)^2(p^2+p+1)\\
    \hline
    16. & p & (p-1)^4p^7(p+1)^2(p^2+p+1)\\
    \hline
    17. & 0 & (p-1)^2p^8(p+1)^2(p^2+p+1)/2\\
    \hline
    18. & 0 & (p-1)^3p^8(p+1)^2(p^2+p+1)\\
    \hline
    19. &0 & (p-1)^4p^8(p+1)^2(p^2+p+1)/4\\
    \hline
    20. &0 & (p-1)^4p^8(p+1)^2(p^2+p+1)/4\\
    \hline
    21. &0 & (p-1)^3p^8(p+1)(p^2+p+1)/2\\
    \hline
    22. &0 & (p-1)^4p^8(p+1)^2(p^2+p+1)/4 \\
    \hline
    23. &0& (p-1)^4p^8(p+1)^2(p^2+p+1)/4\\
  \hline
 \end{array}
\end{equation}
  \item Case $\sO_{D11}$.
   \begin{equation}
 \begin{array}{|l|l|l|}
  \hline
  \text{Orbit} & p^{-12}\widehat{\one_{\max}}( \xi) & \text{Orbit size}\\
  \hline
    1. & 0 & (p-1)^2p^{10}(p+1)^2(p^2+p+1)/2\\
    \hline
    2. & 0 & (p-1)^3p^{10}(p+1)^2(p^2+p+1)/2\\
    \hline
    3. & 0 & (p-1)^4p^{10}(p+1)^2(p^2+p+1)/8\\
    \hline
    4. & 0 & (p-1)^3p^{10}(p+1)^2(p^2+p+1)/2\\
    \hline
    5. & 0 & (p-1)^4p^{10}(p+1)^2(p^2+p+1)/4\\
    \hline
    6. & 0 & (p-1)^4p^{10}(p+1)^2(p^2+p+1)/8\\
    \hline
    7. & 0 & (p-1)^2p^8(p+1)^2(p^2+p+1)\\
    \hline
    8. & 0 & (p-1)^3p^9(p+1)^2(p^2+p+1)/2\\
    \hline
    9. & 0 & (p-1)^3p^8(p+1)^2(p^2+p+1)\\
    \hline
    10. & 0 & (p-1)^3p^8(p+1)^2(p^2+p+1)\\
    \hline
    11. & 0 & (p-1)^4p^9(p+1)^2(p^2+p+1)/2\\
    \hline
    12. & 0 & (p-1)^4p^8(p+1)^2(p^2+p+1)\\
    \hline
    13. & 0 & (p-1)^2p^7(p+1)^2(p^2+p+1)\\
    \hline
    14. & (p-1)p^2 & (p-1)^3p^7(p+1)^2(p^2+p+1)/4\\
    \hline
    15. & -(p-1)p^2 & (p-1)^3p^7(p+1)^2(p^2+p+1)/4\\
    \hline
    16. & 0 & (p-1)^3p^7(p+1)^2(p^2+p+1)\\
    \hline
    17. & -p^2 & (p-1)^4p^7(p+1)^2(p^2+p+1)/4\\
    \hline
    18. & p^2 & (p-1)^4p^7(p+1)^2(p^2+p+1)/4\\
    \hline
    19. &0 & (p-1)^2p^7(p+1)^2(p^2+p+1)/2\\
    \hline
    20. &0 & (p-1)p^7(p+1)^2(p^2+p+1)/2\\
  \hline
 \end{array}
\end{equation}
\item Case $\sO_{D2}$.
   \begin{equation}
 \begin{array}{|l|l|l|}
  \hline
  \text{Orbit} & p^{-12}\widehat{\one_{\max}}( \xi)& \text{Orbit size}\\
  \hline
    1. & 0 & (p-1)^3p^{10}(p+1)(p^2+p+1)/2\\
  \hline
    2. & 0 & (p-1)^4p^{10}(p+1)^2(p^2+p+1)/4\\
  \hline
    3. & 0 & (p-1)^4p^{10}(p+1)^2(p^2+p+1)/4\\
  \hline
    4. & 0 & (p-1)^3p^9(p+1)^2(p^2+p+1)/2\\
  \hline
    5. &0 & (p-1)^4p^9(p+1)^2(p^2+p+1)/2\\
  \hline
    6. & (p-1)p^2 & (p-1)^3p^7(p+1)^2(p^2+p+1)/4 \\
  \hline
    7. & -p^2 & (p-1)^4p^7(p+1)^2(p^2+p+1)/4\\
  \hline 
    8. & -(p-1)p^2 & (p-1)^3p^7(p+1)^2(p^2+p+1)/4\\ 
  \hline
    9. & p^2 & (p-1)^4p^7(p+1)^2(p^2+p+1)/4\\
  \hline
    10. & 0 & (p-1)^3p^7(p+1)(p^2+p+1)/2\\
  \hline
    11. &0 & (p-1)^2p^7(p+1)(p^2+p+1)/2\\
  \hline
 \end{array}
\end{equation}
 \end{enumerate}
\end{theorem}
\begin{proof}
Since $\one_{\max}$ is  $G(\zed/p^2\zed)$-invariant, it is the sum of the indicator functions of some $G(\zed/p^2\zed)$ orbits on $V(\zed/p^2\zed)$.

When $\xi \in \sO_0$ so that $\xi = p\xi_1$, 
\begin{align}
 \widehat{\one_{\max}}(p\xi_1) &= \sum_{x \in V(\zed/p^2\zed)} \one_{\max}(x)e_{p^2}([x,p\xi_1])\\
 \notag &= \sum_{\sO \in G(\zed/p\zed)\backslash V(\zed/p\zed)} \sum_{x \in \sO \bmod p} e_p([x, \xi_1]) \left(\sum_{x_1 \in V(\zed/p\zed)} \one_{\max}(x + px_1) \right).
\end{align}
The density \begin{equation}\mu(\sO)=p^{-12} \sum_{x_1 \in V(\zed/p\zed)} \one_{\max}(x + px_1)\end{equation} of maximal orbits above $x$ depends only on the orbit $\sO$, and is given in the following table.
 \begin{equation}
  \begin{array}{|l|l|}
   \hline
   \text{Orbit} & \text{Density}\\
   \hline
   \sO_{1^4} & \frac{p-1}{p}\\
   \hline
   \sO_{1^31} & \frac{p-1}{p}\\
   \hline
   \sO_{1^21^2} & \left(\frac{p-1}{p}\right)^2\\
   \hline
   \sO_{2^2} & \frac{p^2-1}{p^2}\\
   \hline
   \sO_{1^211} & \frac{p-1}{p}\\
   \hline
   \sO_{1^2 2} & \frac{p-1}{p}\\
   \hline
   \sO_{1111} & 1\\
   \hline
   \sO_{112} & 1\\
   \hline
   \sO_{22} & 1\\
   \hline
   \sO_{13} & 1\\
   \hline
   \sO_{4} & 1 \\
   \hline
  \end{array}
 \end{equation}
 The Fourier transform above $\sO_0$ was thus calculated according to the formula
 \begin{equation}
  \widehat{\one_{\max}}(p\xi_1) = p^{12}\sum_{\sO \in G(\zed/p\zed)\backslash V(\zed/p\zed)}\mu(\sO) \sum_{x \in \sO \bmod p} e_p([x, \xi_1]), 
 \end{equation}
 see \textbf{modp\_sums.nb}.
The orbital exponential sums $\sum_{x \in \sO} e_p([x, \xi_1])$ are obtained as the columns of the matrix $M$ in \cite{TT16} p.27, as a function of the orbits of $\xi_1$ under the $G(\zed/p\zed)$ action.
 The orbit sizes of the $\mod p$ orbits were also determined in \cite{TT16}.

 The orbit sizes of the orbits in $V/V_x$ under the $G_x$ action were obtained in Lemmas 19, 17, and 16. The orbit sizes of the orbits in $V(\zed/p^2\zed)$ under $G(\zed/p^2\zed)$ were obtained from the formula
  \begin{equation}
  \left|\sO_\xi \bmod p^2\right| = \left|\sO_{\xi_0 \bmod p}\right| \cdot p^{\dim V_\xi} \cdot \left|\sO_{\xi_1} \subset V/V_\xi\right|,
 \end{equation}
 from Lemma \ref{stabilizer_size_lemma}. 
 
 When $\xi \not \equiv 0 \bmod p$, it was verified in Lemma \ref{full_mod_p_orbit_lemma} that for $\sO \in G(\zed/p\zed)\backslash V(\zed/p\zed)$ with density of maximal elements 0 or 1,
 \begin{equation}
  \sum_{\substack{x \in V(\zed/p^2\zed)\\ x \in \sO \bmod p}}\one_{\max}(x)e_{p^2}([x,\xi]) = \sM(x, \xi) = 0.
 \end{equation}
Hence, in the formula
\begin{align}
 \widehat{\one_{\max}}(\xi) &= \sum_{x \in V(\zed/p^2\zed)}\one_{\max}(x)e_{p^2}([x, \xi])\\
 \notag &= \sum_{\sO \in G(\zed/p\zed)\backslash V(\zed/p\zed)}\sM(x, \xi)
\end{align}
the sum may be restricted to $\sO \in \{\sO_{1^211}, \sO_{1^22}, \sO_{1^21^2}, \sO_{2^2}, \sO_{1^31}, \sO_{1^4}\}$. This sum is further restricted, since by Lemma \ref{orb_exp_sum_lemma}, $\sM(x, \xi) = 0$ if $G_{x, \xi} = \emptyset$. This proves that $\widehat{\one_{\max}}(\xi) = 0$ unless $\xi \in \sO_0, \sO_{D1^2}, \sO_{D11}, \sO_{D2}, \sO_{Cs}$ or $\sO_{1^4}$. For each of those pairs $x, \xi\neq 0 \bmod p$ such that $G_{x, \xi} \neq \emptyset$, $\sM(x, \xi)$ was calculated.  This demonstrated that in fact the Fourier transform vanishes over $\sO_{Cs}$ and $\sO_{1^4}$, also.  The Fourier transform displayed is the result of adding the columns $\sM(x, \xi)$ from the previous section, Sections \ref{O1211_OD12_section}, \ref{O122_OD12_section}, \ref{O1212_OD12_section}, \ref{O131_OD12_section} and \ref{O14_OD12_section} for the Fourier transform above $\sO_{D1^2}$, Sections \ref{O22_OD11_section} and \ref{O1212_OD11_section} for the Fourier transform above $\sO_{D11}$, and Sections \ref{O22_OD2_section} and \ref{O1212_OD2_section} for the Fourier transform above $\sO_{D2}$.

For the calculations in this proof, see \textbf{summation.nb}.

\end{proof}

\begin{proof}[Proof of Theorem \ref{norm_theorem}]
 The Fourier transform of $\one_{\nonmax}$ is $p^{24}\delta_{\xi = 0} - \widehat{\one_{\max}}(\xi)$. 
The formulas obtained in Theorem \ref{norm_theorem} are the result of combining the formulas for the Fourier transform with the size of the orbit of each frequency.  These calculations are performed in \textbf{summation.nb}.
\end{proof}

\appendix
\section{Computation of annihilator subspaces}\label{annihilator_appendix}
The annihilator subspaces for standard representatives are as follows.  For the computations in this Appendix, see \textbf{annihilator\_spaces.nb}.

\begin{enumerate}
 \item Case of $\sO_{D1^2}$, $x = (0, w^2)$. 
 {\tiny
 \begin{align*}
  &\left(I + p\begin{pmatrix} b_{11} & b_{12} \\ b_{21} & b_{22}\end{pmatrix} \right)\left(\left( I + p \begin{pmatrix} a_{11} & a_{12} & a_{13}\\ a_{21} & a_{22} & a_{23}\\ a_{31} & a_{32} & a_{33}\end{pmatrix}\right) \begin{array}{l} \begin{pmatrix} 0&0&0\\0&0&0\\0&0&0\end{pmatrix}\\ \begin{pmatrix} 0&0&0\\0&0&0\\0&0&1\end{pmatrix} \end{array} \left(I + p \begin{pmatrix} a_{11} & a_{21} & a_{31}\\ a_{12} & a_{22} & a_{32}\\ a_{13} & a_{23} & a_{33}\end{pmatrix}\right)\right)\\
  &= \begin{array}{l} \begin{pmatrix} 0&0&0\\&0&0\\&&0\end{pmatrix}\\ \begin{pmatrix} 0&0&0\\&0&0\\&&1\end{pmatrix} \end{array} + p \begin{array}{l} \begin{pmatrix} 0&0&0\\&0&0\\&&b_{12}\end{pmatrix}\\ \begin{pmatrix} 0&0&a_{13}\\&0&a_{23}\\&&2 a_{33}+b_{22}\end{pmatrix} \end{array}.
 \end{align*}
}
Thus
\begin{equation*}
 V_x = \begin{array}{l} \begin{pmatrix} 0&0&0\\&0&0\\&&*\end{pmatrix}\\ \begin{pmatrix} 0&0& *\\&0& *\\  && *\end{pmatrix} \end{array}.
\end{equation*}
 \item Case of $\sO_{D11}$, $x = (0, vw)$.
  {\tiny
 \begin{align*}
  &\left(I + p\begin{pmatrix} b_{11} & b_{12} \\ b_{21} & b_{22}\end{pmatrix} \right)\left(\left( I + p \begin{pmatrix} a_{11} & a_{12} & a_{13}\\ a_{21} & a_{22} & a_{23}\\ a_{31} & a_{32} & a_{33}\end{pmatrix}\right) \begin{array}{l} \begin{pmatrix} 0&0&0\\0&0&0\\0&0&0\end{pmatrix}\\ \begin{pmatrix} 0&0&0\\0&0&\frac{1}{2}\\0&\frac{1}{2}&0\end{pmatrix} \end{array} \left(I + p \begin{pmatrix} a_{11} & a_{21} & a_{31}\\ a_{12} & a_{22} & a_{32}\\ a_{13} & a_{23} & a_{33}\end{pmatrix}\right)\right)\\
  &= \begin{array}{l} \begin{pmatrix} 0&0&0\\&0&0\\&&0\end{pmatrix}\\ \begin{pmatrix} 0&0&0\\&0&\frac{1}{2}\\&&0\end{pmatrix} \end{array} + p \begin{array}{l} \begin{pmatrix} 0&0&0\\&0&\frac{b_{12}}{2}\\&&0\end{pmatrix}\\ \begin{pmatrix} 0&\frac{a_{13}}{2}&\frac{a_{12}}{2}\\&a_{23}&\frac{a_{22} + a_{33} + b_{22}}{2}\\&&a_{32}\end{pmatrix} \end{array}.
 \end{align*}
}
Thus
\begin{equation*}
 V_x = \begin{array}{l} \begin{pmatrix} 0&0&0\\&0&*\\&&0\end{pmatrix}\\ \begin{pmatrix} 0&*& *\\&*& *\\  & & *\end{pmatrix} \end{array}.
\end{equation*}
\item Case of $\sO_{D2}$, $x = (0, v^2 - \ell w^2)$.
  {\tiny
 \begin{align*}
  &\left(I + p\begin{pmatrix} b_{11} & b_{12} \\ b_{21} & b_{22}\end{pmatrix} \right)\left(\left( I + p \begin{pmatrix} a_{11} & a_{12} & a_{13}\\ a_{21} & a_{22} & a_{23}\\ a_{31} & a_{32} & a_{33}\end{pmatrix}\right) \begin{array}{l} \begin{pmatrix} 0&0&0\\0&0&0\\0&0&0\end{pmatrix}\\ \begin{pmatrix} 0&0&0\\0&1&0\\0&0&-\ell\end{pmatrix} \end{array} \left(I + p \begin{pmatrix} a_{11} & a_{21} & a_{31}\\ a_{12} & a_{22} & a_{32}\\ a_{13} & a_{23} & a_{33}\end{pmatrix}\right)\right)\\
  &= \begin{array}{l} \begin{pmatrix} 0&0&0\\&0&0\\&&0\end{pmatrix}\\ \begin{pmatrix} 0&0&0\\&1&0\\&&-\ell\end{pmatrix} \end{array} + p \begin{array}{l} \begin{pmatrix} 0&0&0\\&b_{12}&0\\&&- b_{12}\ell\end{pmatrix}\\ \begin{pmatrix} 0&a_{12}&- a_{13}\ell\\&2a_{22} + b_{22}&- a_{23}\ell + a_{32}\\&&-2  a_{33}\ell -  b_{22}\ell\end{pmatrix} \end{array}.
 \end{align*}
}
Thus
\begin{equation*}
 V_x = \begin{array}{l} \begin{pmatrix} 0&0&0\\&z&0\\&&- z\ell\end{pmatrix}\\ \begin{pmatrix} 0&*& *\\&*& *\\  & & *\end{pmatrix} \end{array}.
\end{equation*}

\item Case of $\sO_{Dns}$, $x = (0, u^2 - vw)$.
  {\tiny
 \begin{align*}
  &\left(I + p\begin{pmatrix} b_{11} & b_{12} \\ b_{21} & b_{22}\end{pmatrix} \right)\left(\left( I + p \begin{pmatrix} a_{11} & a_{12} & a_{13}\\ a_{21} & a_{22} & a_{23}\\ a_{31} & a_{32} & a_{33}\end{pmatrix}\right) \begin{array}{l} \begin{pmatrix} 0&0&0\\0&0&0\\0&0&0\end{pmatrix}\\ \begin{pmatrix} 1&0&0\\0&0&-\frac{1}{2}\\0&-\frac{1}{2}&0\end{pmatrix} \end{array} \left(I + p \begin{pmatrix} a_{11} & a_{21} & a_{31}\\ a_{12} & a_{22} & a_{32}\\ a_{13} & a_{23} & a_{33}\end{pmatrix}\right)\right)\\
  &= \begin{array}{l} \begin{pmatrix} 0&0&0\\&0&0\\&&0\end{pmatrix}\\ \begin{pmatrix} 1&0&0\\&0&-\frac{1}{2}\\&&0\end{pmatrix} \end{array} + p \begin{array}{l} \begin{pmatrix} b_{12}&0&0\\&0&-\frac{b_{12}}{2}\\&&0\end{pmatrix}\\ \begin{pmatrix} 2 a_{11}+b_{22}&- \frac{a_{13}}{2} + a_{21}&-\frac{a_{12}}{2} + a_{31}\\&-a_{23}&-\frac{a_{22} + a_{33}+b_{22}}{2}\\&&-a_{32}\end{pmatrix} \end{array}.
 \end{align*}
}
Thus
\begin{equation*}
 V_x = \begin{array}{l} \begin{pmatrix} z&0&0\\&0&-\frac{z}{2}\\&&0\end{pmatrix}\\ \begin{pmatrix} *&*& *\\&*& *\\  & & *\end{pmatrix} \end{array}.
\end{equation*}

\item Case of $\sO_{Cs}$, $x = (w^2, vw)$.
  {\tiny
 \begin{align*}
  &\left(I + p\begin{pmatrix} b_{11} & b_{12} \\ b_{21} & b_{22}\end{pmatrix} \right)\left(\left( I + p \begin{pmatrix} a_{11} & a_{12} & a_{13}\\ a_{21} & a_{22} & a_{23}\\ a_{31} & a_{32} & a_{33}\end{pmatrix}\right) \begin{array}{l} \begin{pmatrix} 0&0&0\\0&0&0\\0&0&1\end{pmatrix}\\ \begin{pmatrix} 0&0&0\\0&0&\frac{1}{2}\\0&\frac{1}{2}&0\end{pmatrix} \end{array} \left(I + p \begin{pmatrix} a_{11} & a_{21} & a_{31}\\ a_{12} & a_{22} & a_{32}\\ a_{13} & a_{23} & a_{33}\end{pmatrix}\right)\right)\\
  &= \begin{array}{l} \begin{pmatrix} 0&0&0\\&0&0\\&&1\end{pmatrix}\\ \begin{pmatrix} 0&0&0\\&0&\frac{1}{2}\\&&0\end{pmatrix} \end{array} + p \begin{array}{l} \begin{pmatrix} 0&0&a_{13}\\&0&a_{23} + \frac{b_{12}}{2}\\&&2a_{33} + b_{11}\end{pmatrix}\\ \begin{pmatrix} 0&\frac{a_{13}}{2}&\frac{a_{12}}{2}\\&a_{23}&\frac{a_{22} + a_{33} + b_{22}}{2}\\&&a_{32}+b_{21}\end{pmatrix} \end{array}.
 \end{align*}
}
Thus
\begin{equation*}
 V_x = \begin{array}{l} \begin{pmatrix} 0&0&z\\&0& *\\&& *\end{pmatrix}\\ \begin{pmatrix} 0& \frac{z}{2}& *\\& * & *\\  & & *\end{pmatrix} \end{array}.
\end{equation*}

\item Case of $\sO_{Cns}$, $x = (vw, uw)$.
  {\tiny
 \begin{align*}
  &\left(I + p\begin{pmatrix} b_{11} & b_{12} \\ b_{21} & b_{22}\end{pmatrix} \right)\left(\left( I + p \begin{pmatrix} a_{11} & a_{12} & a_{13}\\ a_{21} & a_{22} & a_{23}\\ a_{31} & a_{32} & a_{33}\end{pmatrix}\right) \begin{array}{l} \begin{pmatrix} 0&0&0\\0&0&\frac{1}{2}\\0&\frac{1}{2}&0\end{pmatrix}\\ \begin{pmatrix} 0&0&\frac{1}{2}\\0&0&0\\\frac{1}{2}&0&0\end{pmatrix} \end{array} \left(I + p \begin{pmatrix} a_{11} & a_{21} & a_{31}\\ a_{12} & a_{22} & a_{32}\\ a_{13} & a_{23} & a_{33}\end{pmatrix}\right)\right)\\
  &= \begin{array}{l} \begin{pmatrix} 0&0&0\\&0&\frac{1}{2}\\&&0\end{pmatrix}\\ \begin{pmatrix} 0&0&\frac{1}{2}\\&0&0\\&&0\end{pmatrix} \end{array} + p \begin{array}{l} \begin{pmatrix} 0&\frac{a_{13}}{2}&\frac{a_{12} + b_{12}}{2}\\&a_{23}&\frac{ a_{22} + a_{33}+b_{11}}{2}\\&&a_{32}\end{pmatrix}\\ \begin{pmatrix} a_{13}&\frac{a_{23}}{2}&\frac{a_{11} + a_{33} + b_{22}}{2}\\&0&\frac{  a_{21}+b_{21}}{2}\\&&a_{31}\end{pmatrix} \end{array}.
 \end{align*}
}
Thus
\begin{equation*}
 V_x = \begin{array}{l} \begin{pmatrix} 0&\frac{z}{2}&*\\&y& *\\&& *\end{pmatrix}\\ \begin{pmatrix}z & \frac{y}{2}& *\\& 0 & *\\  & & *\end{pmatrix} \end{array}.
\end{equation*}

\item Case of $\sO_{B11}$, $x = (v^2, w^2)$.
  {\tiny
 \begin{align*}
  &\left(I + p\begin{pmatrix} b_{11} & b_{12} \\ b_{21} & b_{22}\end{pmatrix} \right)\left(\left( I + p \begin{pmatrix} a_{11} & a_{12} & a_{13}\\ a_{21} & a_{22} & a_{23}\\ a_{31} & a_{32} & a_{33}\end{pmatrix}\right) \begin{array}{l} \begin{pmatrix} 0&0&0\\0&1&0\\0&0&0\end{pmatrix}\\ \begin{pmatrix} 0&0&0\\0&0&0\\0&0&1\end{pmatrix} \end{array} \left(I + p \begin{pmatrix} a_{11} & a_{21} & a_{31}\\ a_{12} & a_{22} & a_{32}\\ a_{13} & a_{23} & a_{33}\end{pmatrix}\right)\right)\\
  &= \begin{array}{l} \begin{pmatrix} 0&0&0\\&1&0\\&&0\end{pmatrix}\\ \begin{pmatrix} 0&0&0\\&0&0\\&&1\end{pmatrix} \end{array} + p \begin{array}{l} \begin{pmatrix} 0& a_{12}&0\\& 2 a_{22}+b_{11} &a_{32}\\&& b_{12}\end{pmatrix}\\ \begin{pmatrix} 0&0&a_{13}\\&b_{21}&a_{23}\\&& 2a_{33}+b_{22}\end{pmatrix} \end{array}.
 \end{align*}
}
Thus
\begin{equation*}
 V_x = \begin{array}{l} \begin{pmatrix} 0& *& 0\\& *& *\\&& *\end{pmatrix}\\ \begin{pmatrix} 0& 0& *\\& * & *\\  & & *\end{pmatrix} \end{array}.
\end{equation*}

\item Case of $\sO_{B2}$, $x = (vw, v^2 + \ell w^2)$.
  {\tiny
 \begin{align*}
  &\left(I + p\begin{pmatrix} b_{11} & b_{12} \\ b_{21} & b_{22}\end{pmatrix} \right)\left(\left( I + p \begin{pmatrix} a_{11} & a_{12} & a_{13}\\ a_{21} & a_{22} & a_{23}\\ a_{31} & a_{32} & a_{33}\end{pmatrix}\right) \begin{array}{l} \begin{pmatrix} 0&0&0\\0&0&\frac{1}{2}\\0&\frac{1}{2}&0\end{pmatrix}\\ \begin{pmatrix} 0&0&0\\0&1&0\\0&0&\ell\end{pmatrix} \end{array} \left(I + p \begin{pmatrix} a_{11} & a_{21} & a_{31}\\ a_{12} & a_{22} & a_{32}\\ a_{13} & a_{23} & a_{33}\end{pmatrix}\right)\right)\\
  &= \begin{array}{l} \begin{pmatrix} 0&0&0\\&0&\frac{1}{2}\\&&0\end{pmatrix}\\ \begin{pmatrix} 0&0&0\\&1&0\\&&\ell\end{pmatrix} \end{array} + p \begin{array}{l} \begin{pmatrix} 0& \frac{a_{13}}{2}& \frac{a_{12}}{2}\\&   a_{23}+b_{12} &  \frac{a_{22} + a_{33}+b_{11}}{2}\\&&   a_{32}+b_{12}\ell\end{pmatrix}\\ \begin{pmatrix} 0&a_{12}& a_{13}\ell\\& 2 a_{22}+b_{22} &  a_{23}\ell + a_{32} + \frac{b_{21}}{2}\\&&   2 a_{33}\ell+b_{22}\ell\end{pmatrix} \end{array}.
 \end{align*}
}
Thus
\begin{equation*}
 V_x = \begin{array}{l} \begin{pmatrix} 0& \frac{z}{2}& \frac{y}{2}\\& *& *\\&& *\end{pmatrix}\\ \begin{pmatrix} 0& y&  z\ell \\& * & *\\  & & *\end{pmatrix} \end{array}.
\end{equation*}

\item Case of $\sO_{1^4}$, $(w^2, uw + v^2)$.
  {\tiny
 \begin{align*}
  &\left(I + p\begin{pmatrix} b_{11} & b_{12} \\ b_{21} & b_{22}\end{pmatrix} \right)\left(\left( I + p \begin{pmatrix} a_{11} & a_{12} & a_{13}\\ a_{21} & a_{22} & a_{23}\\ a_{31} & a_{32} & a_{33}\end{pmatrix}\right) \begin{array}{l} \begin{pmatrix} 0&0&0\\0&0&0\\0&0&1\end{pmatrix}\\ \begin{pmatrix} 0&0&\frac{1}{2}\\0&1&0\\\frac{1}{2}&0&0\end{pmatrix} \end{array} \left(I + p \begin{pmatrix} a_{11} & a_{21} & a_{31}\\ a_{12} & a_{22} & a_{32}\\ a_{13} & a_{23} & a_{33}\end{pmatrix}\right)\right)\\
  &= \begin{array}{l} \begin{pmatrix} 0&0&0\\&0&0\\&&1\end{pmatrix}\\ \begin{pmatrix} 0&0&\frac{1}{2}\\&1&0\\&&0\end{pmatrix} \end{array} + p \begin{array}{l} \begin{pmatrix} 0& 0& a_{13} + \frac{b_{12}}{2}\\&b_{12}  & a_{23} \\&&  2a_{33}+b_{11}\end{pmatrix}\\ \begin{pmatrix} a_{13}&a_{12} + \frac{a_{23}}{2}& \frac{a_{11} + a_{33} + b_{22}}{2}\\& 2 a_{22}+b_{22} &  \frac{a_{21}}{2}+a_{32}\\&& a_{31}+b_{21}\end{pmatrix} \end{array}.
 \end{align*}
}
Thus
\begin{equation*}
 V_x = \begin{array}{l} \begin{pmatrix} 0& 0& y + \frac{z}{2}\\& z& *\\&& *\end{pmatrix}\\ \begin{pmatrix} y& *& * \\& * & *\\  & & *\end{pmatrix} \end{array}.
\end{equation*}

\item Case of $\sO_{1^31}$, $(vw, uw + v^2)$.
  {\tiny
 \begin{align*}
  &\left(I + p\begin{pmatrix} b_{11} & b_{12} \\ b_{21} & b_{22}\end{pmatrix} \right)\left(\left( I + p \begin{pmatrix} a_{11} & a_{12} & a_{13}\\ a_{21} & a_{22} & a_{23}\\ a_{31} & a_{32} & a_{33}\end{pmatrix}\right) \begin{array}{l} \begin{pmatrix} 0&0&0\\0&0&\frac{1}{2}\\0&\frac{1}{2}&0\end{pmatrix}\\ \begin{pmatrix} 0&0&\frac{1}{2}\\0&1&0\\\frac{1}{2}&0&0\end{pmatrix} \end{array} \left(I + p \begin{pmatrix} a_{11} & a_{21} & a_{31}\\ a_{12} & a_{22} & a_{32}\\ a_{13} & a_{23} & a_{33}\end{pmatrix}\right)\right)\\
  &= \begin{array}{l} \begin{pmatrix} 0&0&0\\&0&\frac{1}{2}\\&&0\end{pmatrix}\\ \begin{pmatrix} 0&0&\frac{1}{2}\\&1&0\\&&0\end{pmatrix} \end{array} + p \begin{array}{l} \begin{pmatrix} 0& \frac{a_{13}}{2}&  \frac{a_{12} + b_{12}}{2}\\&a_{23}+b_{12}   & \frac{a_{22} + a_{33} +b_{11}}{2} \\&& a_{32}\end{pmatrix}\\ \begin{pmatrix} a_{13}&a_{12} + \frac{a_{23}}{2}& \frac{a_{11} + a_{33} + b_{22}}{2}\\& 2 a_{22}+b_{22} & a_{32} + \frac{a_{21}+b_{21}}{2}\\&& a_{31}\end{pmatrix} \end{array}.
 \end{align*}
}
Thus
\begin{equation*}
 V_x = \begin{array}{l} \begin{pmatrix} 0& \frac{z}{2}& *\\& *& *\\&& *\end{pmatrix}\\ \begin{pmatrix} z& *& * \\& * & *\\  & & *\end{pmatrix} \end{array}.
\end{equation*}

\item Case of $\sO_{1^21^2}$, $(w^2, uv)$.
  {\tiny
 \begin{align*}
  &\left(I + p\begin{pmatrix} b_{11} & b_{12} \\ b_{21} & b_{22}\end{pmatrix} \right)\left(\left( I + p \begin{pmatrix} a_{11} & a_{12} & a_{13}\\ a_{21} & a_{22} & a_{23}\\ a_{31} & a_{32} & a_{33}\end{pmatrix}\right) \begin{array}{l} \begin{pmatrix} 0&0&0\\0&0&0\\0&0&1\end{pmatrix}\\ \begin{pmatrix} 0&\frac{1}{2}&0\\\frac{1}{2}&0&0\\0&0&0\end{pmatrix} \end{array} \left(I + p \begin{pmatrix} a_{11} & a_{21} & a_{31}\\ a_{12} & a_{22} & a_{32}\\ a_{13} & a_{23} & a_{33}\end{pmatrix}\right)\right)\\
  &= \begin{array}{l} \begin{pmatrix} 0&0&0\\&0&0\\&&0\end{pmatrix}\\ \begin{pmatrix} 0&\frac{1}{2}&0\\&0&0\\&&0\end{pmatrix} \end{array} + p \begin{array}{l} \begin{pmatrix} 0& \frac{b_{12}}{2}&  a_{13}\\&0  & a_{23} \\&&  2a_{33}+b_{11}\end{pmatrix}\\ \begin{pmatrix} a_{12}&\frac{a_{11} + a_{22} + b_{22}}{2}& \frac{a_{32}}{2}\\&a_{21} &  \frac{a_{31}}{2}\\&& b_{21}\end{pmatrix} \end{array}.
 \end{align*}
}
Thus
\begin{equation*}
 V_x = \begin{array}{l} \begin{pmatrix} 0& *& *\\& 0& *\\&& *\end{pmatrix}\\ \begin{pmatrix} *& *& * \\& * & *\\  & & *\end{pmatrix} \end{array}.
\end{equation*}

\item Case of $\sO_{2^2}$, $(w^2, u^2 - \ell v^2)$.
  {\tiny
 \begin{align*}
  &\left(I + p\begin{pmatrix} b_{11} & b_{12} \\ b_{21} & b_{22}\end{pmatrix} \right)\left(\left( I + p \begin{pmatrix} a_{11} & a_{12} & a_{13}\\ a_{21} & a_{22} & a_{23}\\ a_{31} & a_{32} & a_{33}\end{pmatrix}\right) \begin{array}{l} \begin{pmatrix} 0&0&0\\0&0&0\\0&0&1\end{pmatrix}\\ \begin{pmatrix} 1&0&0\\0&-\ell&0\\0&0&0\end{pmatrix} \end{array} \left(I + p \begin{pmatrix} a_{11} & a_{21} & a_{31}\\ a_{12} & a_{22} & a_{32}\\ a_{13} & a_{23} & a_{33}\end{pmatrix}\right)\right)\\
  &= \begin{array}{l} \begin{pmatrix} 0&0&0\\&0&0\\&&1\end{pmatrix}\\ \begin{pmatrix} 1&0&0\\&-\ell &0\\&&0\end{pmatrix} \end{array} + p \begin{array}{l} \begin{pmatrix} b_{12}& 0&  a_{13}\\&- b_{12}\ell  & a_{23} \\&&  2a_{33}+b_{11}\end{pmatrix}\\ \begin{pmatrix} 2a_{11}+b_{22}&- a_{12}\ell+a_{21}& a_{31}\\& -2 a_{22}\ell- b_{22}\ell & - a_{32}\ell\\&& b_{21}\end{pmatrix} \end{array}.
 \end{align*}
}
Thus
\begin{equation*}
 V_x = \begin{array}{l} \begin{pmatrix} z& 0& *\\& - z\ell& *\\&& *\end{pmatrix}\\ \begin{pmatrix} *& *& * \\& * & *\\  & & *\end{pmatrix} \end{array}.
\end{equation*}

\item Case of $\sO_{1^211}$, $(v^2 -w^2, uw)$.
  {\tiny
 \begin{align*}
  &\left(I + p\begin{pmatrix} b_{11} & b_{12} \\ b_{21} & b_{22}\end{pmatrix} \right)\left(\left( I + p \begin{pmatrix} a_{11} & a_{12} & a_{13}\\ a_{21} & a_{22} & a_{23}\\ a_{31} & a_{32} & a_{33}\end{pmatrix}\right) \begin{array}{l} \begin{pmatrix} 0&0&0\\0&1&0\\0&0&-1\end{pmatrix}\\ \begin{pmatrix} 0&0&\frac{1}{2}\\0&0&0\\\frac{1}{2}&0&0\end{pmatrix} \end{array} \left(I + p \begin{pmatrix} a_{11} & a_{21} & a_{31}\\ a_{12} & a_{22} & a_{32}\\ a_{13} & a_{23} & a_{33}\end{pmatrix}\right)\right)\\
  &= \begin{array}{l} \begin{pmatrix} 0&0&0\\&1&0\\&&-1\end{pmatrix}\\ \begin{pmatrix} 0&0&\frac{1}{2}\\&0 &0\\&&0\end{pmatrix} \end{array} + p \begin{array}{l} \begin{pmatrix} 0& a_{12}&  -a_{13} + \frac{b_{12}}{2}\\& 2a_{22} +b_{11} & -a_{23}+a_{32} \\&& -2a_{33}-b_{11}\end{pmatrix}\\ \begin{pmatrix}a_{13}&\frac{a_{23}}{2}& \frac{a_{11} + a_{33} + b_{22}}{2}\\&b_{21} & \frac{a_{21}}{2}\\&& a_{31}-b_{21}\end{pmatrix} \end{array}.
 \end{align*}
}
Thus
\begin{equation*}
 V_x = \begin{array}{l} \begin{pmatrix} 0& *& *\\& *& *\\&& *\end{pmatrix}\\ \begin{pmatrix} *& *& * \\& * & *\\  & & *\end{pmatrix} \end{array}.
\end{equation*}

\item Case of $\sO_{1^22}$, $(v^2 - \ell w^2, uw)$.
  {\tiny
 \begin{align*}
  &\left(I + p\begin{pmatrix} b_{11} & b_{12} \\ b_{21} & b_{22}\end{pmatrix} \right)\left(\left( I + p \begin{pmatrix} a_{11} & a_{12} & a_{13}\\ a_{21} & a_{22} & a_{23}\\ a_{31} & a_{32} & a_{33}\end{pmatrix}\right) \begin{array}{l} \begin{pmatrix} 0&0&0\\0&1&0\\0&0&-\ell\end{pmatrix}\\ \begin{pmatrix} 0&0&\frac{1}{2}\\0&0&0\\\frac{1}{2}&0&0\end{pmatrix} \end{array} \left(I + p \begin{pmatrix} a_{11} & a_{21} & a_{31}\\ a_{12} & a_{22} & a_{32}\\ a_{13} & a_{23} & a_{33}\end{pmatrix}\right)\right)\\
  &= \begin{array}{l} \begin{pmatrix} 0&0&0\\&1&0\\&&-\ell\end{pmatrix}\\ \begin{pmatrix} 0&0&\frac{1}{2}\\&0 &0\\&&0\end{pmatrix} \end{array} + p \begin{array}{l} \begin{pmatrix} 0& a_{12}&  - a_{13}\ell + \frac{b_{12}}{2}\\& 2a_{22}+b_{11}  & - a_{23}\ell+ a_{32} \\&& -2 a_{33}\ell- b_{11}\ell\end{pmatrix}\\ \begin{pmatrix}a_{13}&\frac{a_{23}}{2}& \frac{a_{11} + a_{33} + b_{22}}{2}\\&b_{21} & \frac{a_{21}}{2}\\&& a_{31}- b_{21}\ell\end{pmatrix} \end{array}.
 \end{align*}
}
Thus
\begin{equation*}
 V_x = \begin{array}{l} \begin{pmatrix} 0& *& *\\& *& *\\&& *\end{pmatrix}\\ \begin{pmatrix} *& *& * \\& * & *\\  & & *\end{pmatrix} \end{array}.
\end{equation*}
\end{enumerate}

\section{Determination of action sets}\label{action_set_appendix}
For each of the pairings listed in Lemma \ref{action_set_lemma} the action set $G_{x,\xi}$ is calculated. The calculations in this Appendix are performed in \textbf{action\_sets.nb}.
\begin{enumerate}
 \item [1.] Case $(\sO_{1^211}, \sO_{D1^2})$: For standard representatives 
 \begin{equation}
  x = \begin{pmatrix} 0&0&0\\ & 1 &0\\ &&-1\end{pmatrix} \begin{pmatrix} 0 &0 &\frac{1}{2}\\ &0 &0\\&&0\end{pmatrix}, \qquad \xi = \begin{pmatrix} 1 &0 &0\\ &0&0\\ &&0\end{pmatrix} \begin{pmatrix}0&0&0\\&0&0\\&&0\end{pmatrix}
 \end{equation}
 and 
 \begin{equation}
  V_x = \begin{pmatrix} 0 & * & *\\ & * & *\\ & & *\end{pmatrix} \begin{pmatrix} * & * & *\\ & * & *\\ && *\end{pmatrix}, \qquad V_\xi = \begin{pmatrix} * & * & * \\ & 0 & 0 \\ && 0\end{pmatrix}\begin{pmatrix} * & 0 & 0\\ & 0 &0\\ && 0\end{pmatrix}.
 \end{equation}
 Since the second form is eliminated by $V_x$, in the action on $\xi$, it may be assumed that the $\GL_2$ action acts on the second quadratic form by an arbitrary scalar.  
 
 The action on $\xi$ is thus given by 
 {\tiny
 \begin{equation*}
 \left(\begin{pmatrix} a_{11} & a_{12} & a_{13}\\ a_{21} & a_{22} & a_{23}\\ a_{31} & a_{32} & a_{33}\end{pmatrix},
  \begin{pmatrix} b_{11} & b_{12} \\  & b_{22}\end{pmatrix}\right) \cdot \xi =  \begin{pmatrix} b_{11}a_{11}^2 & b_{11}a_{11}a_{21} & b_{11}a_{11}a_{31}\\ & b_{11}a_{21}^2  & b_{11}a_{21}a_{31}\\ && b_{11}a_{31}^2  \end{pmatrix} \begin{pmatrix} 0 & 0 & 0\\ & 0 & 0\\ && 0 \end{pmatrix}.
 \end{equation*}
} This forces $ a_{21} = a_{31} = 0$, so the action set is
\begin{equation}
 G_{x, \xi}^t = \begin{pmatrix} * & * & * \\ & * & *\\ & * & *\end{pmatrix} \begin{pmatrix} * & *\\ & *\end{pmatrix}, \qquad |G_{x, \xi}| = (p-1)^5 p^4 (p+1).
\end{equation}
 \item [2.] Case $(\sO_{1^2 2}, \sO_{D1^2})$: For standard representatives
 \begin{equation}
  x = \begin{pmatrix} 0&0&0\\ & 1 &0\\ &&-\ell\end{pmatrix} \begin{pmatrix} 0 &0 &\frac{1}{2}\\ &0 &0\\&&0\end{pmatrix}, \qquad \xi = \begin{pmatrix} 1 &0 &0\\ &0&0\\ &&0\end{pmatrix} \begin{pmatrix}0&0&0\\&0&0\\&&0\end{pmatrix}
 \end{equation}
 and 
 \begin{equation}
  V_x = \begin{pmatrix} 0 & * & *\\ & * & *\\ & & *\end{pmatrix} \begin{pmatrix} * & * & *\\ & * & *\\ && *\end{pmatrix}, \qquad V_\xi = \begin{pmatrix} * & * & * \\ & 0 & 0 \\ && 0\end{pmatrix}\begin{pmatrix} * & 0 & 0\\ & 0 &0\\ && 0\end{pmatrix}
 \end{equation}
the action set is, as for 1., 
\begin{equation}
 G_{x, \xi}^t = \begin{pmatrix} * & * & * \\ & * & *\\ & * & *\end{pmatrix} \begin{pmatrix} * & *\\ & *\end{pmatrix}, \qquad |G_{x, \xi}| = (p-1)^5 p^4 (p+1).
\end{equation}
 \item [3.] Case $(\sO_{2^2}, \sO_{D11})$: For standard representatives
 \begin{equation}
  x = \begin{pmatrix} 0 & 0 &0\\ &0&0\\ &&1\end{pmatrix}\begin{pmatrix}1 &0&0\\ &-\ell &0\\&&0\end{pmatrix}, \qquad \xi = \begin{pmatrix} 0 & 1 &0\\ &0&0\\&&0\end{pmatrix}\begin{pmatrix} 0&0&0\\ &0&0\\&&0\end{pmatrix}
 \end{equation}
and
\begin{equation}
 V_x = \begin{pmatrix} z &0&*\\&-z \ell  &*\\ &&*\end{pmatrix}\begin{pmatrix} * & * & *\\ & * & *\\ && *\end{pmatrix}, \qquad V_\xi= \begin{pmatrix} * & * & *\\ & * & *\\ && 0\end{pmatrix} \begin{pmatrix} 0 & * & 0 \\ & 0 & 0\\ && 0\end{pmatrix}.
\end{equation}
 Since the second form is eliminated by $V_x$, in the action on $\xi$, it may be assumed that the $\GL_2$ action acts on the second quadratic form by an arbitrary scalar.  
 
 The action on $\xi$ is thus given by {\tiny
 \begin{align*}
 &\left(\begin{pmatrix} a & b & a_{13}\\ c & d & a_{23}\\ a_{31} & a_{32} & a_{33}\end{pmatrix},
  \begin{pmatrix} b_{11} & b_{12} \\  & b_{22}\end{pmatrix}\right) \cdot \xi  =  \begin{pmatrix} 2b_{11}ab  & b_{11}(ad + bc) & b_{11}(aa_{32} + ba_{31})\\ & 2b_{11}cd  & b_{11}(ca_{32} + da_{31})\\ && 2b_{11}a_{31}a_{32}  \end{pmatrix} \begin{pmatrix} 0 & 0 & 0\\ & 0 & 0\\ && 0 \end{pmatrix}.
 \end{align*}
} 
This forces $ a_{31} = a_{32} = 0$ and $ab = cd\ell $, so that
\begin{equation}
 G_{x, \xi}^t = \begin{pmatrix} a & b & *\\ c & d & * \\ && *\end{pmatrix} \begin{pmatrix} * & *\\ & *\end{pmatrix}, \qquad ab = cd\ell .
\end{equation}
There are $(p-1)^3$ choices of $a,b,c,d$ with $ab \neq 0$ and $2(p-1)^2$ choices with $ab = 0$, so $|G_{x,\xi}| = (p-1)^5 p^3 (p+1)$.  The two parametrizations are equivalent since the cardinalities are the same.

The action set may be parametrized,
\begin{equation}
 \begin{pmatrix} a & c\lambda \ell  & a_{13}\\ c &a\lambda   & a_{23}\\  &  & a_{33}\end{pmatrix},
  \begin{pmatrix} b_{11} & b_{12} \\  & b_{22}\end{pmatrix}, \qquad \lambda \in \bF_p^\times, (a,c) \in \bF_p^2 \setminus \{(0,0)\}.
\end{equation}

\item [4.] Case $(\sO_{2^2}, \sO_{D2})$: For standard representatives
\begin{equation}
 x = \begin{pmatrix} 0 & 0 &0 \\ &0 &0\\ &&1\end{pmatrix}\begin{pmatrix} 1 &0 &0\\ & - \ell & 0\\ &&0\end{pmatrix}, \qquad \xi = \begin{pmatrix} \ell & \beta & 0\\ & 1 &0\\ && 0\end{pmatrix} \begin{pmatrix} 0&0&0\\ &0&0\\ &&0\end{pmatrix},
\end{equation}
with $\ell u^2 + 2\beta uv + v^2$ irreducible the annihilated spaces are
\begin{equation}
 V_x = \begin{pmatrix} z & 0 & *\\ & -z\ell  & *\\ & & *\end{pmatrix}\begin{pmatrix} * & * & *\\ & * & *\\ & & *\end{pmatrix}, \qquad V_\xi = \begin{pmatrix} * & * & *\\ & * & *\\ && 0\end{pmatrix} \begin{pmatrix} z \ell & z \beta & 0\\ & z & 0\\ && 0\end{pmatrix}.
\end{equation}
  Since the second form is eliminated by $V_x$, in the action on $\xi$, it may be assumed that the $\GL_2$ action acts on the second quadratic form by an arbitrary scalar.  
 
 The action on $\xi$ is thus given by {\tiny
 \begin{align*}
 &\left(\begin{pmatrix} a & b & a_{13}\\ c & d & a_{23}\\ a_{31} & a_{32} & a_{33}\end{pmatrix},
  \begin{pmatrix} b_{11} & b_{12} \\  & b_{22}\end{pmatrix}\right) \cdot \xi \\&\notag =  \begin{pmatrix} b_{11}(a^2\ell + 2ab\beta + b^2)  & b_{11}(ac\ell + ad\beta + bc\beta + bd) & b_{11}(a a_{31}\ell + a a_{32}\beta + ba_{31}\beta + ba_{32})\\ & b_{11}(c^2\ell + 2cd\beta + d^2)  & b_{11}(ca_{31}\ell + ca_{32}\beta + da_{31}\beta + da_{32})\\ && b_{11}(a_{31}^2\ell + 2a_{31}a_{32}\beta + a_{32}^2)  \end{pmatrix} \begin{pmatrix} 0& 0 & 0\\ & 0 &  0\\ && 0 \end{pmatrix}.
 \end{align*}
} 
Since the lower right entry in the first quadratic form is a quadratic irreducible on $a_{31}$ and $a_{32}$ which must vanish, this forces $ a_{31}=a_{32}= 0$.
The action set is thus
\begin{equation}
 G_{x,\xi}^t = \begin{pmatrix} a & b & * \\ c & d & * \\ & & *\end{pmatrix} \begin{pmatrix} * & *\\ & *\end{pmatrix}, \qquad \left[ \begin{pmatrix} a & b\\ c & d\end{pmatrix} \cdot (\ell u^2 + 2\beta uv + v^2), u^2 - \ell v^2 \right] = 0.
\end{equation}

To count $G_{x,\xi}$, first note that the number of quadratic irreducibles in $\sym^2(\bF_p)$ is $\frac{(p-1)^2p}{2}$.  The number of irreducibles which are orthogonal to $u^2 - \ell v^2$ is
\begin{align}
 &\#\{(\lambda, \beta): \lambda (\ell u^2 + v^2) + 2 \beta uv \text{ irreducible}\}\\
 \notag &= \#\{(\lambda, \beta): 4\beta^2 - 4\lambda^2 \ell \neq \square\}.
\end{align}
As $\lambda$ runs through $\zed/p\zed$, together $4 \lambda^2 \ell$ and $4 \lambda^2$ run through $\zed/p\zed$ twice.  Thus
\begin{equation}
  \#\{(\lambda, \beta): 4\beta^2 - 4\lambda^2 \ell \neq \square\} + \#\{(\lambda, \beta): 4\beta^2 - 4\lambda^2 \neq \square\} = (p-1)p.
\end{equation}
Since
\begin{align}
 \#\{(\lambda, \beta): 4\beta^2 - 4\lambda^2 \neq \square\} & =\#\{(a,b): ab \neq \square\} = \frac{(p-1)^2}{2},
\end{align}
it follows that
\begin{equation}
 \#\{(\lambda, \beta): 4\beta^2 - 4\lambda^2 \ell \neq \square\} =\frac{(p-1)(p+1)}{2}.
\end{equation}

Since the quadratic irreducibles are a single $\GL_2 \times \GL_1$ orbit in $\sym^2(\bF_p)$, which are uniformly covered by the action of the group on a representative, it follows that
\begin{align}
 &\#\left\{g \in \GL_2(\zed/p\zed): \left[ g \cdot (\ell u^2 + 2 \beta uv + v^2), u^2 - \ell v^2\right] = 0 \right\}\\ \notag &= \frac{\frac{(p-1)(p+1)}{2}}{\frac{(p-1)^2 p}{2}} (p^2-p)(p^2-1) \\ \notag &= (p-1)(p+1)^2.
\end{align}
Thus $|G_{x,\xi}| = (p-1)^4 p^3 (p+1)^2.$
\item [5.] Case $(\sO_{1^2 1^2}, \sO_{D1^2})$: For standard representatives
\begin{equation}
 x = \begin{pmatrix} 0&0&0\\ &0&0\\ &&1\end{pmatrix}\begin{pmatrix} 0 & \frac{1}{2} & 0\\ &0&0\\&&0\end{pmatrix}, \qquad \xi = \begin{pmatrix} 1  & 0 &0\\ &0&0\\&&0\end{pmatrix}\begin{pmatrix} 0 & 0 & 0\\ &0&0\\&&0\end{pmatrix}
\end{equation}
and
\begin{equation}
 V_x = \begin{pmatrix} 0 & * & *\\ & 0 & *\\ & & *\end{pmatrix}\begin{pmatrix} * & * & *\\ & * & *\\ && *\end{pmatrix}, \qquad V_\xi = \begin{pmatrix} * & * & *\\ & 0 & 0\\ && 0\end{pmatrix} \begin{pmatrix} * & 0 & 0\\ & 0 &0\\ &&0\end{pmatrix}.
\end{equation}
Since the second form is eliminated by $V_x$, in the action on $\xi$, it may be assumed that the $\GL_2$ action acts on the second quadratic form by an arbitrary scalar.  
 
 The action on $\xi$ is thus given by  {\tiny
 \begin{align*}
 &\left(\begin{pmatrix} a_{11} & a_{12} & a_{13}\\ a_{21} & a_{22} & a_{23}\\ a_{31} & a_{32} & a_{33}\end{pmatrix},
  \begin{pmatrix} b_{11} & b_{12} \\  & b_{22}\end{pmatrix}\right) \cdot \xi =  \begin{pmatrix} b_{11}a_{11}^2 & b_{11}a_{11}a_{21} & b_{11}a_{11}a_{31}\\ & b_{11}a_{21}^2  & b_{11}a_{21}a_{31}\\ && b_{11}a_{31}^2  \end{pmatrix} \begin{pmatrix} 0& 0 & 0\\ & 0 & 0\\ && 0 \end{pmatrix}.
 \end{align*}
} 
This forces $ a_{31} = a_{11}a_{21}=0$.
The action set is 
\begin{equation}
 G_{x,\xi}^t = \begin{pmatrix} * & * & *\\  & * & *\\  & * & *\end{pmatrix}\begin{pmatrix} * & *\\  & * \end{pmatrix} \sqcup \begin{pmatrix}  & * & *\\ * & * & *\\  & * & *\end{pmatrix} \begin{pmatrix} * & *\\ & *\end{pmatrix}, \qquad |G_{x, \xi}| = 2 (p-1)^5 p^4 (p+1).
\end{equation}

\item[6.] Case $(\sO_{1^21^2}, \sO_{D11})$: For standard representatives
\begin{equation}
 x = \begin{pmatrix} 0 & 0 & 0\\ &0 &0\\ &&1\end{pmatrix} \begin{pmatrix} 0 & \frac{1}{2} & 0\\ &0 &0\\ && 0\end{pmatrix}, \qquad \xi = \begin{pmatrix} 1 & 0 & 0\\ & -1 & 0\\ && 0\end{pmatrix}\begin{pmatrix} 0 & 0 & 0\\ &0 &0\\ &&0\end{pmatrix}
\end{equation}
and
\begin{equation}
 V_x = \begin{pmatrix} 0 & * & *\\ & 0 & *\\ && *\end{pmatrix}\begin{pmatrix} * & * & *\\ & * & *\\ && *\end{pmatrix}, \qquad V_\xi = \begin{pmatrix} * & * & *\\ & * & *\\ && 0\end{pmatrix}\begin{pmatrix} z & 0 &0\\ & -z & 0 \\ && 0\end{pmatrix}.
\end{equation}
Since the second form is eliminated by $V_x$, in the action on $\xi$, it may be assumed that the $\GL_2$ action acts on the second quadratic form by an arbitrary scalar.  
 
 The action on $\xi$ is thus given by {\tiny
 \begin{align*}
 &\left(\begin{pmatrix} a_{11} & a_{12} & a_{13}\\ a_{21} & a_{22} & a_{23}\\ a_{31} & a_{32} & a_{33}\end{pmatrix},
  \begin{pmatrix} b_{11} & b_{12} \\  & b_{22}\end{pmatrix}\right) \cdot \xi \\&\notag =  \begin{pmatrix} b_{11}(a_{11}^2 -a_{12}^2) & b_{11}(a_{11}a_{21} - a_{12}a_{22}) & b_{11}(a_{11}a_{31}-a_{12}a_{32})\\ & b_{11}(a_{21}^2-a_{22}^2)  & b_{11}(a_{21}a_{31}-a_{22}a_{32})\\ && b_{11}(a_{31}^2-a_{32}^2)  \end{pmatrix}
  \begin{pmatrix} 0 & 0 & 0\\ & 0 & 0\\ && 0 \end{pmatrix}.
 \end{align*}
} 
Since the first two columns of the $\GL_3$ matrix are linearly independent, vanishing of the $w$ dependence in the first quadratic form forces $a_{31}=a_{32}=0$.
The action set maps $u^2 - v^2 = (u+v)(u-v)$ to a form $c_1u^2 + c_2 v^2 = c_1(u+\gamma v)(u - \gamma v)$ and thus satisfies
\begin{equation}
 G_{x,\xi}^t \subset \begin{pmatrix} * & * & *\\ * & * & *\\ && *\end{pmatrix} \begin{pmatrix} * & *\\ & *\end{pmatrix}, \begin{array}{l} u + v \mapsto \alpha u + \beta v\\ u-v \mapsto \lambda(\alpha u - \beta v)\end{array}, \alpha, \beta, \lambda \in (\zed/p\zed)^\times.
\end{equation}
The size is $|G_{x,\xi}| = (p-1)^6p^3.$

\item[7.] Case $(\sO_{1^21^2}, \sO_{D2})$: For standard representatives
\begin{equation}
 x = \begin{pmatrix} 0 & 0 & 0\\ & 0 &0\\ && 1\end{pmatrix}\begin{pmatrix}0 & \frac{1}{2} & 0\\ & 0 & 0\\ && 0\end{pmatrix}, \qquad \xi = \begin{pmatrix}-\ell &0 &0\\ & 1 & 0\\ & &0\end{pmatrix}\begin{pmatrix}0 &0 &0\\ &0 &0\\ & & 0\end{pmatrix}
\end{equation}
and
\begin{equation}
 V_x = \begin{pmatrix}0 & * & *\\ & 0 & *\\ & & *\end{pmatrix}\begin{pmatrix} * & * & *\\ & * & *\\ && *\end{pmatrix}, \qquad V_\xi = \begin{pmatrix} * & * & *\\ & * & *\\ & & 0\end{pmatrix}\begin{pmatrix} - z\ell & 0 & 0\\ &  z & 0 \\ &&0\end{pmatrix}.
\end{equation}
Since the second form is eliminated by $V_x$, in the action on $\xi$, it may be assumed that the $\GL_2$ action acts on the second quadratic form by an arbitrary scalar.  
 
 The action on $\xi$ is thus given by {\tiny
 \begin{align*}
 &\left(\begin{pmatrix} a & b & a_{13}\\ c & d & a_{23}\\ a_{31} & a_{32} & a_{33}\end{pmatrix},
  \begin{pmatrix} b_{11} & b_{12} \\ & b_{22}\end{pmatrix}\right) \cdot \xi \\&\notag =  \begin{pmatrix} b_{11}(-a^2\ell +b^2) & b_{11}(-ac\ell + bd) & b_{11}(-aa_{31}\ell+ba_{32})\\ & b_{11}(-c^2\ell+d^2)  & b_{11}(-ca_{31}\ell+da_{32})\\ && b_{11}(-a_{31}^2\ell+a_{32}^2)  \end{pmatrix} 
  \begin{pmatrix} 0 & 0 & 0\\ & 0 & 0\\ && 0 \end{pmatrix}.
 \end{align*}
} 
This forces $a_{31}=a_{32}=0$.
The action set is
\begin{equation}
 G_{x,\xi}^t = \begin{pmatrix} a & b& * \\ c & d & *\\ & & *\end{pmatrix} \begin{pmatrix} * & *\\ & * \end{pmatrix}, \qquad  ac\ell =  bd.
\end{equation}
As for 3., $|G_{x, \xi}| = (p-1)^5 p^3 (p+1),$
and the action set has parametrization,
\begin{equation}
 \begin{pmatrix} a & c\ell& * \\c \lambda &a \lambda & *\\ & & *\end{pmatrix} \begin{pmatrix} * & *\\ & * \end{pmatrix}, \qquad \lambda \in \bF_p^\times, (a,c) \in \bF_p^2 \setminus \{(0,0)\}.
\end{equation}

\item[8.] Case $(\sO_{1^31}, \sO_{D1^2})$: For standard representatives
\begin{equation}
 x = \begin{pmatrix} 0 & 0 & 0\\ & 0 & \frac{1}{2} \\ & & 0\end{pmatrix}\begin{pmatrix} 0 & 0 &\frac{1}{2}\\ & 1 & 0\\ & & 0\end{pmatrix}, \qquad \xi = \begin{pmatrix} 1 & 0 & 0\\ & 0 &0\\ & & 0\end{pmatrix}\begin{pmatrix} 0 & 0 & 0\\ & 0 &0\\ && 0\end{pmatrix},
\end{equation}
and
\begin{equation}
 V_x = \begin{pmatrix} 0 & \frac{z}{2} & * \\ & * & *\\ && *\end{pmatrix}\begin{pmatrix} z & * & *\\ & * & * \\ & & *\end{pmatrix}, \qquad V_\xi = \begin{pmatrix} * & * & *\\ & 0 & 0 \\ & & 0\end{pmatrix}\begin{pmatrix} * & 0 & 0\\ & 0&0\\ &&0\end{pmatrix}.
\end{equation}
Under the action on $\xi$ the two forms are proportional, and this can be orthogonal to $V_x$ only if the second form is 0, and thus it may be assumed that the second form is acted on by a scalar under the $\GL_2$ action.

The action on $\xi$ is thus given by {\tiny
 \begin{equation*}
 \left(\begin{pmatrix} a_{11} & a_{12} & a_{13}\\ a_{21} & a_{22} & a_{23}\\ a_{31} & a_{32} & a_{33}\end{pmatrix},
  \begin{pmatrix} b_{11} & b_{12} \\  & b_{22}\end{pmatrix}\right) \cdot \xi =  \begin{pmatrix} b_{11}a_{11}^2 & b_{11}a_{11}a_{21} & b_{11}a_{11}a_{31}\\ & b_{11}a_{21}^2  & b_{11}a_{21}a_{31}\\ && b_{11}a_{31}^2  \end{pmatrix} \begin{pmatrix} 0 & 0 & 0\\ & 0 & 0\\ && 0 \end{pmatrix}.
 \end{equation*}
} The $v^2$ and $w^2$ coefficients force $a_{21} = a_{31} = 0$, so the action set is
\begin{equation}
 G_{x,\xi}^t = \begin{pmatrix} * & * & *\\ & * & *\\ & * & *\end{pmatrix}\begin{pmatrix} * & *\\ & *\end{pmatrix}, \qquad |G_{x, \xi}| = (p-1)^5 p^4 (p+1).
\end{equation}

\item[9.] Case $(\sO_{1^31}, \sO_{Cs})$: For standard representatives
\begin{equation}
 x = \begin{pmatrix} 0 & 0& 0\\ & 0 & \frac{1}{2}\\ & & 0\end{pmatrix}\begin{pmatrix}0 & 0 & \frac{1}{2} \\ &1 & 0\\ & & 0\end{pmatrix}, \qquad \xi = \begin{pmatrix} 0 & -1 & 0\\ &0 &0\\ & & 0\end{pmatrix}\begin{pmatrix} 1 & 0 &0\\ & 0 &0\\ & & 0\end{pmatrix}
\end{equation}
and
\begin{equation}
 V_x = \begin{pmatrix} 0 & \frac{z}{2} & *\\ & * & *\\ & & *\end{pmatrix}\begin{pmatrix} z & * & *\\ & * & *\\ & & *\end{pmatrix}, \qquad V_\xi = \begin{pmatrix} * & * & *\\ & * & -z \\ & & 0\end{pmatrix}\begin{pmatrix} * & * & z\\ & 0 & 0\\& & 0\end{pmatrix}
\end{equation}
Since $V_x$ annihilates the dependence on $w$ in both forms, the action on $\xi$ has no $w$ dependence in both forms, and thus this remains true of any linear combination of the forms.  Thus the action by $\GL_2$ can be inverted with this still the case, so that the $\GL_3$ action must carry each form to a form independent of $w$.  From the second form, $u^2$, it follows that $u$ is mapped by the $\GL_3$ action to a linear form not containing $w$.  Now from the first form, $2uv$, it follows that $v$ is mapped to a linear form not containing $w$.

 The action on $\xi$ is given by {\tiny
 \begin{align*}
 &\left(\begin{pmatrix} a_{11} & a_{12} & a_{13}\\ a_{21} & a_{22} & a_{23}\\  &  & a_{33}\end{pmatrix},
  \begin{pmatrix} b_{11} & b_{12} \\ b_{21} & b_{22}\end{pmatrix}\right) \cdot \xi\\&\notag =  
  \begin{pmatrix} -2b_{11}a_{11}a_{12} + b_{12}a_{11}^2 & b_{11}(-a_{11}a_{22}-a_{12}a_{21}) + b_{12}a_{11}a_{21} & 0\\ & -2b_{11}a_{21}a_{22}  + b_{12}a_{21}^2 & 0\\ && 0 \end{pmatrix} \\ \notag &
  \begin{pmatrix} -2b_{21}a_{11}a_{12} + b_{22}a_{11}^2 & b_{21}(-a_{11}a_{22}-a_{12}a_{21}) + b_{22}a_{11}a_{21} & 0\\ & -2b_{21}a_{21}a_{22}  + b_{22}a_{21}^2 & 0\\ && 0 \end{pmatrix} .
 \end{align*}
}
 Considering the $v^2$ coefficients gives
\begin{equation}
 a_{21}(-2b_{11}a_{22} + b_{12}a_{21}) = 0, \qquad a_{21}(-2b_{21}a_{22} + b_{22}a_{21}) = 0.
\end{equation}
This forces $a_{21} = 0$, since otherwise either $a_{21} = a_{22} = 0$ or the matrix $\begin{pmatrix} b_{11} & b_{12} \\ b_{21} & b_{22}\end{pmatrix}$ is singular, both of which are impossible.  It now follows from the $uv$ coefficient in the second quadratic form that $-a_{11}a_{22}b_{21} = 0$.  Since $a_{21} = 0$, the matrix on $\GL_3$ would be singular if either $a_{11}= 0$ or $a_{22} = 0$, so that $b_{21} = 0$.

The action has now been simplified to
{\tiny
 \begin{align*}
 &\left(\begin{pmatrix} a & a_{12} & a_{13}\\  & b & a_{23}\\  &  & a_{33}\end{pmatrix},
  \begin{pmatrix} c & b_{12} \\  & d\end{pmatrix}\right) \cdot \xi =  
  \begin{pmatrix} a^2 b_{12} - 2aa_{12}c & -abc & 0\\ & 0 & 0\\ && 0 \end{pmatrix} 
  \begin{pmatrix} a^2d & 0 & 0\\ & 0 & 0\\ && 0 \end{pmatrix} .
 \end{align*}
}
Thus $abc = a^2 d$ so that 
the action set is
\begin{equation}
 G_{x,\xi}^t = \begin{pmatrix} a & * & *\\ & b & * \\ & & *\end{pmatrix}\begin{pmatrix} c & *\\ & \frac{bc}{a}\end{pmatrix}, \qquad |G_{x,\xi}| = (p-1)^4 p^4.
\end{equation}
\item[10.] Case $(\sO_{1^4}, \sO_{D1^2})$: For standard representatives
\begin{equation}
 x = \begin{pmatrix} 0 & 0 & 0\\ & 0 & 0\\ & & 1\end{pmatrix}\begin{pmatrix} 0 & 0 & \frac{1}{2}\\ & 1 & 0 \\ & & 0\end{pmatrix}, \qquad \xi = \begin{pmatrix}1 & 0 & 0\\ & 0 & 0\\ & & 0\end{pmatrix}\begin{pmatrix} 0 & 0 & 0\\ & 0 & 0\\ & & 0\end{pmatrix}
\end{equation}
and
\begin{equation}
 V_x = \begin{pmatrix} 0 & 0 & y + \frac{z}{2}\\ & z & * \\ & & *\end{pmatrix}\begin{pmatrix} y & * & *\\ & * & * \\ & & *\end{pmatrix}, \qquad V_\xi = \begin{pmatrix} * & * & * \\ & 0 & 0\\ && 0\end{pmatrix}\begin{pmatrix} * & 0 & 0\\ & 0 & 0\\ & & 0\end{pmatrix}.
\end{equation}
Since the two forms in the action on $\xi$ are linearly dependent, to be orthogonal to $V_x$ it is necessary that the second form is 0, and thus it may be assumed that the second form is acted on by a scalar under the $\GL_2$ action.

The action on $\xi$ is thus given by {\tiny
 \begin{align*}
 &\left(\begin{pmatrix} a_{11} & a_{12} & a_{13}\\ a_{21} & a_{22} & a_{23}\\ a_{31} & a_{32} & a_{33}\end{pmatrix},
  \begin{pmatrix} b_{11} & b_{12} \\ & b_{22}\end{pmatrix}\right) \cdot \xi =  \begin{pmatrix} b_{11}a_{11}^2 & b_{11}a_{11}a_{21} & b_{11}a_{11}a_{31}\\ & b_{11}a_{21}^2  & b_{11}a_{21}a_{31}\\ && b_{11}a_{31}^2  \end{pmatrix} \begin{pmatrix} 0 & 0 & 0\\ & 0 & 0\\ && 0 \end{pmatrix}.
 \end{align*}
} 
From the $w^2$ coefficient it follows that $a_{31} = 0$.  It now follows from the $v^2$ coefficient that $a_{21}^2 = 0$.
The action set is
\begin{equation}
 G_{x,\xi}^t = \begin{pmatrix} * & * & *\\ & * & *\\ & * & *\end{pmatrix}\begin{pmatrix} * & * \\ & *\end{pmatrix}, \qquad |G_{x,\xi}| = (p-1)^5p^4 (p+1).
\end{equation}

\item[11.] Case $(\sO_{1^4}, \sO_{D11})$: For standard representatives
\begin{equation}
 x = \begin{pmatrix} 0 & 0 & 0\\ & 0 &0\\ & & 1\end{pmatrix}\begin{pmatrix} 0 & 0 & \frac{1}{2}\\ &1&0\\&&0\end{pmatrix}, \qquad \xi = \begin{pmatrix}0 & 1 & 0\\ & 0 &0\\& & 0\end{pmatrix}\begin{pmatrix}0 & 0 &0\\ &0&0\\ &&0\end{pmatrix},
\end{equation}
and
\begin{equation}
 V_x = \begin{pmatrix} 0 & 0 & y + \frac{z}{2}\\ & z & *\\ && *\end{pmatrix}\begin{pmatrix} y & * & *\\ & * & *\\ & & *\end{pmatrix}, \qquad V_\xi = \begin{pmatrix} * & * & *\\ & * & *\\ &&0\end{pmatrix}\begin{pmatrix} 0 & * & 0\\ &0&0\\ &&0\end{pmatrix}.
\end{equation}
Since the action on $\xi$ results in two forms which are linearly dependent, to be orthogonal to $V_x$ the second form is 0, and thus it may be assumed that the second form is acted on by a scalar under the $\GL_2$ action.

The action on $\xi$ is thus given by {\tiny
 \begin{align*}
 &\left(\begin{pmatrix} a_{11} & a_{12} & a_{13}\\ a_{21} & a_{22} & a_{23}\\ a_{31} & a_{32} & a_{33}\end{pmatrix},
  \begin{pmatrix} b_{11} & b_{12} \\  & b_{22}\end{pmatrix}\right) \cdot \xi \\&\notag =  \begin{pmatrix} 2b_{11}a_{11}a_{12}  & b_{11}(a_{11}a_{22} + a_{12}a_{21}) & b_{11}(a_{11}a_{32} + a_{12}a_{31})\\ & 2b_{11}a_{21}a_{22}  & b_{11}(a_{21}a_{32} + a_{22}a_{31})\\ && 2b_{11}a_{31}a_{32}  \end{pmatrix}
  \begin{pmatrix} 0 &0 & 0\\ & 0 &0\\ && 0 \end{pmatrix}.
 \end{align*}
}  It follows from $w^2$ coefficient that one of $a_{31}$ and $a_{32}$ is 0.  Now considering the $uw$, $vw$  coordinates, both $a_{31}$ and $a_{32}$ are zero, or else the $\GL_3$ matrix would have a 0 column.  It follows from the $v^2$ coefficient that $a_{21}a_{22} = 0$, either of which can hold.
The action set is thus 
\begin{equation}
 G_{x,\xi}^t = \begin{pmatrix} * & * & *\\ & * & *\\ && *\end{pmatrix}\begin{pmatrix} * & *\\ & *\end{pmatrix} \sqcup \begin{pmatrix} * & * & *\\ * & & *\\ && *\end{pmatrix}\begin{pmatrix} * & *\\ & *\end{pmatrix},  \qquad |G_{x,\xi}| = 2 (p-1)^5 p^4.
\end{equation}

\item [12.] Case $(\sO_{1^4}, \sO_{1^4})$: For standard representatives
\begin{equation}
 x = \begin{pmatrix} 0 & 0 & 0\\ & 0 & 0\\ & & 1\end{pmatrix}\begin{pmatrix} 0 & 0 & \frac{1}{2}\\ & 1 &0\\ & & 0\end{pmatrix}, \qquad \xi = \begin{pmatrix} 0 & 0 & -1\\ &1 & 0\\ & & 0\end{pmatrix}\begin{pmatrix}2 & 0 &0\\&0 &0\\&&0\end{pmatrix},
\end{equation}
and
\begin{equation}
V_x= \begin{pmatrix} 0 & 0 & y + \frac{z}{2}\\ & z & * \\ & & *\end{pmatrix}\begin{pmatrix} y & * & *\\ & * & *\\ & & *\end{pmatrix}, \qquad V_\xi = \begin{pmatrix} * & * & *\\ & * & *\\ & & y\end{pmatrix} \begin{pmatrix} * & * & -z + y\\ & z &0\\ & & 0\end{pmatrix},
\end{equation}
any combination $\lambda_1 (v^2 - 2uw) + 2\lambda_2 u^2$ with $\lambda_1 \neq 0$ cannot be a double line, and hence $b_{21} = 0$.  By considering the $v^2$ and $w^2$ coefficients in the second quadratic form, it follows that $a_{21} = a_{31}= 0$. 
The action on $\xi$ is thus given by {\tiny
 \begin{align*}
 &\left(\begin{pmatrix} a_{11} & a_{12} & a_{13}\\ & a_{22} & a_{23}\\ & a_{32} & a_{33}\end{pmatrix},
  \begin{pmatrix} b_{11} & b_{12} \\  & b_{22}\end{pmatrix}\right) \cdot \xi \\&\notag =  \begin{pmatrix} b_{11}(a_{12}^2 - 2a_{11}a_{13})  + 2b_{12}a_{11}^2  & b_{11}(-a_{11}a_{23} +a_{12}a_{22})  & b_{11}(-a_{11}a_{33}+a_{12}a_{32})\\ & b_{11}a_{22}^2  & b_{11}a_{22}a_{32}\\ &&  b_{11}a_{32}^2  \end{pmatrix}
  \begin{pmatrix}  2b_{22}a_{11}^2  & 0 & 0\\ & 0 & 0\\ && 0 \end{pmatrix}.
 \end{align*}
} 
Considering the $w^2$ coefficient of the first quadratic form implies $a_{32} = 0$. The action thus reduces to
\begin{align}
 &\left(\begin{pmatrix} a & * & *\\ & b & *\\ && c\end{pmatrix}, \begin{pmatrix} d & *\\ & e\end{pmatrix} \right) \cdot \xi \\\notag &= \begin{pmatrix} 2a^2 b_{12} -2aa_{13}d + a_{12}^2d & -aa_{23}d + a_{12}bd & -acd\\ & b^2 d &0\\&&0\end{pmatrix}\begin{pmatrix} 2a^2e &0&0\\&0&0\\&&0\end{pmatrix}.
\end{align}
The action set is thus
\begin{equation}
 G_{x,\xi}^t = \begin{pmatrix} a & * & *\\ & b & *\\ && c\end{pmatrix} \begin{pmatrix} d & *\\ & e\end{pmatrix}, \qquad b^2 = ac, b^2d = a^2 e, \qquad |G_{x,\xi}| = (p-1)^3p^4.
\end{equation}

\end{enumerate}

\bibliographystyle{plain}

\end{document}